\documentclass[11pt,a4paper]{article}

\pdfoutput=1 % required by arXiv. Must be in the first 5 lines.

\usepackage[draft]{dl-en}
\title{\color{DLtitle}\rmfamily\spacedallcaps{homotopy theory of linear coalgebras}}
\author{\spacedlowsmallcaps{brice le grignou}\and\spacedlowsmallcaps{damien lejay}}
\date{} % pas de date

\tikzcdset{%
	arrow style=tikz,
	>/.tip={Stealth[length=3pt, width=5pt, inset=1.8pt]}
}

\newcommand{\addresses}{{% additional braces for segregating \footnotesize
	\bigskip
	\footnotesize
	
	\textit{E-mail address}, Brice Le Grignou: \textsf{bricelegrignou@gmail.com}
	
	\medskip
	
	\textit{E-mail address}, Damien Lejay: \textsf{lejay@paracompact.space}
}}
\hypersetup{%
	pdftitle={Homotopy theory of linear cogebras},
	pdfauthor={\textcopyright%
		Brice Le Grignou and Damien Lejay},
	pdfsubject={},
	pdfkeywords={},
	pdfcreator={pdfLaTeX},
	pdfproducer={LaTeX with hyperref}
}
\newcommand{\tocconfig}{%
	\addtocontents{toc}{\protect\vspace{\beforebibskip}}
	\addcontentsline{toc}{section}{Acknowledgements}
	\addcontentsline{toc}{section}{\refname}
}
\DeclareFontFamily{U}{wncy}{}
\DeclareFontShape{U}{wncy}{m}{n}{<->wncyr10}{}
\DeclareSymbolFont{mcy}{U}{wncy}{m}{n}
\DeclareMathSymbol{\shaletter}{\mathord}{mcy}{"58} 
\usepackage{scalerel}
\newcommand{\sha}{\operatorname{\scaleobj{0.85}{\shaletter}}}

\newcommand{\leftadjointmark}{\square}
\newcommand{\rightadjointmark}{\triangle}
\newcommand{\ordinalomega}{\omegaup}
\newcommand{\operad}{P}
\newcommand{\coperad}{Q}

\renewcommand{\op}{^\mathrm{op}}
\newcommand{\einfinity}{\ensuremath{\mathsf{E}_∞}}
\newcommand{\ainfinity}{\ensuremath{\mathsf{A}_∞}}
\newcommand{\linfinity}{\ensuremath{\mathsf{L}_∞}}

\newcommand{\diagonal}{\Deltaup}
\newcommand{\homology}{\operatorname{H}}
\newcommand{\bvinfinity}{\ensuremath{\mathsf{BV}_∞}}
\newcommand{\bdinfinity}{\ensuremath{\mathsf{BD}_∞}}
\newcommand{\coend}[1]{\operatorname{Coend}_{#1}}
\newcommand{\relhom}[3]{\operatorname{Hom}_{#1}\mleft(#2,#3\mright)}

\newcommand{\symmoncat}{\mathcal{C}}
\newcommand{\fieldk}{{\mathbf{K}}}

\newcommand{\barfunctor}{\ensuremath{\operatorname{B}}}
\newcommand{\baradj}{\ensuremath{\operatorname{B^\leftadjointmark}}}

\newcommand{\cobaradj}{\ensuremath{\operatorname{C^\rightadjointmark}}}

\newcommand{\cobarfunctor}{\ensuremath{\operatorname{C}}}
\newcommand{\barfunctorfull}{\ensuremath{\operatorname{Bar}}}
\newcommand{\cobarfunctorfull}{\ensuremath{\operatorname{Cobar}}}
\newcommand{\cobaradjfull}{\ensuremath{\operatorname{Cobar}^\rightadjointmark}}
\newcommand{\leftadjoint}[1]{#1^\leftadjointmark}
\newcommand{\rightadjoint}[1]{#1^\rightadjointmark}
\newcommand{\catofalg}[1]{#1\text{-}\mathsf{alg}}
\newcommand{\catofcog}[1]{#1\text{-}\mathsf{cog}}
\newcommand{\catofpermutations}{\mathbf{S}}
\newcommand{\subobject}{\subset}
\newcommand{\freecog}[1]{\operatorname{L}^{#1}}
\newcommand{\freealg}[1]{\operatorname{L}_{#1}}
\newcommand{\mathoverflow}{\textsf{mathoverflow.net}}

\newcommand{\catofmod}[1]{#1\text{-}\mathsf{mod}}

\newcommand{\linearlogos}{\mathcal{L}}
\newcommand{\vect}[1]{\mathsf{Vect}_{#1}}
\newcommand{\treemodule}{\operatorname{\mathsf{T}}}
\newcommand{\planartreemodule}{\operatorname{\mathsf{T}}_\mathrm{pl}}
\newcommand{\treemoduleheight}[1]{\operatorname{\mathsf{T}}^{\leq #1}}

\newcommand{\planartreemoduleheight}[1]{\operatorname{\mathsf{T}}_{\mathrm{pl},\mathrm{h} \leq #1}}

\newcommand{\symgroup}[1]{\mathbf{S}_{#1}}
\newcommand{\symseq}[1]{#1^{\catofpermutations\op}}

\newcommand{\catofop}{\operatorname{\mathsf{Op}}}

\newcommand{\catofcop}{\operatorname{\mathsf{Coop}}}
\newcommand{\catofmon}{\operatorname{\mathsf{Mon}}}
\newcommand{\catofcomon}{\operatorname{\mathsf{Comon}}}

\newcommand{\comonadl}{\mathbf{L}}
\newcommand{\monadl}{\mathbf{L}}
\newcommand{\freefunctor}{\operatorname{L}}
\newcommand{\monoidalunit}{\mathbf{1}}

\newcommand{\coprodw}{\mathrm{w}}

\newcommand{\counitw}{\othertau}
\newcommand{\planartree}{\operatorname{\mathsf{T}_\mathrm{pl}}}
\newcommand{\laxmap}{\operatorname{l}}
\newcommand{\laxmapsha}{\operatorname{l}_{\sha}}
\newcommand{\filtration}[1]{\operatorname{F}_{#1}}
\newcommand{\filtrationd}[1]{\operatorname{F}_{#1}^{\mathsf{d}}}
\newcommand{\cofiltration}[1]{\operatorname{F}^{#1}}

\renewcommand{\lim}{\varprojlim}
\newcommand{\limn}{\varprojlim_{n\in\ordinalomega\op}}
\newcommand{\hlimn}{{{\varprojlim_{n\in\ordinalomega\op}}^{\!\!\!\mathrm{h}}}}

\newcommand{\colimn}{\varinjlim_{n\in\ordinalomega}}

\newcommand{\limover}[1]{\varprojlim_{#1}}

\newcommand{\chain}[1]{\operatorname{\mathrm{Ch}}(#1)}

\newcommand{\gr}[1]{\operatorname{\mathrm{gr}}^{#1}}

\newcommand{\grfilt}[1]{\operatorname{\mathrm{gr}}_{#1}}
\newcommand{\grfiltd}[1]{\operatorname{\mathrm{gr}}_{#1}^{\mathsf{d}}}

\newcommand{\coker}{\operatorname{coker}}
\newcommand{\deltacoprod}{\Deltaup}
\newcommand{\counittau}{\tauup}
\newcommand{\catss}{\mathcal{A}}
\newcommand{\s}{\operatorname{s}}
\newcommand{\inftymorphism}{\rightsquigarrow}
\newcommand{\h}{\mathrm{h}}

\newcommand{\dergrset}[2]{\operatorname{\mathrm{Der}}_{#1}(#2)}

\newcommand{\codergrset}[2]{\operatorname{\mathrm{Coder}}_{#1}(#2)}

\newcommand{\catofcompletealg}[1]{#1\text{-}\widehat{\mathsf{alg}}}

\newcommand{\ucom}{\mathsf{uC}}
\newcommand{\idfunctor}{\operatorname{id}}

\newcommand{\catc}{\mathcal{C}}

\newcommand{\dual}{^\vee}
\newcommand{\interval}{\mathrm{I}}
\newcommand{\compofplanarseq}{\vartriangleleft_\mathrm{pl}}
\newcommand{\image}{\operatorname{Im}}
\newcommand{\uass}{\mathsf{uA}}
\newcommand{\setofobj}{\operatorname{Ob}}
\newcommand{\catofenrichement}{\mathcal{V}}
\newcommand{\nilfunctor}{\operatorname{nil}}
\newcommand{\catofnilalg}[1]{#1\text{-}\mathsf{alg}^\mathrm{nil}}
\newcommand{\procat}[1]{\operatorname{Pro}\mleft(#1\mright)}
\newcommand{\hombracket}[3]{\operatorname{Hom}_{#1}\mleft(#2,#3\mright)}
\newcommand{\fullhombracket}[4]{\operatorname{Hom}_{#1}^{#2}\mleft(#3,#4\mright)}

\newcommand{\evfunctor}{\operatorname{ev}}

\newcommand{\under}[2]{#1_{#2}}
\newcommand{\fullunder}[2]{{\mleft(#1\mright)}_{#2}}
\newcommand{\power}[2]{#1^{#2}}
\newcommand{\fullpower}[2]{{\mleft(#1\mright)}^{#2}}
\newcommand{\internalhom}[2]{\mleft[#1,#2\mright]}

\newcommand{\forget}{\mathrm{U}}
\newcommand{\forgetsym}{\mathrm{U}_{\mathbf{S}}}
\newcommand{\catofcompletesmodalg}[1]{#1\text{-}\widehat{\mathsf{alg}}\mleft(\catofmod {\s^{-1}}\mright)}
\newcommand{\catofsmodalg}[1]{#1\text{-}\mathsf{alg}\mleft(\catofmod {\s^{-1}}\mright)}

\newcommand{\twisting}[2]{\operatorname{Tw}\mleft(#1,#2\mright)}
\newcommand{\alphatwisting}[3]{\operatorname{Tw}^{#1}\mleft(#2,#3\mright)}
\newcommand{\dg}{\mathrm{dg}}

\newcommand{\homotopypullbackmark}{\lrcorner_{_\mathrm{h}}}
\newcommand{\forgetd}{\mathrm{U}_\mathsf{d}}
\newcommand{\dzero}{\mathsf{D}_0}
\newcommand{\counitdzero}{\varepsilonup}
\newcommand{\canonicalaction}{\mathrm{a}}
\newcommand{\canonicaldiff}{\mathrm{d}_\mathrm{c}}
\newcommand{\induceddiff}[1]{\dd_{#1}}
\newcommand{\canonicalproj}{\operatorname{\tauup}}
\newcommand{\canonicalinj}{\mathrm{i}}
\newcommand{\caninjj}{\mathrm{l}}
\newcommand{\ideal}[1]{\operatorname{\mathrm{I}^{#1}}}
\newcommand{\ideali}{I}
\newcommand{\idealj}{J}
\newcommand{\idealk}{K}

\newcommand{\completionmap}{\varphiup}
\newcommand{\triangularmatrices}{\mathbf{T}}
\newcommand{\extendedsummap}{\Sigma}
\newcommand{\actionexample}{\mathrm{S}}
\newcommand{\kronecker}{\deltaup}
\newcommand{\formalx}{\mathrm{X}}

\newcommand{\weakeqofcompletealg}{\mathrm{W}_{\catofcompletealg\coperad}}
\newcommand{\fibrationofcompletealg}{\mathrm{F}_{\catofcompletealg\coperad}}

\newcommand{\planar}{_\mathrm{pl}}
\newcommand{\plan}{\mathrm{pl}}
\newcommand{\unitsymadj}{\zetaup}
\usepackage{xfrac}
\newcommand{\weakeqofcog}{\mathrm{W}_{\catofcog\operad}}
\newcommand{\fibrationofcog}{\mathrm{F}_{\catofcog\operad}}
\newcommand{\cofibrationofcog}{\mathrm{C}_{\catofcog\operad}}
\newcommand{\counitcobaradj}{\varepsilonup}
\newcommand{\eoperad}{\mathsf{E}}
\newcommand{\beoperad}{\mathsf{W}}

\newcommand{\externaltensor}{\boxtimes}
\newcommand{\diagtensor}{\otimes}
\newenvironment{sumtable}{
	\medskip
	\begin{center}
	}
	{\end{center}
	\medskip
	}
\newcommand{\ie}{i.\thinspace e.}
\newcommand{\exchangemap}{\epsilonup}
\newcommand{\unitofcom}{\mathrm{u}}
\newcommand{\unitor}{\upsilonup}

\newcommand{\catofcomod}[1]{#1\text{-}\mathsf{comod}}

\newcommand{\sectionref}[1]{\hyperref[#1]{\S\thinspace\ref{#1}}}
\newcommand{\convolution}{\circledast}

\newcommand{\thesquares}{\mathrm{S}}
\newcommand{\thediagramre}{\mathrm{RE}}
\newcommand{\thediagramex}{\mathrm{EX}}
\newcommand{\thediagramox}{\mathrm{OC}}
\newcommand{\thealgo}{\mathrm{L}}

\newcommand{\formalphi}{\phiup}

\newcommand{\treemonad}{\bm{\mathsf{T}}}
\newcommand{\compofsymseq}{\mathbin{\vartriangleleft}}
\newcommand{\corolla}{\mathrm{c}}
\newcommand{\attach}{\alphaup}
\newcommand{\copcomodstruct}{w_{\reducedtreemodule}}
\newcommand{\grafting}{\gammaup}

\newcommand{\rightindsym}[2]{\otimes_{#1} #2}
\newcommand{\laudayvallettecobar}{\Omegaup}
\newcommand{\fressecobar}{B^c}
\newcommand{\planarconvolution}{\circledast_\mathrm{pl}}
\newcommand{\compofseq}{\compofsymseq_\mathrm{pl}}

\newcommand{\naturalstosym}{\xiup}

\newcommand{\degrafting}{\deltaup}

\newcommand{\picking}{\mathrm{p}}
\newcommand{\reducedtreemonad}{\overline{\bm{\mathsf{T}}}}

\newcommand{\reducedtreemodule}{\operatorname{\overline{\mathsf{T}}}}
\newcommand{\reducedpltreemodule}{\operatorname{\overline{\mathsf{T}}_\mathrm{pl}}}
\newcommand{\pltosym}[1]{#1 \otimes\catofpermutations}
\newcommand{\pltosymfunctor}{(- \otimes\catofpermutations)}

\newcommand{\wrestricted}[1]{w_{|#1}}
\newcommand{\fncaninclusion}[1]{\iotaup_{#1}}
\newcommand{\wgeqn}[1]{w_{>#1}}

\newcommand{\specialprojto}[1]{\mathrm{p}_{#1}}
\newcommand{\projto}[1]{\pi_{#1}}
\newcommand{\degreeonemap}{\varsigmaup}
\newcommand{\degreeminusonemap}{\varsigmaup^{-1}}

\newcommand{\pseudounitcobaradj}[1]{\mathrm{j}}
\newcommand{\augmentation}{\mathrm{z}}
\newcommand{\specialquotient}{\mathrm{q}}

\newcommand{\restrictionone}{\mathrm{R}_1}

\newcommand{\restrictiontwo}{\mathrm{R}_2}
\newcommand{\extensionone}{\mathrm{E}_1}
\newcommand{\extensiontwo}{\mathrm{E}_2}
\newcommand{\operadcompositionone}{\mathrm{O}_1}
\newcommand{\operadcompositiontwo}{\mathrm{O}_2}
\newcommand{\thehomotopy}{\mathrm{h}}
\newcommand{\thebighomotopy}{\mathrm{H}}
\newcommand{\mapg}{\mathrm{g}}
\newcommand{\bigmapg}{\mathrm{G}}
\newcommand{\theboundary}{\mathrm{B}}

\newcommand{\specialoperad}{\mathrm{P}}
\newcommand{\smalldiff}{\mathrm{d}}
\newcommand{\smalldiffex}{\mathrm{d}_{\mathrm{ex}}}

\newcommand{\smalldiffw}{\mathrm{d}_{w}}

\newcommand{\bigdiff}{\mathrm{D}}
\newcommand{\bigdiffex}{\mathrm{D}_{\mathrm{ex}}}
\newcommand{\bigdiffin}{\mathrm{D}_{\mathrm{in}}}
\newcommand{\bigdiffw}{\mathrm{D}_w}
\newcommand{\bigdiffcop}{\mathrm{D}_{\coperad}}
\newcommand{\bigdiffv}{\mathrm{D}_{V}}
\newcommand{\bigdiffa}{\mathrm{D}_{a}}
\newcommand{\bigdifftheta}{\mathrm{D}_{\theta}}

\newcommand{\canonicaltwisting}{\alphaup}
\newcommand{\canonicaltwistinginverse}{\betaup}

\newcommand{\specialkernelkmap}{\mathrm{k}}
\newcommand{\tensorbycounitofd}{\rhoup}
\newcommand{\modulispace}[1]{\mathscr{M}^{#1}}
\newcommand{\setofder}[1]{#1\text{-}\mathsf{der}}
\newcommand{\dertag}{\textsf{der}}
\newcommand{\cogtag}{\textsf{cog}}
\newcommand{\groundofe}{\sigmaup}
\newcommand{\simplexcat}{\mathbf{\Delta}}
\newcommand{\nervefunctor}{\operatorname{N}}
\newcommand{\gammafunctor}{\Gammaup}
\newcommand{\positivechain}[1]{\mathrm{Ch}_{\geq 0}(#1)}
\newcommand{\counitofd}{\rhoup}
\newcommand{\gradedd}{\mathsf{d}}

\newcommand{\counitsymadj}{\lambdaup}
\newcommand{\isophi}{\phiup}
\newcommand{\isopsi}{\psiup}
\newcommand{\natabrev}{\mathrm{nat.}}
\newcommand{\imshuffle}[3]{\operatorname{Im}\mleft(\sha\mleft(#1,#2\mright)^{#3}\mright)}
\newcommand{\smallimshuffle}[3]{\operatorname{Im}(\sha(#1,#2)^{#3})}

\newcommand{\lcc}{\mathsf{lcc}}

\newcommand{\catoflpccccoop}[1]{\mathsf{Coop}_\bullet^{\mathsf{lccc}} (#1)}

\newcommand{\unitspecialoperad}{\etaup}
\newcommand{\bigretract}{\mathrm{R}}
\newcommand{\bigsection}{\mathrm{S}}

\newcommand{\node}[1]{\mathrm{N}_{#1}}

\newcommand{\prop}{P}
\newcommand{\catofgebras}[1]{{#1}\text{-}\mathsf{geb}}
\newcommand{\sinv}{\s^{\scalebox{0.6}{$-1$}}\!}

\begin{document}

%----[ Title ]----------------------------------------------------------

\maketitle

%----[ Footnote ]-------------------------------------------------------

%\footnotefirstpage%

%----[ Abstract ]-------------------------------------------------------

\begin{abstract}
We study extensively the homotopy theory of coalgebras. By coalgebras, we
mean the full theory of coalgebras: with counits and not necessarily
locally conilpotent. For example \einfinity-coalgebras,
\ainfinity-coalgebras, \linfinity-coalgebras etc. To do so, we
define the category of complete curved algebras --- where the notion
of quasi-isomorphims does not make sense --- and endow it with a model
category structure, equivalent to that of the category of coalgebras.
\end{abstract}

%----[ TOC ]------------------------------------------------------------

\tableofcontents

%----[ Body ]-----------------------------------------------------------

\section{Introduction}
%=======================================================================

This article is the first of a series about the homotopy theory of
differential graded coalgebras. Another article shall follow, focused on
providing examples and applications to well-know types of coalgebras like
$\einfinity$, $\ainfinity$, $\linfinity$-coalgebras
etc.

\subsection{Mathematical context}
%-----------------------------------------------------------------------

\subsubsection{Coalgebras are ubiquitous}

As mathematicians, we are very familiar with the theory of algebras.
Associative algebras, Lie algebras, commutative algebras\ldots We know
how they behave, we know what to expect from them and we know when they
are peculiar. We are so used to them that we recognise them whenever we
encounter them in mathematical experiments.

This is less the case with homotopy algebras, for several reasons.
The homotopy theory of algebras is a much younger field in the
history of mathematics. Stasheff gave the definition of an
$\ainfinity$-algebra in 1963%
~\cite{doi:10.2307/1993609} and Quillen created model categories --- one
of the major tools used to understand homotopy theories --- in 1967
~\cite{doi:10.1007/bfb0097438}. Hence homotopy algebras are not in the
basic curriculum of a mathematician. Furthermore, describing homotopy
structures involves inherently a thicker amount of computations, for
example when displaying the $\ainfinity$\=/structure on the
Fukaya category%
~\cite{doi:10.1090/amsip/046.1}.

This is even less the case with homotopy coalgebras. Since coalgebras are
unfamiliar, people usually prefer to study them through algebras. Indeed
given a coalgebra $C$ it is often possible to build an algebra out of it.
For example if $C$ is linear, its dual $C\dual$ is always an algebra.
More generally, given an algebra $A$, one can study the convolution
product induced on the set of functions from $C$ to $A$. However during
mathematical experiments, homotopical coalgebras tend to appear as often
as their algebraic counterparts. Let us describe an inexhaustible source
of homotopy coalgebras: spaces. In any geometrical context, a space $X$
has the fundamental property of being a cocommutative coalgebra thanks
to its diagonal $\diagonal : X \to X^2$. Consequently, whenever one
wants to study spaces by algebraic homotopy invariants covariantly
computed out of $X$, one often encounters homotopy coalgebras. This is the
case with the homology functor: given a topological space $X$ and an
absolutely flat ring $R$, $\homology_\ast(X, R)$ has a natural structure
of homotopy cocommutative coalgebra or $\einfinity$\=/coalgebra structure.
This mirrors the cup product and Steenrod squares on cohomology. We can
also cite among many other examples, the work on formal geometry through
$\linfinity$-coalgebras by Kontsevitch or Merkulov%
~\cite{doi:10.1007/s00220-010-0987-x}.

The variety of coalgebras is not limited to the examples cited above.
Instead the notion of coalgebra make sense over any operad $\operad$, such
as the unital and non-unital versions of $\ainfinity$ and $\einfinity$,
$\linfinity$, $\bvinfinity$, $\bdinfinity$ etc. For this, given an
object $V$ in an appropriate symmetric monoidal category $(\symmoncat,
\otimes, \monoidalunit)$, we use its operad of coendomorphisms given
in arity $n$
by $\coend V (n) \coloneqq \relhom {\symmoncat} V {V^{\otimes n}}$. A
$\operad$-coalgebra structure on $V$ is then the same data as a morphism
of operads $\operad \to \coend V$.

\subsubsection{Coalgebras are notorious}

In addition to being unfamiliar, coalgebras are also notorious for not
behaving as well as algebras. This comes from a very deep fact of the
mathematics we manipulate: the category of sets is well generated under
colimits --- it is presentable --- and badly generated under limits ---
it is not co-presentable. Thus, there is a strong preference in all
categorical constructions towards colimit preserving constructions and
functors. Tensor structures usually commute with colimits and almost
never with limits, in all non-exotic categories built from the category
of sets.

As a consequence, while it is straightforward to prove that the category
of $\operad$-algebras is monadic and presentable in all cases, it is
far from obvious to prove that the category of $\operad$-coalgebras
is comonadic and presentable. Since 2014 we now have drastic
sufficient conditions on $\symmoncat$ to answer that problem%
~\cite{arXiv:1409.4688}. In the particular case of $\operad$-coalgebras
in chain complexes over a field, one has an explicit formula computing
the comonad associated to the category or $\operad$-coalgebras.

To comonadicity problem of the category of $\operad$-coalgebras
is related to the existence of cofree $\operad$-coalgebras. Even
when they exists they are usually unpleasant to compute. A much
more pleasant coalgebra to compute is the cofree locally conilpotent
coalgebra, which is straightforward to compute.
For example, given a $\fieldk$-vector space $V$, the cofree locally
conilpotent coassociative coalgebra is given by the tensor $\fieldk[V]
\coloneqq \fieldk \oplus V \oplus V^{\otimes 2} \oplus \cdots$ and
the coalgebra structure is given by deconcatenation of tensors. By
contrast, the cofree coassociative coalgebra on one generator is given by
$\under{\fieldk[\formalx]}{(\formalx)} \subset \fieldk[[\formalx]]$.
Because of this, in several references `coalgebra' means by default
`locally conilpotent coalgebra'.

\subsubsection{Using Cobar to study homotopy coalgebras}
In order to study homotopy theory, a common way is to use model
categories. Using a theorem of 2015%
~\cite{doi:10.1090/conm/641/12859},
it is possible given a dg-operad $\operad$ to induce a model category
structure on the category of $\operad$\=/coalgebras from the one of chain
complexes. Once this is done, understanding the homotopy theory of
coalgebras reduces to understanding fibrant\-/cofibrant objects of the
category. These are generated by the cofree $\operad$\=/coalgebras, that
we already have difficulties to handle.

A solution is to use the Koszul duality of operads to find a different
model category equivalent to the first one. This idea has given
successful results in the study of homotopy algebras by Hinich%
~\cite{doi:10.1080/00927879708826055},
Vallette%
~\cite{arXiv:1411.5533},
Lefèvre-Hasegawa%
~\cite{arXiv:0310337}
and LG%
~\cite{arXiv:1612.02254}.
Let $\coperad$ be a cooperad Koszul dual to $\operad$. By dualising
the Bar construction, we obtain a functor \cobarfunctorfull{}
from the category of $\operad$\=/coalgebras to the category of
complete $\coperad$\=/algebras. Given a $\operad$\=/coalgebra
$V$, $\cobarfunctorfull V$ is built by taking the free graded
$\coperad$\=/algebra on $V$ which then receives a differential encoding
the structure of $\operad$\=/coalgebra. After endowing the category of
complete $\coperad$-algebras with a model structure,
we show that \cobarfunctorfull{} can be
promoted to an equivalence of model categories. This means that the
homotopy theory of $\operad$-coalgebras is equivalent to the homotopy
theory of $\coperad$-algebras.

It turns out that the fibrant-cofibrant objects are easy to describe:
they are the $\coperad$-algebras that are free as graded-algebras. In
addition, contrarily to the case of cofree $\operad$-coalgebras, the free
$\coperad$-algebras have a simple description. We may then use the
\cobarfunctorfull{} functor as a dictionary between $\operad$-coalgebras
and complete $\coperad$-algebras to understand the homotopy theory of
$\operad$-coalgebras.

\subsection{Mathematical content}
%-----------------------------------------------------------------------

In this article we study the homotopy theory of coalgebras over a linear
operad by dualising the Bar construction used in the study of the
homotopy theory of algebras. Given a dg-operad $\operad$ and a
locally conilpotent curved cooperad $\coperad$ Koszul dual to $\operad$,
the Bar construction is a functor $\barfunctor : \catofalg P\to
\catofcog\coperad$ between the categories of $\operad$\=/algebras
and $\coperad$\=/coalgebras. It admits a left adjoint that we denote
$\baradj$. Once the category of $\operad$-algebras is
given its natural model structure, the \barfunctorfull{} adjunction can
be promoted to an equivalence of model categories by transferring the
model structure of $\catofalg\operad$ to $\catofcog\coperad$%
~\cite{arXiv:1612.02254}.
This allows the homotopy theory of $\operad$-algebras to be read using
cofree graded $\coperad$\=/coalgebras.

We build the dual theory in several steps. The first one is to
to define the notion of $\coperad$-algebras in order to play a role
symmetric to the one of $\coperad$-coalgebras. For this we introduce
a tool: the cotensor of the category of symmetric sequences
$\symseq\symmoncat$ onto the category $\symmoncat$. Thanks to this
cotensor, a $\operad$-coalgebra is the data of a map $\Lambda \to
\Lambda^{\operad}$ satisfying certain conditions. Likewise, we define a
$\coperad$-algebra as the data of a map $V^{\coperad} \to V$ among other
things. In \sectionref{fig:tableau_algebres_cogebres}, we draw a table
summarising the different definitions of algebras and coalgebras.

The second step is the construction of a functor $\cobarfunctor :
\catofcog\operad \to \catofalg\coperad$. As it is built
using the coconstruction of the \barfunctorfull{} functor, we call it
the \cobarfunctorfull{} 
functor. The construction of its right adjoint, noted
$\cobaradj$ is more
problematic as it involves the cofree $\coperad$-coalgebra functor.
This is the source of many technical computations that were easier in
the case of $\operad$-algebras.

The next step is to endow the category of $\operad$-coalgebras with a
model structure induced from the one on chain complexes on $\symmoncat$.
Since this is a transfer along a \emph{left} adjoint, it requires
the use of a 2015 theorem of Bayeh, Hess, Karpova, Kedziorek, Riehl and
Shipley%
~\cite{doi:10.1090/conm/641/12859}.
Doing so, we answer Porta's Question~1 and Question~3 on
\mathoverflow{}%
~\cite{mathoverflow:148626}: we define coadmissible operads in
\cref{def:operade_coadmissible} and in
\cref{thm:les_operades_planaires_ou_cofibrantes_sont_coadmissibles} we
show that $\eoperad$-split operads such as planar operads and cofibrant
operads are coadmissible.

\begin{theorem-for-introduction}
{\Cref{bigthm:weq_of_operads_eq_of_models}}
Cofibrant dg-operads are coadmissible. Weak equivalences
$f : \operad \to \operad'$ between cofibrant dg-operads yield
model equivalences
\[
	\begin{tikzcd}[ampersand replacement=\&]
		\catofcog\operad
		\arrow[r, shift right, swap, "f_\ast"]
		\&
		\catofcog{\operad'}.
		\arrow[l, shift right, swap, "f^\ast"]
	\end{tikzcd}
\]
\end{theorem-for-introduction}

In addition, we show in \cref{thm:com_n_est_pas_coadmissible} that
the commutative operad is not coadmissible in the category of chain
complexes over an algebraically closed field of characteristic zero.
This is in stark contrast with the algebra case where any operad is
admissible in characteristic zero.

We then transfer the model structure on $\operad$-coalgebras to
$\coperad$\=/algebras, along the functor $\cobaradj$.
However, as $\coperad$ is locally conilpotent, any $\coperad$\=/coalgebra
is also locally conilpotent but not every $\coperad$-algebra
is complete. We give a counter-example of this fact in
\sectionref{sec:contre_exemple}. Hence we shall replace the category of
$\coperad$-algebras with the category of complete $\coperad$\=/algebras.

\begin{theorem-for-introduction}{\Cref{bigthm:theoreme_de_transfert}}
The category of complete
$\coperad$-algebras can be endowed with a combinatorial model category
structure transferred along the \cobarfunctorfull{} adjunction:
\[
	\begin{tikzcd}[ampersand replacement=\&]
		\catofcog\operad
		\arrow[rr, shift left=1.5,"\cobarfunctorfull"]
		\&
		\&
		\catofcompletealg\coperad.
		\arrow[ll, shift left=1.5, "\rightadjoint\cobarfunctorfull"]
	\end{tikzcd}
\]
As a consequence, the \cobarfunctorfull{} adjunction is promoted to a
model adjunction.
\end{theorem-for-introduction}

\sectionref{sec:equivalence_cobar} is dedicated to proof of the model
equivalence.

\begin{theorem-for-introduction}{\Cref{bigthm:equivalence}}
When $\operad$ is cofibrant and
$\coperad$ is a Koszul dual
of $\operad$, the \cobarfunctorfull{} adjunction is a model equivalence.
\end{theorem-for-introduction}

Finally in \sectionref{sec:theorie_homotopique_des_cogebres_lineaires},
we deduce the consequences of the model equivalence for the homotopy
theory of coalgebras.

\begin{theorem-for-introduction}{%
\Cref{bigthm:homotopy_theory_of_coalgebras}}
The following full subcategories of the
category of complete $\coperad$-algebras are equivalent:
\begin{itemize}
	\item the subcategory of cofibrant objects;
	\item the subcategory whose objects are in the image of the
	      \cobarfunctorfull{} functor;
	\item the subcategory whose objects are the complete
	      $\coperad$-algebras that are free as graded $\coperad$-algebras.
\end{itemize}
\end{theorem-for-introduction}

\subsection{Non-standard notations and terminology}
%-----------------------------------------------------------------------

\begin{itemize}
	\item We use lateral marks from the international navigation system
	      to designate adjoint functors:
	      \begin{center}
	      \begin{tikzpicture}
	      	\node[draw = black, rectangle, inner sep = 8pt,
	      	rounded corners]
	      	{%
	      	\begin{minipage}{0.35\textwidth}
	      		\begin{center}
	      			$\square \coloneqq \text{left}
	      			\qquad \triangle \coloneqq \text{right}$
	      		\end{center}
	      	\end{minipage}
	      	};
	      \end{tikzpicture}
	      \end{center}
	      Thus $\rightadjoint F$ designates the right adjoint to a functor
	      $F$ and $\leftadjoint F$ is the left adjoint of $F$;
	\item In this article we define a functor named \cobarfunctorfull{}.
	      It is not the adjoint of the \barfunctorfull{} functor,
	      which we denote by $\leftadjoint\barfunctorfull$. The same
	      functor is denoted by $\laudayvallettecobar$ by Loday and
	      Vallette~%
	      \cite{doi:10.1007/978-3-642-30362-3} and $\fressecobar$ by
	      Fresse~%
	      \cite{doi:10.1090/conm/504/09879};
	\item We shall denote by $\compofsymseq$ the composition of symmetric
	      sequences instead of the traditional $\circ$. This is done in
	      order to avoid any confusion in formulas with the composition
	      of maps. We also find it more appropriate to denote a
	      non-symmetric monoidal structure by a non-symmetric notation;
\end{itemize}

\subsection{Standard hypotheses}%
%-----------------------------------------------------------------------
\label{sec:characteristic_zero}%
\label{sec:logos}

The main context we are interested in is unbounded chain
complexes of vector spaces over a field.
However, our construction
may be performed
in a more general linear context: we shall denote by
$(\linearlogos, \otimes, \monoidalunit)$ any
closed symmetric monoidal category such that
\begin{itemize}
	\item $\linearlogos$ is a presentable additive category;
	\item $\linearlogos$ is exact;
	\item filtered colimits in $\linearlogos$ are exact;
	\item products in $\linearlogos$ are exact;
	\item $\otimes : \linearlogos \times \linearlogos \to
	      \linearlogos$ is exact in both variables;
	\item the adjoint inner-hom bifunctor $\internalhom {-} {-}
	      : \linearlogos\op \times \linearlogos \to \linearlogos$
	      is exact in both variables.
\end{itemize}

Although it is not possible to impose the exactness of cofiltered limits
in general, thanks to the presentability condition%
~\cite{doi:10.1007/s002220100197} one has the following theorem of
Roos.

\begin{theorem}[\cite{doi:10.1112/S0024610705022416}]%
\label{thm:exactness_inverse_limit}
Let
\[
	0 \longrightarrow A^\ast \longrightarrow B^\ast
	\longrightarrow C^\ast \longrightarrow 0
\]
be an exact sequence in the category $\linearlogos^{\ordinalomega\op}$.
If for every $n \in \ordinalomega$, the transition map
\[
	A^{n+1} \longrightarrow A^n
\]
is an epimorphism in $\linearlogos$, then the sequence
\[
	0 \longrightarrow \limn A^\ast \longrightarrow \limn B^\ast
	\longrightarrow \limn C^\ast \longrightarrow 0
\]
is also exact.
\end{theorem}

Besides, our work will be tightened by the following restriction.

\begin{center}
\begin{tikzpicture}
	\node[draw = black, rectangle, inner sep = 10pt, rounded corners]
	{
	\begin{minipage}{0.7\textwidth}
		We shall either suppose that
\medskip
		\begin{itemize}
			\item we restrict our attention to planar operads;
			\item or that all categories are $\rationals$-linear.
		\end{itemize}
	\end{minipage}
	};
\end{tikzpicture}
\end{center}

For instance, if $\linearlogos$
is the category of chain
complexes of vector spaces over a field,
we assume the characteristic of this field
to be zero when dealing with non planar operads.
\section{Operads and cooperads}
%=======================================================================

Our main reference for operads is the
book \emph{Algebraic Operads} by Loday and
Vallette%
~\cite{doi:10.1007/978-3-642-30362-3}.

\subsection{Definitions and notations}
%-----------------------------------------------------------------------

We shall recall well-known definitions in the context of
linear operads, setup notations and make some remarks that will be
used later.

\begin{itemize}
	\item Let $\catofpermutations$ denote the groupoid of finite
	      ordinals with bijections. It is endowed with a natural
	      symmetric monoidal category structure
	      $(\catofpermutations, + ,0)$;
	\item The category of symmetric sequences in $\linearlogos$ is the
	      functor category $\symseq\linearlogos$. It is endowed with
	      a canonical symmetric monoidal structure via convolution
	      $(\symseq\linearlogos, \convolution,
	      \monoidalunit_\convolution)$. If
	      $M$ and $N$ are two symmetric sequences and $n$ is a natural
	      number then
	      \[
	      	(M \convolution N)[n] \coloneqq \bigoplus_{p+q=n}
	      	(M(p)\otimes N(q))
	      	\rightindsym {\symgroup p \times \symgroup q} {\symgroup n}.
	      \]
	      The unit is given by the sequence
	      $\monoidalunit_\convolution \coloneqq (\monoidalunit, 0,\dots)$;
	\item The composition of two symmetric sequences $M$ and $N$ is
	      defined by
	      \[
	      	M \compofsymseq N \coloneqq \bigoplus_{n\in\naturals}
	      	M(n) \otimes_{\symgroup n} N^{\convolution n},
	      \]
	      where for every $n \in \naturals$, $N^{\convolution n}$ is
	      given its natural left $\symgroup n$-action by permutation.
	      The composition endows the category of symmetric sequences
	      with a monoidal structure
	      $(\symseq\linearlogos,
	      \compofsymseq, \monoidalunit_{\compofsymseq})$, where the unit
	      is the sequence $\monoidalunit_{\compofsymseq} \coloneqq
	      (0,\monoidalunit,0, \dots)$.
\end{itemize}

\begin{definition}
The category of operads and cooperads in $\linearlogos$ are defined as
\[
	\catofop(\linearlogos)
	\coloneqq \catofmon
	\left(\symseq\linearlogos, \compofsymseq, \monoidalunit\right)
	\qand \catofcop(\linearlogos)
	\coloneqq \catofop(\linearlogos\op)\op.
\]
Since $\linearlogos$ is assumed $\rationals$-linear and filtered
colimits are exact, we have a full embedding
\[
	\begin{tikzcd}
		\catofcomon
		\left(\symseq\linearlogos, \compofsymseq, \monoidalunit\right)
		\ar[r, hook] & \catofcop(\linearlogos).
	\end{tikzcd}
\]
By abuse, we shall always use the word \emph{cooperad} to refer to
such comonoids.
\end{definition}

\subsection{Trees}%
%-----------------------------------------------------------------------
\label{sec:trees}

The category of operads is monadic over the category of symmetric
sequences: operads are the modules over a monad
$\treemonad \coloneqq (\treemodule, \grafting,
\attach)$ called the tree monad. We shall recall some well-known
results about it and setup a few notations. A more thorough description
can be found in \emph{Algebraic Operads}%
~\cite[5.5]{doi:10.1007/978-3-642-30362-3}.
The notion of tree that we use is defined in details in
\emph{Simplicial Methods for Operads and Algebraic Geometry}%
~\cite[2]{doi:10.1007/978-3-0348-0052-5},
in particular trees are finite, non-empty and rooted such as

\begin{center}
	\begin{tikzpicture}
		\tikzset{every node/.style={draw, circle, thick, inner sep=0,
		minimum size=4}, every path/.style={thick}}
		\node (a) at (0,0) {};
		\node (b) at (1,1) {};
		\node (c) at (-1,1) {};
		\draw (a) -- (b);
		\draw (a) -- (c);
		\draw (a) -- (0,-1);
		\draw (c) -- (-2,2);
		\draw (c) -- (0,2);
	\end{tikzpicture}
\end{center}
which has three vertices and two leaves.

Every tree $t$ induces an endofunctor $t(-)$ of the category of
$\catofpermutations$\=/modules. It can be defined by induction on the
height of the trees, that is the maximum number of vertices that one can
meet along an ascending path from the root to a leaf:
\begin{itemize}
\item If the height of $t$ is zero, that is $t$ the trivial tree, then
\[
	t (-) \coloneqq \monoidalunit;
\]
\item Suppose that we have defined $t'(-)$ for any tree $t'$ of
height equal or lower than $n\in \naturals$. Let $t$ be a tree
of height $n+1$. It is obtained by grafting $k$ subtrees to a corolla
with one of these subtrees of height $n$.
Let us label these subtrees
$t_1, \ldots, t_k$ in such a way that isomorphic subtrees
are gathered. More precisely, there exists natural integers
$p,n_1, \dots,n_p$ so that $k = n_1 + \cdots + n_p$, and
for any $0\leq i< p$, the subtrees
\[
	t_{n_1+\cdots+ n_i+1}, t_{n_1+\cdots+ n_i+2}, \ldots, 
	t_{n_1+\cdots+ n_i+n_{i+1}}
\]
are all isomorphic to each other and not isomorphic 
to any other subtree.
Then, for any $\catofpermutations$-modules $M$,
\[
	t(M) \coloneqq M(k) \otimes_{\symgroup k}
	\left( \bigoplus_{\sigma}  t_{\sigma^{-1}(1)} (M)
	\convolution \cdots \convolution t_{\sigma^{-1}(k)}(M) 
	\right),
\]
where the sum is taken over the $(n_1,\ldots,n_p)$-shuffle
permutations. The construction of $t(M)$ does not
depend on the choice of labelling of the subtrees of $t$.
\end{itemize}

Given an $\catofpermutations$-module $M$,
the tree module functor $\treemodule$ is then given by
taking the sum over all equivalence classes of trees
\[
	\treemodule M \coloneqq \bigoplus_{[t]\text{ tree}} t(M),
\]
while $\grafting M : \treemodule(\treemodule M) \to \treemodule M$ is
given by the grafting of trees and the attachment map $\attach M : M
\to \treemodule M$ is the map sending $M$ to the sum of $\corolla_n(M)$
where $\corolla_n$ is the $n$-corolla tree. Finally, given an integer $n
\geq 0$ and $\catofpermutations$-module $M$, the weight $n$ tree module
is the sub $\catofpermutations$-module of $\treemodule M$ made up of
trees with $n$-vertices
\[
	\treemodule_n M
	\coloneqq \bigoplus_{[t]\text{ tree with n vertices}} t(M).
\]
In particular $\treemodule_0 M = \monoidalunit$.

\subsection{Locally conilpotent cooperads}
%-----------------------------------------------------------------------

There exists several equivalent ways to define locally
conilpotent cooperads. It will be convenient to be
able to switch between
the different descriptions that we give here.
We also give the definition of the
coradical filtration of a locally conilpotent 
cooperad; ours is different from the one of Loday-Valette%
~\cite[5.7.6]{doi:10.1007/978-3-642-30362-3}.

Given a symmetric sequence $M$, the attachment map $\attach M : 
M \to \treemodule M$ has a canonical section, the picking map
\[
	\picking M : \treemodule M \longrightarrow M.
\]
Together with the degrafting of trees
\[
	\degrafting :
	\treemodule M \longrightarrow \treemodule (\treemodule M)
\]
obtained by summing over all possible cuts of a given tree,
they endow the tree monad $\treemonad$ with the structure of
a bimonad $(\treemodule$, $\grafting$, $\degrafting$, $\attach$,
$\picking)$.

We denote by $\reducedtreemonad$ the sub-bimonad of $\treemonad$
obtained by removing the unique tree with no vertices: the unit tree.

\begin{definition}%
\label{def:loc_conilpotent_coperad}
The functor $\overline\coperad \mapsto \coperad \coloneqq
\overline\coperad \oplus \monoidalunit$ sends fully faithfully
comodules over $\reducedtreemonad$ to pointed  cooperads%
~\cite[5.7.13]{doi:10.1007/978-3-642-30362-3}
\[
	\begin{tikzcd}
		\catofcomod\reducedtreemonad
		\left(\symseq\linearlogos\right)
		\ar[r, hook] &
		\catofcomon_\bullet
		\left(\symseq\linearlogos, \compofsymseq, \monoidalunit\right).
	\end{tikzcd}
\]
The objects of this subcategory shall be called \emph{locally
conilpotent cooperads}.
\end{definition}

\begin{notation}
Given a $\reducedtreemonad$-comodule
\[
	\left(\overline \coperad, \copcomodstruct : \overline \coperad
	\longrightarrow \reducedtreemodule
	\left(\overline\coperad\right)\right) 
\]
we shall denote by
$\overline w_2$ the projection of the decomposition $\copcomodstruct$
on trees with two vertices
\[
	\overline w_2 : \overline\coperad \xrightarrow{\copcomodstruct}
	\reducedtreemodule \left(\overline\coperad\right) \twoheadrightarrow
	\treemodule_2\left(\overline\coperad\right),
\]
and by $\overline w$, its projection on trees of level two
\[
	\overline w : \overline\coperad \xrightarrow{\copcomodstruct}
	\reducedtreemodule \left(\overline\coperad\right) \longrightarrow
	\overline\coperad \compofsymseq \coperad.
\]
The associated pointed cooperad $\coperad$ shall be denoted
by
\[
	\left(\coperad \coloneqq \overline\coperad \oplus \monoidalunit,
	w : \coperad \to \coperad \compofsymseq
	\coperad, \tau : \coperad \to \monoidalunit,
	\iota :\monoidalunit \to \coperad \right).
\]
If $\projto \monoidalunit$ denotes the projector $\iota \circ \tau$,
one has on $\overline\coperad$
\[
	\overline w = (\id{\coperad\compofsymseq\coperad}
	- \projto\monoidalunit \compofsymseq
	\id\coperad - \id\coperad \compofsymseq \projto\monoidalunit)
	\circ w.
\]
\end{notation}

\begin{definition}%
\label{def:coradical_filtration}
Let $(W,\copcomodstruct)$ be a comodule over $\reducedtreemonad$. The
tree module
$\reducedtreemodule W$ is naturally filtered by the number of vertices
of trees
\[
	\reducedtreemodule_{\leq0} W \coloneqq 0
	\hookrightarrow \reducedtreemodule_{\leq 1} W \hookrightarrow
	\cdots \hookrightarrow \reducedtreemodule_{\leq n} W \hookrightarrow
	\cdots \hookrightarrow \reducedtreemodule W.
\]
One can then induce an exhaustive filtration on $W$ via pullback
\[
	\begin{tikzcd}[ampersand replacement=\&]
		\filtration n W
		\arrow[r, ""]
		\arrow[d, "",swap]
		\arrow[rd, very near start, phantom, "\lrcorner"]
		\& \reducedtreemodule_{\leq n} W
		\arrow[d, hook] \\
		W
		\arrow[r, hook, "\copcomodstruct", swap]
		\& \reducedtreemodule W.
	\end{tikzcd}
\]
The induced filtration on the cooperad $W \oplus \monoidalunit$
shall be called the \emph{coradical filtration}.
\end{definition}

\subsection{Planar setting versus symmetric setting}%
\label{sec:from_planar}
%-----------------------------------------------------------------------
There exists a notion of planar operads and
of planar cooperads that we recall here. Planar techniques will
prove to be key in the demonstration of \cref{bigthm:equivalence}.
The relation between operads and planar operads is well-known%
~\cite[5.8.12]{doi:10.1007/978-3-642-30362-3}; we shall
build here the bridge linking  cooperads and planar 
cooperads.

\begin{definition}[(Planar analogues)]
Replacing $(\catofpermutations, +, 0)$ with $(\naturals, +,0)$, one
can define planar analogues of all the definitions described so far. One
has $\linearlogos^\naturals$ the category of planar sequences, endowed
with planar convolution $\planarconvolution$ and planar composition of
sequences $\compofseq$.

One then has planar operads in $\linearlogos$, planar cooperads
and planar  cooperads. The tree monad also admits a
planar version: given a planar sequence $M$ and a planar tree $t$, one
can define
\[
	t(M) \coloneqq \bigotimes_{v} M(|v|).
\]
where the tensor is taken over the totally ordered set of vertices of
$t$ and where for a vertex $v$, $|v|$ denotes its valence.

One can then define the planar tree module
$\planartree M$ as the sum over all equivalence classes of planar trees%
~\cite[5.8.6]{doi:10.1007/978-3-642-30362-3}
\[
	\planartree M \coloneqq \bigoplus_{[t]\text{ planar}} t(M).
\]
Finally, one has planar locally conilpotent cooperads
as comodules over the comonad of reduced planar trees.
\end{definition}

The canonical functor
$\naturalstosym :\naturals \to \catofpermutations$ induces a forgetful
functor
$\forgetsym \coloneqq \naturalstosym^\ast
: \symseq\linearlogos \to \linearlogos^\naturals$. Since $\linearlogos$
is additive, it has the same left and right adjoints
\[
	\begin{tikzcd}[ampersand replacement=\&]
		\linearlogos^{\naturals}
		\arrow[rr, shift left=3, "\pltosymfunctor"]
		\arrow[rr, shift right=3, "\pltosymfunctor"']
		\&
		\&
		\symseq\linearlogos .
		\arrow[ll,"\forgetsym" description]
	\end{tikzcd}
\]
The top adjunction is well-known; by straightforward inspection, one has
the following description of the bottom adjunctions.

\begin{proposition}%
\label{thm:sym_adjunctions}
Given a planar sequence $M$, the symmetric sequence $\pltosym M$ is
given in arity $n \in \naturals$ by
\[
	(\pltosym M) (n) \coloneqq \bigoplus_{\sigma\in\symgroup n} M(n)
\]
endowed with its natural right $\symgroup n$-action.

\begin{itemize}
	\item The unit of the adjunction $\forgetsym \dashv 
	      (\pltosym{-})$ is given in arity $n \in \naturals$ by
	      \[
	      	\unitsymadj(n) : M(n) \xrightarrow{\sum_{\sigma\in
	      	\symgroup n} \id {M(n)}}\bigoplus_{\sigma\in\symgroup n}
	      	M(n) \isonat \pltosym {\forgetsym(M)},
	      \]
	      for any symmetric sequence $M$;
	\item The counit of the adjunction $\forgetsym \dashv 
	      (\pltosym{-})$ is given in arity $n \in \naturals$ by
	      \[
	      	\counitsymadj(n) : \forgetsym (\pltosym M)(n)
	      	\isonat \bigoplus_{\sigma \in \symgroup n} M (n)
	      	\xrightarrow{\sigma = 1} M(n),
	      \]
	      for any sequence $M$.
\end{itemize}
\end{proposition}

\begin{remark}
A remarkable fact of this setting is that
$\forgetsym : \symseq\linearlogos \to \linearlogos^\naturals$ and
$(\pltosym{-}) : \linearlogos^\naturals \to \symseq\linearlogos$ are
both bimonadic functors.
\end{remark}

Since $\naturalstosym :(\naturals, +, 0)
\to (\catofpermutations, +, 0)$
is symmetric monoidal, the left adjoint $\naturalstosym_{!}$ is
symmetric monoidal with respect to the convolution products: for any
two sequences $M$ and $N$ there are natural isomorphism
\[
	\pltosym{(M \planarconvolution N)}
	\isonat (\pltosym M) \convolution (\pltosym N).
\]
Since it also commutes with colimits, it is monoidal with respect
to the composition products
\[
	\pltosym{(M \compofplanarseq N)}
	\isonat (\pltosym M) \compofsymseq (\pltosym N).
\]
Thus $\pltosymfunctor$ is promoted to a monoidal functor.

Since $\pltosymfunctor$ is monoidal, the unit and the counit of the
adjunction $\forgetsym \dashv \pltosymfunctor$ are monoidal natural
transformations of colax functors. As a consequence we have the
following adjunction.

\begin{proposition}%
\label{thm:from_planar}
The functor $\pltosymfunctor$ supplies a bimonadic functor
\[
	\begin{tikzcd}[ampersand replacement=\&]
		\catofcomon
		\left(\symseq\linearlogos, \compofsymseq, \monoidalunit\right)
		\arrow[rr, shift left=3, "\forgetsym"]
		\arrow[rr, shift right=3,
		"{\internalhom \monoidalunit -}^\catofpermutations"']
		\& \&
		\arrow[ll, "\pltosymfunctor" description]
		\catofcomon
		\left(\linearlogos^\naturals, \compofseq, \monoidalunit\right).
	\end{tikzcd}
\]
\end{proposition}

Let us describe the comonadic functor
$\forgetsym$ in a concrete way.
It sends any  cooperad $(\coperad,w,\tau)$ to the
planar  cooperad whose underlying sequence is
$\forgetsym (\coperad)$
(that is $\coperad$ without the actions of the symmetric groups), and
equipped with the decomposition
\[
	\begin{tikzcd}[ampersand replacement=\&]
		\forgetsym (\coperad)
		\arrow[r,"\forgetsym (w)"]
		\&
		\forgetsym \left(\coperad \compofsymseq \coperad \right)
		\dar{\unitsymadj\compofsymseq\unitsymadj}
		\&
		\\
		\&
		\forgetsym \left(
		\left(\pltosym{\forgetsym(\coperad)}\right)\compofsymseq
		\left(\pltosym{\forgetsym(\coperad)}\right)\right)
		\dar[equal]
		\&
		\\
		\&
		\forgetsym \left( \pltosym{\left(\forgetsym(\coperad)
		\compofplanarseq \forgetsym(\coperad) \right)}\right)
		\rar{\counitsymadj}
		\&
		\forgetsym(\coperad)
		\compofplanarseq \forgetsym(\coperad). 
	\end{tikzcd}
\]

\begin{remark}%
\label{rmk:droppl}
Since the functor $\pltosymfunctor$ commutes to all operations, as
with planar operads, one can
fully study the planar theory of  cooperads through the
symmetric one.  This allows us to drop the
`pl' subscript.
The `pl' subscript shall be used only when emphasising the relation
between the symmetric framework and the planar framework.
\end{remark}

\begin{lemma}%
\label{thm:planar_trees_vs_trees}
One has an equivalence of cooperads
\[
	\treemodule(\pltosym M)
	\isonat \pltosym{\left(\planartreemodule M\right)}
\]
natural in $M \in \linearlogos^\naturals$.
\end{lemma}

\begin{proof}
We can prove this by induction on the height of the trees.
For any natural integer $n$ let us denote by
$\treemoduleheight n (\pltosym M)$
be the sub cooperad of 
$\treemodule(\pltosym M)$
made up of the trees whose height is equal or lower
than $n$. We define the sub planar cooperad
$\planartreemoduleheight n M$ of $\planartreemodule M$
in the same way.
There exist natural isomorphisms of pointed cooperads
\[
	\treemoduleheight 0 (\pltosym M)
	\isonat  \monoidalunit
	\isonat \pltosym{\left(\planartreemoduleheight 0 M\right)}.
\]
Suppose that we have built an isomorphism of pointed cooperads
\[
	\treemoduleheight n (\pltosym M)
	\isonat \pltosym{\left(\planartreemoduleheight n M\right)},
\]
for a natural integer $n$.
Then, we have an isomorphism of
pointed symmetric sequences
\begin{align*}
 	\treemoduleheight{n+1} (\pltosym M)
	&\isonat \monoidalunit \oplus ( \pltosym M ) \compofsymseq
	\treemoduleheight n (\pltosym M) \\
	&\isonat \monoidalunit \oplus  (\pltosym M) \compofsymseq
	\left(\pltosym{\left(\planartreemoduleheight n M\right)} \right) \\
	&\isonat \monoidalunit \oplus \pltosym{\left( M \compofplanarseq
	\planartreemoduleheight n M \right)}\\
	&\isonat \pltosym{\left( \planartreemoduleheight{n+1}  M \right)} .
\end{align*}
A straightforward checking shows that
this is actually an
isomorphism of pointed cooperads.
\end{proof}

\begin{remark}
One can also prove that
the cooperad $\pltosym{\left(\reducedpltreemodule M\right)}$
is conilpotent. Then
both functors
\begin{align*}
 	M &\mapsto \reducedtreemodule(\pltosym M)
	\\
	M &\mapsto 
	\pltosym{\left(\reducedpltreemodule M\right)}
\end{align*}
are right adjoint the forgetful functor from
conilpotent cooperads to $\naturals$-modules.
$\forget$. Thus, they are isomorphic.
\end{remark}

As a straightforward consequence, the symmetrisation functor
$\pltosymfunctor$ preserves locally conilpotent cooperads.

\begin{proposition}%
\label{thm:symmetrisation_of_loc_conilpotent_coperads}
The following diagram is commutative
\[
	\begin{tikzcd}[ampersand replacement=\&]
		\catofcomod\reducedtreemonad
		\left(\linearlogos^\naturals\right)
		\arrow[rr, "\pltosymfunctor"]
		\arrow[d, hook]
		\&
		\&
		\catofcomod\reducedtreemonad
		\left(\symseq\linearlogos\right)
		\arrow[d, hook] \\
		\catofcomon_\bullet
		\left(\linearlogos^\naturals, \compofseq, \monoidalunit\right)
		\arrow[rr, "\pltosymfunctor"]
		\&
		\&
		\catofcomon_\bullet
		\left(\symseq\linearlogos, \compofsymseq, \monoidalunit\right)
	\end{tikzcd}
\]
\end{proposition}
\section{Algebras and coalgebras}
%=======================================================================

\begin{center}
\begin{tikzpicture}
	\node[draw = black, rectangle, inner sep = 8pt, rounded corners]
	{%
	\begin{minipage}{0.8\textwidth}
		\begin{center}
			In this section, we shall be given $(\operad, m,
			\eta)$ an operad and $(\coperad, w, \tau)$ a  cooperad.
		\end{center}
	\end{minipage}
	};
\end{tikzpicture}
\end{center}

The category $\linearlogos$ is tensored over 
$(\symseq\linearlogos, \compofsymseq, \monoidalunit)$. The tensorisation
functor $ \compofsymseq : \symseq\linearlogos \times
\linearlogos
\to \linearlogos$ is given by the formula
$ M \compofsymseq X \coloneqq \bigoplus_{n\in\naturals} M(n)
\otimes_{\symgroup n} X^{\otimes n}$,
where $X^{\otimes n}$ is given its natural structure of left
$\symgroup n$-module.

One can then define the category of $\operad$-algebras as the category
of modules over the monoid $\operad$. That is a $\operad$-algebra is
the data of an object $\Lambda \in \linearlogos$ endowed with a
structure map $\operad \compofsymseq \Lambda \to \Lambda$ subject to
compatibility conditions. Likewise, one can define $\coperad$-coalgebras
given by structure maps $V \to \coperad \compofsymseq V$.

We shall now describe in a dual way the categories of $\operad$-coalgebras
and $\coperad$-algebras.

\subsection{Definitions}
%-----------------------------------------------------------------------

\begin{definition}%
\label{def:cotensor}
Let $\wedge : (\symseq\linearlogos)\op\times
\linearlogos \to \linearlogos$ be the functor defined as
\[
	X^M \coloneqq \prod_{n\in\naturals}
	{\left[M(n),X^{\otimes n}\right]}^{\symgroup n},
\]
where $X^{\otimes n}$ is given its natural structure of right
$\symgroup n$-module.

We shall call it the cotensor.
\end{definition}

\begin{definition}
Let $\linearlogos$ be a category and $(\catofenrichement, \compofsymseq,
\monoidalunit)$ be a monoidal category. We shall say that $\linearlogos$
is lax cotensored over $\catofenrichement$ if there is given a functor
\[
  \wedge : \catofenrichement\op \times \linearlogos \to \linearlogos
\]
and two natural transformations $l$ and $\varepsilon$,
\[
	\begin{tikzcd}[ampersand replacement=\&]
		\catofenrichement\op \times
		\catofenrichement\op \times \linearlogos 
		\arrow[rr, "\id {\catofenrichement\op} \times \wedge"]
		\arrow[d, "\compofsymseq \times \id {\linearlogos}", swap]
		\& \& \catofenrichement\op \times \linearlogos
		\arrow[d, "\wedge"]
		\arrow[lld, Rightarrow,"l", swap] \\
		\catofenrichement\op \times \linearlogos 
		\arrow[rr, "\wedge", swap]
		\& \& \linearlogos ;
	\end{tikzcd}
\]
\[
	\begin{tikzcd}[ampersand replacement=\&]
		\linearlogos \ar[rr, "\monoidalunit  \times  \id {\linearlogos}"]
		\arrow[rrrr, bend right, "\id {\linearlogos}"',""{name=E}] 
		\& \& \catofenrichement\op
		\times \linearlogos \arrow[Leftarrow, to=E,
		"\varepsilon~"']
		\ar[rr, "\wedge"]
		\& \& \linearlogos ,
	\end{tikzcd}
\]
such that
\[
	\begin{tikzcd}[ampersand replacement=\&]
		{\left({\left(X^M\right)}^N\right)}^S
		\arrow[rr, "{l(N{,}M{,}X)}^S"]
		\arrow[d, "l\left(Q{,} N{,}X^M\right)", swap]
		\&
		\& {\left(X^{N \compofsymseq M}\right)}^S
		\arrow[rr, "l(Q{,}N\compofsymseq M{,}X)"] 
		\&
		\&
		X^{Q \compofsymseq (N \compofsymseq M)}
		\arrow[d, "X^{\alpha(Q{,}N{,}M)}"]
		\arrow[d, shift right=1.5, phantom, "\rotatebox{90}{\(\iso\)}"] \\
		{\left(X^M\right)}^{Q \compofsymseq N}
		\arrow[rrrr, "l(Q\compofsymseq N{,} M {,} X)", swap]
		\&
		\&
		\&
		\&
		X^{(Q \compofsymseq N) \compofsymseq M};
	\end{tikzcd}
\]
\begin{align*}
&
	\begin{tikzcd}[ampersand replacement=\&]
		X^M \arrow[rr, "{\varepsilon(X)}^M"] 
		\arrow[rrrr, bend right, "X^{\rho(M)}"']
		\arrow[rrrr, bend right, shift left=1.5, phantom, "\iso"]
		\& \& {\left(X^{\monoidalunit}\right)}^M
		\arrow[rr, "l(M{,}\monoidalunit{,}X)"]
		\& \& X^{M \compofsymseq \monoidalunit}
	\end{tikzcd}
\\
\text{and}\quad &
	\begin{tikzcd}[ampersand replacement=\&]
		X^M \arrow[rr, "\varepsilon(X^M)"] 
		\arrow[rrrr, bend right, "X^{\lambda(M)}"']
		\arrow[rrrr, bend right, shift left=1.5, phantom, "\iso"]
		\& \& {\left(X^M\right)}^{\monoidalunit}
		\arrow[rr, "l(\monoidalunit{,}M{,}X)"]
		\& \& X^{\monoidalunit \compofsymseq M}
	\end{tikzcd}
\end{align*}
are commutative for every $(Q,N,M,X) \in \setofobj (\catofenrichement
\times \catofenrichement \times \catofenrichement \times \linearlogos)$
and where $\monoidalunit$ is the monoidal unit of $\catofenrichement$;
$\alpha(Q,N,M) : (Q \compofsymseq N) \compofsymseq M \to Q \compofsymseq
(N \compofsymseq M)$ is the monoidal associator; $\lambda(M) :
\monoidalunit \compofsymseq M \to M$ is the left unitor and $\rho(M) : M
\compofsymseq \monoidalunit \to M$ is the right unitor.
\end{definition}

\begin{proposition}%
\label{thm:cotenseur_sym_lax}
Let $X$ be an object of $\linearlogos$, then the functor
\[
	X^{(-)} : \left(\symseq\linearlogos, \convolution,
	\monoidalunit_\convolution\right)\op
	\longrightarrow (\linearlogos, \otimes, \monoidalunit)
\]
is naturally endowed with a lax symmetric structure, where
\[
	X^{\monoidalunit_\convolution} \isonat \monoidalunit
\]
and for a pair $(M,N)$ of symmetric sequences, the lax map
\[
	\laxmap^\convolution(M,N,X) : X^M \otimes X^N \longrightarrow
	X^{M \convolution N}
\]
is obtained by distribution:
\[
	\begin{tikzcd}
		X^M \otimes X^N \isonat \displaystyle\prod_{p\in\naturals}
		{\internalhom {M(p)} {X^{\otimes p}}}^{\symgroup p}
		\otimes \prod_{q\in\naturals}
		{\internalhom {N(q)} {X^{\otimes q}}}^{\symgroup q}
		\dar
		\\
		\displaystyle
		\prod_{p,q\in\naturals}\fullpower{
		\internalhom{M(p)}{X^{\otimes p}} \otimes
		\internalhom{N(q)}{X^{\otimes q}}
		}{\symgroup p \times \symgroup q}
		\dar
		\\
		\displaystyle
		\prod_{p,q\in\naturals}\power{
		\internalhom{M(p) \otimes N(q)}{X^{\otimes p+q}}
		}{\symgroup p \times \symgroup q}
		\isonat X^{M \convolution N}.
	\end{tikzcd}
\]
\end{proposition}

\begin{proof}
Since limits distribute over bifunctors in an associative manner, the
associativity of the lax map results from the commutativity of the
diagram
\[
	\begin{tikzcd}[ampersand replacement=\&]
		[A, B] \otimes [C,D] \otimes [E,F]
		\arrow[r]
		\arrow[d]
		\&
		{[A,B]} \otimes [C \otimes E, D \otimes F]
		\arrow[d] \\
		{[A \otimes C, B \otimes D]} \otimes [E,F]
		\arrow[r]
		\& {[A \otimes C \otimes E,
		B \otimes C \otimes F]} .
	\end{tikzcd}
\]
for any objects $A, B, C, D, E, F$ of $\linearlogos$, which itself
comes from the associativity of the tensor structure on $\linearlogos$.
Likewise, the symmetry of the lax map is a direct
consequence of the symmetry of the tensor structure on $\linearlogos$.
\end{proof}

\begin{corollary}
The category $\linearlogos$ is naturally lax cotensored over the
category of symmetric sequences
$(\symseq\linearlogos, \compofsymseq, \monoidalunit_{\compofsymseq})$,
where for $X \in \linearlogos$
\[
	X^{\monoidalunit_{\compofsymseq}} \isonat X
\]
and for two symmetric sequences $M,N$ the lax map $\power {(\power X N)}
M \to X^{M\compofsymseq N}$ is given by
\[
	\laxmap(M,N,X) \coloneqq \prod_n \power{\internalhom{M(n)}
	{\laxmap^\convolution(N,\dots,N,X)}}{\symgroup n}.
\]
\end{corollary}

\begin{remark}%
\label{rmk:cotenseur_fonctoriel}
By assumptions%
~[\sectionref{sec:characteristic_zero}],
the cotensor is continuous in the first variable and
if both left and right exact in each variable.
\end{remark}

\begin{definition}
The category $\catofalg\coperad$ of $\coperad$\=/algebras is the
category of modules over $\coperad$. That
is, an algebra over $\coperad$ is the data of an object $\Lambda$ and a
map
\[
	a : \Lambda^\coperad \longrightarrow \Lambda ,
\]
such that the following diagrams commute:
\[
	\begin{tikzcd}[ampersand replacement=\&]
		{\left(\Lambda^\coperad\right)}^\coperad 
		\arrow[d, "a^\coperad"']
		\arrow[rr, "\laxmap(\coperad{,}\coperad{,}\Lambda)"]
		\&\& \Lambda^{\coperad\compofsymseq\coperad}
		\arrow[r, "\Lambda^w"]
		\& \Lambda^\coperad
		\arrow[d,"a"]
		\\
		\Lambda^\coperad \arrow[rrr, "a"']
		\&\&\&
		\Lambda ;
	\end{tikzcd}
\]
\[
	\begin{tikzcd}[ampersand replacement=\&]
		\Lambda
		\arrow[r, "\varepsilon(\Lambda)"]
		\arrow[r, shift right=1.7, phantom, "\iso"]
		\arrow[d, equal]
		\& \Lambda^{\monoidalunit}
		\arrow[d, "\Lambda^\tau"] \\
		\Lambda
		\& \Lambda^\coperad.
		\arrow[l, "a"]
	\end{tikzcd}
\]
A morphism of $\coperad$-algebras $(\Lambda, a) \longrightarrow
(\Gamma,b)$ is the data of a map $f : \Lambda \to \Gamma$ such that the
following diagram commutes:
\[
	\begin{tikzcd}[ampersand replacement=\&]
		\Lambda^\coperad 
		\arrow[r, "f^\coperad"]
		\arrow[d, "a", swap]
		\& \Gamma^\coperad
		\arrow[d, "b"] \\
		\Lambda
		\arrow[r, "f", swap]
		\& \Gamma.
	\end{tikzcd}
\]
\end{definition}

\begin{definition}%
\label{def:categorie_des_cogebres}
The category of $\operad $\=/coalgebras, that we shall denote
by $\catofcog\operad$, is the category whose objects
are pairs $(V,a)$ of an object $V$ and a map
\[
	a : V \longrightarrow V^\operad ,
\]
such that the following diagrams are commutative:
\[
	\begin{tikzcd}[ampersand replacement=\&]
		V \arrow[r, "a"] \arrow[d, "a"']
		\& V^\operad  \arrow[r, "a^\operad "]
		\& {\left(V^\operad \right)}^\operad 
		\arrow[d,"\laxmap(\operad {,}\operad {,}V)"]
		\\
		V^\operad  \arrow[rr, "V^m"]
		\& \&
		V^{\operad \compofsymseq\operad } ;
	\end{tikzcd}
\]
\[
	\begin{tikzcd}[ampersand replacement=\&]
		V \arrow[r, "a"]
		\arrow[rr, bend right, "\varepsilon(V)"']
		\arrow[rr, bend right, shift left=1.5, phantom, "\iso"]
		\& V^\operad  \arrow[r, "V^\eta"]
		\& V^{\monoidalunit}
	\end{tikzcd}
\]
and whose morphisms $(V, a) \to (W,b)$ are the maps $f : V \to W$ such
that
\[
	\begin{tikzcd}[ampersand replacement=\&]
		V^\operad  
		\arrow[r, "f^\operad "]
		\arrow[d, "a", swap]
		\& W^\operad 
		\arrow[d, "b"] \\
		V 
		\arrow[r, "f", swap]
		\& W
	\end{tikzcd}
\]
commutes.
\end{definition}

The definition of $\operad$-coalgebras above coincide with the one
given in \emph{Algebraic Operads}%
\cite[5.2.17]{doi:10.1007/978-3-642-30362-3}.

\begin{theorem}%
\label{thm:def_of_coalgebras}
The data of a $\operad$-coalgebra structure on $V$ amounts to the data
of an operad morphism
\[
	\operad \longrightarrow \coend V.
\]
\end{theorem}

\begin{proof}
On the one hand, a morphism of symmetric sequence
$f :\operad \to \coend V$
amounts to the data of $\symgroup n$-equivariant morphisms
\[
	g(n) : \operad (n)  \otimes V \to V^{\otimes n}.
\]
given by the composition
\[
	\operad (n)  \otimes V \to V^{\otimes n}
	\xrightarrow{f(n) \otimes \id{}} [V, V^{\otimes n} ] \otimes V
	\to V^{\otimes n}.
\]
For $k, i_1, \ldots, i_k, n$ natural integers such that $n=
i_1 + \cdots +i_k$, one can draw
\[
\begin{tikzcd}
	\operad (k) \otimes \left( \operad (i_1)
	\otimes \cdots \otimes \operad (i_k)\right) \otimes V
	\arrow[rr,"m\otimes \id{}"] \arrow[d,"g(k)\otimes \id{}"']
	&& \operad (n) \otimes V
	\arrow[d,"g(n)"]
	\\
	\left( \operad (i_1) \otimes \cdots \otimes \operad (i_k)\right)
	\otimes V^{\otimes k}
	\arrow[rr,"g(i_1) \otimes \cdots \otimes g(i_k)"']
	&& V^{\otimes n},
\end{tikzcd}
\]
which commutes for every such family of natural integers if
and only if $f$ is a morphism of operads.

On the other hand, equivariant
morphisms $g(n): \operad (n) \otimes V \to V^{\otimes n}$ amount to the
data of $\symgroup n$-equivariant morphisms
\[
	h(n): V \to \internalhom{\operad(n)}{V^{\otimes n}},
\]
given by the composition
\[
	V \to [\operad(n), \operad(n) \otimes V]
	\xrightarrow{[\id{}, g(n)]}
	\internalhom {\operad(n)} {V^{\otimes n}}.
\]
Since the action of  $\symgroup n$ on $V$ is trivial, the
$\symgroup n$-equivariance of $h(n)$ is equivalent to the fact that 
it factors through the invariants
\[
	\internalhom{\operad(n)}{V^{\otimes n}}^{\symgroup n}.
\]
Besides, the above square involving $g$ commutes if and only if
its adjoint square
\[
\begin{tikzcd}
	V
	\arrow[rr,"h(k)"] 
	\arrow[dd,"h(n)"']
	&& \internalhom{\operad (k)}{V^{\otimes k}}
	\arrow[d,"\internalhom{\idfunctor}{h(i_1) \otimes
	\cdots \otimes h(i_k)}"] 
	\\
	&&\internalhom{\operad (k)} {\internalhom{\operad (i_1) \otimes \cdots
	\otimes \operad (i_k)}{V^{\otimes n}}}
	\arrow[d,equal]
	\\
	\internalhom{\operad (n)}{V^{\otimes n}}
	\arrow[rr,"{[m, \id{}]}"] 
	&&
	\internalhom {\operad (k) \otimes \left( \operad (i_1) \otimes \cdots
	\otimes \operad (i_k)\right)}{V^{\otimes n}}
\end{tikzcd}
\]
commutes.
Finally, the maps $h(n)$ can be gathered into a map
$h : V \to V^\operad$. It gives $V$ a $\operad$-coalgebra structure if
and only if the above square commutes for any natural integers $k,
i_1,\ldots, i_k$.
\end{proof}

\paragraph{Summary}
We summarise the four different notions we have just defined by a table
of their structure maps.

\begin{sumtable}%
	\label{fig:tableau_algebres_cogebres}
	\begin{tikzpicture}
		\matrix (mat) [matrix of nodes, inner sep=0cm,
			nodes={align=center, text width=3cm,
			inner sep=0.3cm, text depth=0.2cm, text height=0.4cm}]{
			\emph{The musketeers} & over $\operad$
			& over $\coperad$ \\
			Algebra $\Lambda$ & $\operad \compofsymseq\Lambda
			\longrightarrow \Lambda$
			& $\Lambda^\coperad \longrightarrow \Lambda$ \\
			Coalgebra $V$ & $V \longrightarrow V^\operad $ 
			& $V \longrightarrow \coperad\compofsymseq V$ \\
		};
		\draw[gray, thick] (mat-1-2.north west) -- (mat-3-1.south east);
		\draw[gray, thick] (mat-1-3.north west) -- (mat-3-3.south west);
		\draw[gray, thick] (mat-2-1.north west) -- (mat-2-3.north east);
		\draw[gray, thick] (mat-3-1.north west) -- (mat-3-3.north east);
		\draw[gray, thick, rounded corners=.5em]
		(mat.north west) rectangle (mat.south east);
	\end{tikzpicture}
\end{sumtable}

\subsection{Free functors}
%----------------------------------------------------------------------

For each of the four notions that we have been describing, there is
a corresponding `free object' generated by some object in the ground
category $\linearlogos$. However, not all notions are equal: although
it is always possible to build the free algebra on an operad \& a
 cooperad or the coalgebra over a  cooperad, it is not
clear how to build a cofree coalgebra over an operad. For this we need to
add special assumptions on the category $\linearlogos$.

Thanks to the tensorisation of $\linearlogos$ over the category
of sequences, the functor $\freefunctor_\operad \coloneqq \operad
\compofsymseq (-)$ can be endowed with the structure of a monad
$\monadl_\operad$, the category of $\operad$-algebras turns out to be
equivalent to the one of $\monadl_\operad$-modules. Likewise,
the category of $\coperad$-coalgebras is equivalent to the category
of comodules over a comonad with underlying functor
$\freefunctor_\coperad \coloneqq \coperad \compofsymseq (-)$.

The functor $\freecog \coperad
\coloneqq{(-)}^\coperad$ can be also endowed with the structure
of a monad $\monadl^\coperad \coloneqq (\freecog \coperad, \power{(-)}w,
\power{(-)}\tau)$, thanks to the fact that the cotensorisation is a lax
functor, hence it sends algebras to algebras.

\begin{theorem}%
\label{thm:algebre_libre_sur_une_coperade}
The category of $\coperad$-algebras is equivalent to the category of
modules over the monad $\monadl^\coperad$.

As a consequence, for every object $X$ of $\linearlogos$,
the free algebra over $\coperad$ is given by $X^\coperad$ together with
the map
\[
	\begin{tikzcd}[ampersand replacement=\&]
		{\left(X^\coperad\right)}^\coperad
		\arrow[rr, "\laxmap(\coperad{,} \coperad{,} X)"]
		\&
		\&
		X^{\coperad \compofsymseq \coperad}
		\arrow[r, "{(\id X)}^w"]
		\&
		X^\coperad.
	\end{tikzcd}
\]
\end{theorem}

\subsection{Cofree coalgebra over an operad}
%-----------------------------------------------------------------------

The case of coalgebras over the operad $\operad $ can’t be treated in the
same way. Indeed, because the cotensorisation is only lax, the functor
${(-)}^{\operad}$ cannot be given the structure of a comonad.

The problem of constructing the cofree coalgebra over a linear operad has
been solved by Smith%
~\cite{doi:10.1016/s0166-8641(03)00037-3}
following ideas from Fox%
~\cite{doi:10.1016/0022-4049(93)90038-u}. In what follows, we use
the more conceptual approach of Anel%
~\cite{arXiv:1409.4688} which we rephrase using the cotensor.

\begin{theorem}[(Rewrite conditions)]%
\label{thm:conditions_de_reecriture}
Assume that
\begin{itemize}
	\item for any four objects $A, B, C, D$ of $\linearlogos$, the
	      natural map
	      \begin{equation*}
	      	[A,B] \otimes [C,D] \longrightarrow [A \otimes
	      	C, B \otimes D]
	      	\tag{$\otimes \leftrightarrow [-]$}
	      \end{equation*}
	      is a monomorphism;
	\item the tensor product commutes with countable intersections. That
	      is, given an object $A$ and a countable sequence of subobjects
	      \[
	      	\cdots \subobject B_n \subobject \cdots \subobject B_1
	      	\subobject B_0
	      \]
	      one has the equality
	      \begin{equation*}
	      	A \otimes
	      	\bigcap_{n\in\naturals} B_n
	      	= \bigcap_{n\in\naturals}
	      	A \otimes B_n.
	      	\tag{$\otimes \leftrightarrow \cap$}
	      \end{equation*}
\end{itemize}
Then the category $\catofcog\operad(\linearlogos)$ of $\operad
$-coalgebras in $\linearlogos$, is equivalent to the category of comodules
over a comonad
\[
	\comonadl^\operad \coloneqq
	\left(\freecog\operad , \coprodw^\operad , \counitw^\operad \right).
\]
which can be computed via a specific algorithm, described below.
\end{theorem}

This is proved by Anel in \emph{Cofree coalgebras over operads and
representative functions}%
~\cite[Hypothesis 2.7.5]{arXiv:1409.4688}.
The assumptions on $(\linearlogos, \otimes, \monoidalunit)$ guaranty
the following facts. The first rewrite condition ensures that the lax
map
\[
	\laxmap(N, M, X) : {\left(X^M\right)}^N
	\longrightarrow X^{N \compofsymseq  M}
\]
becomes a monomorphism for every objects $X, M, N$.

The second rewrite condition ensures that given a symmetric sequence
$M$, the cotensor ${(-)}^M$ preserves countable intersections in the
sense that it preserves monomorphisms and given a sequence of subobjects
\[
	\cdots \subobject Y_n \subobject \cdots \subobject Y_1
	\subobject Y_0
\]
we get the equality
${\left(\bigcap_{n\in\naturals} Y_n\right)}^M
=
\bigcap_{n \in\naturals} Y_n^M$.

Let us spend a few words on why those conditions may be seen as natural.
Given an operad $(\operad , m, \eta)$
in $\linearlogos$ and a coalgebra $(V,a)$ on $\operad $,
the associativity condition reads
\[
	\begin{tikzcd}[ampersand replacement=\&]
		V \arrow[r, "a"] \arrow[d, "a"']
		\& V^\operad  \arrow[r, "a^\operad "]
		\& {\left(V^\operad \right)}^\operad 
		\arrow[d,"\laxmap(\operad {,}\operad {,}V)"]
		\\
		V^\operad  \arrow[rr, "V^m"]
		\& \&
		V^{\operad \compofsymseq\operad }
	\end{tikzcd}
\]
The fact that the lax map $\laxmap$ goes in the ‘wrong’ direction is
what stops us from having a comonadic structure for the functor
${(-)}^\operad $. Thanks to the first rewrite condition, we see that
the composition $V^m \circ a$ actually lends in
$\power{(V^\operad )}\operad $
which allows the rewriting to take place. In this case, given an object
$X$ in $\linearlogos$, $\freefunctor_0^\operad (X)
\coloneqq X^\operad $ is very
close to be the cofree coalgebra on $X$. A putative candidate would be the
fibre product
\[
	\begin{tikzcd}[ampersand replacement=\&]
		\freefunctor_1^\operad (X)
		\arrow[r, "\deltaup_1"]
		\arrow[d, "\iotaup_1"',hookrightarrow]
		\arrow[rd, very near start, phantom, "\lrcorner"]
		\& {\left(X^\operad \right)}^\operad 
		\arrow[d, "\laxmap(\operad {,}\operad {,}X)", hookrightarrow] \\
		X^\operad 
		\arrow[r, "X^m"]
		\& X^{\operad  \compofsymseq \operad } .
	\end{tikzcd}
\]
In some cases this is the correct candidate.

\begin{proposition}[{\cite[Corollary 3.3.2]{arXiv:1409.4688}}]%
\label{thm:algorithme_arrete_cran_un}
Let $\operad$ be an operad in the category of chain complexes over a
field. Then the category of $\operad$-coalgebras is comonadic and
the underlying functor of the comonad $\comonadl^{\operad}$ is
given by $\freefunctor_1^{\operad}$.
\end{proposition}

In the
general case the process must be iterated until it stabilises.
For $n \geq 1$, draw the diagram:
\[
\label{fig:free_coalgebra_algorithm}
\begin{tikzpicture}
\node[draw = black, rectangle, rounded corners] (b) at (1.2,-0.5) {
$\thealgo$
};
\node (a) at (0,0){
	\begin{tikzcd}[ampersand replacement=\&]
	\freefunctor_{n+1}^\operad 
	\arrow[d, "\iotaup_{n+1}"']
	\arrow[r, "\deltaup_{n+1}"]
	\&
	\freefunctor_n^\operad  \circ \freefunctor_n^\operad 
	\arrow[rr, shift left,
	"\iotaup_n\circ \deltaup_n"]
	\arrow[rr, shift right,
	"\deltaup_n\circ\iotaup_n"']
	\arrow[d, "\iotaup_n\circ \iotaup_n"]
	\&\&
	\freefunctor_{n-1}^\operad  \circ \freefunctor_{n-1}^\operad
	\circ \freefunctor_{n-1}^\operad 
	\\
	\freefunctor_n^\operad 
	\arrow[r, "\deltaup_n"]
	\&
	\freefunctor_{n-1}^\operad  \circ \freefunctor_{n-1}^\operad  .
	\&\&
	\end{tikzcd}
};
\end{tikzpicture}
\]

\begin{remark}
The symbol $\circ$ here denotes the so-called `horizontal composition'
of natural transformations.
\end{remark}

The functor $\freefunctor_{n+1}^\operad $ is defined as the limit, in
the category
of functors, of the remaining diagram, this defines also the natural
transformations $\iotaup_{n+1}$ and $\deltaup_{n+1}$. This is where the
second rewrite condition is needed: it allows passing to the limit in
this construction.

\begin{proposition}[{\cite[Corollary 2.7.12]{arXiv:1409.4688}}]%
\label{thm:LOmega-intersection}
Let $\operad $ be an operad in $\linearlogos$, suppose the rewrite
conditions are met. Then the category of $\operad$-coalgebras is
comonadic and the underlying functor of the comonad $\comonadl^\operad
$ is the intersection $\freecog\operad \coloneqq \bigcap_{n\in\naturals}
\freefunctor_n^\operad $.
\end{proposition}

\begin{remark}
For any $X$, the map $\coprodw^\operad (X) : \freecog\operad (X) \to
\freecog\operad \circ \freecog\operad (X)$ determines the $\operad
$-coalgebra structure of $\freecog\operad (X)$,
\[
	\begin{tikzcd}[ampersand replacement=\&]
		\freecog\operad (X)
		\arrow[r, "\canonicalaction"]
		\arrow[rd, "\coprodw^\operad (X)", swap]
		\&
		{\left(\freecog\operad (X)\right)}^\operad 
		\\
		\& \freecog\operad  \circ \freecog\operad (X) .
		\arrow[u, hookrightarrow]
	\end{tikzcd}
\]
hence we shall abuse notations and speak of $\canonicalaction$ when we
mean $\coprodw^\operad $.
\end{remark}

\paragraph{Summary}%
\label{fig:tableau_algebres_cogebres_libres}

Let us summarise the different constructions
of free functors.

\begin{sumtable}
\begin{tikzpicture}
	\matrix (mat) [matrix of nodes, inner sep=0cm,
		nodes={align=center, text width=3cm, text depth=0.2cm,
		text height=0.4cm, inner sep=0.3cm}]{
		\emph{Free} [...] & over $\operad$
		& over $\coperad$ \\
		Algebra & $\operad  \compofsymseq X$
		& $X^\coperad$ \\
		Coalgebra & $\freecog\operad (X) \subobject X^\operad $ 
		& $\coperad \compofsymseq X$ \\
	};
	\draw[gray, thick] (mat-1-2.north west) -- (mat-3-2.south west);
	\draw[gray, thick] (mat-1-3.north west) -- (mat-3-3.south west);
	\draw[gray, thick] (mat-2-1.north west) -- (mat-2-3.north east);
	\draw[gray, thick] (mat-3-1.north west) -- (mat-3-3.north east);
	\draw[rounded corners=.5em, gray, thick]
	(mat.north west) rectangle (mat.south east);
\end{tikzpicture}
\end{sumtable}

\subsection{Presentability}
%-----------------------------------------------------------------------

\begin{theorem}%
\label{thm:les_mousquetaires_sont_presentables}
The category
$\catofalg\coperad(\linearlogos)$ is presentable.
Assuming rewrite conditions on $(\linearlogos, \otimes, [-,-])$ are
met, $\catofcog\operad(\linearlogos)$ is also
presentable.
\end{theorem}

\begin{proof}
Categories of modules over an accessible monad \& categories of
comodules over an accessible comonad are known to be presentable when
the ground category is%
~\cite{doi:10.1017/CBO9780511600579,doi:10.4310/HHA.2014.v16.n2.a9}.
Hence we shall prove that both free functors are accessible.

Let $Y$ be any sequence and let us show that the functor ${(-)}^Y :
\linearlogos \longrightarrow \linearlogos$ is accessible. For this
we shall use the fact that because $\linearlogos$ is presentable, every
object of $\linearlogos$ is small for a certain cardinal. Let $\kappa$
be an infinite cardinal such that $\linearlogos$ is $\kappa$-accessible.
The accessibility of ${(-)}^Y$ reduces to the following sequence of
facts:
\begin{itemize}
	\item The functor $X \mapsto X^{\otimes n}$
	      commutes with sifted colimits;
	\item For every $n$, the functor
	      $[Y(n), - ] : \linearlogos \longrightarrow \linearlogos$
	      is a right adjoint, hence there exists a cardinal $\kappa_n$
	      such that this functor is $\kappa_n$\=/accessible;
	\item When the indexing category is finite --- such as
	      $\symgroup n$ --- the end functor becomes
	      $\ordinalomega$-accessible
	      since ends on finite categories can be computed with finite
	      limits; \item For any object $Z$ of $\linearlogos$, the functor
	      $Z \times (-)$ is $\kappa$-accessible;
	\item countable limits commute with $\tau$-filtered colimits in 
	      $\linearlogos$ for $\kappa \ll \tau$, since $\kappa$ is
	      infinite.
\end{itemize}
We choose a cardinal $\tau$ such that $\kappa \ll \tau$ and $\kappa_n
\ll \tau$ for every $n$. Then all functors described above become
$\tau$-accessible and $\tau$-filtered colimits commute with countable
limits. From this, we deduce that ${(-)}^Y$ is $\tau$-accessible.

Now remains the case of the functor $\freecog Y$ assuming the
rewrite conditions. We have already proved that $\freefunctor_0^Y$
is accessible. The functors $\freefunctor_n^Y$ are then inductively
built using limits of accessible functors, hence they are accessible
for every integer $n$ since limits are computed pointwise in functor
categories. For the same reason the intersection $\freecog Y =
\bigcap_{n\in\naturals} \freefunctor_n^Y$ is accessible.
\end{proof}

\subsection{External tensor product of coalgebras}
%-----------------------------------------------------------------------

Given two operad $\operad $ and $\operad '$ the arity-wise tensor
product $\operad  \diagtensor \operad '$ is defined in arity $n$ by
\[
	(\operad  \diagtensor \operad ') (n) \coloneqq \operad (n)
	\otimes \operad '(n).
\]
endowed with a diagonal action of $\symgroup n$ with unit
the operad $\ucom$. The resulting
sequence is naturally endowed with a structure of operad induced by
those of $\operad $ and $\operad '$. This is a consequence of the
fact that $\ucom$ has a unit
\[
	\unitofcom : \monoidalunit_{\compofsymseq} \longrightarrow \ucom
\]
and that the $\compofsymseq$ tensor is colax with respect to this new
tensor product: there are natural exchange maps of symmetric sequences
\[
	\exchangemap : (A \diagtensor B) \compofsymseq (C \diagtensor D)
	\longrightarrow (A \compofsymseq C)
	\diagtensor (B \compofsymseq D) .
\]
for any tuple $(A,B,C,D)$ of symmetric sequences; they satisfy the
usual relations of a 2-monoidal structure%
~\cite[Proposition 6]{doi:10.1515/CRELLE.2008.051}.

\begin{proposition}
The cotensor
\[
	\wedge : {(\symseq\linearlogos, \diagtensor, \ucom)}\op \times
	(\linearlogos, \otimes, \monoidalunit) \longrightarrow
	(\linearlogos, \otimes, \monoidalunit)
\]
is lax 2-monoidal \ie{} $(\linearlogos, \otimes, \monoidalunit)$ is
a module over $(\symseq\linearlogos, \diagtensor, \ucom)\op$ in the
2\=/category of monoidal categories with lax functors. In particular
there is a unitor $\unitor : \monoidalunit \to \monoidalunit^\ucom$
and exchange maps
\[
	\exchangemap : X^M \otimes Y^N \longrightarrow
	\fullpower{X\otimes Y}{M \diagtensor N}
\]
natural in $X, Y \in \linearlogos$, $M, N \in (\symseq\linearlogos)\op$.
\end{proposition}

\begin{proof}[Sketch of proof]
This proof is analogous to the proof of the 2\=/monoidal structure of
the tensor%
~\cite[Proposition 6]{doi:10.1515/CRELLE.2008.051}. The
unitor is obtained as $\unitor \coloneqq \monoidalunit^\unitofcom$ and
the exchange map is intimately related to the natural maps
\[
	\internalhom M {X^{\otimes n}} \otimes \internalhom N {Y^{\otimes n}}
	\longrightarrow \internalhom {M\otimes N} {X^{\otimes n} \otimes
	Y^{\otimes n}} \isonat
	\internalhom {M\otimes N}{\fullpower{X \otimes Y}{\otimes n}}
\]
for $X, Y, M, N \in \linearlogos$ and $n \in \naturals$.
\end{proof}

\begin{corollary}
Let $\operad$ and $\operad'$ be two operads.
The lax 2-monoidal nature of the cotensor allows us to define a
bifunctor
\[
	\externaltensor : \catofcog\operad \times \catofcog{\operad'}
	\longrightarrow \catofcog{(\operad\diagtensor\operad')}
\]
that we shall call the \emph{external tensor product} of coalgebras.
Where,
for $(V, a)$ a $\operad$-coalgebra and $(W, b)$ a $\operad'$-coalgebra,
the external tensor product $V \externaltensor W$ is the tensor
product $V \otimes W$ endowed with a
$\operad \diagtensor \operad'$-coalgebra structure given by
\[
	(V \otimes W) \xrightarrow{a \otimes b} (V^\operad \otimes
	W^{\operad'}) \xrightarrow{\exchangemap}
	\fullpower{V\otimes W}{\operad \diagtensor \operad'}.
\]
\end{corollary}
\section{Algebras over  cooperads}
%=======================================================================

Since the monad governing algebras over a  cooperad is better
behaved that the comonad governing
coalgebras over an operad, we shall start by studying the easiest
setting --- algebras over  cooperads --- and
increase the number of required hypotheses gradually.

In this section we shall let $(\coperad, w, \tau)$ be a
 cooperad and $(\Lambda, a)$ be a $\coperad$-algebra.

\subsection{The shuffle notation}%
%-----------------------------------------------------------------------
\label{sec:the_shuffle_notation}

Let $X$ and $Y$ be two objects of $\linearlogos$ and $n \geq 1$
an integer. We shall write ${\sha(X, Y)}^{\otimes n}$ to designate
the object
\[
	{\sha(X, Y)}^{\otimes n} \coloneqq \bigoplus_{i+1+j=n}
	X^{\otimes i} \otimes Y \otimes X^{\otimes j}.
\]
In particular ${\sha(X, Y)}^{\otimes 1} \coloneqq Y$.
Moreover, for any symmetric sequence $M$, we shall write
\[
	{\sha(X, Y)}^M
	\coloneqq \prod_{n \geq 1}{\left[M(n), \bigoplus_{i+1+j=n}
	X^{\otimes i} \otimes Y \otimes
	X^{\otimes j} \right]}^{\symgroup n}.
\]
In particular the zero component is trivial. 
Besides, when $Y \subset X$ we shall denote the image of the induced
map
\[
	\begin{tikzcd}
		{\sha(X, Y)}^M \ar[r, hook] & {\sha(X, X)}^M \ar[r, "+"]
		& X^M.
	\end{tikzcd}
\]
by $\imshuffle X Y M$.
\subsection{Ideals of an algebra over a  cooperad}
%-----------------------------------------------------------------------

\begin{definition}
An ideal of $\Lambda$ is
a subobject $\ideali$ such that the quotient $\Lambda/\ideali$ can be
endowed with a structure of $\coperad$-algebra
\[
	\overline{a} : {(\Lambda/\ideali)}^\coperad \to \Lambda/\ideali
\]
such that the quotient map $\varpiup
: \Lambda \longrightarrow \Lambda/\ideali$ becomes a morphism of
$\coperad$-algebras.
\end{definition}

\begin{remark}
Since $\varpiup^\coperad$ is an epimorphism, the map $\overline{a}$ is
uniquely determined by $a$. Hence we shall often forget to mention it.
\end{remark}

\begin{proposition}%
\label{thm:ideal-associativite-auto}
A subobject of $\ideali$ of $\Lambda$ is an ideal
if and only if there exists a map
\[
	\overline{a} : {(\Lambda/I)}^\coperad \to \Lambda/I
\]
in $\linearlogos$ such that the following square commutes,
\[
	\begin{tikzcd}[ampersand replacement=\&]
		\Lambda^\coperad
		\arrow[r, "\varpiup^\coperad"]
		\arrow[d, "a", swap]
		\& {(\Lambda/I)}^\coperad
		\arrow[d, "\overline{a}"] \\
		\Lambda
		\arrow[r, "\varpiup", swap]
		\& \Lambda/I . 
	\end{tikzcd}
\]
\end{proposition}

\begin{proof}
We already have a map
\[
	\overline{a}: {\left(\Lambda/I\right)}^\coperad \to \Lambda/I.
\]
We have to show that it is actually a $\coperad$-algebra structure.
Consider the following cubical diagram
% Très mauvaise solution pour réduire la taille de ce diagramme.
% Mais bon…
\begin{center}
\begin{tikzpicture}[baseline= (a).base]
\node[scale=.75] (a) at (0,0){
	\begin{tikzcd}
		&\left(\left(\Lambda/I\right)^{\coperad}\right)^{\coperad}
		\arrow[rr, "\laxmap(\coperad{,}\coperad{,}\Lambda/I)"]
		\arrow[ddd, "\overline{a}^\coperad" description, near end]
		&& \left(\Lambda/I\right)^{\coperad\compofsymseq\coperad}
		\arrow[rr, "(\Lambda/I)^w"]
		&& \left(\Lambda/I\right)^\coperad
		\arrow[ddd, "\overline{a}"]
		\\
		\left(\Lambda^{\coperad}\right)^{\coperad}
		\arrow[rr, "\laxmap(\coperad{,}\coperad{,}\Lambda)"
		description, crossing over]
		\arrow[ddd,"a^\coperad"']
		\arrow[ru, twoheadrightarrow, "(\varpiup^\coperad)^\coperad"]
		&& \Lambda^{\coperad\compofsymseq\coperad} \arrow[rr, "\Lambda^w"]
		\arrow[ru, "\varpiup^{\coperad\compofsymseq\coperad}" description]
		&& \Lambda^\coperad
		\arrow[ru, "\varpiup^\coperad"]
		\\
		\\
		&\left(\Lambda/I\right)^\coperad
		\arrow[rrrr, "\overline{a}"] &&&& \Lambda/I
		\\
		\Lambda^\coperad \arrow[rrrr, "a"] \arrow[ru, "\varpiup^\coperad"]
		&&&& \Lambda. \arrow[ru,"\varpiup"']
		\arrow[from=uuu, "a", crossing over]
	\end{tikzcd}
};
\end{tikzpicture}
\end{center}
We have to show that the back face is a commutative square knowing that
all other faces are commutative. This can be proved by diagram chasing
after having noticed that the map
\[
	{\left(\varpiup^\coperad\right)}^\coperad :
	{\left(\Lambda^{\coperad}\right)}^{\coperad}
	\to {\left({\left(\Lambda/I\right)}^{\coperad}\right)}^{\coperad}
\]
is an epimorphism. The same argument applied to the following diagram
ends the proof:
\[
	\begin{tikzcd}
	& \Lambda^{\monoidalunit}
	\arrow[rrr, twoheadrightarrow, "\varpiup^{\monoidalunit}"]
	&&& {(\Lambda/I)}^{\monoidalunit}
	\arrow[ddr, "{(\Lambda/I)}^w"]
	\\
	\Lambda
	\arrow[rrr,"\varpiup", near start]
	\arrow[ru, equal]
	&&& \Lambda/I
	\arrow[ru, equal]
	\\
	&& \Lambda^\coperad
	\arrow[rrr, "\varpiup^\coperad"']
	\arrow[llu, "a" description]
	\arrow[from=uul, crossing over, "\Lambda^\tau", near end]
	&&& {(\Lambda/I)}^\coperad .
	\arrow[llu, "\overline{a}" description]
	\end{tikzcd}
\]
\end{proof}

\begin{lemma}%
\label{lemma_condition_ideal}
A subobject $\ideali$ of $\Lambda$ is an ideal
if and only if
\[
	a \left(\imshuffle \Lambda \ideali \coperad\right)
	\subobject \ideali .
\]
\end{lemma}

\begin{proof}
$\ideali$ is an ideal of
$\Lambda$ if we can find a dotted arrow filling the diagram%
~[\ref{thm:ideal-associativite-auto}]
\[
	\begin{tikzcd}[ampersand replacement=\&]
		\Lambda^\coperad
		\arrow[r, "\varpiup^\coperad"]
		\arrow[d, "a", swap]
		\& {(\Lambda/I)}^\coperad
		\arrow[d, dotted] \\
		\Lambda
		\arrow[r, "\varpiup", swap]
		\& \Lambda/I . 
	\end{tikzcd}
\]
This happens if and only if the composition $\varpiup \circ a$ vanishes
on the kernel of $\varpiup^\coperad$. One has
$\ker \varpiup^\coperad = \smallimshuffle\Lambda\ideali\coperad$, so
one may find the dotted arrow if and only if one has $a (\smallimshuffle
\Lambda\ideali\coperad) \subobject \ideali$.
\end{proof}

\begin{proposition}%
\label{thm:somme-d-ideaux}
Let $\Lambda$ be a $\coperad$-algebra and let $\idealj$ and $\idealk$
be two
ideal of $\Lambda$, then the sum $\idealj + \idealk$, image of the map
$\idealj \oplus \idealk \longrightarrow \Lambda$
is also an ideal of $\Lambda$.
\end{proposition}

\begin{proof}
This follows directly from the previous lemma and the fact that
\[
	\imshuffle \Lambda {\ideali + \idealk} \coperad
	= \imshuffle \Lambda \ideali \coperad
	+ \imshuffle \Lambda \idealk \coperad .
\]
\end{proof}

\begin{proposition}%
\label{thm:image-ideal}
Let $f : \Lambda \to \Gamma$ be a morphism of $\coperad$-algebras. If
$f$ is an epimorphism in $\linearlogos$ and $\ideali$ is an ideal of
$\Lambda$, then $f(\ideali)$ is an ideal of $\Gamma$.
\end{proposition}

\begin{proof}
Since $f$ is an epimorphism, its kernel $\idealk$ is an ideal of
$\Lambda$ and we have
$\Gamma/f(\ideali) \isonat \Lambda/(\ideali + \idealk)$
and by the previous proposition, a sum of ideals is again an ideal.
\end{proof}

\begin{lemma}
Let $\Lambda$ be a $\coperad$-algebra and let
${\{\ideali_p\}}_{p\in
P}$ be a set of ideals of $\Lambda$, then the intersection
$\bigcap_{p} \ideali_p$ is again an ideal of $\Lambda$.
\end{lemma}

\begin{proof}
For every $p \in P$, let $\varpiup_p : \Lambda \to \Lambda/\ideali_p$
be the morphism of $\coperad$\=/algebras associated to the ideal
$\ideali_p$. Since the category of $\coperad$-algebras is a category of
algebras over a monad, products may be computed in the ground category.
As a consequence, the intersection $\bigcap_p \ideali_p$ is the
kernel of the composition of $\coperad$-algebras maps
\[
	\begin{tikzcd}
		\Lambda \arrow[r, "\diagonal"]
		& \displaystyle\prod_{p\in P} \Lambda
		\arrow[r, "\prod_p \varpiup_p"]
		& \displaystyle\prod_{p\in P} \Lambda/\ideali_p
	\end{tikzcd}
\]
hence an ideal.
\end{proof}

Consequently, it makes sense to define the ideal generated by a graded
subobject of an algebra.

\begin{definition}
Given a subobject $X$ of $\Lambda$,
we shall denote by $(X)$ the smallest ideal of $\Lambda$ that contains
$X$. It is the intersection of all the ideals of $\Lambda$ that contain
$X$.
\end{definition}

\begin{theorem}[(Generation of ideals)]%
\label{thm:generation_d_ideal}
Assume $X$ is a subobject of$\Lambda$.
Then the ideal generated by $X$ is the image
\[
	(X) = a \left(\imshuffle \Lambda X \coperad\right).
\]
\end{theorem}

\begin{proof}
Let us first remark that given $X$ a graded subobject of $\Lambda$,
the following sequence is exact
\[
	\imshuffle \Lambda X \coperad \longrightarrow \Lambda^\coperad
	\longrightarrow {(\Lambda/X)}^\coperad .
\]
As a consequence $X$ is an ideal of $\Lambda$
if and only if one has
~[\ref{thm:ideal-associativite-auto}]
\[
	a \left(\imshuffle \Lambda X \coperad\right) \subobject X.
\]
Let us denote by $\ideali$ the image $a(\power{\sha(\Lambda,
X)}\coperad)$. By the unit property of $\Lambda$, the subobject
$X$ sits inside of $\ideali$ and since $(X)$ is an ideal we get
\[
	X \subobject \ideali \subobject
	a\left(\imshuffle \Lambda X \coperad\right) \subobject (X).
\]
Hence we only need to show that $\ideali$ is an ideal. That is
\[
	a\left(
	\imshuffle
	\Lambda
	{a\left(\imshuffle\Lambda X \coperad\right)}\coperad
	\right)
	\subobject a\left(\imshuffle\Lambda X\coperad \right).
\]
That is a direct consequence of the associativity condition for
$(\Lambda,a)$.
\end{proof}

\subsection{Generated ideals of free algebras}%
\label{sec:ideal_of_a_free_algebra}
When $Y \subset X^\coperad$ is a subobject of a free
$\coperad$\=/algebra, the ideal generated by $Y$ can be obtained by
looking at the image of $\power{\sha(X, Y)}\coperad \to X^\coperad$
instead of the bigger $\power{\sha(X^\coperad, Y)}\coperad$. This is
what we shall prove now.

Similarly to the fact that there exists an $\symgroup n$-equivariant
lax map
$\power{(X^\coperad)}{\otimes n} \to X^{\coperad^{\convolution n}}$
for any natural $n$%
~[\ref{thm:cotenseur_sym_lax}],
there exists an $\symgroup n$\=/equivariant map
\[
	\laxmap^\convolution_n(\coperad, X, Y)
	: {\sha\left(X^\coperad, Y\right)}^{\otimes n}
	\longrightarrow {\sha(X, Y)}^{\coperad^{\convolution n}}
\]
whose projection on the component $\power{[{\coperad^{\convolution
n}}(k), \power{\sha(X, Y)}{\otimes k}]}{\symgroup k}$ is the sum over
$\gamma +1 +\delta = n$ of the following composite map
\[
\begin{tikzcd}
	\fullpower{X^\coperad}{\otimes \gamma} \otimes
	Y \otimes \fullpower{X^\coperad}{\otimes \delta}
	\arrow[d]
	\\
	X^{\coperad^{\convolution \gamma}} \otimes Y^\monoidalunit
	\otimes X^{\coperad^{\convolution \delta}}
	\arrow[d]
	\\
	\displaystyle\prod_{a+1+b=k} 
	\internalhom{\coperad^{\convolution \gamma} (a) \otimes \monoidalunit
	\otimes \coperad^{\convolution \delta}(b)}{
	X^{\otimes a} \otimes Y \otimes X^{\otimes b}}
	^{\symgroup a \times \symgroup b}
	\arrow[d,hook]
	\\
	\displaystyle\prod_{a+1+b=k} 
	\internalhom{\coperad^{\convolution \gamma} (a) \otimes \coperad (1)
	\otimes \coperad^{\convolution \delta}(b)}{{\sha(X,Y)}^{\otimes k}
	}^{\symgroup a \times \symgroup b}
	\arrow[d,equal]
	\\
	\displaystyle\prod_{a+1+b=k} 
	\internalhom{\left( \coperad^{\convolution \gamma} (a)
	\otimes \coperad (1)
	\otimes \coperad^{\convolution \delta}(b)\right)
	\rightindsym{\symgroup a \times \symgroup b} {\symgroup k}}{
	{\sha(X,Y)}^{\otimes k}}^{\symgroup k}
	\arrow[d,hook]
	\\
	\internalhom{{\coperad^{\convolution n}}(k)}{
	{\sha(X, Y)}^{\otimes k}}^{\symgroup k}
\end{tikzcd}
\]

\begin{definition}%
\label{def:lax_map_shuffle}
Let $\laxmapsha (\coperad,X,Y): \power{\sha (X^\coperad,Y)}\coperad
\to {\sha (X,Y)}^{\coperad\compofsymseq\coperad}$ be
the morphism in $\linearlogos$ given by the following formula
\[
	\begin{tikzcd}
		{\sha\left(X^\coperad, Y\right)}^\coperad =
		\displaystyle\prod_{n \geq 0} {\internalhom {\coperad(n)}
		{{\sha\left(X^\coperad, Y\right)}^{\otimes n}}}^{\symgroup n}
		\ar[d]
		\\
		\displaystyle\prod_{n \geq 0} {\internalhom {\coperad(n)}
		{{\sha\left(X, Y\right)}^{\coperad^{\convolution n}}}}^{\symgroup n}
		= 
		\displaystyle\prod_{n \geq 0}
		{
		\internalhom {\coperad(n)}
		{\displaystyle\prod_{k \geq 0}{\internalhom
		{\coperad^{\convolution n}(k)}
		{{\sha(X,Y)}^{\otimes k}}}^{\symgroup k}}
		}^{\symgroup n}
		\ar[d]
		\\
		{\sha(X,Y)}^{\coperad \compofsymseq \coperad} =
		\displaystyle\prod_{k \geq 0}
		{
		\internalhom {(\coperad \compofsymseq \coperad) (k)}
		{{\sha(X,Y)}^{\otimes k}}
		}^{\symgroup k}.
	\end{tikzcd}
\]
\end{definition}

\begin{lemma}%
\label{thm:idea_of_a_free_algebra}
The following diagram commutes
\[
\begin{tikzcd}
	{\sha (X^\coperad,Y)}^\coperad
	\arrow[rr,"\laxmapsha(\coperad{,}X{,}Y)"] \arrow[ddd]
	&& {\sha (X,Y)}^{\coperad\compofsymseq\coperad}
	\arrow[r,"\id{}^{w}"]
	\arrow[d,hook]
	& {\sha (X,Y)}^{\coperad}
	\arrow[d,hook]
	\\
	&&{\sha (X^\coperad,Y)}^{\coperad\compofsymseq\coperad}
	\arrow[r,"\id{}^{w}"]
	\arrow[d]
	&{\sha (X^\coperad,Y)}^{\coperad}
	\arrow[d]
	\\
	&&X^{\coperad\compofsymseq\coperad\compofsymseq\coperad}
	\arrow[r,"\id{}^{w \compofsymseq \id{}}"]
	\arrow[d,"\id{}^{\id{}\compofsymseq w}"]
	& X^{\coperad\compofsymseq\coperad}
	\arrow[d,"\id{}^{w}"]
	\\
	\fullpower{X^\coperad}\coperad
	\ar[rr, "\laxmap(\coperad{,}\coperad{,}X)"]
	&& X^{\coperad\compofsymseq\coperad}
	\arrow[r,"\id{}^{w}"']
	& X^{\coperad}.
\end{tikzcd}
\]
\end{lemma}

\begin{proof}
Since the inclusion $X \hookrightarrow X^\coperad$ is given by
$X^\tau$,
the commutation of the left rectangle comes from the counit identity
\[
	\left(\id{\coperad} \compofsymseq \tau \right) \circ w
	=\id{\coperad},
\]
while the commutation of the right rectangles follow from the
definitions and the coassociativity
$(w \compofsymseq \id \coperad) \circ w = (\id \coperad \compofsymseq
w) \circ w$.
\end{proof}

\begin{proposition}%
\label{thm:other_presentation_of_the_ideal_j}
The ideal $(Y)$ is the image of the following map
\[
	{\sha(X,Y)}^\coperad \hookrightarrow
	{\sha\left(X^\coperad,Y\right)}^\coperad
	\hookrightarrow {\left(X^\coperad\right)}^\coperad
	\xrightarrow{\canonicalaction} X^\coperad .
\]
\end{proposition}

\begin{proof}
The ideal $(Y)$ can be obtained as the image of the composite map,
\[
\begin{tikzcd}
	{\sha(X^\coperad, Y)}^\coperad 
	\arrow[r]
	& \fullpower{X^\coperad}\coperad \rar
	& X^{\coperad \compofsymseq \coperad}
	\arrow[r,"X^w"]
	& X^\coperad~[\ref{thm:generation_d_ideal}].
\end{tikzcd}
\]
By the previous lemma%
~[\ref{thm:idea_of_a_free_algebra}], this map factors through
$\power{\sha(X, Y)}\coperad$.
\end{proof}

\subsection{Complete algebras}
%-----------------------------------------------------------------------

Let $(\coperad, w, \tau, \iota)$ be a locally conilpotent cooperad.

\begin{definition}
For any natural $n$ let
$\varpiup_n : \coperad \twoheadrightarrow \coperad/\filtration
n\coperad$
denote the quotient map where
${\{\filtration n \coperad\}}_{n \in \ordinalomega}$
denotes the coradical filtration%
~[\ref{def:coradical_filtration}] of $\coperad$.

For any integer $n \geq 0$, let $\ideal n  \Lambda$ be the image of the
composition $a \circ \Lambda^{\varpiup_n}$:
\[
	\begin{tikzcd}[ampersand replacement=\&]
		\Lambda^{\coperad/\filtration n\coperad}
		\arrow[r, "\Lambda^{\varpiup_n}"]
		\arrow[d, "", swap, twoheadrightarrow]
		\& \Lambda^\coperad
		\arrow[d, "a"] \\
		\ideal n \Lambda
		\arrow[r, "", swap, hookrightarrow]
		\& \Lambda .
	\end{tikzcd}
\]
\end{definition}

\begin{lemma}
For
every natural $n$, the subobject $\ideal n  \Lambda$ is an ideal
of $\Lambda$.
\end{lemma}

\begin{proof}
Let $n$ be a natural integer.
The object $\Lambda^\coperad$ may
be endowed with the structure of a free $\coperad$-algebra%
~[\ref{thm:algebre_libre_sur_une_coperade}]. Moreover the map
$\varpiup_n$ is an epimorphism, so that $\Lambda^{\coperad/\filtration
n\coperad}$ is a subobject of $\Lambda^\coperad$; it is in fact an
ideal.

Indeed, the exact sequence
\[
	\filtration n\coperad \subobject
	\coperad \twoheadrightarrow \coperad/\filtration n\coperad .
\]
yields the following exact sequence,
\[
	\Lambda^{\coperad/\filtration n\coperad}
	\subobject \Lambda^\coperad \twoheadrightarrow
	\Lambda^{\filtration n\coperad} .
\]
Hence we shall only find a map
$\power{(\Lambda^{\filtration n\coperad})}\coperad
\to \Lambda^{\filtration n\coperad}$
that is compatible with the algebra map
$\power{(\Lambda^\coperad)}\coperad \to \Lambda^\coperad$.

By definition of the coradical filtration on $\coperad$%
~[\ref{def:coradical_filtration}], $\filtration n\coperad$ is a
subcooperad of $\coperad$, then
$\filtration n\coperad$ is a comodule over $\coperad$. Let
$\wrestricted n$ be the comodule map
\[
	\wrestricted n : \filtration n \coperad \longrightarrow \coperad
	\compofsymseq \filtration n \coperad
\]
then the algebra map that we are looking for is given by
\[
	\Lambda^{\wrestricted n}
	\circ \laxmap(\coperad, \filtration n\coperad,\Lambda)
\]
which fits into the following commutative diagram,
\[
	\begin{tikzcd}[ampersand replacement=\&]
		{\left(\Lambda^\coperad\right)}^\coperad
		\arrow[rr, "{\left(\Lambda^{\fncaninclusion n}\right)}^\coperad"]
		\arrow[d, "\laxmap(\coperad{,}\coperad{,}\Lambda)",
		swap]
		\& \&{\left(\Lambda^{\filtration n\coperad}\right)}^\coperad
		\arrow[d, "\laxmap(\coperad{,} \filtration n\coperad{,} \Lambda)"]
		\\
		\Lambda^{\coperad\compofsymseq\coperad}
		\arrow[rr, "\Lambda^{(\id \coperad\compofsymseq\fncaninclusion n)}"]
		\arrow[d, "\Lambda^w", swap]
		\& \&\Lambda^{\coperad\compofsymseq \filtration n\coperad}
		\arrow[d, "\Lambda^{\wrestricted n}"]
		\\
		\Lambda^\coperad \arrow[rr, "\Lambda^{\wrestricted n}"]
		\& \&\Lambda^{\filtration n\coperad}
	\end{tikzcd}
\]
where $\fncaninclusion n$ is the inclusion $\filtration n \coperad
\subobject \coperad$. This proves that $\Lambda^{\coperad/\filtration n
\coperad}$ is an ideal of $\Lambda^\coperad$.

In addition to this, the map $a : \Lambda^\coperad \to \Lambda$ is a
map of $\coperad$-algebras by definition and it is also an epimorphism
in $\linearlogos$ by the unit property. As a consequence, $\ideal n
\Lambda$ is the image of an ideal by an epimorphism; it is an ideal of
$\Lambda$%
~[\ref{thm:image-ideal}].
\end{proof}

\begin{definition}%
\label{defintion:canonical-topology}
We shall call the decreasing sequence of ideals
\[
	\dots \subobject \ideal n \Lambda
	\subobject \dots \subobject \ideal 1\Lambda
	\subobject \ideal 0\Lambda \subobject \Lambda
\]
the canonical topology on $\Lambda$. The associated diagram of quotients
\[
	\Lambda \to \dots \twoheadrightarrow \cofiltration n\Lambda
	\twoheadrightarrow \dots \twoheadrightarrow \cofiltration 1\Lambda
	\twoheadrightarrow \cofiltration 0\Lambda
\]
shall be called the radical cofiltration of $\Lambda$. The associated
graded object is defined by
\[
	\gr 0 \Lambda \coloneqq \cofiltration 0 \Lambda \qand
	\gr n \Lambda \coloneqq \ker(\cofiltration n \Lambda 
	\longrightarrow \cofiltration {n-1} \Lambda), \text{ for } n \geq 1.
\]
\end{definition}

\begin{remark}%
\label{remark:les-morphismes-d-algebres-sont-continus}
The following commutative diagram,
\[
	\begin{tikzcd}
		& \Gamma^{\coperad/\filtration n \coperad}
		\arrow[rr, "\Gamma^{\varpiup_n}"]
		\arrow[ddd]
		&& \Gamma^\coperad
		\arrow[ddd, "m"]
		\\
		\Lambda^{\coperad/\filtration n \coperad}
		\arrow[rr, "\Lambda^{\varpiup_n}" description]
		\arrow[ddd, twoheadrightarrow]
		\arrow[ru, "f^{\coperad/\filtration n \coperad}"]
		&& \Lambda^\coperad \arrow[ru, "f^\coperad"]
		\\
		\\
		&\ideal n  \Gamma \arrow[rr, hookrightarrow] && \Gamma
		\\
		\ideal n  \Lambda \arrow[rr] \arrow[ru, "f"]
		&& \Lambda. \arrow[ru,"f"']
		\arrow[from=uuu, "a", crossing over]
	\end{tikzcd}
\]
shows that every morphism $f : (\Lambda,a) \to (\Gamma,m)$ of
$\coperad$-algebras is `continuous' in the sense that for every natural
$n$, $f(\ideal n \Lambda)$ is a subobject of $\ideal n \Gamma$:
\[
	f(\ideal n \Lambda) \subobject \ideal n \Gamma .
\]
Thus, the constructions $\ideal \ast \Lambda$, $\cofiltration \ast
\Lambda$ and $\gr \ast \Lambda$ are functorial.
\end{remark}

\begin{definition}
A $\coperad$-algebra $\Lambda$ shall be called nilpotent
if $\ideal n \Lambda = 0$ for a natural number $n$.
We shall call the full subcategory of $\catofalg\coperad$
whose objects are the nilpotent algebras $\catofnilalg\coperad$.
\end{definition}

\begin{proposition}
The realisation functor from the pro-category of nilpotent
$\coperad$-algebras to the category of $\coperad$-algebras has a left
adjoint,
\[
	\begin{tikzcd}[ampersand replacement=\&]
		\procat{\catofnilalg\coperad}
		\arrow[r, shift right=1.5, "\varprojlim"']
		\&
		\catofalg\coperad
		\arrow[l, shift right=1.5,"\nilfunctor"']
	\end{tikzcd}
\]
that associates to a $\coperad$-algebras $\Lambda$ its radical
cofiltration.
\end{proposition}

\begin{proof}
Let $\Lambda$ be a $\coperad$-algebra and
$\Gamma : I \to \catofnilalg\coperad$ be a pro-object. By definition we
have
\[
	\hombracket{\procat{\catofnilalg\coperad}}
	{\nilfunctor(\Lambda)} \Gamma \isonat \limover{i\in I} \colimn
	\hombracket{\catofalg\coperad}{\cofiltration n\Lambda}{\Gamma(i)}
\]
For every $i \in I$, since $\Gamma(i)$ is nilpotent, there exists a
natural number $n$ such that $\ideal n \Gamma(i) = 0$ so that for
every $m \geq n$, we get
\[
	\hombracket{\catofalg\coperad}{\cofiltration m\Lambda}{\Gamma(i)}
	\isonat
	\hombracket{\catofalg\coperad}{\Lambda}{\Gamma(i)}.
\]
which implies that
\[
	\colimn
	\hombracket{\catofalg\coperad}{\cofiltration n\Lambda}{\Gamma(i)}
	\isonat 
	\hombracket{\catofalg\coperad}{\Lambda}{\Gamma(i)}.
\]
As a conclusion we get
\[
	\hombracket{\procat{\catofnilalg\coperad}}{
	\nilfunctor(\Lambda)} \Gamma
	\isonat
	\hombracket{\catofalg\coperad}\Lambda {\varprojlim \Gamma}. 
\]
\end{proof}

\begin{definition}
Let us denote by
\[
	\completionmap_\Lambda : \Lambda \longrightarrow
	\widehat{\Lambda} \coloneqq \limn \cofiltration n\Lambda.
\]
the unit of the monad
\[
	\widehat{(\phantom{a})} :
	\catofalg\coperad(\linearlogos)
	\longrightarrow \catofalg\coperad(\linearlogos)
\]
associated to the adjunction described above. We shall say that
the topology of $\Lambda$ is separated if $\completionmap_\Lambda$ is
a monomorphism and that it is complete if $\completionmap_\Lambda$ an
isomorphism.

We shall write $\catofcompletealg\coperad(\linearlogos)$
to denote the full subcategory of complete $\coperad$\=/algebras.
\end{definition}

\begin{proposition}%
\label{thm:cofiltration_des_algebres_libres}
For any free $\coperad$-algebra $\Lambda = X^\coperad$, one has
\[
	\ideal n  X^\coperad = X^{\coperad/\filtration n\coperad},
\]
so that
\[
	\cofiltration n \Lambda  \isonat X^{\filtration n \coperad}.
\]
\end{proposition}

\begin{proof}
 Let $n$ be a natural integer and consider the composition
\[
	(\varpiup_n \compofsymseq \id \coperad) \circ w : \coperad
	\longrightarrow \coperad \compofsymseq \coperad \longrightarrow
	(\coperad/\filtration n\coperad) \compofsymseq \coperad .
\]
since by definition of the coradical filtration of $\coperad$%
~[\ref{def:coradical_filtration}], $\filtration n \coperad$ is a
subcooperad of $\coperad$, this map factors
through $\coperad/\filtration n\coperad$. Let us call $\wgeqn n
: \coperad/\filtration n\coperad
\longrightarrow \coperad/\filtration n \coperad \compofsymseq \coperad$
the resulting map.
Since $\varpiup_n$ is an epimorphism, the following diagram shows that
$X^{\coperad/\filtration n\coperad}$ lies inside $\ideal n
X^\coperad$,
\[
	\begin{tikzcd}[ampersand replacement=\&]
		{\left(X^\coperad\right)}^{\coperad/\filtration n\coperad}
		\arrow[rr, "{\left(X^\coperad\right)}^{\varpiup_n}"]
		\arrow[d, "\laxmap(\coperad/\filtration n\coperad{,}
		\coperad{,} X)", swap]
		\& \& {\left(X^\coperad\right)}^\coperad
		\arrow[d, "\laxmap(\coperad{,}\coperad{,} X)"] \\
		X^{\coperad/\filtration n\coperad\compofsymseq\coperad}
		\arrow[rr, "X^{(\varpiup_n\compofsymseq\id \coperad)}"]
		\arrow[d, "X^{\wgeqn n}", swap]
		\& \& X^{\coperad\compofsymseq\coperad}
		\arrow[d, "X^w"]
		\\
		X^{\coperad/\filtration n\coperad}
		\arrow[rr, "X^{\varpiup_n}", hookrightarrow]
		\& \& X^\coperad .
	\end{tikzcd}
\]
In addition, the composite morphism
\[
	X^{\coperad/\filtration n\coperad}
	\xrightarrow{{\left(X^\tau\right)}^{\coperad/\filtration n\coperad} }
	{\left(X^\coperad\right)}^{\coperad/\filtration n\coperad}
	\to X^{\coperad/\filtration n\coperad} 
\]
obtained using the counit of the cooperad $\coperad$ is the identity
of $X^{\coperad/\filtration n\coperad}$. Therefore, the morphism
\[
	{\left(X^\coperad\right)}^{\coperad/\filtration n\coperad}
	\to X^{\coperad/\filtration n\coperad} 
\]
is an epimorphism. We conclude the proof by using the
unicity of the epi-mono factorisation of a morphism in $\linearlogos$.
Besides, by using the exactness property of the cotensor, we get that
$\cofiltration n X^\coperad \isonat X^{\filtration n \coperad}$.
\end{proof}

\begin{proposition}%
\label{thm:les_algebres_libres_sont_completes}
Any free $\coperad$-algebra is complete.
\end{proposition}

\begin{proof}
Let $X$ be an object of $\linearlogos$ and let us prove that
$X^\coperad$ is complete.
For every
natural $n$,
$\cofiltration n X^\coperad \isonat X^{\filtration n \coperad}$%
~[\ref{thm:cofiltration_des_algebres_libres}].
Then using the continuity of the
cotensor~[\ref{rmk:cotenseur_fonctoriel}]
and the local conilpotency of $\coperad$, we conclude that
\[
	X^\coperad \isonat X^{\underset{}{\varinjlim} 
	 \filtration n \coperad}
	\isonat \limn
	 X^{\filtration n \coperad}
	\isonat \limn
	 \cofiltration n X^\coperad .
\]
\end{proof}

\begin{lemma}%
\label{thm:epi-cofiltration}
Let $f : \Lambda \to \Gamma$ be a morphism of $\coperad$-algebras.
Suppose that for any integer $n \geq 0$, the maps
\[
	\gr n f : \gr n \Lambda \to\gr n \Gamma
\]
are epimorphisms in $\linearlogos$. Then, the morphism
$ \widehat{f} : \widehat{\Lambda} \longrightarrow \widehat{\Gamma}$
is also an epimorphism in $\linearlogos$.
\end{lemma}

\begin{proof}
By assumption $\cofiltration 0 \Lambda \to \cofiltration 0 \Gamma$ is
an epimorphism. Now let $n \geq 1$ and let us contemplate the diagram
\[
	\begin{tikzcd}
		\ker\left(\cofiltration {n-1} f\right)
		\arrow[r, hookrightarrow]
		& \cofiltration {n-1}\Lambda
		\arrow[rr, "\cofiltration {n-1} f", twoheadrightarrow]
		&& \cofiltration {n-1}\Gamma
		\\
		\ker\left(\cofiltration n f\right)
		\arrow[u, "\specialkernelkmap^{n-1}"]
		\arrow[r, hookrightarrow]
		& \cofiltration n\Lambda
		\arrow[u, "", twoheadrightarrow]
		\arrow[rr, "\cofiltration n f", twoheadrightarrow]
		&& \cofiltration n\Gamma
		\arrow[u, "", twoheadrightarrow]
		\\
		& \gr n \Lambda
		\arrow[u, hookrightarrow]
		\arrow[rr,"\gr n f", twoheadrightarrow]
		&& \gr n \Gamma ,
		\arrow[u, hookrightarrow]
	\end{tikzcd}
\]
If $\filtration {n-1} f$ is an epimorphism, then by the two-four
lemma, $\filtration n f$ is also an epimorphism.
Moreover, by the $3 \times 3$ lemma, the map $\specialkernelkmap^n$
is also an epimorphism. We have
subsequently proven by induction that for any integer $n$, the maps
$\cofiltration n f$ and $\specialkernelkmap^n$ are epimorphisms.
Combining this
with the fact that countable products are exact in $\linearlogos$ and
$\linearlogos$ is presentable, one can conclude that
\[
	\limn
	\cofiltration n f :
	\limn
	\cofiltration n\Lambda \longrightarrow
	\limn
	\cofiltration n\Gamma
\]
is an epimorphism%
~[\ref{thm:exactness_inverse_limit}].
\end{proof}

\begin{proposition}%
\label{thm:epi-chapeau}
Let $f : \Lambda \to \Gamma$ be a morphism of $\coperad$-algebras. If
$f$ is an epimorphism in $\linearlogos$, then
\[
	\widehat{f} : \widehat{\Lambda} \longrightarrow \widehat{\Gamma}
\]
is also an epimorphism in $\linearlogos$.
\end{proposition}

\begin{proof}
By looking at the
diagram in \cref{remark:les-morphismes-d-algebres-sont-continus} we
notice that the map
\[
	\ideal n f : \ideal n \Lambda \to \ideal n \Gamma
\]
is an epimorphism. Thus, the quotient
\[
	\gr n f : \gr n \Lambda \to \gr n \Gamma
\]
is also an epimorphism. We conclude using the previous lemma%
~[\ref{thm:epi-cofiltration}].
\end{proof}

\begin{proposition}%
\label{thm:les_algebres_sont_completes}
For all $\coperad$-algebras $\Lambda$, the unit map
$\completionmap_\Lambda$ is an epimorphism in $\linearlogos$.
\end{proposition}

\begin{proof}
Let $\Lambda$ be a $\coperad$-algebra. Then by definition the action
map
\[
	a : \Lambda^\coperad \longrightarrow \Lambda
\]
is an algebra map, when $\Lambda^\coperad$ is endowed with the
structure of a free
$\coperad$\=/algebra. Moreover by the unit property, the map $a$
is an epimorphism in $\linearlogos$.

We can draw the commutative diagram
~[\ref{thm:les_algebres_libres_sont_completes},%
~\ref{thm:epi-chapeau}]:
\[
	\begin{tikzcd}[ampersand replacement=\&]
		\Lambda^\coperad
		\arrow[r, equal]
		\arrow[d, "a",twoheadrightarrow, swap]
		\& \widehat{\Lambda^\coperad}
		\arrow[d, "\widehat{a}", twoheadrightarrow] \\
		\Lambda
		\arrow[r, "\completionmap_\Lambda", swap]
		\& \widehat{\Lambda} ,
	\end{tikzcd}
\]
from which we deduce that $\completionmap_\Lambda$ is an epimorphism is
$\linearlogos$.
\end{proof}

\begin{proposition}%
\label{thm:la_categorie_des_algebres_completes_est_reflexive}
The category of complete $\coperad$-algebras is a reflective and
accessible localisation
of the category of $\coperad$-algebras.
\end{proposition}

\begin{proof}
This amounts to show that the monad
$\widehat{(\phantom{a})} : \catofalg\coperad(\linearlogos)
\longrightarrow \catofalg\coperad(\linearlogos)$
is idempotent and accessible.

To see this let $\Lambda$ be a $\coperad$-algebra and let $\ideal
\infty\Lambda \coloneqq \bigcap_{n\in \naturals} \ideal n \Lambda$.
Since by the previous proposition the unit map $\completionmap_\Lambda$
is an epimorphism, we get the natural equivalence $\Lambda/\ideal
\infty\Lambda \isonat \widehat{\Lambda}$, as $\ideal \infty\Lambda$
is nothing but the kernel of $\completionmap_\Lambda$. The canonical
topology of the quotient algebra is separated by construction, from
which we deduce that the monad is idempotent.

As the monad is defined with an $\ordinalomega_1$-small limit, it is
$\ordinalomega_1$-accessible.
\end{proof}

Since the category of $\coperad$-algebras
is presentable~[\ref{thm:les_mousquetaires_sont_presentables}], we
deduce the following corollary.

\begin{corollary}%
\label{thm:la_categorie_des_algebres_completes_est_presentable}
The category of complete $\coperad$-algebras is presentable.
\end{corollary}

\subsection{A (counter) example}%
\label{sec:contre_exemple}
%-----------------------------------------------------------------------

In most cases, filtered colimits are exact in the ground category. As a
consequence all coalgebras over a locally conilpotent
cooperad are pointed and locally
conilpotent themselves. Dually, if cofiltered limits were exact,
all algebras over a locally conilpotent cooperad would
be complete. However, the exactness of cofiltered limits is rare in
nature. Here we wish to work out an example, in order to show that
the subcategory of complete algebras is a strict one and that the
characterisation theorem cannot be improved.

For this example, let us choose the category
$\vect \reals$ of real vector spaces. We then
consider a cofree locally conilpotent cooperad
in arity one, that is the cofree locally conilpotent coalgebra on
one generator
$(\reals[\formalx] = \formalx\reals[\formalx]\oplus\reals ,
\deltacoprod, \counittau, \iota)$ with
$\deltacoprod(\formalx) = 1 \otimes \formalx + \formalx \otimes 1$.
Its coradical filtration is given by
$\filtration n \reals[\formalx] = \reals_n[\formalx]$ for $n \geq 1$.

An algebra over this cooperad is the data of a real vector space
$\Lambda$ and a linear map
\[
	S : \Lambda^{\naturals} \longrightarrow \Lambda,
\]
satisfying the following two conditions. The unit condition says that
for every element $a \in \Lambda$ we have:
\[
	S (a, 0, 0 , \cdots) = a .
\]
The associativity condition amounts to the following fact: given a
matrix ${\{a_{ij}\}}_{i,j\in\naturals}$ of elements of $\Lambda$, we
have:
\[
	S_i(S_j (a_{ij})) = S_n
	\left(\sum_{i+j=n} a_{ij}\right).
\]
In other words, the `sum' of a matrix may be computed either by
`summing' columns with $S$ then summing lines, or by first summing the
anti-diagonals with the sum of $\Lambda$ and then using $S$. By symmetry
it is also possible to sum first on the lines and then on columns (this
is because the coalgebra we are working with is in fact cocommuative).

The key data in this construction is the following endomorphism
of $\Lambda$:
\[
	\varepsilon(a) = S (0,a,0,0, \cdots) .
\]
Indeed from the axioms above it follows that for any sequence
$\underline{a} \in \Lambda^{\naturals}$ that is eventually zero, we get:
\[
	S(\underline{a}) = \sum_{n\in\naturals}
	\varepsilon^n(a_n) .
\]
Moreover $S$ intertwines the right shift function and $\varepsilon$:
$S \circ [-1] = \varepsilon \circ S$. From this we deduce that the
canonical topology on $\Lambda$ is given by: $\ideal n \Lambda =
\image(\varepsilon^n)$. Thus $\Lambda$ is a nilpotent algebra if and
only if $\varepsilon$ is a nilpotent endomorphism and it is complete if
and only if
\[
	\ideal \infty\Lambda = \bigcap_{n\in\naturals}
	\image(\varepsilon^n) = 0 ,
\]
in which case it can be identified with the inverse limit of the
nilpotent algebras $\Lambda/\image(\varepsilon^n)$.

Nevertheless, $\varepsilon$ satisfies a different nilpotence condition
in general.

\begin{proposition}
Let $V$ be a subspace of $\Lambda$ stable under $\varepsilon$. Suppose
that
\[
	\varepsilon : V \longrightarrow V
\]
is surjective, then $V = 0$.
\end{proposition}

\begin{proof}
Let $x_0 \in V$, since $\varepsilon$ is surjective on $V$, we can find a
sequence of elements $x_i \in V$ such that 
$x_i = \varepsilon(x_{i+1})$
for all $i \in \naturals$. Then summing the matrix
\[
	\begin{bmatrix}
	x_0 & - x_1 & 0 & 0 \\
	0   & x_2 & - x_3 & 0 \\
	0 & 0 & \ddots & \ddots
	\end{bmatrix}
\]
in two different ways gives us $x_0 = 0$.
\end{proof}

As a consequence, the only eigenvalue of $\varepsilon$ is zero. Hence
in case $\Lambda$ is finite dimensional, it is automatically nilpotent.
In case it is infinite dimensional we can only write the nilpotence
condition
\[
	\bigcap_{n\in\naturals} \image
	(\varepsilon^n) \subset \ker(\varepsilon) .
\]

We will now present an example of such a vector space endowed with an
endomorphism that is not complete but satisfies the nilpotence
condition above. We will then show that it exhibits an example of a
non-complete algebra. For this we shall use a typical example
coming from operator theory.

Let $\triangularmatrices$ be the real vector space of lower triangular
double sequences; its elements are of the form
\[
	\begin{bmatrix}
	\ast & 0 & 0 & 0 \\
	\ast   & \ast & 0 & 0 \\
	\ast & \ast & \ast & \ddots
	\\
	\ast & \ast & \ddots & \ddots
	\end{bmatrix}
	.
\]
We endow $\triangularmatrices$ with a linear operator $[-1]
:\triangularmatrices \longrightarrow \triangularmatrices$, the
right shift on columns. In details, if $a$ is an element of
$\triangularmatrices$, then we have ${(a[-1])}_{ij} = a_{i,j-1}$
if $j \neq 0 \,\&\, j \leq i$ and ${(a[-1])}_{ij} = 0$ otherwise. This
endomorphism has three key properties:
\begin{itemize}
	\item the intersection
	      $\bigcap_{n\in\naturals}
	      \image({[-1]}^{\circ n})$ is trivial;
	\item for every $n\in\naturals$,
	      $\ker({[-1]}^{\circ n})$
	      is different from $\ker({[-1]}^{\circ n+1})$;
	\item for any sequence $\underline{a}
	      \in \triangularmatrices^{\naturals}$ the sum
	      \[
	      	\sum_{n\in\naturals} a_n [-n]
	      \]
	      is well defined in $\triangularmatrices$.
\end{itemize}
The last point comes from the fact that the sum reduces to a finite sum
on each column so that for $i,j\in \naturals$ we have
\[
	{\left(\sum_{n\in\naturals} a_n[-n]\right)}_{ij} =
	\sum_{n \leq j} {(a_n)}_{i,j-n} .
\]
Additionally, let $\extendedsummap : \triangularmatrices \longrightarrow
\reals$ be your preferred choice of linear extension of the usual sum of
finitely supported sequences. Although not invariant under the shift we
still have that
\[
	\image(\extendedsummap \circ [-1] ) =
	\image(\extendedsummap) = \reals,
\]
with the consequence that
\[
	\bigcap_{n\in\naturals}
	\image(\extendedsummap \circ {[-1]}^{\circ n}) \neq 0 .
\]
Our algebra is now declared to be the vector space
$\Lambdaup = \triangularmatrices \oplus \reals$
with the endomorphism
\[
	\varepsilonup =
	\begin{pmatrix}
	[-1] & 0 \\
	\extendedsummap & 0
	\end{pmatrix}.
\]
By construction we have:
\[
	\bigcap_{n\in\naturals}
	\image(\varepsilonup^n) = 0 \oplus \reals .
\]
Indeed, for $n \in \naturals$ let $a^n$ be the element of
$\triangularmatrices$ such that $a^n_{ij} = \kronecker_{in}
\kronecker_{j0}$, that is $a^n$ has only one non-zero coefficient in
position $(n,0)$. Since
\[
	\varepsilonup^n =
	\begin{pmatrix}
	{[-1]}^{\circ n} & 0 \\
	\extendedsummap \circ {[-1]}^{\circ n-1} & 0
	\end{pmatrix},
\]
we have for any real $\lambda$,
$\varepsilonup^n(\lambda a^n,0) = (0,\lambda)$.

Finally we construct the map
$\actionexample : \Lambdaup^{\naturals} \longrightarrow \Lambdaup$
in the following way. Let $\underline{a}$ be a sequence in $\Lambdaup$,
this is the data of sequences $\underline{a}'
\in \triangularmatrices^{\naturals}$ and $\underline{a}''
\in \reals^{\naturals}$, such that
$\underline{a} = (\underline{a}', \underline{a}'')$.
We let
\[
	\actionexample(\underline{a})
	\coloneqq \sum_{n\in\naturals} \varepsilonup^n(\underline{a})
	\coloneqq\left(\sum_{n\in\naturals} a_n'[-n],
	\extendedsummap\left(
	\sum_{n\in\naturals} a_{n+1}'[-n]\right) + a''_0
	\right) .
\]
The map $\actionexample : \Lambdaup^{\naturals}
\longrightarrow \Lambdaup$
is linear and satisfies the unit axiom by construction. Let us check
that it also satisfies the associativity axiom.

Let ${(a_{ij})}_{i,j\in\naturals}$ be a matrix with
coefficients in $\Lambdaup = \triangularmatrices \oplus \reals$. Summing
all columns, we get a sequence in $\Lambdaup$ whose general term is:
\[
	\actionexample_{j}(a_{ij}) =
	\left(\sum_{j\in\naturals} a_{ij}'[-j],
	\extendedsummap\left(\sum_{j\in\naturals}
	a_{i,j+1}'[-j]\right) + a_{i0}''\right) .
\]
Summing now on lines, we get:
\begin{align*}
	\actionexample_i \actionexample_j(a_{ij})' &= \sum_{i,j\in\naturals}
	a'_{ij}[-i-j];
	\\
	\actionexample_i \actionexample_j(a_{ij})'' &=
	\extendedsummap\left(\sum_{i,j\in\naturals}
	a_{i+1,j}'[-i-j]\right) + \extendedsummap\left(
	\sum_{j\in\naturals}a_{0,j+1}'[-j]\right)+ a_{00}'';
	\\
	&=
	\extendedsummap\left(\sum_{i+j
	\geq 1} a'_{ij} [-i-j+1]\right) + a_{00}'' .
\end{align*}

In parallel, by first summing on the anti-diagonal --- that is, using
the action of $\deltacoprod$ --- we get a sequence of general term
\[
	{\left({(a_{ij})}^\deltacoprod\right)}_n = \sum_{i+j = n} a_{ij} ,
\]
so that:
\begin{align*}
\actionexample\left(a_{ij}^\deltacoprod\right)' &= \sum_{n\in\naturals}
\sum_{i+j = n} a_{ij}'[-n];
\\
\actionexample\left(a_{ij}^\deltacoprod\right)'' &=
\extendedsummap
\left(\sum_{n\in\naturals} \sum_{i+j=n+1} a_{ij}'[-n]\right)
+ a_{00}'' .
\end{align*}
This ends the proof that $(\Lambdaup, \actionexample)$ is an algebra
over $(\reals[\formalx], \deltacoprod)$; it is not complete.
\section{General framework}
%=======================================================================

From now on we work with a closed symmetric monoidal additive category
$(\catss, \otimes, \monoidalunit)$ that is close enough to a category of
type $(\vect\fieldk, \otimes, \fieldk)$; that is we shall assume that
$\catss$ is semi-simple.

\subsection{Sign rules and notations for graded objects and chain
complexes}
%-----------------------------------------------------------------------

We shall use the homological grading convention for chain complexes:
\[
	\dots \to X_2 \to X_1 \to X_0 \to X_{-1} \to X_{-2} \to \dots
\]
As a consequence the shift functor $\s$ shifts complexes `to the left',
that is $\s X_{n} = X_{n-1}$.
Let us denote by $\catss^{\integers}$
the category of graded objects in $\catss$.
The category $\catss^{\integers}$ of
sequences is endowed, as usual, with the
symmetric monoidal product
\[
	{( X \otimes_{{\catss}^{\integers}} Y)}_n
	= \bigoplus_{i+j=n} X_i \otimes_{\catss} Y_j,
\]
for any graded objects $X$ and $Y$. The braiding follows
the `sign rule': if $\sigma_{X_i,Y_j} : X_i \otimes Y_j \to
Y_j \otimes X_i$ is the braiding of $\otimes_{\catss}$,
then
\[
	\sigma^{\integers}_{X,Y} =
	\bigoplus_{i+j=n} {(-1)}^{ij} {\sigma}_{X_i,X_j}.
\]
The associated internal hom is given by
\[
	{[X,Y]}^{\catss^{\integers}}_n =
	=\prod_{p\in\integers}
	{\left[{(s^n X)}_p, Y_p\right]}^{\catss}.
\]
\begin{definition}
We shall denote by $\catofmod{\s^p}(\catss)$ the
category of modules over $\s^p$, for $p \in \integers$. An
object of this category is the data of a graded object $X$
with a degree $p$ map $d : \s^p X \to X$.
The category of chain complex $\chain\catss$
is the full subcategory of $\catofmod{\sinv}(\catss)$
whose objects $(X,d)$ satisfy the equation $d^2=0$;
\end{definition}

\begin{definition}[(Morphisms with degree)]
Let $X$ and $Y$ be two graded objects in $\catss$. A graded
morphism of degree $p$ (or degree $p$ map for short) from $X$ to $Y$ is
a morphism of graded objects $f: \s^p \monoidalunit \otimes X \to Y$.
We will usually denote $f$ as an arrow $f: X \to Y$. We shall denote
the degree $1$ morphism induced by the identity of $X$ by a map
$\degreeonemap : X \longrightarrow \s X$.
\end{definition}

Graded morphisms may be composed, given $g: Y \to Z$ of degree
$q$, $g\circ f$ is the degree $p+q$ map defined as
\[
	\begin{tikzcd}[ampersand replacement=\&]
		s^q \monoidalunit \otimes s^{p} \monoidalunit \otimes X
		\arrow[r]
		\arrow[r, "{\idfunctor \otimes f}"]
		\& 
		s^q \monoidalunit \otimes Y
		\arrow[d,"g"] \\
		s^{p+q} \monoidalunit \otimes X
		\arrow[u]
		\arrow[u, shift left=1.5, phantom, "\rotatebox{90}{\(\iso\)}"]
		\arrow[r, "g \circ f"]
		\& 
		Z.
	\end{tikzcd}
\]

They can also be tensored: $f \otimes g$ is again of degree $p+q$ and
is given by
\[
	\begin{tikzcd}[ampersand replacement=\&]
		s^{p} \monoidalunit \otimes s^{q} \monoidalunit \otimes X \otimes X'
		\arrow[r]
		\arrow[r, phantom, shift left=1.5, "\iso"]
		\& 
		\s^p \monoidalunit \otimes X \otimes  s^q \monoidalunit \otimes X
		\arrow[d,"f\otimes g"] \\
		s^{p+q} \monoidalunit \otimes X \otimes X'
		\arrow[u]
		\arrow[u, shift left=1.5, phantom, "\rotatebox{90}{\(\iso\)}"]
		\arrow[r, "f \otimes g"]
		\& 
		Y \otimes Y'
	\end{tikzcd}
\]
Equivalently, $f \otimes g$ can be defined by the composition
\[
	\begin{tikzcd}[ampersand replacement=\&]
		s^{p} \monoidalunit \otimes s^{q} \monoidalunit
		\arrow[r]
		\& 
		{[X,Y]} \otimes {[X',Y']}
		\arrow[d]
		\\
		s^{p+q}\monoidalunit
		\arrow[u]
		\arrow[u, shift left=1.5, phantom, "\rotatebox{90}{\(\iso\)}"]
		\arrow[r]
		\& 
		{[X \otimes X',Y \otimes Y']}.
	\end{tikzcd}
\]

Likewise, one defines $[f,g]$ of degree $p+q$ as the map associated to
\[
	\begin{tikzcd}[ampersand replacement=\&]
	s^{q} \monoidalunit  \otimes [X',Y]
	\otimes s^{p} \monoidalunit \otimes X
	\arrow[r, "\idfunctor\otimes f"]
	\&
	s^{q} \monoidalunit  \otimes [X',Y] \otimes X'
	\arrow[d, "\idfunctor\otimes\evfunctor_{X'}"]
	\\
	s^{q} \monoidalunit \otimes s^{p}
	\monoidalunit \otimes [X',Y] \otimes X
	\arrow[u]
	\arrow[u, shift left=1.5, phantom, "\rotatebox{90}{\(\iso\)}"]
	\&
	s^q \monoidalunit \otimes Y
	\arrow[d,"g"]
	\\
	s^{p+q} \monoidalunit \otimes [X',Y] \otimes X
	\arrow[u]
	\arrow[u, shift left=1.5, phantom, "\rotatebox{90}{\(\iso\)}"]
	\arrow[r]
	\&
	Y'.
	\end{tikzcd}
\]

\begin{lemma}%
\label{thm:sign}
The above notations satisfy the sign rules
\[
	\begin{cases}
		(f \otimes g) \circ (f' \otimes g') = {(-1)}^{|g||f'|} (f\circ f')
		\otimes (g \circ g');\\
		[f,g] \circ [f',g'] = {(-1)}^{|f||f'|} {(-1)}^{|f||g'|}
		[f' \circ f , g \circ g'];\\
		[f,[g,h]] = {(-1)}^{|f||g|} [f \otimes g , h].
	\end{cases}
\] 
\end{lemma}

Let us introduce some new notations to distinguish between two possible
compositions: given two maps of symmetric sequences $f : M \to N$
and $g : P \to P$, there are two different ways to build a map from
$M \compofsymseq P$ to $N \compofsymseq P$ by either inserting $g$
$n$-times in arity $n$ (that would be called $f \compofsymseq g$) or by
inserting $g$ only once in each arity and completing with identities
(that would be called $f \compofsymseq'g$).

\begin{itemize}
	\item Let $f: M \to N$ be a morphism of graded symmetric sequences
	      and let $g:M \to N$ be a degree $p$ map which commutes with the
	      actions of the symmetric groups. Then
	      $\sha (f,g) : M^{\convolution n}
	      \longrightarrow N^{\convolution n}$
	      is the degree $p$ map which commutes with
	      the right actions of the symmetric groups and which is defined
	      as
	      \[
	      	\sha (f,g) = \sum_i f^{\convolution i}\convolution g
	      	\convolution f^{ \convolution n-1-i}.
	      \]
	      The map $\sha (f,g)$ commutes also with the left action
	      of the symmetric group $\symgroup n$;
	\item Suppose that we also have a morphism of graded symmetric
	      sequences $f':M'\to N'$. Then
	      \[
	      	f' \compofsymseq \sha(f,g) : M' \compofsymseq M
	      	\longrightarrow N' \compofsymseq N,
	      \]
	      is the degree $p$ map which commutes with the actions
	      of the symmetric groups and given on
	      $M'(n) \otimes_{\symgroup n} M^{\convolution n}$ by
	      $f \otimes_{\symgroup n} \sha (f,g)$;
	\item In the same vein, we will use the notation
	      \[
	      	f \compofsymseq' g
	      	\coloneqq f \compofsymseq \sha(\idfunctor,g).
	      \]
	\item Let $f: X \to Y$ be a morphism of graded objects in
	      $\catss$,
	      let $g:X \to Y$ be a degree $p$ map and let $M$ be a symmetric
	      sequence. Then we will denote by
	      \[
	      	{\sha (f,g)}^M : X^M \longrightarrow Y^M
	      \]
	      the degree $p$ map whose projection
	      on $\left[M(n),X^{\otimes n}\right]$ is
	      given by $\left[M(n), \sha (f,g)\right]$;
	\item Let $f: M \to N$ be a degree $p$ map of graded symmetric
	      sequences and let $X$ be a graded object of $\catss$.
	      Then we will denote by
	      \[
	      	X^f : X^N \longrightarrow X^M
	      \]
	      the degree $p$ map whose projection on
	      $[M(n),X^{\otimes n}]$ is given by $[f, \id {X^{\otimes n}}]$.
\end{itemize}

These notations satisfy the following sign rules.

\begin{lemma}%
\label{lemma:calcul-inversion}
Let $f: X \to Y$ be a morphism of graded objects of $\catss$,
let $g:X \to Y$ be a degree $p$ map and let $h:N \to M$ be a
degree $q$ map of graded symmetric
sequences. Then, we have the following equality between degree
$p+q$ maps from $X^M$ to $Y^N$
\[
	Y^{h} \circ {\sha (f,g)}^{N} = {(-1)}^{pq} {\sha (f,g)}^{M}
	\circ X^h.
\]
\end{lemma}

\begin{lemma}%
\label{thm:signe-pour-les-puissances}
Let $X $ be a graded object of $\catss$ and
let $f: P \to N$ and $g:N \to M$ be maps of graded symmetric
sequences of respective degrees $p$ and $q$.
Then, we have the following equality between degree
$p+q$ maps from $X^M$ to $X^P$
\[
	X^f \circ X^g = {(-1)}^{pq} X^{g\circ f}.
\]
\end{lemma}

\subsection{Graded coalgebras versus dg-coalgebras}
%-----------------------------------------------------------------------

The goal of this subsection is to investigate the relation between
differential graded coalgebras over a dg-operad and the
graded coalgebras on its underlying graded operad

\begin{definition}
By a graded operad, we mean an operad in $\catss^\integers$; by a
differential graded operad, we mean an operad in $\chain \catss$.
Given a dg-operad $\operad$, we shall denote by
$\forgetd\operad$ its underlying graded operad.
\end{definition}

Let us begin by setting a notation and recalling well-known facts.

\begin{definition}%
\label{definition-D0}
Let us denote by $\dzero$ the following chain complex concentrated
in degree $0$ and $-1$.
\[
	\begin{tikzcd}
		\cdots \arrow[r] & 0 \arrow[r] & 0
		\arrow[r] & \monoidalunit \arrow[r, "\id {\monoidalunit}"]
		& \monoidalunit \arrow[r]
		& 0 \arrow[r] & \cdots
	\end{tikzcd}
\]
and by $\gradedd$ its underlying graded object. If we identify
$\gradedd$ with $\monoidalunit[\formalphi]$ where $\formalphi$ is a
formal variable of degree $-1$, its commutative algebra structure is
given by $\formalphi^2 = 0$ and its cocommutative coalgebra structure
is given by $\formalphi \mapsto \monoidalunit \otimes \formalphi +
\formalphi \otimes \monoidalunit$. Endowed with these two canonical
structures, $\gradedd$ becomes a bicommutative bigebra. The projection
on the degree zero part $\counitofd : \dzero \to \monoidalunit$ is
actually a chain map and $\dzero$ is in fact a cocommutative coalgebra in
the category of chain complexes.
\end{definition}

\begin{remark}%
\label{rmk:graded_versus_dg}
Since we have the equivalence of categories
\[
	\chain \catss
	\isonat \catofmod \gradedd \left({\catss^\integers}\right)
\]
and $\gradedd$ is a bicommutative bigebra, then
\begin{itemize}
	\item the forgetful functor $\forgetd: \chain {\catss} \to
	      \catss^{\integers}$ is bimonadic with adjoints given by
	      \[
	      	\begin{tikzcd}[ampersand replacement=\&]
	      		\chain {\catss} \arrow[rr, "\forgetd" description]
	      		\& \&
	      		\catss^{\integers};
	      		\arrow[ll, shift right=3, swap, "\gradedd
	      		\otimes -"]
	      		\arrow[ll, shift left=3, "{[\gradedd,-]}"]
	      	\end{tikzcd}
	      \]
	\item the forgetful functor $\forgetd$ is symmetric monoidal and
	      commutes also with internal homs;
	\item both adjoints $\gradedd \otimes -$ and $[\gradedd, -]$ are bilax
	      symmetric monoidal.
\end{itemize}
\end{remark}

\begin{remark}%
\label{rmk:free_graded_versus_free_dg}
As a straightforward consequence of these properties, for any symmetric
sequences $M,N$ and chain complex $X$, we have
\[
	\forgetd (M \compofsymseq N)
	\isonat \forgetd M \compofsymseq \forgetd N
	\qand \forgetd X^M \isonat \fullpower{\forgetd X}{\forgetd M}.
\]
In particular since $\forgetd$ commutes with all limits, for any
dg-operad $\operad$ and any chain complex we have
\[
	\forgetd\left(\freecog{\operad} X\right)
	\isonat \freecog{\forgetd\operad} \forgetd X.
\]
\end{remark}

\begin{remark}
Since both $\gradedd \otimes -$ and $[\gradedd,-]$ are lax symmetric
monoidal, they are also lax monoidal
for the $\compofsymseq$ product; it follows that they send
graded operads to dg-operads in such a way that the forgetful functor
\[
	\begin{tikzcd}[ampersand replacement=\&]
		\catofop\left({\chain\catss}\right)
		\arrow[rr, "\forgetd" description]
		\&\&
		\catofop\left({\catss^\integers}\right)
		\arrow[ll, shift left=3, "{[\gradedd{,}-]}"]
		\arrow[ll, shift right=3, "\gradedd \otimes -"']
	\end{tikzcd}
\]
is bimonadic.
\end{remark}

Let us show an analoguous statement for the categories of coalgebras.

\begin{proposition}%
\label{thm:forget_coderivation_preserves_limits}
Let $\operad$ be a dg-operad. Then the forgetful functor
$\forgetd$ is bimonadic:
\[
	\begin{tikzcd}[ampersand replacement=\&]
		\catofcog\operad
		\arrow[rr, "\forgetd" description]
		\&\&
		\catofcog{\forgetd\operad}
		\arrow[ll, shift left=3, "{[\gradedd{,}-]}"]
		\arrow[ll, shift right=3, "\gradedd \otimes -"']
	\end{tikzcd}
\]
\end{proposition}

\begin{proof}[Sketch of proof]
Since $\forgetd$ is known to be conservative, it is enough to
show that it has a left adjoint and a right adjoint. Let us sketch
the left adjoint (the right adjoint is built in an analogous
manner).

We will use natural compatibility relations between the tensor by
$\gradedd$ and the cotensor. Let $X$ be a graded object and
$M$ be a graded sequence, then there are natural maps
\[
	\gradedd \otimes X^M \longrightarrow \fullpower{\gradedd \otimes X}
	M\qand \gradedd \otimes X^M \longrightarrow X^{[\gradedd, M]}
\]
where the first one uses the fact that $\gradedd \otimes -$ is colax
symmetric monoidal:
\[
	\begin{tikzcd}
		\gradedd \otimes \displaystyle\prod_{n\in\naturals} {\internalhom
		{M(n)} {X^{\otimes n}}}^{\symgroup n}
		\ar[d]
		&
		\\
		\displaystyle\prod_{n\in\naturals}
		\fullpower{\gradedd \otimes \internalhom
		{M(n)} {X^{\otimes n}}}{\symgroup n}
		\ar[d] \ar[r]
		&
		\displaystyle\prod_{n\in\naturals} {\internalhom
		{\internalhom \gradedd {M(n)}} {X^{\otimes n}}}^{\symgroup n}
		\\
		\displaystyle\prod_{n\in\naturals} {\internalhom
		{M(n)} {\gradedd \otimes X^{\otimes n}}}^{\symgroup n}
		\ar[d]
		& \\
		\displaystyle\prod_{n\in\naturals} {\internalhom
		{M(n)} {\fullpower{\gradedd \otimes X}{\otimes n}}}^{\symgroup n}
	\end{tikzcd}
\]
Thanks to the monadicity of $\forgetd$ for dg-operads, we have the
unit morphism $\operad \to [\gradedd,\forgetd \operad]$. Let $V$ be
a $\forgetd \operad$-coalgebra, one can endow $\gradedd \otimes V$ with a
$\operad$-coalgebra structure using the coproduct $\gradedd \to \gradedd
\otimes \gradedd$:
\[
	\gradedd \otimes V \to
	\gradedd \otimes \left(V^{\forgetd\operad}\right)
	\to \gradedd \otimes \gradedd \otimes \left(V^{\forgetd\operad}\right)
	\to {(\gradedd \otimes V)}^{[\gradedd,\forgetd\operad]}
	\to {(\gradedd \otimes V)}^\operad.
\]
Given a $\forgetd\operad$-coalgebra $V$,
we have
$\forgetd(\power{(\gradedd \otimes V)}\operad)
\isonat \fullpower{\forgetd(\gradedd
\otimes V)}{\forgetd \operad}$%
~[\ref{rmk:free_graded_versus_free_dg}].
The unit of the adjunction is then given by the unit
at the level of the underlying graded objects $V \to \forgetd (\gradedd
\otimes V)$:
\[
	\begin{tikzcd}[ampersand replacement=\&]
		V
		\arrow[r, ""]
		\arrow[d, "a", swap]
		\& \forgetd(\gradedd \otimes V)
		\arrow[d, "\forgetd(\gradedd \otimes a)"] \\
		V^{\forgetd \operad}
		\arrow[r, "", swap]
		\& 
		\fullpower{\forgetd(\gradedd \otimes V)}{\forgetd \operad}
	\end{tikzcd}
\]
The same happens for the counit.
\end{proof}

\begin{remark}[(Wording)]%
\label{rmk:graded_epi_versus_dg_epi}
Since $\forgetd$ is both conservative and a left adjoint, it preserves
and reflects isomorphisms and epimorphisms. And since we have the
commutative square of conservative left adjoints
\[
	\begin{tikzcd}[ampersand replacement=\&]
		\catofcog\operad
		\arrow[r, "\forgetd"]
		\arrow[d, "\forget_\operad", swap]
		\& \catofcog{\forgetd\operad}
		\arrow[d, "\forget_{\forgetd\operad}"] \\
		\chain \catss
		\arrow[r, "\forgetd", swap]
		\& \catss^\integers,
	\end{tikzcd}
\]
we can freely switch between the terms ‘epimorphism’
‘degree-wise epimorphism’ when speaking of morphisms of
$\operad$-coalgebras or of a morphism of chain complexes.

On the contrary, the forgetful functor $\forget_\operad$ does not
a priori preserve monomorphisms and we shall always say precisely which
morphism is a monomorphism.
\end{remark}

\begin{definition}[(Degree-wise monomorphism)]%
\label{def:degree-wise_monomorphism_of_coalgebras}
We shall say that a morphism of $\operad$-coalgebras $f : V \to W$ is
a degree-wise monomorphism if $\forget_\operad f : \forget_\operad
V \to \forget_\operad W$ is a monomorphism of chain complexes. Since
$\forget_\operad$ is conservative, it is in particular a monomorphism
of $\operad$-coalgebras.
\end{definition}

A similar issue appears when dealing with epimorphisms of
$\coperad$\=/algebras%
~[\ref{def:degree-wise_epimorphism_of_algebras}].

\subsection{Semi-simplicity}%
%-----------------------------------------------------------------------
\label{sec:hypothese_semi_simple}

\begin{center}
\begin{tikzpicture}
	\node[draw = black, rectangle, inner sep = 10pt, rounded corners]
	{%
	\begin{minipage}{0.8\textwidth}
		\begin{center}
			From now on we shall assume that $\catss$ is semi-simple.
		\end{center}
	\end{minipage}
	};
\end{tikzpicture}
\end{center}

\begin{definition}
We shall say that an abelian category $\catss$ is semi-simple if
every short exact sequence in $\catss$ splits.
\end{definition}

\begin{theorem}%
\label{thm:structure_de_modele_sur_ch}
There exists a combinatorial model structure on
the category $\chain {\catss}$ where
\begin{itemize}
	\item weak equivalences are the quasi-isomorphisms;
	\item fibrations are the degree-wise epimorphisms;
	\item cofibrations are the degree-wise monomorphisms.
\end{itemize}
\end{theorem}

\subsection{Monomorphy conditions}
%-----------------------------------------------------------------------

\begin{center}
\begin{tikzpicture}
	\node[draw = black, rectangle, inner sep = 10pt, rounded corners]
	{%
	\begin{minipage}{0.7\textwidth}
		\begin{center}
			From now on we shall assume that
			$(\catss, \otimes, \monoidalunit)$ satisfies the monomorphy
			conditions.
		\end{center}
	\end{minipage}
	};
\end{tikzpicture}
\end{center}

\begin{definition}%
\label{def:mono_cond}
We shall say that $(\catss, \otimes , \monoidalunit)$ satisfies the
monomorphy conditions if:
\begin{itemize}
	\item for any $A, B, C, D \in \catss$ the natural map
	      \begin{equation*}
	      	\tag{$\otimes \leftrightarrow [-]$}\qquad
	      	\internalhom A B \otimes \internalhom C D \longrightarrow
	      	\internalhom {A \otimes C} {B \otimes D}
	      \end{equation*}
	      is a monomorphism;
	\item for any sequence $A, B_0, B_1, \dots \in \catss$, the
	      natural morphism
	      \begin{equation*}
	      	A \otimes \prod_{n \in \naturals} B_n \longrightarrow
	      	\prod_{n \in \naturals} A \otimes B_n
	      	\tag{$\otimes \leftrightarrow \prod$}
	      \end{equation*}
	      is a monomorphism.
\end{itemize}
\end{definition}

\begin{example}
If $R$ is a semi-simple ring and $\catss$ is the category of left
$R$-modules, then any closed symmetric monoidal structure on
$\catss$ satisfies the monomorphy conditions. One can see this
using the classification of bilinear functors on the category
of $R$-modules: they are given by tensoring over $R$ with
a left $R$-module $\Xi$ endowed with two other right $R$-module
structures%
~\cite{MR:125148, doi:10.1090/s0002-9939-1960-0118757-0,
doi:10.1016/j.jpaa.2010.06.024}.
One then has
\[
	M \otimes N = (\Xi \otimes_R M) \otimes_R N.
\]
This can be extended to the case where $R$ is non-unital and regular,
with $\catss$ being the category of regular left $R$-modules, on the
condition that the $R$-module $\Xi$ be regular for the
three $R$-actions%
~\cite{zbMATH:01203017}.
\end{example}

\begin{remark}
In any exact logos, because of the exactness of filtered colimits,
the canonical map
\begin{equation*}
	\bigoplus_{i \in I} \prod_{j \in J} A_{i,j} \longrightarrow
	\prod_{j\in J}\bigoplus_{i \in I} A_{i,j}
	\tag{$\oplus \leftrightarrow \prod$}
\end{equation*}
is a monomorphism for any family of objects indexed by
a set $I \times J$.
\end{remark}

\begin{proposition}%
\label{thm:ch_monomorphy}
If $(\catss, \otimes, \monoidalunit)$ satisfies the monomorphy
conditions, so does $(\chain {\catss}, \otimes, \monoidalunit)$.
\end{proposition}

\begin{proof}
Let us prove $(\otimes \leftrightarrow [-])$.
For any chain complexes $A,B,\cobarfunctor,D$, the natural 
composition
\[
	\begin{tikzcd}
		\displaystyle\prod_{p\in\integers} [A_p, B_{n+p}]
		\otimes \prod_{q\in\integers} [\cobarfunctor_q, B_{m+q}]
		\arrow[d]
		\\
		\displaystyle\prod_{p,q\in\integers} [A_p, B_{n+p}]
		\otimes [\cobarfunctor_q, D_{m+q}]
		\arrow[d]
		\\
		\displaystyle\prod_{p,q\in\integers}
		[A_p \otimes \cobarfunctor_q, B_{n+p} \otimes D_{m+q}]
		\end{tikzcd}
\]
is a monomorphism for any integers $n,m$ since $\catss$
satisfies the monomorphy conditions $(\otimes \leftrightarrow
\prod)$ and $(\otimes \leftrightarrow [-])$.
As a consequence the composition
\[
	[A,B] \otimes [\cobarfunctor, D] \longrightarrow
	\bigoplus_{n,m\in\integers}
	\displaystyle\prod_{p,q\in\integers}
	[A_p \otimes \cobarfunctor_q, B_{n+p} \otimes D_{m+q}]
\]
is a monomorphism.

Using $(\oplus \leftrightarrow \prod)$, one gets that
\[
	\begin{tikzcd}
		\displaystyle\bigoplus_{n,m\in\integers}
		\displaystyle\prod_{p,q\in\integers}
		\internalhom{A_p \otimes C_q}{B_{n+p} \otimes D_{m+q}}
		\arrow[d]
		\\
		\displaystyle\bigoplus_{\alpha\in\integers}
		\displaystyle\prod_{p,q\in\integers}
		\bigoplus_{n+m=\alpha}
		\internalhom{A_p \otimes C_q}{B_{n+p} \otimes D_{m+q}}
		\arrow[d]
		\\
		\displaystyle\bigoplus_{\alpha\in\integers}
		\displaystyle\prod_{p,q\in\integers}
		\internalhom {A_p \otimes C_q} {\bigoplus_{n+m = \alpha} B_{n+p}
		\otimes D_{m+q}} \isonat \internalhom {A \otimes \cobarfunctor}
		{B \otimes D}
	\end{tikzcd}
\]
is also a monomorphism.

The monomorphy condition $(\otimes \leftrightarrow \prod)$ is a direct
consequence of the
products in $\chain \catss$ being computed degree-wise%
~[\ref{rmk:graded_versus_dg}] and of $(\otimes \leftrightarrow \prod)$
and $(\oplus \leftrightarrow
\prod)$ from $\catss$.
\end{proof}

\begin{corollary}
The category of chain complexes $(\chain \catss, \otimes,
\monoidalunit)$ satisfies the rewrite conditions%
~[\ref{thm:conditions_de_reecriture}].
\end{corollary}

\begin{proof}
For this we only need to prove $(\otimes \leftrightarrow \cap)$.

First we remark that the tensor structure is exact because
$\catss$ is semi-simple. Then
let $A$ be a chain complex and let
\[
	\dots \subset B_n \subset \dots \subset B_1 \subset B_0 = B
\]
be a decreasing sequence of chain complexes. The intersection can be
obtained as the kernel of the natural map
\[
	\bigcap_{n\in\naturals} B_n = \ker \left(
	B \longrightarrow \prod_{n\in\naturals} B/B_n\right).
\]
By exactness one writes
\[
	A \otimes \bigcap_{n\in\naturals} B_n = \ker \left(
	A \otimes B \longrightarrow A \otimes
	\prod_{n\in\naturals} B/B_n\right).
\]
Thanks to the monomorphy condition $(\otimes \leftrightarrow
\prod)$%
~[\ref{thm:ch_monomorphy}], this is equal to
\[
	A \otimes \bigcap_{n\in\naturals} B_n = \ker \left(
	A \otimes B \longrightarrow
	\prod_{n\in\naturals} A \otimes B/B_n\right).
\]
Finally by exactness again, for every natural $n$
\[
	A \otimes B/B_n \isonat \coker \left(A \otimes B_n \longrightarrow
	A \otimes B\right),
\]
from which one gets
\[
	A \otimes \bigcap_{n\in\naturals} B_n = \ker \left(
	A \otimes B \longrightarrow
	\prod_{n\in\naturals} (A \otimes B)/(A \otimes B_n)\right) =
	\bigcap_{n\in\naturals}A \otimes B_n.
\]
\end{proof}
\section{Derivations and coderivations}
%=======================================================================

This section deals with the transition from graded operads to operads in
$\catofmod {\s^p}(\catss)$. The added data corresponds to what is called
a derivation. We develop in particular the theory of coderivations of
graded coalgebras over an operad and the theory of derivations of algebras
over  cooperads.

Let us first recall well-known facts about derivations and coderivations
on operads and  cooperads%
~\cite{doi:10.1007/978-3-642-30362-3}.

\begin{definition}
By graded operad, graded cooperad, graded algebra and graded coalgebra, we
mean an operad, cooperad etc.\ in the category $(\catss^{\integers},
\otimes, \monoidalunit)$ of graded objects in $\catss$.
\end{definition}

\begin{definition}
Let $(\operad ,m,\eta)$ be a graded operad. A degree $p$
derivation $d$ on this operad is the data of degree $p$ maps $d(n):
\operad (n) \to \operad (n)$ such that the following equation between
maps from $\operad  \compofsymseq \operad $ to $\operad $ holds:
\[
	d \circ m = m \circ
	\left( d \compofsymseq \id {\operad } + \id {\operad }
	\compofsymseq' d  \right).
\]
\end{definition}

\begin{definition}
Let $(\coperad,w,\tau)$ be a graded  cooperad. A degree $p$
coderivation $d$ on this cooperad is the data of degree $p$ maps $d(n):
\coperad(n) \to \coperad(n)$ such that the following equation between
maps from $\coperad$ to $\coperad \compofsymseq \coperad$ holds:
\[
	w \circ d = \left( d \compofsymseq \id \coperad +
	\id \coperad \compofsymseq' d  \right) \circ w.
\]
\end{definition}

\begin{proposition}
The category of graded operads endowed with a degree $p$ derivations
and whose morphisms commute with the derivations is isomorphic to the
category operads in $\catofmod {\s^p}(\catss)$. Similarly,
the category of graded  cooperads with degree $p$
coderivations and whose morphisms commutes with the coderivations is
equivalent to the category of  cooperads in
$\catofmod {\s^p}(\catss)$.
\end{proposition}

\begin{proposition}%
\label{thm:derivation-operad-lie}
Let $(\operad ,m,\eta)$ be a graded operad. The derivations of $\operad$
may be organised into a graded Lie algebra
$(\dergrset \ast \operad, [-,-])$, where the bracket is given by
\[
	[d,d'] = d\circ d' - {(-1)}^{|d||d'|} d' \circ d .
\]
Similarly, for $(\coperad,w,\tau)$ a graded  cooperad. The
coderivations of $\coperad$ may be organised into a graded Lie algebra
$(\codergrset \ast \coperad,[-,-])$ whose bracket is defined
similarly.
\end{proposition}

\subsection{Derivations of algebras over  cooperads}
%-----------------------------------------------------------------------

Let $(\coperad,w,\tau)$ be a graded  cooperad
and let $(\Lambda,a)$ be a graded $\coperad$\=/algebra.

\begin{definition}
Let $d_\coperad$ be a
degree $p$ coderivation of $\coperad$
We call a derivation of $\Lambda$ relatively to
$d_\coperad$ a degree $p$ map $d$ from $\Lambda$ to itself such that
the following equality holds between maps from $\Lambda^\coperad$ to
$\Lambda$
\[
	d \circ a = a \circ \left( {\sha(\idfunctor, d)}^{\coperad}
	- \Lambda^{d_{\coperad}} \right).
\]
\end{definition}

\begin{remark}
Let $d_\coperad$ be a coderivation of $\coperad$. Then, the category of
$\coperad$\=/algebras in
$\catofmod{\s^p}(\catss)$ is canonically isomorphic to the category
of graded $\coperad$\=/algebras with a derivation relatively to
$d_\coperad$ and whose morphisms commute with the derivations.
\end{remark}

\begin{definition}
We shall denote by $\dergrset\ast{\coperad, \Lambda}$ the graded set
whose degree $p$ part is made up of pairs $(d_\coperad, d_\Lambda)$
where $d_\coperad$ is a degree $p$ coderivation of $\coperad$
and $d_\Lambda$ is a degree $p$ derivation 
of $\Lambda$ relatively to the coderivation $d_\coperad$ of
$\coperad$.
\end{definition}

\begin{proposition}%
\label{thm:addition_of_derivations}
For any integer $p$, the set $\dergrset{p}{\coperad, \Lambda}$ has 
a canonical structure of commutative group whose sum is
given by the formula
\[
	(d_\coperad, d_\Lambda) + (d_\coperad', d_\Lambda')
	= (d_\coperad + d_\coperad, d_\Lambda + d_\Lambda).
\]
\end{proposition}

\begin{proof}
One has
\begin{align*}
	\left(d_\Lambda + d_\Lambda' \right) \circ a
	&=  a \circ \left( 
	{\sha(\idfunctor, d_\Lambda)}^{\coperad}
	- \Lambda^{d_{\coperad}}
	\right) + 
	a \circ \left( 
	{\sha(\idfunctor, d_\Lambda')}^{\coperad}
	- \Lambda^{d_{\coperad}'}
	\right)
	\\
	&= a \circ \left( 
	{\sha(\idfunctor, d_\Lambda+ d_\Lambda')}^{\coperad}
	- \Lambda^{d_{\coperad}+d_{\coperad}'}
	\right).
\end{align*}
\end{proof}

\begin{proposition}%
\label{thm:bracket_derivation_algebras}
The graded commutative group 
$\dergrset\ast{\coperad, \Lambda}$ has a canonical
structure of a graded Lie algebra whose bracket is given by
\[
	\left[(d_\coperad, d_\Lambda),(d_\coperad', d_\Lambda')\right] 
	\coloneqq
	\left( [d_\coperad,d_\coperad'], [d_\Lambda,d_\Lambda']\right),
\]
where
\[
	\left[d_\Lambda,d_\Lambda'\right] \coloneqq 
	d_\Lambda\circ d_\Lambda'
	- {(-1)}^{|d_\Lambda||d_\Lambda'|} d_\Lambda' \circ d_\Lambda.
\]
\end{proposition}

\begin{proof}
In view of the definition of the bracket, it is enough to check that
for any two derivations $d$ and $d'$ of $\Lambda$, the bracket
$[d,d']$ is again a derivation of $\Lambda$.

For any natural $n$, one has
\begin{align*}
	{\sha(\idfunctor,[d,d'])}^{\otimes n}
	&= {\sha(\idfunctor,d\circ d')}^{\otimes n} - {(-1)}^{|d||d'|}
	{\sha(\idfunctor, d' \circ d )}^{\otimes n}
	\\
	\begin{split}
		&={\sha(\idfunctor, d)}^{\otimes n}
		\circ {\sha(\idfunctor, d')}^{\otimes n}
		\\
		& \qquad - {(-1)}^{|d||d'|} {\sha(\idfunctor, d')}^{\otimes n}
		\circ {\sha(\idfunctor, d)}^{\otimes n},
	\end{split}
\end{align*}
from which one gets%
~[\ref{lemma:calcul-inversion},~\ref{thm:signe-pour-les-puissances}]
\[
	\left[d_\Lambda, d_\Lambda'\right] \circ a
	=
	a \circ \left(
	{\sha\left(\idfunctor,
	\left[d_\Lambda,d_\Lambda'\right]\right)}^\coperad
	-\Lambda^{\left[d_\coperad,d_\coperad'\right]}
	\right).
\]
\end{proof}

\subsection{Ideals and derivations}
%-----------------------------------------------------------------------

Let $\coperad$ be a graded  cooperad with coderivation and let
$(\Lambda,a,D)$ be a $\coperad$-algebra.

\begin{proposition}
A graded ideal $\ideali$ of $(\Lambda,a)$ is an ideal of $(\Lambda, a,
D)$ if and only if $D(\ideali) \subobject \ideali$.
\end{proposition}

\begin{proof}
We already know that $\ideali$ is a graded ideal, so that the
quotient map $\Lambda \longrightarrow \Lambda/\ideali$ is a map of
graded $\coperad$-algebras. Since $\ideali$ is the kernel of that
map, the descend condition for $\Lambda/\ideali$ to inherit the
derivation of $\Lambda$ is that $D(\ideali) \subobject \ideali$.
\end{proof}

\begin{example}
For any natural $n$, $D(\ideal n \Lambda) \subset \ideal n \Lambda$.
\end{example}

\begin{proposition}%
\label{thm:derivation_sur_ideal_engendre}
Let $X$ be a graded subobject of $\Lambda$ and let us write $\ideali
= (X)$ the graded ideal generated by $X$. If $D(X)\subobject \ideali$,
then $D(\ideali)\subobject \ideali$.
\end{proposition}

\begin{proof}
Since $D$ is a derivation of $\Lambda$ and we have both
$X \subobject \ideali$ and $D(X) \subobject \ideali$,
one has
\[
	D(\ideali)
	= D\left(a\left({\sha(\Lambda, X)}^\coperad\right)\right)
	\subobject a\left({\sha(\Lambda, \ideali)}^\coperad\right)
	 = \ideali .
\]
\end{proof}

\subsection{Derivations of free algebras}
%-----------------------------------------------------------------------

In this subsection $(\coperad,w,\tau,d_\coperad)$ is again a graded
 cooperad equipped with a degree $p$ coderivation. We deal
here with the
derivations of free graded $\coperad$-algebras, that is algebras of
the form $(X^\coperad, \canonicalaction)$, for $X$ a graded object. We
will denote by $\canonicalinj$ the map $X^\tau : X^{\monoidalunit} \to
X^\coperad$. In the case where $\Lambda=X^\coperad$ is a free graded
$\coperad$-algebra, the affine
space of derivations relative to $d_\coperad$ has a canonical base
point.

\begin{definition}%
\label{def:derivation-canonique-induite}
The degree $p$ derivation of $X^\coperad$ relative to $d_\coperad$
\[
	\canonicaldiff \coloneqq - X^{d_\coperad}
\]
shall be called the canonical derivation.
\end{definition}

\begin{proposition}%
\label{thm:derivation-algebre-libre}
The map $d \mapsto d \circ \canonicalinj$,
sending derivations of $X^\coperad$ relatively to
$d_\coperad$ to degree $p$ maps $X \to X^\coperad$, admits an
inverse given by
\[
	\induceddiff f
	:= \canonicaldiff + \canonicalaction
	\circ   {\sha(\canonicalinj,f)}^\coperad.
\]
\end{proposition}

\begin{proof}
It is enough to show the proposition in the case where $d_\coperad$
is trivial%
~[\ref{thm:addition_of_derivations}].

Let us first show that $d \mapsto d \circ \canonicalinj$ is injective.
Given a derivation $d$ relatively to zero,
one can draw the commutative digram
\[
	\begin{tikzcd}[ampersand replacement=\&]
		{\left(X^{\monoidalunit}\right)}^\coperad
		\arrow[r,"\canonicalinj^\coperad"]
		\arrow[rd]
		\arrow[rd, phantom, shift left, "\rotatebox{-30}{\(\sim\)}"]
		\& {\left(X^\coperad\right)}^\coperad  \arrow[d,"\canonicalaction"]
		\arrow[rrr," {
		{\sha (\idfunctor, d)}^{\coperad}}"]
		\&\&\& {\left(X^\coperad\right)}^\coperad
		\arrow[d,"\canonicalaction"] \\
		\& X^\coperad  \arrow[rrr,"d"] \&\&\& X^\coperad.
	\end{tikzcd}
\]
Then,
\[
	d= \canonicalaction \circ  {\sha (\idfunctor, d)}^\coperad \circ
	\canonicalinj^{\coperad}
	=\canonicalaction\circ
	{\sha(\canonicalinj,d\circ\canonicalinj)}^\coperad.
\]
Then $d$ is determined by $d \circ\canonicalinj$.
Conversely, let $f$ be a degree $p$ map from
$X$ to $X^\coperad$. Let us prove that the map 
\[
	\induceddiff f
	= \canonicalaction  \circ   {\sha(\canonicalinj,f)}^\coperad
\]
is a derivation. This is equivalent to the following square 
being commutative.
\[
	\begin{tikzcd}[ampersand replacement=\&]
		{\left(X^\coperad\right)}^\coperad
		\arrow[rrr," {{\sha(\idfunctor,\induceddiff f )}^\coperad}"]
		\arrow[d]
		\&\&\& {\left(X^\coperad\right)}^\coperad \arrow[d]\\
		X^{ \coperad} \arrow[rrr,"\induceddiff f "']
		\&\&\& X^{\coperad }.
	\end{tikzcd}
\]
This square can be decomposed as follows.
\[
	\begin{tikzcd}[ampersand replacement=\&]
		{\left(X^\coperad\right)}^\coperad
		\arrow[rrr,"{\sha
		\left(\canonicalinj^\coperad,{\sha
		(\canonicalinj,f)}^\coperad\right)}^\coperad"]
		\arrow[d]
		\&\&\& {\left({\left(X^\coperad\right)}^\coperad\right)}^\coperad
		\arrow[r]
		\arrow[d]
		\& {\left( X^{\coperad \circ \coperad} \right)}^\coperad
		\arrow[r]
		\arrow[d]
		\&  {\left(X^\coperad \right)}^\coperad
		\arrow[d]
		\\
		X^{\coperad \compofsymseq \coperad}
		\arrow[rrr,
		"{{\sha (\canonicalinj,f)}^{\coperad \compofsymseq \coperad}}"']
		\arrow[d,"{X^w}"]
		\&\&\& {\left(X^\coperad \right)}^{\coperad \compofsymseq \coperad}
		\arrow[r]
		\arrow[d,"{\left(X^\coperad \right)}^w"]
		\& X^{\coperad \compofsymseq \coperad \compofsymseq \coperad}
		\arrow[r]
		\arrow[d,"X^{ w \compofsymseq \idfunctor}"]
		\& X^{\coperad \compofsymseq \coperad} \arrow[d,"X^w"]
		\\
		X^{\coperad} \arrow[rrr,"{{\sha(\canonicalinj,f)}^\coperad}"']
		\&\&\& {\left(X^\coperad\right)}^{\coperad }
		\arrow[r]
		\& X^{\coperad  \compofsymseq \coperad}
		\arrow[r,"X^w"']
		\& X^{\coperad}
	\end{tikzcd}
\]
All these squares are commutative; $\induceddiff f $ is a derivation.
Finally, one has $\induceddiff f \circ \canonicalinj = f$.
\end{proof}

\begin{proposition}%
\label{thm:morphism_free_algebra}
Let $(\Lambda,a,d_\Lambda)$ be a \mbox{$\coperad$-algebra} equipped with
a derivation $d_\Lambda$ and let
$(X^\coperad,\canonicalaction,\induceddiff f )$ be a free graded
$\coperad$-algebra equipped with a derivation induced by a degree $p$
map $f: X \to X^\coperad$. The morphism of graded $\coperad$-algebras
$\tilde g: X^\coperad \to \Lambda$ induced by a graded map
$g: X\to \Lambda$ commutes with derivations if and only if
\[
	\tilde g \circ  f = d_\Lambda \circ g.
\]
\end{proposition}

\begin{proof}
Suppose that $\tilde g \circ \induceddiff f = d_\Lambda
\circ \tilde g$. Then, $\tilde g \circ \induceddiff f
\circ\canonicalinj= d_\Lambda \circ \tilde g \circ\canonicalinj$,
that is, $\tilde g \circ f = d_\Lambda \circ g$. Conversely,
suppose that $\tilde g \circ f = d_\Lambda \circ g$. On the one
hand, $\tilde g \circ \induceddiff f $ equals the sum of the two
following compositions:
\[
	\begin{tikzcd}[ampersand replacement=\&]
		X^\coperad \arrow[rr,"{{\sha(\canonicalinj,f)}^\coperad}"]
		\& \& {\left(X^\coperad\right)}^\coperad
		\arrow[r,"{{\left(g^\coperad\right)}^\coperad}"]
		\& {\left(\Lambda^\coperad\right)}^\coperad \arrow[r] \& \Lambda.\\
		X^\coperad \arrow[rr,"{-X^{d_\coperad}}"]
		\& \& X^\coperad \arrow[r,"g^\coperad"]
		\& \Lambda^\coperad \arrow[r] \& \Lambda.
	\end{tikzcd}
\]
On the other hand, since $d_\Lambda$ is a derivation,
$d_\Lambda \circ \tilde g$ equals the sum of the two following
compositions:
\[
	\begin{tikzcd}[ampersand replacement=\&]
		X^\coperad \arrow[r,"g^\coperad"]
		\& \Lambda^\coperad
		\arrow[rr,"{{\sha ( \idfunctor, d_\Lambda )}^\coperad}"]
		\&\& \Lambda^\coperad \arrow[r] \& \Lambda.\\
		X^\coperad \arrow[r,"g^\coperad"]
		\& \Lambda^\coperad \arrow[rr,"{-\Lambda^{d_\coperad}}"]
		\&\& \Lambda^\coperad \arrow[r] \& \Lambda.
	\end{tikzcd}
\]
The first composite map equals the third one, and the
second one equals the fourth one.
\end{proof}

\begin{proposition}%
\label{thm:squarederiv}%
\label{thm:derivation_de_carre_nul}
Assume $p$ is odd, then
\[
	\induceddiff f ^2
	=\canonicalaction \circ \sha(\canonicalinj, \induceddiff f \circ f )
	+ \canonicaldiff^2.
\]
In particular, if $d_\coperad^2 = 0$, then $\induceddiff f ^2 = 0$ if
and only if $\induceddiff f \circ f = 0$.
\end{proposition}

\begin{proof}
The square $\induceddiff f ^2$ is the following map
\begin{align*}
	\induceddiff f ^2 &= \canonicaldiff^2 
	- {\left(X^\coperad\right)}^{d_\coperad} \circ\canonicalaction \circ
	{\sha(\canonicalinj,f)}^\coperad
	-\canonicalaction \circ {\sha(\canonicalinj,f)}^\coperad
	\circ {\left(X^\coperad\right)}^{d_\coperad}\\
	&\qquad+\canonicalaction \circ {\sha(\canonicalinj,f)}^\coperad
	\circ\canonicalaction \circ {\sha(\canonicalinj,f)}^\coperad.
\end{align*}
Since $\canonicaldiff$ is a derivation, we have the following equality
between maps from ${\left(X^\coperad\right)}^\coperad$ to $X^\coperad$
\[
	X^{d_\coperad} \circ \canonicalaction
	=\canonicalaction \circ {\left(X^\coperad\right)}^{d_\coperad}
	+\canonicalaction \circ
	\sha {\left(\idfunctor, X^{d_\coperad}\right)}^\coperad.
\]
Since the degrees of $d_\coperad$ and $f$ are odd, then 
\[
	{\left(X^\coperad\right)}^{d_\coperad}
	\circ  {\sha(\canonicalinj,f)}^\coperad
	+  {\sha(\canonicalinj,f)}^\coperad
	\circ X^{d_\coperad} = 0.
\]
So
% Some manual spacing has been done below
\begin{align*}
	{\left(X^\coperad\right)}^{d_\coperad}\circ\canonicalaction\!\circ
	{\sha(\canonicalinj,f)}^\coperad 
	+\canonicalaction \circ {\sha(\canonicalinj,f)}^\coperad
	\circ {\left(X^\coperad\right)}^{d_\coperad}
	&=\canonicalaction \circ
	{\sha (\idfunctor, X^{d_\coperad})}^\coperad
	\!\!\circ {\sha(\canonicalinj,f)}^\coperad \\
	&= \canonicalaction \circ {\sha (\canonicalinj,
	X^{d_\coperad} \circ f)}^\coperad.
\end{align*}
The last equality follows from the fact that $X^{d_\coperad}
\circ \canonicalinj=0$. Moreover, since $\canonicalaction \circ
{\sha(\canonicalinj,f)}^\coperad$ is a derivation relatively to zero,
then,
\[
	\canonicalaction \circ {\sha(\canonicalinj,f)}^\coperad \circ
	\canonicalaction \circ {\sha(\canonicalinj,f)}^\coperad
	= \canonicalaction \circ {\sha \left(\idfunctor, \canonicalaction
	\circ {\sha (\canonicalinj,f)}^\coperad\right)}^\coperad
	\circ {\sha(\canonicalinj,f)}^\coperad.
\]
Since the degree of $f$ is odd and since $f = \canonicalaction \circ
{\sha(\canonicalinj,f)}^\coperad \circ \canonicalinj$, then
\[
	{\sha \left(\idfunctor, \canonicalaction
	\circ {\sha (\canonicalinj,f)}^\coperad\right)}^\coperad
	\circ {\sha(\canonicalinj,f)}^\coperad
	= {\sha\left(\canonicalinj , \canonicalaction
	\circ {\sha(\canonicalinj,f)}^\coperad \circ f \right)}^\coperad.
\]
Hence,
\[
	\canonicalaction^\coperad \circ {\sha \left(\canonicalinj^\coperad,
	{\sha (\canonicalinj,f)}^\coperad\right)}^\coperad
	\circ {\sha(\canonicalinj,f)}^\coperad
	= \sha(\canonicalinj, \canonicalaction
	\circ  {\sha(\canonicalinj,f)}^\coperad \circ f) .
\]
Finally, one has
\[
	\induceddiff f ^2 = \canonicalaction \circ
	{\sha(\canonicalinj, (\canonicalaction
	\circ {\sha(\canonicalinj,f)}^\coperad-
	X^{d_\coperad})\circ f)}^\coperad +\canonicaldiff^2
	=\canonicalaction  \circ \sha(\canonicalinj, \induceddiff f \circ f )
	+ \canonicaldiff^2.
\]
\end{proof}

\subsection{Coderivation of a coalgebra over an operad}
%-----------------------------------------------------------------------

Let $(\operad, m, \eta)$ be a graded operad and let $(V,a)$ be a
$\operad$-coalgebra.

\begin{definition}%
\label{def:coderivation}
Let $d_\operad $ be a degree $p$ derivation of $\operad $.
A coderivation of $V$ relatively to $d_\operad $ is a degree $p$ map
$d: V\to V$ such that
\[
	a \circ d =  \left( {\sha (\id {V}, d)}^\operad  -V^{d_\operad }
	\right) \circ  a.
\]
\end{definition}

\begin{remark}%
\label{rmk:coderaffine}
Similarly to the case of derivations of algebras over 
cooperads%
~[\ref{thm:bracket_derivation_algebras}],
the coderivations of $V$ relative to derivations of $\operad$ can be
organised into a graded Lie algebra $\codergrset \ast
{\operad, V}$.
\end{remark}

\subsection{Coderivations of cofree coalgebras}
%-----------------------------------------------------------------------

In this subsection $(\operad ,m,\eta,d_\operad )$ is again a graded
operad equipped with a degree $p$ derivation. We deal here with the
coderivations of the cofree graded $\operad $-coalgebras, that is the
graded coalgebras of the form $\freecog\operad X$. As before, we shall use
$\canonicalproj$ for the map induced by $X^\eta$
\[
	\canonicalproj : X^\operad  \longrightarrow X
\]
or its composite
\[
	\canonicalproj : \freecog\operad (X) \subobject X^\operad
	\longrightarrow X .
\]

\begin{lemma}%
\label{thm:definition-de-sha-sur-LOmega}
Let $X,Y$ be two objects of $\catss^{\integers}$. Let $f: X\to Y$ be
a graded morphism and let $g: X \to Y$ be a degree $p$ map. Then, the
composite map
\[
	\freecog\operad  (X) \subobject X^\operad 
	\xrightarrow{{\sha(f, g)}^\operad }
	Y^\operad 
\]
factorises through $\freecog\operad (Y)$. The induced map from
$\freecog\operad X$ to $\freecog\operad Y$ shall also be denoted
${\sha(f, g)}^\operad $.
\end{lemma}

\begin{proof}
The strategy of the proof is to show by induction that the
shuffle map ${\sha(f,g)}^\operad $ is well defined as map from
$\freefunctor_n^\operad (X)$ to $\freefunctor_n^\operad (Y)$. Let us
reuse the notations $\iotaup_n$ and $\deltaup_n$ to denote the maps
constructed inductively
\begin{align*}
	\iotaup_n(X) : \freefunctor_n^\operad  (X)
	&\to \freefunctor_{n-1}^\operad  (X)\\
	\deltaup_n(X): \freefunctor_n^\operad  (X)
	&\to \freefunctor_{n-1}^\operad 
	\circ \freefunctor_{n-1}^\operad  (X).
\end{align*}
Level $n=0$ is given by assumption, here is why it factors through
$\freefunctor_1^\operad $.
Let us recall the construction of $\freefunctor_1(X)$
\[
	\begin{tikzcd}[ampersand replacement=\&]
		\freefunctor_1(X)
		\arrow[r, "\deltaup_1(X)"]
		\arrow[d, "\iotaup_1(X)",swap]
		\arrow[rd, very near start, phantom, "\lrcorner"]
		\& {\left(X^\operad \right)}^\operad 
		\arrow[d, "\laxmap(\operad {,}\operad {,}X)"] \\
		X^\operad 
		\arrow[r, "X^m", swap]
		\& X^{\operad \compofsymseq\operad }
	\end{tikzcd}
\]
Hence the map ${\sha(f,g)}^\operad $ factors through $\freefunctor_1(Y)$
if we can define another map $\freefunctor_1^\operad (X) \to
\power {(Y^\operad)}\operad$ such that both maps coincide in
$Y^{\operad \compofsymseq\operad }$. Our candidate is the following map
\[
	\freefunctor_1^\operad  X \xrightarrow{\deltaup_1(X)}
	{\left(X^\operad \right)}^\operad 
	\xrightarrow{{\sha\left(f^\operad ,
	{\sha(f,g)}^\operad \right)}^\operad }
	{\left(Y^\operad \right)}^\operad  .
\]
Then,
\begin{align*}
	Y^m \circ {\sha(f,g)}^\operad  \circ \iotaup_1(X)
	& = {\sha(f,g)}^{\operad  \compofsymseq \operad }
	\circ  X^m \circ \iotaup_1(X)\\
	&= {\sha(f,g)}^{\operad  \compofsymseq \operad }\circ
	\laxmap(\operad , \operad , X) \circ \deltaup_1(X)\\
	&= \laxmap(\operad , \operad , Y)\circ
	{\sha\left(f^\operad , {\sha(f,g)}^\operad \right)}^\operad 
	\circ \deltaup_1(X).
\end{align*}
So that ${\sha(f,g)}^\operad  \circ \iotaup_1(X)$ factors through
$\freecog\operad _1(Y)$.

As the diagrams involved for the general case are pretty inconvenient to
display, we shall only show how to induce to $\freefunctor_2^\operad $;
the general case is treated in the same way. For this let us write the
full diagram of definition of $\freecog\operad _2(X)$~%
\cite[Section 2.7 Diagram D]{arXiv:1409.4688}:
\[
	\begin{tikzcd}[ampersand replacement=\&]
		\freefunctor_2^\operad (X)
		\arrow[r, ""]
		\arrow[d, "", hookrightarrow]
		\& \freefunctor_1^\operad  \circ \freefunctor_1^\operad (X)
		\arrow[d, hookrightarrow]
		\arrow[r, shift left]
		\arrow[r, shift right]
		\& {\left({\left(X^\operad \right)}^\operad \right)}^\operad 
		\arrow[dd, hookrightarrow]
		\\
		\freefunctor_1^\operad (X)
		\arrow[d, hookrightarrow]
		\arrow[r, "", swap]
		\& {\left(X^\operad \right)}^\operad 
		\arrow[d, hookrightarrow]
		\&
		\\
		X^\operad 
		\arrow[r]
		\& X^{\operad \compofsymseq\operad }
		\arrow[r, shift right]
		\arrow[r, shift left]
		\& X^{\operad  \compofsymseq \operad 
		\compofsymseq \operad } .
	\end{tikzcd}
\]
As $\freefunctor_2^\operad (X)$ is the limit of a diagram of the shape
above, in order to show that ${\sha(f,g)}^\operad $ can be factored
through as a map
\[
	{\sha(f,g)}^\operad  : \freefunctor_2^\operad (X)
	\longrightarrow \freefunctor_2^\operad (Y) .
\]
it is enough to construct of morphism between the two diagrams that
define them. Using what we have built before, we already have what to
use for the lower left corner. For the third column we shall use the
following morphisms
\begin{align*}
	{\sha(f,g)}^{\operad  \compofsymseq \operad 
	\compofsymseq \operad } :  X^{\operad \compofsymseq
	\operad  \compofsymseq \operad } & \longrightarrow
	Y^{\operad  \compofsymseq \operad 
	\compofsymseq \operad } \\
	\sha\left(f^\operad ,
	\sha\left(f^\operad ,{\sha(f,g)}^\operad \right)\right) :
	{\left({\left(X^\operad \right)}^\operad \right)}^\operad 
	& \longrightarrow
	{\left({\left(Y^\operad \right)}^\operad \right)}^\operad  .
\end{align*}
Finally, for the middle top spot, we use the composition
\[
	\sha\left(f^\operad , {\sha(f,g)}^\operad \right) :
	\freefunctor_1^\operad \circ \freefunctor_1^\operad (X)
	\longrightarrow \freefunctor_1^\operad
	\circ \freefunctor_1^\operad (Y)
\]
where by $f^\operad $ we mean $\freefunctor_1^\operad (f) :
\freefunctor_1^\operad (X) \to \freefunctor_1^\operad (Y)$.

We now have to check the commutativity of this middle top arrow with
each of the top parallel pair of arrows. Fortunately, thanks to the
monomorphisms in the diagram, this can be checked using the bottom
parallel pair where it becomes nothing but the equations
\begin{align*}
	Y^{\id \operad \compofsymseq m} \circ
	{\sha(f,g)}^{\operad \compofsymseq\operad } & =
	{\sha(f,g)}^{\operad  \compofsymseq \operad 
	\compofsymseq \operad } \circ X^{\id \operad 
	\compofsymseq\operad } ;
	\\
	Y^{m\compofsymseq \id \operad } \circ 
	{\sha(f,g)}^{\operad \compofsymseq\operad } & =
	{\sha(f,g)}^{\operad  \compofsymseq \operad 
	\compofsymseq \operad } \circ X^{m
	\compofsymseq\id \operad } .
\end{align*}
Such a morphism of diagrams induces a morphism between their limits
\[
	{\sha(f,g)}^\operad  : \freefunctor_2^\operad (X) \longrightarrow
	\freefunctor_2^\operad (Y) .
\]
Repeating this proof for every $n$, we obtain factorisation
\[
	{\sha(f,g)}^\operad  : \freefunctor_n^\operad (X)
	\to \freefunctor_n^\operad (Y)
\]
for every natural $n$ and thus a map on the intersection:
\[
	{\sha(f,g)}^\operad  :
	\bigcap_{n\in\naturals} \freecog\operad (X) \longrightarrow
	\bigcap_{n\in\naturals}
	\freefunctor_n^\operad (Y) .
\]
\end{proof}

We also need to know that we can induce a coderivation on
$\freecog\operad (X)$ by defining it on $X^\operad $.

\begin{lemma}%
\label{thm:reduction-coderivation-a-LOmega}
The following composite map
\[
	\freecog\operad  X \subobject X^\operad 
	\xrightarrow{X^{d_\operad }} X^{\operad }
\]
factorises through $\freecog\operad X$. The induced map from
$\freecog\operad X$ to itself will also be denoted $X^{d_\operad }$.
\end{lemma}

\begin{proof}
We use the same notations and the same layout as in the proof of
\cref{thm:definition-de-sha-sur-LOmega}.
Consider the two following degree $p$ composite maps
\begin{align*}
	& \freefunctor_1^\operad  X \xrightarrow{\iotaup_1(X)} X^\operad 
	\xrightarrow{X^{d_\operad }} X^\operad .\\
	& \freefunctor_1^\operad  X \xrightarrow{\deltaup_1(X)}
	{\left(X^\operad \right)}^\operad 
	\xrightarrow{{\left(X^\operad \right)}^{d_\operad }
	+ {\sha\left(\id X, X^{d_\operad }\right)}^\operad }
	{\left(X^\operad \right)}^\operad .
\end{align*}
Then,
\begin{align*}
	X^m \circ X^{d_\operad }
	\circ \iotaup_1(X)& = X^{d_\operad  \circ
	 m}  \circ \iotaup_1(X)
	\\
	&=X^{m \circ (d_\operad  \compofsymseq
	\id \operad  + \id \operad 
	\compofsymseq' d_\operad )} \circ \iotaup_1(X)
	\\
	&= X^{d_\operad  \compofsymseq \id \operad  +
	\id \operad  \compofsymseq' d_\operad }
	\circ X^m \circ \iotaup_1(X)
	\\ 
	&= X^{d_\operad  \compofsymseq \id \operad  +
	\id \operad  \compofsymseq' d_\operad }
	\circ \laxmap(\operad , \operad , X) \circ \deltaup_1(X)
	\\
	&= \laxmap(\operad , \operad , X) \circ
	\left({\left(X^\operad \right)}^{d_\operad }
	+ {\sha(\id {X^\operad },
	X^{d_\operad })}^\operad \right) \circ \deltaup_1(X).
\end{align*}
Thus, both maps and in particular $X^{d_\operad } \circ \iotaup_1(X)$
factorise through $\freefunctor_1^\operad  X$.

For the induction to $\freefunctor_2^\operad $, we shall only describe
the relevant maps that induce a morphism of diagrams and refer to the
proof of the previous lemma. For the third column we shall use the
following morphisms
\[
	X^{d_\operad \compofsymseq\id \operad \compofsymseq
	\id \operad  +
	\id \operad \compofsymseq'
	(d_\operad \compofsymseq
	\id \operad ) +
	\id \operad \compofsymseq'(\id \operad 
	\compofsymseq'  d_\operad )
	}
	: X^{\operad \compofsymseq
	\operad  \compofsymseq \operad } \longrightarrow
	X^{\operad  \compofsymseq \operad 
	\compofsymseq \operad }
\]
and for the map $
{\left({\left(X^\operad \right)}^\operad \right)}^\operad 
\longrightarrow
{\left({\left(X^\operad \right)}^\operad \right)}^\operad $:
\[
	{\left({\left(X^\operad \right)}^\operad \right)}^{d_\operad }
	+ {\sha\left(\id {{(X^\operad )}^\operad },
	{\left(X^\operad \right)}^{d_\operad }\right)}^\operad 
	+ {\sha\left(\id {{(X^\operad )}^\operad },
	{\sha\left(\id {X^\operad },
	X^{d_\operad }\right)}^\operad \right)}^\operad  .
\]
Finally, for the middle top spot, we use the map
\[
	{\left(\freefunctor_1^\operad (X)\right)}^{d_\operad }
	+ {\sha\left(\id {\freefunctor_1^\operad (X)},
	X^{d_\operad }\right)}^\operad 
	: \freefunctor_1^\operad  \circ \freefunctor_1^\operad (X)
	\longrightarrow \freefunctor_1^\operad
	\circ \freefunctor_1^\operad (X) .
\]
The same type of construction allows us to reduce to
$\freefunctor_n^\operad (X)$ and then to the intersection
\[
	X^{d_\operad } : \bigcap_{n\in\naturals}
	\freefunctor_n^\operad (X)
	\longrightarrow \freefunctor_n^\operad (X) .
\]
\end{proof}

In the case where $V=\freecog\operad  X$ is free, the affine space
of coderivations relatively to $d_\operad$ has a canonical base point.

\begin{lemma}
Given any chain complex $X$, there exists a degree $p$ coderivation
$\canonicaldiff$ of the cofree $\operad $-coalgebra $\freecog\operad X$
such that the following diagram is commutative
\[
	\begin{tikzcd}[ampersand replacement=\&]
		\freecog\operad (X)
		\arrow[r, "\canonicaldiff"]
		\arrow[d, "", swap, hookrightarrow]
		\& \freecog\operad (X)
		\arrow[d, "", hookrightarrow] \\
		X^\operad 
		\arrow[r, "-X^{d_\operad }", swap]
		\& X^\operad  .
	\end{tikzcd}
\]
We shall call the coderivation $\canonicaldiff$ the canonical
coderivation of $\freecog\operad (X)$.
\end{lemma}

\begin{proof}
We already know that the morphism
$-X^{d_\operad }$ factors through $\freecog\operad (X)$%
~[\ref{thm:reduction-coderivation-a-LOmega}]; we
call it $\canonicaldiff$ and we only need to show that it is indeed a
coderivation of $\freecog\operad (X)$.

Consider the following cubical diagram:
\[
	\begin{tikzcd}
		& X^{\operad  \compofsymseq \operad }
		\arrow[from=ddd, "X^m" description, near start]
		\arrow[rrrr, "-X^{d_\operad  \compofsymseq 
		\idfunctor+ \idfunctor \compofsymseq'  d_\operad }"
		description]
		&&&& X^{\operad  \compofsymseq \operad }
		\\
		{(\freecog\operad (X))}^\operad 
		\arrow[rrrr,"{{\sha (\id X,
		\canonicaldiff)}^\operad  +{(\freecog\operad  X)}^{-d_\operad }}"
		description, crossing over]
		\arrow[ru]
		&&&& 
		{(\freecog\operad (X))}^\operad 
		\arrow[ru, hookrightarrow]
		\\
		\\
		&  X^\operad 
		\arrow[rrrr, "-X^{d_\operad }" description, near start]
		&&&&
		X^\operad 
		\arrow[uuu,"X^m"' description]
		\\
		\freecog\operad  X
		\arrow[uuu, "a" description]
		\arrow[rrrr, "\canonicaldiff" description]
		\arrow[ru]
		&&&& \freecog\operad  X.
		\arrow[ru]
		\arrow[uuu, "a"' description, crossing over]
	\end{tikzcd}
\]
We know that all the faces but the front one are commutative.
Since the map $\power{(\freecog\operad (X))}\operad 
\to X^{\operad  \compofsymseq
\operad }$ is a monomorphism, then the front face
is also commutative. 
\end{proof}

\begin{proposition}%
\label{thm:coderivation_colibre}
The map
$d \longmapsto \canonicalproj \circ d$, sending
coderivations relative to $d_\operad $ to degree $p$ maps
$\freecog\operad X \to X$, had an inverse.

Let $f: \freecog\operad  X \to X$ be a degree $p$ map. Then the
composite map
\[
	\induceddiff f:\freecog\operad  X \xrightarrow{\coprodw^\operad (X)}
	\freecog\operad  \freecog\operad  X
	\xrightarrow{{\sha(\canonicalproj, f)}^\operad } \freecog\operad  X
\]
is a coderivation of $\freecog\operad  X$ relatively to zero.
Where $\coprodw^\operad
: \freecog\operad  \to \freecog\operad  \circ \freecog\operad $
is the coproduct of the
comonad $\comonadl^\operad $. Abusing notations, one can say that
the inverse map is given by
\[
	\induceddiff f  = \canonicaldiff
	+ {\sha (\canonicalproj, f)}^\operad   \circ \canonicalaction.
\]
\end{proposition}

\begin{proof}
It is enough to show the proposition in the case where $d_\operad$ is
trivial%
~[\ref{rmk:coderaffine}].

Notice that by the coreflection theorem~%
\cite[Theorem~2.7.11]{arXiv:1409.4688}
that extends the monomorphism $\freecog\operad \subobject {(-)}^\operad
$ into a monomorphism of lax comonads, the map $\induceddiff f $ can be
rewritten using the following commutative diagram
\[
	\begin{tikzcd}[ampersand replacement=\&]
		\freecog\operad  X
		\arrow[rr, "\canonicalaction"]
		\arrow[rrd, "\coprodw^\operad (X)"']
		\&\& {\left(\freecog\operad  X\right)}^\operad 
		\arrow[rr, "{\sha(\canonicalproj{,} f)}^\operad "]
		\&\& X^\operad 
		\\
		\&\& \freecog\operad  \freecog\operad  X
		\arrow[u, hookrightarrow]
		\arrow[rr, "{\sha(\canonicalproj{,} f)}^\operad ", swap]
		\&\& \freecog\operad  X
		\arrow[u, hookrightarrow]
	\end{tikzcd}
\]
Let us also recall that the coalgebra structure of $\freecog\operad  X$ is
given by
\[
	\begin{tikzcd}[ampersand replacement=\&]
		\freecog\operad  X
		\arrow[r, "\canonicalaction"]
		\arrow[d, "", swap, hookrightarrow]
		\& {\left(\freecog\operad  X\right)}^\operad 
		\arrow[d, "\caninjj", hookrightarrow] \\
		X^\operad 
		\arrow[r, "X^m", swap]
		\& X^{\operad  \compofsymseq \operad }
	\end{tikzcd}
\]
Using all the monomorphisms at our disposal, we deduce that
$\induceddiff f $ is a coderivation if and only if
\[
	X^m \circ {\sha(\canonicalproj,f)}^\operad  \circ \canonicalaction =
	\caninjj \circ
	{\sha\left(\id {\freecog\operad (X)},{\sha(\canonicalproj, f)}^\operad
	\circ \canonicalaction \right)}^\operad 
	\circ \canonicalaction.
\]
To conclude, it suffices to notice that the following diagrams are
commutative:
\[
	\begin{tikzcd}
		\freecog\operad  X \ar[r,"\canonicalaction"]
		& {\left(\freecog\operad  X\right)}^\operad 
		\ar[rrr,"{{\sha ( \canonicalproj , f )}^{\operad }}"]
		\ar[d,"{{\left(\freecog {\operad } X \right)}^m}"']
		&&& X^\operad  \ar[d, "{X^m}"] \\
		&{\left(\freecog\operad  X\right)}^{\operad  \compofsymseq \operad }
		\ar[rrr,
		"{\sha {(\canonicalproj , f )}^{\operad\compofsymseq\operad }}"']
		&&& X^{\operad  \compofsymseq \operad } .
	\end{tikzcd}
\]
\[
	\begin{tikzcd}[ampersand replacement=\&]
		\freecog\operad  X \arrow[r, "\canonicalaction"]
		\&
		{\left(\freecog\operad  X\right)}^\operad 
		\arrow[rrr, "{\sha\left(\id {\freecog\operad (X)}{,}
		{\sha(\canonicalproj{,}f)}^\operad \right)}^\operad "]
		\arrow[d, "{\left(\freecog\operad  X\right)}^m", swap]
		\&\&\& 
		{\left(\freecog\operad  X\right)}^\operad 
		\arrow[d, "\caninjj"] \\
		\&
		{\left(\freecog\operad  X\right)}^{\operad  \compofsymseq \operad }
		\arrow[rrr,
		"{\sha(\canonicalproj{,}f)}^{\operad\compofsymseq\operad}",swap]
		\&\&\& X^{\operad  \compofsymseq \operad } .
	\end{tikzcd}
\]
\end{proof}

\begin{proposition}%
\label{thm:morphisms_to_a_free_coalgebra}
Consider a graded \mbox{$\operad $-coalgebra} $(V,a)$ with coderivation
$d_V$ and a cofree graded $\operad $-coalgebra $\freecog\operad X$ with
coderivation $\induceddiff f $ induced by a degree $p$ map $f:
\freecog\operad X \to X$. Let $g : V \longrightarrow X$ be graded map.
Its extension as a morphism of graded $\operad $-coalgebras
$\tilde g$ commutes with the coderivations if and only if
$f \circ \tilde g = g \circ d_V$.
\end{proposition}

\begin{proof}
Suppose that $\tilde g$ is a morphism of $\operad $-coalgebras.
Then,
\[
	\canonicalproj \circ \induceddiff f  \circ \tilde g
	=\canonicalproj \circ  \tilde g \circ d_V
	\Longrightarrow f \circ \tilde g = g \circ d_V.
\]
Conversely, suppose tha $f \circ \tilde g = g \circ d_V$. Since
$d_V$ is a coderivation, we get
\[
	\tilde g \circ d_V = g^\operad  \circ
	\left({\sha(\id V,d_V)}^\operad  - V^{d_V}\right) \circ a .
\]
hence using the assumption
\[
	\tilde g \circ d_V = \left({\sha(\canonicalproj, f)}^\operad 
	- X^{d_\operad }\right) \circ {\left(\tilde g\right)}^\operad
	\circ a = \induceddiff f  \circ \tilde g .
\]
\end{proof}

\begin{proposition}%
\label{thm:coderivation_de_carre_nul}
Assume $p$ is odd, then
\[
	\induceddiff f ^2 = \sha (\canonicalproj, f \circ \induceddiff f )
	\circ \canonicalaction +\canonicaldiff^2.
\]
In particular, if $d_\operad^2 =0$, then $\induceddiff f^2 = 0$ if
and only if $f \circ \induceddiff f = 0$.
\end{proposition}

\begin{proof}
The map $\induceddiff f ^2$ is as follows
\begin{align*}
	\induceddiff f ^2 &=  \canonicaldiff^2
	- X^{d_\operad } \circ {\sha (\canonicalproj, f)}^\operad
	\circ \canonicalaction
	- {\sha (\canonicalproj, f)}^\operad
	\circ \canonicalaction \circ X^{d_\operad }\\
	&\qquad+ {\sha (\canonicalproj, f)}^\operad
	\circ \canonicalaction \circ {\sha (\canonicalproj, f)}^\operad 
	\circ \canonicalaction.
\end{align*}
Since $\canonicaldiff$ is a coderivation
\[
	a \circ X^{d_\operad } = \left( {\sha
	\left(\idfunctor, X^{d_\operad } \right)}^\operad 
	+ {\left(\freecog\operad  X\right)}^{d_\operad } \right) \circ
	\canonicalaction.
\]
Besides, since the degree of $f$ and $d_\operad $ is odd, then we have
the following equality between maps from ${(\freecog\operad X)}^{\operad
}$ to $X^\operad $
\[
	X^{d_\operad } \circ {\sha (\canonicalproj, f)}^\operad  
	+ {\sha (\canonicalproj, f)}^\operad 
	\circ {\left(\freecog\operad  X\right)}^{d_\operad }=0.
\]
So
\[
	 X^{d_\operad } \circ {\sha (\canonicalproj, f)}^\operad
	 \circ\canonicalaction + {\sha (\canonicalproj, f)}^\operad
	 \circ\canonicalaction \circ X^{d_\operad } =
	{\sha (\canonicalproj, f)}^\operad 
	\circ {\sha\left(\idfunctor, X^{d_\operad } \right)}^{\operad }\circ
	\canonicalaction .
\]
Since $\canonicalproj \circ X^{d_\operad }=0$,
\[
	{\sha (\canonicalproj, f)}^\operad   
	\circ {\sha (\idfunctor, X^{d_\operad } )}^\operad 
	= {\sha\left(\canonicalproj, f \circ X^{d_\operad }\right)}^\operad .
\]
So
\[
	X^{d_\operad } \circ {\sha (\canonicalproj, f)}^\operad
	\circ\canonicalaction + {\sha (\canonicalproj, f)}^\operad
	\circ\canonicalaction \circ X^{d_\operad } =
	{\sha \left(\canonicalproj, f \circ X^{d_\operad }\right)}^\operad
	\circ \canonicalaction.
\]
Besides, since ${\sha (\canonicalproj, f)}^\operad \circ
\canonicalaction$ is a coderivation relatively to zero, then
\begin{align*}
	{\left({\sha (\canonicalproj, f)}^\operad
	\circ \canonicalaction\right)}^2
	&=  {\sha (\canonicalproj, f)}^\operad 
	\circ {\sha \left(\idfunctor, {\sha (\canonicalproj, f)}^\operad 
	\circ \canonicalaction\right)}^{\operad }
	\circ \canonicalaction\\
	&= {\sha \left(\canonicalproj, f \circ
	{\sha (\canonicalproj, f)}^\operad
	\circ \canonicalaction\right)}^{\operad } \circ
	\canonicalaction.
\end{align*}
The last equality follows from the fact that $\canonicalproj \circ {\sha
(\canonicalproj, f)}^\operad \circ a =f$ and the fact that the degree of
$f$ and $\induceddiff f $ is odd. Finally,
\begin{align*}
	\induceddiff f ^2
	&= {\sha \left(\canonicalproj, f
	\circ {\sha (\canonicalproj, f)}^\operad 
	\circ \canonicalaction\right)}^{\operad } \circ \canonicalaction 
	- {\sha \left(\canonicalproj, f \circ X^{d_\operad }\right)}^\operad 
	\circ\canonicalaction  +\canonicaldiff^2\\
	&= {\sha (\canonicalproj, f \circ \induceddiff f )}^\operad
	\circ \canonicalaction+\canonicaldiff^2.
\end{align*}
\end{proof}

\section{Curved  cooperads}%
%=======================================================================
\label{sec:coperades_courbees}

In order to manipulate the homotopy theory of dg-operads, one may
use the operadic Bar adjunction. There are several models for this; the
one we use was first introduced by Ginzburg \& Kapranov%
~\cite{doi:10.1215/s0012-7094-94-07608-4}, then extended to the
case of augmented dg-operads by Getzler \& Jones%
~\cite{arXiv:hep-th/9403055}, semi-augmented dg-operads by
Hirsh \& Millès%
~\cite{doi:10.1007/s00208-011-0766-9} and to all dg-operads in
\emph{Algebraic operads up to
homotopy}%
~\cite{arXiv:1707.03465}.

When a dg-operad $\operad$ is not augmented, the cooperad
$\barfunctorfull \operad$ is no longer a  dg-cooperad; it is
curved. In this section we shall recall the notion of curved 
cooperad and then study the properties of complete algebras over locally
conilpotent curved cooperads. We then also recall the
main points of the operadic Bar adjunction.

\subsection{Definitions and properties}
%-----------------------------------------------------------------------

\begin{definition}
A curved  cooperad is the data of a graded
 cooperad with coderivation
$(\overline\coperad, \overline w,d)$ equipped with a degree $-2$ arrow
$\theta : \overline\coperad(1) \rightarrow \monoidalunit$ such that
$ \theta \circ d = 0$
and $d$ satisfies the curvature equation
\[
	d^2 = \left( \idfunctor \otimes \theta
	-  \theta \otimes \idfunctor \right)
	\circ \overline w_2,
\]
where, given a two vertices tree, $\theta$ is first applied to the
top vertex and then $-\theta$ is applied to the vertex next to the root.
This may be rephrased as
\[
	d^2 = \left( \id \coperad \compofsymseq \sha(\tau, \theta)
	- \theta \compofsymseq \id \coperad\right) \circ w .
\]
\end{definition}

\begin{definition}
A (curved) $\coperad$-algebra is
the data of a graded $\coperad$\=/algebra with derivation
$(\Lambda,a, d_\Lambda)$
whose derivation satisfies the curvature equation
\[
	d_\Lambda^2 = a \circ \Lambda^{\theta}.
\]
The category of $\coperad$-algebras is the full subcategory of the
category of graded algebras with derivation over $\coperad$ whose
objects are the curved algebras over $\coperad$. We define likewise
the full subcategory of complete $\coperad$\=/algebras when $\coperad$
is locally conilpotent. The first category shall be denoted by
$\catofalg\coperad$ and the second by $\catofcompletealg\coperad$.
\end{definition}

\begin{remark}%
\label{rmk:gr_is_dg}
Curved  cooperads and their algebras are not far from the
dg-case when cooperads are pointed and locally conilpotent.
Indeed,
the graded object associated to the coradical filtration
$\grfilt \ast \coperad$ is then differential graded. Like-wise, if
$\Lambda$ is an algebra over $\coperad$, then the restriction of
$d_\Lambda$ to the graded object associated to its canonical
cofiltration $\gr \ast \Lambda$ also satisfies
$\power{(\gr \ast d_\Lambda)}2 = 0$%
~[\ref{thm:gr_is_dg}].
\end{remark}

\begin{proposition}%
\label{thm:algebre_courbee_libre}
Let $(X^\coperad, \canonicalaction)$
be a free graded $\coperad$-algebra and let $\induceddiff f $
be a derivation of $X^\coperad$ induced by a degree $-1$ map $f: X \to
X^\coperad$. Then, $(X^\coperad, \canonicalaction,
\induceddiff f )$ is a curved $\coperad$-algebra
if and only if
\[
	\induceddiff f  \circ f = X^{\theta}.
\]
\end{proposition}

\begin{proof}
Suppose first that $X^\coperad$ is curved, that is $\induceddiff f ^2
=\canonicalaction
\circ {(X^\coperad)}^\theta$. Then by restriction,
\[
	\induceddiff f ^2 \circ \canonicalinj =\canonicalaction
	\circ {\left(X^\coperad\right)}^\theta \circ \canonicalinj
	\Longrightarrow \induceddiff f  \circ f = X^\theta.
\]
Conversely, suppose that $\induceddiff f \circ f = X^\theta$. We know
that
\[
	\induceddiff f ^2 =\canonicalaction
	\circ {\sha(\canonicalinj, \induceddiff f \circ f )}^\coperad
	+ \canonicaldiff^2%
	~[\ref{thm:squarederiv}].
\]
Since
\[
	\canonicaldiff^2 = X^{d_\coperad} \circ X^{d_\coperad}
	= - X^{d_{\coperad}^2}
\]
and
\[
	d_\coperad^2 = \left( \id \coperad \compofsymseq \sha(\tau, \theta)
	-\theta \compofsymseq \id \coperad\right) \circ w ,
\]
one has
\[
	\canonicaldiff^2 =
	\canonicalaction \circ {\left(X^\coperad\right)}^\theta
	-\canonicalaction
	\circ {\sha\left(\canonicalinj, X^\theta\right)}^\coperad
\]
so that $\induceddiff f ^2=\canonicalaction
\circ \power{(X^\coperad)} \theta$.
\end{proof}

\begin{theorem}%
\label{thm:categorie_des_algebres_completes_courbees_est_presentable}
Let $\coperad$ be a locally conilpotent curved 
cooperad, then the category of complete curved $\coperad$-algebras is
presentable.
\end{theorem}

\begin{proof}
The category of complete graded $\coperad$-algebras with
derivation
is a reflective and accessible localisation of the presentable category
of graded $\coperad$-algebras with derivation~%
[\ref{thm:la_categorie_des_algebres_completes_est_presentable}]:
\[
	\begin{tikzcd}[ampersand replacement=\&]
		\catofcompletesmodalg\coperad
		\arrow[r, shift right, swap, ""]
		\&
		\catofsmodalg\coperad
		\arrow[l, shift right, swap, ""]
	\end{tikzcd}
\]
We shall now prove that the full subcategory of curved
$\coperad$\=/algebras is also reflective and accessible. For this
let $(\Lambda, a, d_\Lambda)$ be a graded $\coperad$\=/algebra with
derivation. Let $\ideali$ be the smallest ideal of $\Lambda$ that
contains the image of the map $d^2_\Lambda - a \circ \Lambda^\theta$.
Then by construction the quotient algebra $\Lambda/\ideali$ is a curved
$\coperad$-algebra. Moreover any morphism of graded $\coperad$-algebras
with derivation $f : \Lambda \longrightarrow \Gamma$ where $\Gamma$ is
curved must be trivial on the image
of $d^2_\Lambda - a \circ \Lambda^\theta$ and thus~%
[\ref{thm:generation_d_ideal}], factors uniquely to a map of
curved $\coperad$-algebras $\overline{f} : \Lambda/\ideali
\longrightarrow \Gamma$. Hence, the category of curved 
$\coperad$-algebras is an accessible localisation of the category
of graded $\coperad$-algebras with derivation
\[
	\begin{tikzcd}[ampersand replacement=\&]
		\catofalg\coperad
		\arrow[r, shift right, swap, ""]
		\&
		\catofsmodalg\coperad
		\arrow[l, shift right, swap, ""]
	\end{tikzcd}
\]
The category of complete curved $\coperad$-algebra becomes then the
intersection
of two reflective and accessible full subcategories of a presentable
category, hence it is also presentable.
\end{proof}

A similar attention to wording must be taken when speaking of
epimorphisms of $\coperad$-algebras as it has been for
monomorphisms of $\operad$-coalgebras%
~[\ref{def:degree-wise_monomorphism_of_coalgebras}].

\begin{definition}[(Degree-wise epimorphism)]%
\label{def:degree-wise_epimorphism_of_algebras}
We shall say that a morphism $f : \Lambda \to \Gamma$
of $\coperad$-algebras is a degree-wise epimorphism if the
corresponding morphism $\forget_\coperad f :
\forget_\coperad \Lambda \to \forget_\coperad \Gamma$
in $\catss^\integers$, is an epimorphism.
In particular, since $\forget_\coperad$ is faithful, it is also
an epimorphism of $\coperad$-algebras.
\end{definition}

\subsection{Functoriality}

Consider a morphism of curved conilpotent cooperads
\[
	\coperad \xrightarrow{f} \coperad' .
\]
It induces an adjunction relating
$\coperad$-algebras in $\sinv$-modules
to  $\coperad'$-algebras in $\sinv$-modules
\[
\begin{tikzcd}
 	\catofsmodalg{\coperad'}
	\arrow[rr, shift left, "\freealg{f,\sinv}"]
	&& \catofsmodalg\coperad
	\arrow[ll, shift left,"\forget^{f,\sinv}"]
\end{tikzcd}
\]
Indeed, for any $\coperad$-algebras in $\sinv$-modules
$(\Lambda,a)$, the $\sinv$-module
$\Lambda$
inherits a structure of a $\coperad'$-algebra
through the map
\[
	\Lambda^{\coperad'} \xrightarrow{\Lambda^f}\Lambda^\coperad
	\xrightarrow{a} \Lambda
\]
This gives the functor $\forget^{f,\sinv}$. Its left adjoint send a $\coperad'$-algebra
$\Lambda$ to the coequalizer in $\coperad$-algebras of the following reflexive diagram
 that exists since $\catofsmodalg\coperad$ is presentable
\[
\begin{tikzcd}
 	\left(\Lambda^{\coperad'}\right)^\coperad 
	\arrow[r, shift left]
	\arrow[r, shift right]
	& \Lambda^\coperad .
\end{tikzcd}
\]

\begin{remark}
 On can also notice that the adjunction $\freealg{f,\sinv} \dashv \forget^{f,\sinv}$
 is monadic.
\end{remark}

\begin{lemma}\label{lemma-uf-curved}
 If the $\coperad$-algebras in $\sinv$-modules
$(\Lambda,a)$ satisfies the equation
$d_\Lambda^2 = a \circ \Lambda^{\theta}$,
then, its image $(\Lambda,a')$ under the functor $\forget^{f,\sinv}$,
satisfies the equation
$d_\Lambda^2 = a' \circ \Lambda^{\theta'}$.
\end{lemma}

\begin{proof}
It follows from the fact that
\[
	\theta = \theta' \circ f.
\]
\end{proof}

\begin{proposition}\label{proposition-uf-complete}
Suppose that the $\coperad$-algebras in $\sinv$-modules
$(\Lambda,a)$ is complete.
Then, its image $(\Lambda,a')$ under the functor $\forget^{f,\sinv}$,
is complete.
\end{proposition}

\begin{proof}
Consider the following diagram
\[
\begin{tikzcd}
 	\Lambda^{\coperad'/ \filtration n \coperad'}
	\arrow[r]
	\arrow[d]
	& \Lambda^{\coperad'}
	\arrow[d]
	\\
	\Lambda^{\coperad/ \filtration n \coperad}
	\arrow[r]
	\arrow[d]
	& \Lambda^{\coperad}
	\arrow[d]
	\\
	\ideal n \Lambda 
	\arrow[r]
	& \Lambda
\end{tikzcd}
\]
For any natural integer $n$, the morphism
$\Lambda^{\coperad'/ \filtration n \coperad'}
\to \Lambda$
 factorises through
$\ideal{n}_\coperad \Lambda$. Subsequently
$\ideal{n}_{\coperad'} \Lambda$ is a sub-object of
$\ideal{n}_\coperad \Lambda$. Thus
$\ideal{\infty}_{\coperad'} \Lambda$ is a sub-object of
\[
	\ideal{\infty}_{\coperad} \Lambda = 0.
\]
So, $\ideal{\infty}_{\coperad'} \Lambda=0$.
\end{proof}

Let us denote respectively $s$ and $s'$ the inclusions
\begin{align*}
 \catofcompletealg{\coperad} &\xrightarrow{s}
	\catofsmodalg{\coperad};
\\
 \catofcompletealg{\coperad'} &\xrightarrow{s'}
	\catofsmodalg{\coperad'}	.
\end{align*}
We know that $s$ and $s'$ are accessible reflective fully faithful embeddings.
Let us denote respectively $r$ and $r'$ their left adjoints.
The previous \cref{proposition-uf-complete} and the previous \cref{lemma-uf-curved}
tells us that the composite functor
\[
	\catofcompletealg{\coperad} \xrightarrow{s}
	\catofsmodalg{\coperad} \xrightarrow{\forget^{f,\sinv}}
	\catofsmodalg{\coperad'}
\]
actually factorises as
\[
	\catofcompletealg{\coperad} \xrightarrow{\forget^f}
	\catofcompletealg{\coperad'} \xrightarrow{s'}
	\catofsmodalg{\coperad'}.
\]

\begin{proposition}
The functor
\[
	\freealg f \coloneqq r \circ \freealg{f,\sinv}\circ s'
\]
is left adjoint to $\forget^f$. 
\end{proposition}

\begin{proof}
This is a consequence of the fact that the functor $s'$ is fully faithful.
Indeed, for any complete curved $\coperad$-algebra $\Lambda$ and
any complete curved $\coperad'$-algebra $\Lambda'$, we have a sequence
of canonical natural isomorphisms
\begin{align*}
 \hombracket{\catofalg\coperad}
	{\freealg f \Lambda'} \Lambda
	&\isonat   \hombracket{\catofalg\coperad}
	{r\freealg{f,\sinv}s'( \Lambda')} \Lambda
	\\
	&\isonat   \hombracket{\catofsmodalg{\coperad}}
	{s'( \Lambda')}{\forget^{f,\sinv}s(\Lambda)}
	\\
	&\isonat  \hombracket{\catofsmodalg{\coperad}}
	{s'( \Lambda')}{s'\forget^{f}(\Lambda)}
	\\
	&\isonat  \hombracket{\catofalg{\coperad}}
	{\Lambda'}{\forget^{f}(\Lambda)} .
\end{align*}
\end{proof}

\begin{example}
Let $n$ be a natural integer and let us consider the canonical morphism
of curved conilpotent cooperad
$$
f : \filtration n \coperad \to \coperad .
$$
In this context, the functor $\forget^f$ is a faithful embedding
whose image consists in complete curved $\coperad$-algebras
$\Lambda$ so that the $n^{th}$ radical ideal is zero: $\ideal n \Lambda = 0$.

Moreover, $\freealg f$ sends a curved complete algebra $\Lambda$
to its quotient by the $n^{th}$ radical ideal
$$
\cofiltration n \Lambda  = \Lambda / \ideal n \Lambda .
$$
\end{example}

\subsection{Operadic Bar adjunction}
%-----------------------------------------------------------------------

In this subsection, we recall the main points of the operadic Bar
adjunction introduced in
\emph{Algebraic operads up to homotopy}~%
\cite{arXiv:1707.03465}
and which relates dg-operads with locally conilpotent curved
 cooperads:
\[
	\begin{tikzcd}[ampersand replacement=\&]
		\catofop_{\dg}(\catss)
		\arrow[rr, shift right=1.5, swap, "\barfunctorfull"]
		\&
		\&
		\catoflpccccoop\catss. 
		\arrow[ll, shift right=1.5, swap, "\leftadjoint\barfunctorfull"]
	\end{tikzcd}
\]
The importance of this adjunction
comes from the fact that for any dg-operad $\operad$,
the counit morphism
\[
	\leftadjoint\barfunctorfull \barfunctorfull \operad
	\longrightarrow \operad
\]
is a cofibrant replacement of $\operad$. For this reason we shall only
recall the construction of $\leftadjoint\barfunctorfull$.

The left adjoint to the Bar construction is given as a graded
operad by
\[
	\leftadjoint\barfunctorfull(\coperad)
	\coloneqq \treemodule\left(\sinv \overline{\coperad}\right),
\]
which is then equipped with the derivation whose
restriction to the generators $\sinv\overline\coperad$ is the sum
of the three following compositions
\begin{align*}
	& \smalldiff_w : \sinv\overline{\coperad}
	\xrightarrow{\sinv\degreeminusonemap}
	\s^{-2}\overline{\coperad}
	\xrightarrow{\s^{-2}\overline w_2} \s^{-2}
	\treemodule_2\overline{\coperad}
	\isonat \treemodule_2\left(\sinv\overline{\coperad}\right);
	\\
	& \smalldiff_{\sinv\overline\coperad} : \sinv\overline
	\coperad \xrightarrow{d_{\sinv\overline\coperad}} \sinv
	\overline\coperad;
	\\
	& \smalldiff_\theta: \sinv\overline{\coperad}
	\xrightarrow{\degreeminusonemap \otimes \overline \coperad}
	\overline{\coperad} \xrightarrow{-\theta}
	\monoidalunit .
\end{align*}

\begin{definition}
A twisting morphism from a pointed curved 
cooperad $(\coperad, w, d_\coperad, \theta)$ to a
differential graded operad $(\operad , m, \eta, d_\operad )$
is a degree $-1$ map
\[
	\alpha : \overline\coperad \longrightarrow \operad 
\]
such that
\[
	\eta \circ \theta + \partial \alpha
	+ m \circ (\alpha \otimes \alpha ) \circ \overline w_2 = 0.
\]
where $\partial\alpha \coloneqq d_\operad  \circ \alpha + \alpha
\circ d_\coperad$.

We shall also denote by $\alpha$ the twisting morphism extended to
$\coperad$ by zero.
Let us write $\twisting\coperad\operad$ for the set of
twisting morphisms between
them.
\end{definition}

\begin{proposition}[{\cite[Proposition 21]{arXiv:1612.02254}}]
There exists functorial isomorphisms:
\[
	\hombracket{\dg}
	{\leftadjoint\barfunctorfull(\coperad)}\operad
	\isonat \twisting\coperad\operad
	\isonat \fullhombracket {\bullet} \lcc
	\coperad {\barfunctorfull(\operad )}.
\]
\end{proposition}

\begin{notation}
For the sake of brevity, we will often simply denote
the bar functor by $\barfunctor$ and its left adjoint by
$\leftadjoint\barfunctor$. 
\end{notation}

\subsection{Transferred model structure on conilpotent curved cooperads and ladders}

\begin{theorem}\cite[Theorem 5]{arXiv:1707.03465}
 There exists a combinatorial model structure on
 the category of conilpotent curved cooperads that is transferred along the adjunction
 $\baradj\dashv \barfunctor$, that is
\begin{itemize}
 \item the cofibrations are the morphisms $f$ so that $\baradj(f)$ is a cofibration
 of dg-operads; they are actually inclusions that is aritywise and degreewise monomorphisms;
 \item the weak equivalences are the morphisms $f$ so that $\baradj(f)$ is a weak equivalence
 of dg-operads.
\end{itemize}
Moreover, the adjunction  $\baradj\dashv \barfunctor$ becomes a model equivalence.
\end{theorem}

\begin{definition}
 A morphism of curved conilpotent cooperads $f:\coperad \to \coperad'$
 is an infinitesimal inclusion if it is an inclusion (that is an
 aritywise degreewise monomorphism) so that the map
 $$
 w: \coperad' \to \treemodule (\overline{\coperad'})  \twoheadrightarrow
 \treemodule_{\geq 2} (\overline{\coperad'})
 $$
 factorises through the subobject $\treemodule_{\geq 2} (\overline{\coperad})$. In particular,
the coderivation
squares to zero on the quotient $\coperad'/\coperad$.
\end{definition}

\begin{lemma}\label{lemma:ladeq}
 Let us consider a square diagram of curved conilpotent cooperads
 $$
\begin{tikzcd}
 \coperad
 \ar[r, "f"]  \ar[d, "i"']
 & R
  \ar[d, "j"]
  \\
  \coperad'
  \ar[r, "g"']
  & R'
\end{tikzcd}
 $$
 and let us suppose that
\begin{itemize}
\item the vertical maps $i,j$ are infinitesimal inclusions;
 \item the map $f : \coperad \to R$ is a weak equivalence;
 \item the map $\coperad'/\coperad \to R'/R$ induced by $g$
 is aritywise a quasi-isomorphism.
\end{itemize}
Then $g$ is a weak equivalence.
\end{lemma}

\begin{proof}
The operad $\baradj \coperad'$ is made up of sums trees labeled
by $\coperad'$. One can filter this operad as follows: $\filtration{n}^{(i)}\baradj \coperad'$
is made up of labeled trees with $m \in \naturals$ nodes so that the labels of at least
$m-n$ nodes are in the suboject $\coperad$ of $\coperad'$. One can 
filter $\baradj R'$ similarly. Then, the map $\baradj g$ preserves these filtrations.

In that context, the square diagram
$$
\begin{tikzcd}
 \baradj \coperad
 \ar[r, "f"]  \ar[d, "i"']
 & \baradj R
  \ar[d, "j"]
  \\
  \baradj(\coperad \oplus \coperad'/\coperad)
  \ar[r] \ar[d, equal]
  &
    \baradj(R \oplus R'/R)
    \ar[d, equal]
  \\
  \grfilt{}  \baradj\coperad'
  \ar[r, "\grfilt{}\baradj g"']
  & \grfilt{}  \baradj R'
\end{tikzcd}
$$
is a homotopy pushout of operads.
Thus, the map $\grfilt{}\baradj g$ is a quasi
isomorphism. Subsequently, the map
$\baradj g$ is also a quasi-isomorphism.
\end{proof}

\begin{definition}
A ladder of curved conilpotent cooperads is the data of a functor
$$
G : \ordinalomega \to \catoflpccccoop\catss
$$
so that $G(0)= \monoidalunit_{\compofsymseq}$ and any map
$$
G(n) \to G(n+1), \quad n \geq 9
$$
is an infinitesimal inclusion. Then, in particular, the quotient $G(n+1)/G(n)$
is a dg-symmetric sequence. Finally, we denote
$$
G(\infty) = \varinjlim_{n \in \ordinalomega} G(n) .
$$
\end{definition}

\begin{definition}
 A ladder equivalence is a morphism of ladders $f: G \to G'$ so that
the map
$$
G(n+1)/G(n) \to G'(n+1)/G'(n)
$$
is aritywise a quasi-isomorphism for any natural integer $n$.
More generally, we say that a morphism of conilpotent curved cooperads
is a ladder equivalence if it appears as the colimit of a ladder equivalence between
ladders.
\end{definition}

\begin{proposition}\label{thm:filtered-qis-equivalences}
Weak equivalences of conilpotent curved cooperads
is the smallest subset of morphisms that contains ladder equivalences and that
follows the 2-out-of-3 rule.
\end{proposition}

\begin{proof}
The fact that it contains ladder equivalences is a direct consequence of
\cref{lemma:ladeq}.
Conversely, let $f: \coperad \to \coperad'$ be a weak equivalence.
Let us contemplate
 the following diagram
 $$
\begin{tikzcd}
 \barfunctor \baradj \coperad
 \ar[r]
 & \barfunctor \baradj \coperad'
 \\
 \coperad
  \ar[r]  \ar[u]
 & \coperad' .
  \ar[u]
\end{tikzcd}
 $$
 The vertical arrows are compositions of ladder equivalences by
 \cite[Proposition 15 and Proposition 17]{arXiv:1707.03465}.
 Moreover, the top horizontal arrow is a ladder equivalence.
 Hence, $f$ belongs to this subset that contains ladder equivalences
 and that follows the 2-out-of-3 rule.
\end{proof}

\section{Model structure on coalgebras over an operad}
%=======================================================================

In this section we discuss the necessary conditions so that
the category of coalgebras over a dg-operad $\operad$ can be endowed with
a model structure transferred along the forgetful functor from the
standard model structure on the category of chain complexes.
When the transferred model structure on the category of
$\operad$-algebras exists, the dg-operad $\operad$ is called
\emph{admissible}; when the transferred model structure on the
category of coalgebras exists, we shall call it \emph{coadmissible}.

The first general admissibility result is due to Hinich%
~\cite{doi:10.1080/00927879708826055},
he proves in particular that
all planar dg-operads and all dg-operads in characteristic zero are
admissible. A more conceptual approach has been then given by
Berger and Moerdijk%
~\cite{doi:10.1007/s00014-003-0772-y}.
Their result increases the
family of admissible dg-operads by adding all cofibrant operads.

An early coadmissibility result is due to
Getzler and Goerss%
~\cite{goerss1999model},
who prove that $\uass$ is coadmissible. More recently, Smith
has given a general coadmissibility result%
~\cite{zbMATH:05934769}.
However, his result requires finiteness conditions on the dg-operad%
~\cite[Condition 4.3]{zbMATH:05934769} that are unnecessary.
Using a similar statement to the one of Berger and Moerdijk, we shall
prove that all planar and all cofibrant dg-operads are coadmissible%
~[\ref{thm:les_operades_planaires_ou_cofibrantes_sont_coadmissibles}].
We then show that Hinich's result cannot be proved for coalgebras as
$\ucom$ is \emph{not} coadmissible in the category $\chain \fieldk$ for
$\fieldk$ an algebraically closed field of characteristic zero%
~[\ref{thm:com_n_est_pas_coadmissible}].

Contrarily to what happens for algebras over dg-operads, where
any weak equivalence of dg-operads $f: \operad \to \operad'$ yields a
model equivalence%
~\cite[4.7.4]{doi:10.1080/00927879708826055}
\[
	\begin{tikzcd}[ampersand replacement=\&]
		\catofalg\operad \arrow[r, shift left,"{f_!}"]
		\&
		\catofalg{\operad'} \arrow[l, shift left, "f^!"]
	\end{tikzcd}
\]
we shall show that one cannot hope for such a general statement in the
case of coalgebras%
~[\sectionref{sec:counter_example_weq_operads}]. One has to rely instead
on a minimal version of the theorem.

Yalin has shown that a weak equivalence $f : \prop \to \prop'$ between
\emph{cofibrant} props induce an equivalence
${\catofgebras {\prop'}}^\infty \to {\catofgebras \prop}^\infty$
between the ∞\=/categories associated to their respective categories of
gebras%
~\cite{doi:10.1017/s0305004114000437}.
In the case of coalgebras over operads, Yalin's
result can be lifted at the level of model categories.

%bigtheorem
\begin{theorem}%
\label{bigthm:weq_of_operads_eq_of_models}
Suppose that $f: \operad \to \operad'$
is a weak equivalence between cofibrant dg-operads. Then the
adjunction
\[
	\begin{tikzcd}[ampersand replacement=\&]
		\catofcog\operad
		\arrow[r, shift right, swap, "f_\ast"]
		\&
		\catofcog{\operad'}
		\arrow[l, shift right, swap, "f^\ast"]
	\end{tikzcd}
\]
is a model equivalence.
\end{theorem}

The fact that homotopy equivalent dg-operads have model equivalent
categories of algebras is a good indicator that the associated
∞\=/categories are the ∞\=/categories of algebras over the ∞\=/operads
associated to the dg-operads; in short: that the model structure
on $\operad$-algebras actually presents the ∞-category that it is
suppposed to. A partial result in that direction has been obtained by
Hinich in the
case where the dg-operad stems from a topological operad, following
a strategy first implemented by Lurie for the operads
$\uass$ and $\ucom$.

\begin{theorem}[{\cite[4.1.1]{zbMATH:06572171}}]
Let $\operad$ be a fibrant simplicial operad and let $\operad^\infty$
be the associated ∞-operad. Let $\fieldk$ be a field of characteristic
zero. Then there exists an equivalence of ∞-categories
\[
	{\catofalg\operad\left(\chain\fieldk\right)}^\infty \longrightarrow
	\catofalg{\operad^\infty}\left(\chain\fieldk^\infty\right)
\]
where we denoted for simplicity $\catc^\infty$ the ∞-category associated
with a category $\catc$ with a set of weak equivalences.
\end{theorem}

\Cref{bigthm:weq_of_operads_eq_of_models} naturally leads us to
formulate an equivalent conjecture in the case of coalgebras.

\begin{conjecture}%
\label{conj:infinity_cat_of_cogs}
If $\operad$ is a cofibrant dg-operad and if
$\operad^\infty$ is its associated ∞-operad, then there exists
an equivalence of ∞-categories
\[
	{\catofcog\operad\left(\chain\catss\right)}^\infty \longrightarrow
	\catofcog{\operad^\infty}\left({\chain\catss}^\infty\right)
\]
\end{conjecture}

\subsection{Coadmissible operads}
%-----------------------------------------------------------------------

Given an operad $\operad $ in $\chain {\catss}$, we may wish to
transfer the standard model structure on $\chain {\catss}$ to
$\catofcog\operad$ using the adjunction
\[
	\begin{tikzcd}[ampersand replacement=\&]
		\chain {\catss}
		\arrow[r, shift right,"\eta_{\ast}"']
		\&
		\catofcog\operad .
		\arrow[l, shift right, "\eta^\ast"']
	\end{tikzcd}
\]
That is, thanks to the adjunction we can define three sets of arrows
$(\weakeqofcog,\fibrationofcog,\cofibrationofcog)$ where
\begin{itemize}
	\item an arrow $f$ is a weak equivalence if
	      $\eta^\ast f$ is a quasi-isomorphism of chain complexes in
	      $\chain {\catss}$:
	      \[
	      	f \in \weakeqofcog \Longleftrightarrow \eta^\ast f
	      	\text{ is a quasi-isomorphism;}
	      \]
	\item an arrow $f$ is a cofibration if
	      $\eta^\ast f$ is a degree-wise monomorphism of chain complexes:
	      \[
	      	f \in \cofibrationofcog \Longleftrightarrow \eta^\ast f
	      	\text{ is a degree-wise monomorphism;}
	      \]
	\item an arrow $f$ is a fibration if it has the
	      right lifting property against all elements of $\weakeqofcog
	      \cap \cofibrationofcog$. That is
	      \[
	      	\fibrationofcog \coloneqq
	      	{(\cofibrationofcog \cap \weakeqofcog)}^\boxslash .
	      \]
\end{itemize}

\begin{definition}%
\label{def:operade_coadmissible}
We shall call a dg-operad $\operad $ coadmissible if the category
of $\operad $-coalgebras endowed with the three sets of arrows
$(\weakeqofcog$, $\fibrationofcog$, $\cofibrationofcog)$ is a model
category.
\end{definition}

In the article
\emph{Left-induced model structures and diagram categories}%
~\cite[Theorem~2.23]{doi:10.1090/conm/641/12859}, the authors
apply a result of Makkai and Rosicky%
~\cite[Remark~3.8]{doi:10.1016/j.jpaa.2014.01.005} to show that
the transferred model structure exists and is combinatorial when a
certain acyclicity condition is satisfied. Using standard techniques
for proving the acyclicity condition%
~\cite[Theorem~2.2.1]{doi:10.1112/topo.12011}, we get the following
proposition.

\begin{proposition}~%
\label{thm:coadmissible}
A dg-operad $\operad $ is coadmissible if and only if the
category $\catofcog\operad$ of $\operad $-coalgebras has functorial
cylinder objects. Moreover, since the model structure on $\chain
{\catss}$ is combinatorial, so is the transferred one.
\end{proposition}

\begin{proposition}[(Strong coadmissibility)]%
\label{thm:les_fibrations_entre_cogebres_sont_des_epis}
The fibrations of $\operad$-coalgebras are in particular degree-wise
epimorphisms.
That is, the forgetful functor $\catofcog\operad
\longrightarrow \chain\catss$
preserves fibrations.
\end{proposition}

\begin{proof}
Let $f : V \to W$ be a fibration of $\operad $-coalgebras. Since
$\gradedd$ is a cocommutative coalgebra%
~[\ref{rmk:graded_versus_dg}],
the external tensor product $\gradedd \externaltensor W$ is a coalgebra
over $\ucom \diagtensor \operad \isonat \operad $ and the induced
map $\counitdzero \externaltensor W : \gradedd \externaltensor W
\longrightarrow W$ is a $\operad$-coalgebra morphism; it is also
a degree-wise epimorphism. In the meantime since $\gradedd$ is
nullhomotopic, $\gradedd \externaltensor W$ is also nullhomotopic. This
allows us to lift $\counitdzero \externaltensor W$ against $f$:
\[
	\begin{tikzcd}[ampersand replacement=\&]
		0
		\arrow[rr, ""]
		\arrow[d, "", hookrightarrow, swap]
		\arrow[d, shift left=1.7, phantom, "\rotatebox{90}{\(\sim\)}"]
		\& \&
		V
		\arrow[d, "f"]
		\\
		\gradedd \externaltensor W
		\arrow[rr, "\counitdzero \externaltensor
		 W", swap] \arrow[urr, dashed]
		\& \&
		W,
	\end{tikzcd}
\]
proving that $f$ is a degree-wise epimorphism.
\end{proof}

\begin{remark}%
\label{rmk:L_preserve_les_epi}
As a straightforward consequence, for any $\operad$ the
functor $\freecog \operad$ preserves degree-wise epimorphisms.
\end{remark}

\subsection{Sufficient condition}
%-----------------------------------------------------------------------

We shall now focus on the sufficient conditions for the existence of
a functorial cylinder object.
In the category of chain
complexes, the canonical way to build cylinders is obtained by
tensorisation with the standard interval $\interval$.
It is the chain complex
\[
	\begin{tikzcd}
		\cdots \ar[r] & 0 \ar[r] & \monoidalunit_{\{0,1\}} \ar[r,"\dd"]
		& \monoidalunit_{\{0\}} \oplus \monoidalunit_{\{1\}} \ar[r]
		& 0 \ar[r]
		& 0 \ar[r]
		& \cdots
	\end{tikzcd}
\]
with $\dd \coloneqq (-\id\monoidalunit, \id\monoidalunit)$.

Since the model structure on coalgebras is transferred from the one on
chain complexes, it is natural to wish for a transfer of the cylinder
functor. However, given a $\operad$-coalgebra $V$, the
chain complex $\interval \otimes V$ does not usually bear a
$\operad$-coalgebra structure. Hence we need to require more data;
this will lead us to a coadmissibility theorem, akin to the
(incomplete) one of Berger and Moerdijk%
~\cite[Proposition~4.1]{doi:10.1007/s00014-003-0772-y}
and its avatars%
~\cite[Theorem~6.3.1]{doi:10.1112/topo.12011}.

The interval $\interval$ has a natural coalgebra structure over the
operad $\eoperad$%
~\cite[Theorem 2.1.1]{doi:10.1017/S0305004103007138}, where
$\eoperad$ is a linear analogue introduced by Berger and Fresse%
~\cite{doi:10.1017/S0305004103007138}
of the topological operad $\beoperad$ of
Barratt and Eccles%
~\cite{doi:10.1016/0040-9383(74)90036-6}.
One of its remarkable properties is to factorise the canonical morphism
$\uass \to \ucom$:
\[
	\begin{tikzcd}
		\uass \ar[r, hook]& \eoperad \ar[r, two heads]
		\ar[r, phantom, shift left, "\sim"] & \ucom
	\end{tikzcd}
\]
in such way that the $\uass$-interval structure of $\interval$
is in fact induced by its $\eoperad$-interval structure.

\begin{definition}
Let us denote by $\sigma$ the map $\eoperad \to \ucom$.
Then a dg\=/operad $\operad$ shall be called $\eoperad$-split if the map
$\groundofe \diagtensor \id\operad : \eoperad \diagtensor \operad \to
\operad$ admits a section in the category of dg-operads.
\end{definition}

\begin{theorem}%
\label{thm:les_operades_planaires_ou_cofibrantes_sont_coadmissibles}
Any $\eoperad$-split dg-operad is coadmissible. In particular
all planar operads, all cofibrant operads
and the dg-operad $\eoperad$ itself are coadmissible.
\end{theorem}

\begin{proof}
It is enough to show the
existence of functorial cylinders%
~[\ref{thm:coadmissible}]. Let $V$ be any $\operad $-coalgebra.
Since $\catss$ is semi-simple, the tensor product of
$\chain {\catss}$ preserves both cofibrations and weak
equivalences so that we get
\[
	\begin{tikzcd}
		(\groundofe^\ast\monoidalunit \externaltensor V)
		\oplus (\groundofe^\ast\monoidalunit \externaltensor V)
		\arrow[r, hookrightarrow]
		& \interval \externaltensor V
		\arrow[r]
		\arrow[r,phantom,"\sim", shift left]
		& \groundofe^\ast\monoidalunit \externaltensor V.
	\end{tikzcd}
\]
Moreover these maps are maps of $(\eoperad \diagtensor
\operad )$-coalgebras by construction. Let $s : \operad \to
\eoperad\diagtensor\operad$ be a section of $\groundofe\diagtensor
\id\operad$, then we get morphisms of $\operad$-coalgebras
\[
	\begin{tikzcd}
		s^\ast(\groundofe^\ast\monoidalunit \externaltensor V)
		\oplus s^\ast(\groundofe^\ast\monoidalunit \externaltensor V)
		\arrow[r, hookrightarrow]
		& s^\ast(\interval \externaltensor V)
		\arrow[r]
		\arrow[r,phantom,"\sim", shift left]
		& s^\ast(\groundofe^\ast\monoidalunit \externaltensor V).
	\end{tikzcd}
\]
By asumption $(\groundofe\diagtensor\id\operad)\circ s = \id\operad$
hence
$s^\ast(\sigma^\ast\monoidalunit\externaltensor V) = V$ and the functor
\[
	s^\ast(\interval\externaltensor -) : \catofcog\operad \longrightarrow
	\catofcog\operad
\]
creates cylinder objects of $\operad$-coalgebras.

Let us quickly recall why all the dg-operads cited above are
$\eoperad$-split:
\begin{itemize}
	\item The operad $\eoperad$ is $\eoperad$-split because it is a Hopf
	      operad;
	\item Planar operads are $\uass$-split, hence $\eoperad$-split;
	\item The map $\groundofe \diagtensor \id\operad
	      : \eoperad \diagtensor
	      \operad \to \operad$ is a trivial fibration of dg-operads. If
	      $\operad$ is cofibrant, it admits a section.
\end{itemize}
\end{proof}

\subsection{A counter-example}%
%-----------------------------------------------------------------------
\label{sec:no_cocom_interval}%
\label{sec:counter_example_weq_operads}

Notice that the above proposition%
~[\ref{thm:les_operades_planaires_ou_cofibrantes_sont_coadmissibles}]
eludes the case of symmetric operads in
characteristic zero, which are known to be admissible in the algebraic
case%
~\cite[4.2]{doi:10.1080/00927879708826055}.
There is actually no corresponding theorem of coadmissibility for a good
reason.

\begin{proposition}%
\label{thm:com_n_est_pas_coadmissible}
The operad $\ucom$ is \emph{not} coadmissible in the category of
chain complexes over an algebraically closed field $\fieldk$
of characteristic zero.
\end{proposition}

\begin{proof}
In order to prove this claim, we shall show that there cannot exists a
cocommutative interval object in $\chain\fieldk$.
Since the unit $\fieldk$ is a cocommutative coalgebra in
$\chain\fieldk$, if $\ucom$ were
coadmissible, there would exists a cocommutative coalgebra
$I$ such that the sum map $\fieldk \oplus \fieldk \to
\fieldk$ factors as
\[
	\begin{tikzcd}
		\fieldk \oplus \fieldk
		\arrow[r, hookrightarrow]
		& I
		\arrow[r]
		\arrow[r,phantom,"\sim", shift left]
		& \fieldk .
	\end{tikzcd}
\]
This is impossible. By a theorem of structure of cocommutative
coalgebras~%
\cite[Theorem 8]{arXiv:1612.02254}, every cocommutative coalgebra
$(I, \Delta)$ can be decomposed as a sum
\[
	I \isonat \bigoplus_{\alpha \in A} I_\alpha
\]
where $A$ is the set of atoms of $I$, each $I_\alpha$ is an irreducible
subcoalgebra. Each of these irreducible subcoalgebras can be split into
$I_\alpha
\iso \fieldk.\alpha \oplus \overline{I}_\alpha$ where
$\Delta(\alpha) =
\alpha \otimes \alpha$ and $\overline{I}_\alpha$ is non-counital and
locally conilpotent.

Using the injection $\fieldk \oplus \fieldk \subobject I$
we deduce that there
exists $\alpha, \beta \in A$ non-zero with $\alpha \neq \beta$ and
$d\alpha = d\beta = 0$. But since $I$ is quasi-isomorphic to
$\fieldk$, we either have $\homology_0(I_\alpha) = 0$ or
$\homology_0(I_\beta) = 0$. In either case, this means that the map
$\fieldk \oplus \fieldk \to \fieldk$ is zero on one of the
factors which contradicts the definition of $I$.
\end{proof}

The fact that $\ucom$ is not coadmissible allows us to create a
conter-example for the functoriality of the model categories of
coalgebras:
given a quasi-isomorphism of dg-operads $f : \operad \to \operad'$, it
is straightforward to prove that the induced morphism $f \compofsymseq
(-) : \operad \compofsymseq (-) \to \operad' \compofsymseq (-)$ is
a point-wise weak equivalence, which by ripple effect gives a model
equivalence between the respective model categories of algebras when
both dg-operads are admissible.

The same phenomenon is false in general for coalgebras.
Indeed, the
obvious weak equivalence $X^f : X^{\operad'} \to X^\operad$ might not
restrict to a weak equivalence $\freecog f : \freecog {\operad'} X \to
\freecog \operad X$ for every $X$. As $\groundofe :
\eoperad \to \ucom$ is a trivial fibration,
one could legitimately hope that for any fibrant
$\eoperad$-coalgebra $V$, the counit morphism
\[
	\groundofe^\ast \groundofe_\ast V \longrightarrow V
\]
be a quasi-isomorphism. This cannot happen. Indeed, we know that the
dg-operad $\eoperad$ is coadmissible%
~[\ref{thm:les_operades_planaires_ou_cofibrantes_sont_coadmissibles}],
hence one can find a fibrant interval $I$
\[
	\begin{tikzcd}
		\groundofe^\ast \monoidalunit \oplus
		\groundofe^\ast \monoidalunit
		\ar[r, hook]
		& I
		\ar[r]
		\ar[r, phantom, shift left, "\sim"]
		&
		\groundofe^\ast \monoidalunit.
	\end{tikzcd}
\]
As $\groundofe$ is a fibration of dg-operads,
$\groundofe^\ast \groundofe_\ast I$ is the biggest cocommuative
subcoalgebra of $I$. Then, one would have the following factorisation
\[
	\begin{tikzcd}
		\groundofe^\ast \monoidalunit \oplus
		\groundofe^\ast \monoidalunit
		\ar[r, hook]
		&
		\groundofe^\ast \groundofe_\ast I
		\ar[r, hook]
		\ar[r, phantom, shift left, "\sim"]
		& I
		\ar[r]
		\ar[r, phantom, shift left, "\sim"]
		&
		\groundofe^\ast \monoidalunit.
	\end{tikzcd}
\]
Again because $\groundofe$ is a fibration, $\groundofe^\ast$
is fully faithful and creates cofibrations
and weak equivalences so that would mean that $\groundofe_\ast I$ is
a cocommutative interval, which we know does not exist%
~[\sectionref{sec:no_cocom_interval}].

We actually think that this counter example is not an exception.

\begin{conjecture}
The following trivial fibrations $f$ of dg-operads do not
induce weak equivalences $\comonadl^f$ of comonads:
\begin{itemize}
	\item $\eoperad \longrightarrow \ucom$;
	\item $\uass_\infty \longrightarrow \uass$.
\end{itemize}
In particular, we conjecture that the ∞-categories associated to
counital
coassociative dg-coalgebras and to $\uass_\infty$-coalgebras are not
equivalent, even though $\uass_\infty \to \uass$ is a trivial fibration
between two coadmissible dg\=/operads.
\end{conjecture}

\begin{remark}
A result of Soré%
~\cite{doi:10.1007/s40062-018-0210-x}
shows that the model adjunction
\[
	\begin{tikzcd}[ampersand replacement=\&]
		\catofcog\uass\left(\catss^{\simplexcat\op}\right)
		\arrow[r, shift right,"\nervefunctor"']
		\&
		\catofcog\uass\left(\positivechain\catss\right)
		\arrow[l, shift right, "\gammafunctor"']
	\end{tikzcd}
\]
between the categories of
simplicial coalgebras and dg-coalgebras endowed with their natural
model structures, is not a model equivalence.
This corroborates our conjecture
that the ∞-category of $\uass$-coalgebras is not equivalent to the
one of $\uass_\infty$-coalgebras.
\end{remark}

\subsection{Model equivalence for cofibrant dg-operads}
%-----------------------------------------------------------------------

In this subsection, we prove \cref{bigthm:weq_of_operads_eq_of_models}.
The idea underlying this proof is that the set of morphisms between cofibrant
operads that induces a model equivalence between the related categories of coalgebras
contains cellular trivial cofibrations satisfies the 2-out-of-3 rule and is stable through retracts.
Then, it contains all weak equivalences
between cofibrant operads.

\begin{lemma}\label{thm:comorita-equivalences-2-out-of-3}
 The class of morphisms $f: \operad \to \operad'$ between cofibrant operads
 so that the adjunction $f^\ast \dashv f_\ast$ is a model equivalence 
 satisfies the 2-out-of-3 property adn is stable through retracts.
\end{lemma}

\begin{proof}
This is a direct consequence of the fact
that equivalence of categories satisfy these properties.
\end{proof}

\begin{lemma}%
\label{thm:free_quasi_iso_induces_weak_equivalence_of_comonads}
Let $\operad$ be an $\eoperad$-split dg-operad and
let $K$ be a contractible symmetric sequence.
Then the free morphism of free dg-operads
$\monoidalunit \to \treemodule(K)$ induces a model equivalence
equivalence
\[
	\begin{tikzcd}[ampersand replacement=\&]
		\catofcog\operad
		\arrow[r, shift right=1.5]
		\&
		\catofcog{(\operad \amalg \treemodule K)}.
		\arrow[l, shift right=1.5]
	\end{tikzcd}
\]
\end{lemma}

\begin{proof}
By $\rationals$-linearity the map $\monoidalunit \to \treemodule(K)$ is
a retract of $\monoidalunit \to \treemodule(\pltosym{\forgetsym K})
\isonat \pltosym{(\planartreemodule \forgetsym K)}$%
~[\ref{thm:planar_trees_vs_trees}].
Hence we can restrict our attention to the planar case.

Let $h$ be a contracting homotopy of $K$. For simplicity let us write
$\bigretract : \operad\amalg\treemodule(K) \to \operad$ for the
map coming from the canonical augmentation of $\treemodule(K)$ and let
$\bigsection :\operad \to \operad \amalg \treemodule(K)$ be its
canonical section. By functoriality we deduce that
\[
	\bigretract^\ast: \catofcog\operad \longrightarrow
	\catofcog{\operad\amalg\treemodule(K)}
\]
is faithful at the homotopy category level. So the
only remaining thing to show is that this functor is homotopy
full and homotopy surjective.

Let $V$ be a $(\operad \amalg \treemodule(K))$-coalgebra, that is a chain
complex with both the structure of a $\operad$-coalgebra and the structure
of a $\treemodule(K)$-coalgebra. Since by assumption $\bigretract^\ast$
and $\bigsection^\ast$ are the identity on the $\operad$-coalgebra
structure part, we shall only focus on the $\treemodule(K)$-coalgebra
structure part.

Since $\treemodule(K)$ is a free
dg-operad, the coalgebra structure on $V$ amounts to the data of maps
\[
	\delta^K_n : K(n) \otimes V \longrightarrow V^{\otimes n}
\]
with no relations between them. By construction $\bigretract^\ast
\bigsection^\ast \delta^M_n$ is the zero map and $\bigretract^\ast
\bigsection^\ast V$ has a trivial $\treemodule(K)$\=/structure.

We now need to show that those two coalgebras are isomorphic in the
homotopy category of $\treemodule(K)$\=/coalgebras. For this we use the
homotopy $h$ the zero map and $\id K$. It gives us a map
\[
	\delta_n^K \circ (h \otimes \id V) :
	K(n) \otimes \interval \otimes V \longrightarrow V^{\otimes n}
\]
which we extend into a $\treemodule(K)$-coalgebra structure on
$\interval \otimes V$ using the standard $\uass$-coalgebra structure
of the interval:
\[
	K(n) \otimes \interval \otimes V \to
	V^{\otimes n} \to V^{\otimes n} \oplus V^{\otimes n} \to
	\interval \otimes V^{\otimes n} \to
	\fullpower{\interval \otimes V}{\otimes n}.
\]
Since $\operad$ is $\eoperad$-split, one can extend the
$\operad$-coalgebra structure of $V$ to a $\operad$-coalgebra structure on
$\interval \otimes V$ such that the canonical maps
\[
	\begin{tikzcd}
		V \oplus V\ar[r, hook]
		& \interval \otimes V
		\ar[r] \ar[r, phantom, shift left, "\sim"]
		& V
	\end{tikzcd}
\]
are all $\operad$-coalgebra maps.

Let us denote by $\under{(\interval \otimes V)}h$ the resulting
$(\operad \amalg \treemodule(K))$-coalgebra.
By construction, one gets two weak equivalences of coalgebras
\[
	\begin{tikzcd}
		V \ar[r, hook] \ar[r, shift left, phantom, "\sim"]
		& \fullunder{\interval \otimes V}h
		& \bigretract^\ast \bigsection^\ast V. \ar[l, hook']
		\ar[l, phantom, shift right, "\sim"]
	\end{tikzcd}
\]
This shows that $\bigretract^\ast$ is homotopically
essentially surjective.

Now let $f : V \to W$ be a morphism of $\operad \amalg
\treemodule(K)$-coalgebras. Then we obtain a new morphism
$\bigretract^\ast \bigsection^\ast f : \bigretract^\ast
\bigsection^\ast V \to \bigretract^\ast \bigsection^\ast W$ that is
equivalent to $f$ in the derived category of $\catss$. Using the same
type of construction from the homotopy $h$, one can build a commutative
diagram of $(\operad\amalg\treemodule(K))$\=/coalgebras
\[
	\begin{tikzcd}
		V \ar[d, hook] \ar[rr, "f"]
		\ar[d, phantom, "\rotatebox{90}{$\sim$}", shift left=1.5] && W
		\ar[d, phantom, "\rotatebox{90}{$\sim$}", shift right=0.75]
		\ar[d, hook']
		\\
		\fullunder{\interval \otimes V}h \ar[rr, "\interval \otimes f"]
		&& \fullunder{\interval \otimes W}h
		\\
		R^\ast S^\ast V
		\ar[u, hook'] \ar[rr, "R^\ast S^\ast f"']
		\ar[u, phantom, "\rotatebox{90}{$\sim$}", shift right=1.5] &&
		R^\ast S^\ast W
		\ar[u, phantom, "\rotatebox{90}{$\sim$}", shift left=0.75]
		\ar[u, hook]
	\end{tikzcd}
\]
showing that $R^\ast$ is homotopically full.
\end{proof}

Now, we can prove \cref{bigthm:weq_of_operads_eq_of_models}. Let us consider a
weak equivalence of cofibrant operads $f : \operad \to \operad'$. The map
$$
f \sqcup \id{} : \operad \sqcup \operad' \to \operad'
$$
may be factorised as a cofibration followed by a weak equivalence $p$.
$$
\operad \sqcup \operad' \xrightarrow{(u,v)} \operad'' \xrightarrow{p} \operad'
$$
Moreover, since the operads $\operad , \operad'$ are cofibrant both injections
$$
\operad \xrightarrow{i} \operad \sqcup \operad' \xleftarrow{j} \operad'
$$
are cofibrations. Then, the maps $u \circ i : \operad \to \operad''$
and $v \circ j : \operad' \to \operad''$ are trivial cofibration. As retracts of morphisms
of the form described in \cref{thm:free_quasi_iso_induces_weak_equivalence_of_comonads},
they induces model equivalence between their model categories of coalgebras. Using the fact
that $\id{\operad'} = p \circ v \circ j$ and the 2-out-of-3 property, $p$ also yields a model equivalence
between categories of coagebras. So this is also the case for $p \circ u \circ i$.

\section{Cobar and its right adjoint}
%=======================================================================

\begin{center}
\begin{tikzpicture}
	\node[draw = black, rectangle, inner sep = 10pt, rounded corners]
	{
	\begin{minipage}{0.9\textwidth}
		\begin{center}
			From now on, we shall consider a locally conilpotent
			curved  cooperad
			$(\overline\coperad,\overline w,d_\coperad,\theta)$, a
			dg-operad $(\operad , m,\eta,d_\operad )$ and
			a twisting morphism $\alpha: \sinv\overline\coperad
			\longrightarrow \operad $.
		\end{center}
	\end{minipage}
	};
\end{tikzpicture}
\end{center}

In this section we are going to define the Cobar functor. Its
construction is entirely dual to the one of the Bar functor between the
category of dg-algebras over $\operad $ and the category of curved
coalgebras on $\coperad$~%
\cite{arXiv:1707.03465}. Given the twisting morphism $\alpha$, it is
possible to construct
\[
	\cobarfunctorfull{}_\alpha : \catofcog\operad
	\longrightarrow
	\catofalg\coperad .
\]
This functor admits a right adjoint that we shall also construct
\[
	\cobaradjfull_\alpha : \catofalg\coperad
	\longrightarrow
	\catofcog\operad .
\]
Since there is no ambiguity here, we shall forget to make a reference to
the twisting morphism $\alpha$ in what follows.
When necessary, we shall shorten $\cobaradjfull$ to $\cobaradj$ and
$\cobarfunctorfull{}$ to $\cobarfunctor$.

\subsection{The Cobar functor}

Let $(V,a_V,d_V)$ be a $\operad$-coalgebra in the category of chain
complexes.  We shall define a
\mbox{$\coperad$-algebra}
$(\cobarfunctorfull{} (V), a, d_b)$ as follows:
\begin{itemize}
	\item the underlying graded algebra is the free algebra over $V$
	      \[
	      	\cobarfunctorfull{} V \coloneqq V^\coperad.
	      \]
	\item the derivation $\induceddiff b$ is freely generated by a degree
	      $-1$ map
	      $b : V \to V^\coperad \isonat V \oplus V^{\overline\coperad}$
	      given by the sum of the maps
	      \begin{itemize}
	      	\item $V \xrightarrow{d_V} V$;
	      	\item $V \xrightarrow{a_V} V^\operad \xrightarrow{-V^\alpha}
	      	      V^{\overline \coperad}$,
	      \end{itemize}
	      that is
	      \[
	      	b \coloneqq \canonicalinj \circ d_V - V^\alpha \circ a_V .
	      \]
\end{itemize}

\begin{proposition}%
\label{thm:definition_de_cobar}
The free graded $\coperad$-algebra $V^\coperad$ together with the
derivation $d_b$ is a
curved $\coperad$-algebra. Hence the Cobar construction yields a
functor
\[
	\cobarfunctorfull{} : \catofcog\operad
	\longrightarrow \catofalg\coperad .
\]
\end{proposition}

\begin{proof}
In what follows, we shall forget to mention the use of the lax
map $\laxmap$ for the reader's sake.

It is enough to show that $\induceddiff b \circ b = V^\theta$%
~[\ref{thm:algebre_courbee_libre}].
We know that
\begin{align*}
	\induceddiff b &= \canonicaldiff + \canonicalaction_{V^\coperad}
	\circ {\sha (\canonicalinj ,b)}^\coperad;\\
	&= -V^{d_\coperad} + {\sha (\id V ,d_V)}^\coperad
	- \canonicalaction_{V^\coperad}
	\circ {\sha (\canonicalinj , V^\alpha \circ a_V)}^\coperad.
\end{align*}
Hence,
\begin{align*}
	\begin{split}
		\induceddiff b \circ b
		&= \canonicaldiff \circ \canonicalinj\circ d_V
		- \canonicaldiff \circ V^\alpha
		\circ a_V + \canonicalinj\circ d_V^2
		\\
		& \qquad- {\sha (\id V ,d_V)}^\coperad \circ
		V^\alpha \circ a_V
		- V^\alpha \circ a_V \circ d_V
		\\
		& \qquad \qquad+ \canonicalaction_{V^\coperad}
		\circ {\sha (\canonicalinj , V^\alpha \circ a_V)}^\coperad \circ
		V^\alpha \circ a_V;
	\end{split}
	\\
	\begin{split}
		&= -V^{\alpha \circ d_\coperad} \circ a_V
		\\
		& \qquad + V^\alpha \circ {\sha (\id V ,d_V)}^\operad 
		\circ a_V
		- V^\alpha \circ a_V \circ d_V
		\\
		& \qquad\qquad + \canonicalaction_{V^\coperad}
		\circ {\sha (\canonicalinj , V^\alpha \circ a_V)}^\coperad \circ
		V^\alpha \circ a_V
	\end{split}
\end{align*}
Since $d_V$ is a coderivation, one has
\[
	V^\alpha \circ {\sha (\id V ,d_V)}^\operad  -V^\alpha
	\circ a_V \circ d_V
	\circ a_V = V^\alpha \circ V^{d_\operad } \circ a_V 
	= -V^{d_\operad  \circ \alpha} \circ a_V.
\]
Besides, since
$\canonicalinj = V^{\eta\circ\tau} \circ a_V$, one has
% Some manual spacing has been done below
\begin{align*}
\canonicalaction_{V^\coperad} \circ
{\sha (\canonicalinj , V^\alpha \circ a_V)}^\coperad
\!\circ V^\alpha \!\!\circ a_V
	&= \canonicalaction_{V^\coperad} \circ {\sha(V^{\eta\circ\tau}
	\circ a_V, V^\alpha \circ a_V)}^\coperad \circ V^\alpha
	\!\!\circ a_V
	\\
	&= \canonicalaction_{V^\coperad} \circ {\sha(V^{\eta\circ\tau},
	V^\alpha)}^\coperad \circ a_V^\coperad \circ V^\alpha
	\circ a_V
	\\
	&= \canonicalaction_{V^\coperad}\circ {\sha(V^{\eta\circ\tau},
	V^\alpha)}^\coperad \circ {\left(V^\operad \right)}^\alpha
	\circ a_V^\operad
	\circ a_V
	\\
	&= V^w \circ V^{\id \coperad \compofsymseq \sha(\eta\circ\tau,
	\alpha)} \circ V^{\alpha \compofsymseq \idfunctor} \circ V^m
	\circ a_V
	\\
	&=- V^{m\circ (\alpha\compofsymseq\sha(\eta\circ\tau,\alpha))
	\circ w}
	\circ a_V
	\\
	&=- V^{  m \circ (\alpha  \otimes \alpha) \circ
	\overline w_2} \circ a_V.
\end{align*}
So
\begin{equation*}
	\induceddiff b \circ b
	= - V^{\partial \alpha + m \circ (\alpha  \otimes \alpha)
	\circ \overline w_2} \circ a_V =  V^{\eta\circ\theta}\circ  a_V =
	V^\theta.
\end{equation*}
\end{proof}

\subsection{The \texorpdfstring{$\cobaradjfull$}{Cobartip} functor}
%-----------------------------------------------------------------------

Let $(\Lambda,a_\Lambda,d_\Lambda)$ be a
$\coperad$-algebra. We shall define the \mbox{$\operad $-coalgebra}
\[
	(\cobaradjfull  (\Lambda), a, \induceddiff b)
\]
as follows:
\begin{itemize}
	\item the underlying graded algebra is the cofree coalgebra over
	      $\Lambda$:
	      \[
	      	\cobaradjfull\Lambda \coloneqq \freecog\operad \Lambda.
	      \]
	\item the coderivation $d_b$ is freely generated by a degree $-1$ map
	      $b : \freecog\operad \Lambda \to \Lambda$ given by
	      \[
	      	b \coloneqq d_\Lambda \circ \canonicalproj + a_\Lambda \circ
	      	\Lambda^\alpha.
	      \]
\end{itemize}

\begin{proposition}~%
\label{thm:coderivationloop}
The cofree graded $\operad $-coalgebra $\freecog\operad \Lambda$
together with the derivation $\induceddiff b$ is a
differential graded $\operad $-coalgebra. Then, the construction
$\cobaradjfull$ defines a functor
\[
	\cobaradjfull : \catofalg\coperad \longrightarrow
	\catofcog\operad .
\]
\end{proposition}

\begin{proof}
It is enough to show that
$b \circ \induceddiff b=0$~[\ref{thm:coderivation_de_carre_nul}].
One has
\begin{align*}
	\induceddiff b &=  \canonicaldiff + {\sha (\canonicalproj, b)}^\operad
	\circ \canonicalaction_{\freecog\operad
	\Lambda}~[\ref{thm:coderivation_colibre}]
	\\
	&=- \Lambda^{d_\operad }
	+ {\sha(\id \Lambda, d_\Lambda)}^\operad 
	+{\sha (\canonicalproj, a_\Lambda \circ  \Lambda^{\alpha})}^\operad
	\circ \canonicalaction_{\freecog\operad \Lambda}.
\end{align*}
Hence we get
\begin{align*}
	\begin{split}
		b\circ \induceddiff b
		&= d_\Lambda\circ\canonicalproj \circ \canonicaldiff
		+ a_\Lambda \circ  \Lambda^{\alpha} \circ \canonicaldiff
		+ d_\Lambda \circ \canonicalproj \circ
		{\sha(\id \Lambda, d_\Lambda)}^\operad
		\\
		& \qquad + a_\Lambda \circ  \Lambda^{\alpha} \circ
		{\sha(\id \Lambda, d_\Lambda)}^\operad 
		+ d_\Lambda \circ \canonicalproj \circ {\sha (\canonicalproj,
		a_\Lambda \circ \Lambda^{\alpha})}^\operad
		\circ \canonicalaction_{\freecog\operad \Lambda}
		\\
		& \qquad \qquad + a_\Lambda \circ  \Lambda^{\alpha} \circ
		{\sha (\canonicalproj, a_\Lambda \circ  \Lambda^{\alpha})}^\operad
		\circ \canonicalaction_{\freecog\operad \Lambda};
	\end{split}
	\\
	\begin{split}
		&= a_\Lambda \circ \Lambda^{d_\operad  \circ \alpha}
		+d_\Lambda^2 \circ \canonicalproj
		\\
		& \qquad - a_\Lambda \circ {\sha(\id \Lambda,
		d_\Lambda)}^\coperad \circ \Lambda^{\alpha}
		+ d_\Lambda \circ a_\Lambda \circ  \Lambda^{\alpha}
		\\
		& \qquad\qquad + a_\Lambda \circ
		\Lambda^{m \circ (\alpha \otimes \alpha ) \circ \overline w_2}.
	\end{split}
\end{align*}
The fact that
\[
	a_\Lambda \circ  \Lambda^{\alpha} \circ
	{\sha (\canonicalproj, a_\Lambda \circ  \Lambda^{\alpha})}^\operad
	\circ \canonicalaction_{\freecog\operad \Lambda}
	= a_\Lambda \circ  \Lambda^{m \circ (\alpha \otimes \alpha) \circ
	\overline w_2}
\]
follows from the same arguments as those used to prove that
$\cobarfunctorfull V$ is a $\coperad$-algebra%
~[\ref{thm:definition_de_cobar}].
Besides, since $d_\Lambda$ is a derivation, the second line can be
rewritten as
\[
	- a_\Lambda \circ {\sha(\id \Lambda,
	d_\Lambda)}^\coperad \circ \Lambda^{\alpha}
	+ d_\Lambda \circ a_\Lambda \circ  \Lambda^{\alpha}
	= - a_\Lambda \circ \Lambda^{d_\coperad} \circ  \Lambda^{\alpha}
	= a_\Lambda \circ \Lambda^{ \alpha \circ d_\coperad}.
\]
Finally we make the following replacement on the first line,
\[
	d_\Lambda^2 \circ \canonicalproj
	=  a_\Lambda \circ \Lambda^\theta \circ
	\canonicalproj = a_\Lambda \circ \Lambda^{\eta \circ \theta }.
\]
This gives us
\[
	b\circ \induceddiff b = a_\Lambda \circ \Lambda^{\eta \circ \theta +
	\partial \alpha + m \circ (\alpha \otimes \alpha ) \circ
	\overline w_2 } = 0 .
\]
\end{proof}

\subsection{Properties}
%-----------------------------------------------------------------------

Let $(V, a_V, d_V)$ be a $\operad$-coalgebra and let $(\Lambda, a_\Lambda,
d_\Lambda)$ be a $\coperad$-algebra.

\begin{definition}
An $\alpha$-twisting
morphism from $V$ to $\Lambda$ is a degree $0$ map $\phi: V \to \Lambda$
such that
\[
	\partial \phi + a_\Lambda \circ \phi^\alpha\circ a_V=0.
\]
The set of $\alpha$-twisting morphisms from $V$ to $\Lambda$ is denoted
$\alphatwisting\alpha V \Lambda$.
\end{definition}

\begin{theorem}%
\label{thm:adjonction_cobar}
There exists natural isomorphisms
\[
	\hombracket{\catofalg\coperad}
	{\cobarfunctorfull{}
	 V} \Lambda
	\isonat \alphatwisting\alpha V \Lambda
	\isonat
	\hombracket{\catofcog\operad}
	V {\cobaradjfull \Lambda}.
\]
In particular, the functor $\cobarfunctorfull{}$ is left adjoint to 
the functor $\cobaradjfull$.
\end{theorem}

\begin{proof}
One has%
~[\ref{thm:morphism_free_algebra}]
\begin{align*}
f : \cobarfunctorfull{} V \to \Lambda
&\Leftrightarrow f : V^\coperad \to \Lambda, \quad d_\Lambda \circ
f = f \circ \induceddiff b
\\
&\Leftrightarrow
\phi : V \to \Lambda, \quad d_\Lambda \circ \phi = a_\Lambda \circ
\phi^\coperad \circ (\canonicalinj \circ d_V - V^\alpha \circ a_V)\\
&\Leftrightarrow \phi \text{ is an $\alpha$-twisting morphism}
\end{align*}
One has the same bijection for $\cobaradjfull$%
~[\ref{thm:morphisms_to_a_free_coalgebra}].
\end{proof}

\begin{remark}
Notice that  the
Cobar functor takes its values in the full subcategory of complete
$\coperad$-algebras
[\ref{thm:les_algebres_libres_sont_completes}].
Hence we get another adjunction
\[
	\begin{tikzcd}[ampersand replacement=\&]
		\catofcog\operad
		\arrow[rr, shift left=1.5,"\cobarfunctorfull{}"]
		\&
		\&
		\catofcompletealg\coperad.
		\arrow[ll, shift left=1.5, "\cobaradjfull"]
	\end{tikzcd}
\]
\end{remark}

\begin{proposition}%
\label{thm:cobar_est_fidele}
Then the unit morphism $V
\to \cobaradj \cobarfunctor V$ is a degree-wise
monomorphism.
In particular, the $\cobarfunctorfull{}$ functor is faithful.
\end{proposition}

\begin{proof}
The graded map
\[
	\freecog {\operad} V^{\coperad}
	\xrightarrow{{\left(V^{\coperad}\right)}^\eta}
	V^{\coperad}
	\xrightarrow{V^{\tau}}
	V,
\]
is a left inverse of the unit morphism $V \to \freecog {\operad}
V^{\coperad}$.
Subsequently, this unit morphism is a degree-wise monomorphism.
\end{proof}

\begin{proposition}~%
\label{thm:cobaradj_est_fidele}
Assume that $P$ admits a graded augmentation,
then the counit $\cobarfunctor\cobaradj \Lambda \to
\Lambda$ is a degree-wise epimorphism.
In particular, the $\cobaradjfull$ functor is faithful.
\end{proposition}

\begin{proof}
The counit map is given by
\[
	a_\Lambda \circ \fullpower{\counitw^\operad (\Lambda)}\coperad :
	\fullpower{\freecog\operad\Lambda}\coperad \longrightarrow
	\Lambda^\coperad \longrightarrow \Lambda.
\]
The map $\Lambda^\tau$ is a graded section of $a_\Lambda$ so that
$a_\Lambda$ is a degree-wise epimorphism.
Since $\operad$ admits a graded augmentation, $\counitw^\operad
(\Lambda)$ is a degree-wise epimorphism and the cotensor by $Q$ is again
a degree-wise epimorphism%
~[\ref{rmk:cotenseur_fonctoriel}].
By composition the counit is a degree-wise epimorphism.
\end{proof}

\subsection{Functoriality}

Consider the following commutative square
\[
\begin{tikzcd}
 	\coperad 
	\arrow[r,"f"]
	\arrow[d,"\alpha"']
	& \coperad'
	\arrow[d,"\beta"]
	\\
	\operad
	\arrow[r,"g"]
	&
	\operad'
\end{tikzcd}
\]
where vertical arrows are twisting morphisms and horizontal
arrows are morphisms of curved conilpotent cooperads
and operads. This induces several adjunctions
\begin{itemize}
 \item $\freealg f \dashv \forget^f$ relating
 	complete $\coperad$-algebras
	to complete $\coperad'$-algebras ;
\item $\forget_g \dashv \freecog g$ relating $\operad$-coalgebras
	to $\operad'$-coalgebras ;
\item $\cobarfunctor_\alpha \dashv \cobaradj_\alpha$ relating
 	complete $\coperad$-algebras
	to $\operad$-coalgebras ;
\item $\cobarfunctor_\beta \dashv \cobaradj_\beta$ relating
 	complete $\coperad'$-algebras
	to $\operad'$-coalgebras ;
\item $\cobarfunctor_{\gamma} \dashv
	\cobaradj_{\gamma}$ relating
 	complete $\coperad$-algebras
	to $\operad'$-coalgebras, where $\gamma=g\circ \alpha
	=\beta \circ f$.
\end{itemize}

\begin{lemma}
There exists canonical isomorphisms of functors
\begin{align*}
	\cobarfunctor_{\gamma} \isonat& \freealg f \circ \cobarfunctor_{\beta}
	\isonat
	\cobarfunctor_\alpha \circ \forget_g ;
	\\
	\cobaradj_{\gamma}
	\isonat& \cobaradj_\beta \circ \forget^f \isonat
	\freecog g \circ \cobaradj_\alpha .
\end{align*}
\end{lemma}

\begin{proof}
On the one hand, to show that $\cobarfunctor_{\gamma}$ is canonically isomorphic
to $\cobarfunctor_\alpha \circ \forget_g$ it suffices to notice that both functors send
a $\operad'$-coalgebra $V$ to the $\coperad$-algebra whose underlying graded coalgebra
is $V^\coperad$ and whose derivation is freely generated by a degree
$-1$ map from $V$ to $V^\coperad \isonat V \oplus V^{\overline\coperad}$
given by the sum of the maps
	      \begin{itemize}
	      	\item $V \xrightarrow{d_V} V$;
	      	\item $V \xrightarrow{a_V} V^{\operad'} 
			\xrightarrow{V^g} V^\operad  \xrightarrow{-V^\alpha}
	      	      V^{\overline \coperad}$.
	      \end{itemize}
This gives us by adjunction a canonical isomorphism $\cobaradj_{\gamma} \isonat 
\freecog g \circ \cobaradj_\alpha $.

On the other hand, to show that $\cobaradj_{\gamma}$
is canonically isomorphic
to $\cobaradj_\beta \circ \forget^f$ it suffices to notice that both functors send
a $\coperad$-algebra $\Lambda$ to the $\operad'$-coalgebra whose underlying graded coalgebra
is $\freecog{\operad'} \Lambda$ and whose coderivation is freely generated by a degree
$-1$ map from $\freecog{\operad'} \Lambda$ to $\Lambda$
given by the sum of the maps
	      \begin{itemize}
	      	\item $\freecog{\operad'} \Lambda \twoheadrightarrow \Lambda \xrightarrow{d_\Lambda} \Lambda$;
	      	\item $\freecog{\operad'} \Lambda \hookrightarrow \Lambda^{\operad'}
		\xrightarrow{\Lambda^{\beta}} \Lambda^{\overline{\coperad'}}
		\xrightarrow{\Lambda^{f}} \Lambda^{\overline{\coperad}} 
			\xrightarrow{a_\Lambda} \Lambda$.
	      \end{itemize}
This gives us by adjunction a canonical isomorphism $\cobarfunctor_{\gamma} \isonat 
\freealg f \circ \cobarfunctor_{\beta}$.
\end{proof}

\begin{proposition}
\label{proposition:factorisation_unit_morphism}
The canonical morphism of functors
\[
	\forget_g \xrightarrow{\eta \circ \forget_g}
	\cobaradj_\alpha \circ \cobarfunctor_\alpha \circ \forget_g
\]
is the composition
\[
\begin{tikzcd}
 	\forget_g 
	\arrow[r,"\forget_g \circ \eta"]
	& \forget_g \circ \cobaradj_\beta \circ \cobarfunctor_\beta
	\arrow[rr,"\forget_g \circ \cobaradj_\beta \circ \eta \circ \cobarfunctor_\beta"]
	&& \forget_g \circ \cobaradj_\beta \circ
	\forget^f \circ \freealg f \circ \cobarfunctor_\beta
	\arrow[d,equal]
	\\
	&&& \forget_g \circ \freecog g \circ \cobaradj_\alpha
	\circ \cobarfunctor_\alpha \circ \forget_g
	\arrow[d,"\epsilon \circ \cobaradj_\alpha
	\circ \cobarfunctor_\alpha \circ \forget_g"']
	\\
	&&& \cobaradj_\alpha \circ \cobarfunctor_\alpha \circ \forget_g .
\end{tikzcd}
\]
\end{proposition}

\begin{proof}
 This follows from the commutation of the following diagram
\[
\begin{tikzcd}
 	\forget_g \circ \cobaradj_\beta \circ \cobarfunctor_\beta
	\arrow[rrr,"\forget_g \circ \cobaradj_\beta \circ \eta \circ\cobarfunctor_\beta"]
	&&&\forget_g \circ \cobaradj_\beta \circ \forget^f \circ \freealg f
	\circ  \cobarfunctor_\beta
	\arrow[d, equal]
	\\
	\forget_g \arrow[rrr] \arrow[u,"\forget_g \circ \eta"]
	\arrow[d,"\forget_g \circ \eta"]
	&&& \forget_g \circ \cobaradj_\gamma \circ \cobarfunctor_\gamma
	\arrow[d, equal]
	\\
	\forget_g \circ \freecog g \circ \forget_g
	\arrow[rrr,"\forget_g \circ \freecog g \circ \eta \circ \forget_g"]
	\arrow[d,"\epsilon \circ \forget_g"]
	&&& \forget_g \circ \freecog g \circ
	\cobaradj_\alpha \circ \cobarfunctor_\alpha \circ \forget_g
	\arrow[d,"\epsilon \circ \cobaradj_\alpha \circ \cobarfunctor_\alpha \circ \forget_g"]
	\\
	\forget_g 
	\arrow[rrr, "\eta \circ \forget_g"]
	&&& \cobaradj_\alpha \circ \cobarfunctor_\alpha \circ \forget_g ,
\end{tikzcd}
\]
and the fact that the composition
\[
	\forget_g 
	\xrightarrow{\forget_g \circ \eta}
	\forget_g \circ \freecog g \circ \forget_g
	\xrightarrow{\epsilon \circ \forget_g}
	\forget_g
\]
is the identity of $\forget_g$.
\end{proof}
\section{Model structure on complete algebras}
%=======================================================================

In the previous section, we endowed the category of coalgebras over an
operad $\operad$ with a model structure that is expected to
describe the
relevant ∞\=/category of $\operad$-coalgebras in the ∞-category of chain
complexes%
~[\ref{conj:infinity_cat_of_cogs}].
However, this model structure is cumbersome
due to the impracticability of the comonad
$\comonadl^\operad$.
For instance, identifying the fibrant objects or computing fibrant
replacements is a real challenge.

To bypass this obstacle, we get inspiration from an idea that was
applied with success to describe the homotopy theory of algebras
over a dg-operad: to transfer the model structure of
$\operad$-algebras to the category of $\coperad$-coalgebras along the
$\barfunctorfull$ adjunction
\[
	\begin{tikzcd}[ampersand replacement=\&]
		\catofalg\operad
		\arrow[rr, shift right=1.5, swap, "\barfunctorfull"]
		\& \&
		\catofcog\coperad
		\arrow[ll, shift right=1.5, swap, "\leftadjoint\barfunctorfull"]
	\end{tikzcd}
\]
and then to study the homotopy theory of $\operad$-algebras
in the framework of $\coperad$-coalgebras. Indeed, under some conditions
we obtained a model equivalence, for instance in the case when
$\operad$ is the dg-operad $\leftadjoint\barfunctor \coperad$.
This result has been given in
\emph{Homotopy theory of unital algebras}%
~\cite{arXiv:1612.02254},
following earlier results for Lie algebras by Hinich%
~\cite{doi:10.1016/S0022-4049(00)00121-3}, associative algebras by
Lefèvre-Hasegawa%
~\cite{arXiv:0310337}
and general augmented dg-operads by Vallette%
~\cite{arXiv:1411.5533} and Drummond-Cole \& Hirsh%
~\cite{doi:10.1090/proc/12823}.

Thus, we are going to transfer
the model structure on $\operad $-coalgebras to the
category of complete $\coperad$-algebras along the Cobar adjunction,
\[
	\begin{tikzcd}[ampersand replacement=\&]
		\catofcog\operad
		\arrow[rr, shift left=1.5,"\cobarfunctorfull"]
		\&
		\&
		\catofcompletealg\coperad.
		\arrow[ll, shift left=1.5, "\cobaradjfull"]
	\end{tikzcd}
\]
In the next section we shall prove that this defines a model equivalence
in the case where $\operad$ is the dg-operad $\baradj\coperad$%
~[\sectionref{sec:equivalence_cobar}].
We shall then explain how the homotopy theory of $\operad$-coalgebras
can be read from the homotopy theory of complete $\coperad$-algebras%
~[\sectionref{sec:theorie_homotopique_des_cogebres_lineaires}].

So from now on, we are concerned with proving the following theorem.

%bigtheorem
\begin{theorem}%
\label{bigthm:theoreme_de_transfert}
If $\operad$ is coadmissible,
the category of complete $\coperad$\=/algebras can be endowed with
a combinatorial
model structure transferred along the $\cobarfunctorfull{}$ adjunction,
where a morphism $f$
of complete $\coperad$-algebras is
\begin{itemize}
	\item a weak equivalence if $\cobaradjfull(f)$ is a
	      quasi-isomorphism of chain complexes
	      \[
	      	f \in \weakeqofcompletealg \Longleftrightarrow \cobaradj f \in
	      	\weakeqofcog;
	      \]
	\item a fibration if $\cobaradjfull(f)$ is a
	      degree-wise epimorphism of chain complexes.
	      \[
	      	f \in \fibrationofcompletealg \Longleftrightarrow \cobaradj f
	      	\in \fibrationofcog.
	      \]
\end{itemize}
As a consequence, the Cobar adjunction is promoted to a model
adjunction.
\end{theorem}

\begin{definition}[(Canonical model structure)]
When $\specialoperad \coloneqq \baradj \coperad$ and
$\canonicaltwisting$ is the
canonical twisting morphism, the transferred model structure on the
category of complete $\coperad$-algebras shall be called the
\emph{canonical model structure}.
\end{definition}

\begin{proof}[Outline of the proof]
For this proof we may just check that the assumptions of the
acyclicity theorem%
~\cite[Sec. 2.5 \& 2.6]{doi:10.1007/s00014-003-0772-y}
are satisfied:
\begin{itemize}
	\item The model category of $\operad $-coalgebras
	      is combinatorial~[\ref{thm:coadmissible}];
	\item The category of complete $\coperad$-algebras is
	      presentable~[\ref{thm:categorie_des_algebres_completes%
	      _courbees_est_presentable}];
	\item Every arrow having the left lifting property against all
	      fibrations is a weak equivalence,
	      ${}^\boxslash(\fibrationofcompletealg)\subset
	      \weakeqofcompletealg$.
\end{itemize}

For the last item
in the planar case, we shall use the path object argument: we
shall show that every complete $\coperad$-algebra is fibrant and that
for any complete $\coperad$-algebra $\Lambda$, we shall show that
$[\interval, \Lambda]$ provides a path-object on $\Lambda$.

The characteristic zero case is less straightforward and we shall show
the acyclicity by direct computation instead.
\end{proof}

\subsection{Infinitesimal extensions}
%-----------------------------------------------------------------------

Using a dévissage argument, any fibration can be decomposed into
a succession of elementary fibrations given by infinitesimal extensions.
The prime example of an infinitesimal extension will be the canonical
map
\[
	\cofiltration {n+1} \Lambda \longrightarrow \cofiltration n \Lambda
\]
for any natural $n$ and any $\coperad$-algebras $\Lambda$.
In the case where $\Lambda$ is a complete associative algebra,
a quotient map $\varpi : \Lambda \to \Lambda/I$ is an infinitesimal
extension when $I$ is an absorbent ideal of $\Lambda$. We shall
generalise this definition to the case of any  cooperad.

\begin{definition}
Let $\Lambda$ be a $\coperad$-algebra and $\ideali$ be an ideal of
$\Lambda$. The
corresponding degree-wise epimorphism
\[
	\varpi : \Lambda \longrightarrow \Lambda/\ideali 
\]
is said to be an infinitesimal extension if the following composition
equals zero
\[
	{\sha(\Lambda,\ideali)}^{\overline{\coperad}} \subobject
	\Lambda^\coperad
	\xrightarrow{a} \Lambda.
\]
By graded splitting of an infinitesimal extension
$\varpi : \Lambda \to \Gamma$
we shall mean a graded map $s : \Gamma \to \Lambda$ such that
$\varpi \circ s$ is the identity of $\Gamma$ as a graded object.
\end{definition}

\begin{remark}
By semi-simplicity of $\catss$, any degree-wise epimorphism admits a
graded splitting.
\end{remark}

\begin{remark}
Notice that by definition, the curvature of $\coperad$ acts trivially on
$\ideali$; hence $\ideali$ is a chain complex.
\end{remark}

We shall now see that infinitesimal extensions are fibrations. The
proof is rather technical so we start with a sketch.

\begin{proof}[Sketch of the proof]
Although we do not have a complete knowledge of fibrations in the
model category of $\operad$-coalgebras, we can still describe some of
them. In the model category of chain complexes, the ‘universal’
fibration is the map $\counitofd : \dzero \to \monoidalunit$; it
generates by tensorisation other fibrations $\counitofd \otimes X :
\dzero \otimes X \to X$ between chain complexes. By ripple effect,
the cofree morphism of cofree $\operad$-coalgebras
$\freecog\operad (\counitofd \otimes X) : \freecog \operad
(\dzero \otimes X) \to \freecog\operad X$ is a fibration of
$\operad$-coalgebras. The goal of the proof is to show that for
a given infinitesimal extension $\varpi : \Lambda \to \Gamma$, the
map $\cobaradj \varpi : \cobaradj \Lambda \to \cobaradj \Gamma$ can
be obtained as a pull-back of such a ‘generation fibration’
\[
	\begin{tikzcd}[ampersand replacement=\&]
		\cobaradj\Lambda
		\arrow[r, ""]
		\arrow[d, "\cobaradj\varpi",swap]
		\arrow[rd, very near start, phantom, "\lrcorner"]
		\& \freecog\operad (\dzero \otimes X)
		\arrow[d, "\freecog\operad(\counitofd \otimes X)"] \\
		\cobaradj\Gamma
		\arrow[r, "", swap]
		\& \freecog\operad X.
	\end{tikzcd}
\]

Consider such an infinitesimal extension
$\varpi: \Lambda \to \Gamma$ and a
graded splitting $s : \Gamma \to \Lambda$. This induces a graded
isomorphism
\[
	\Lambda \iso \Gamma \oplus \ideali .
\]
where $\ideali$ is the kernel of $\varpi$.
Under this isomorphism the derivation of $\Lambda$ is given by
the following matrix
\[
	d_\Lambda = 
	\begin{pmatrix}
	d_\Gamma & 0 \\
	d_t & d_{\ideali}
	\end{pmatrix}.
\]
where $d_t$ is a degree $-1$ map from $\Gamma$ to $\ideali$.
Moreover, since $\varpi$ is an infinitesimal extension,
the map $a_\Lambda :\Lambda^{\overline{\coperad}} \to \Lambda$ 
decomposes as follows
\[
	\begin{tikzcd}[ampersand replacement=\&]
		a_\Lambda : \Lambda^{\overline{\coperad}}
		\arrow[r, twoheadrightarrow, "\varpi^{\overline{\coperad}}"]
		\& \Gamma^{\overline{\coperad}}
		\arrow[rr, "(a_\Gamma{,} a_t)"]
		\& \& \Gamma \oplus \ideali .
	\end{tikzcd}
\]
where $a_t$ is a map from $\Gamma^{\overline{\coperad}}$ to $\ideali$.

The derivation condition for $\Lambda$ rewrites now as the following
equalities between maps from $\Gamma^{\overline{\coperad}}$ to
respectively $\ideali$ and $\Gamma$,
\[
	\begin{cases}
		d_t \circ a_{\Gamma} + d_\ideali \circ  a_t 
		= a_t \circ
		\left({\sha(\id \Gamma, d_{\Gamma})}^{\overline{\coperad}}
		- \Gamma^{d_\coperad}  \right)\\
		d_{\Gamma} \circ a_{\Gamma} 
		= a_{\Gamma} \circ \left({\sha(\id \Gamma,
		d_{\Gamma})}^{\overline{\coperad}} - \Gamma^{d_\coperad}  \right).
	\end{cases}
\]

Given the maps $\varpi$ and $s$, the underlying graded $\operad
$-coalgebra of $\cobaradj\Lambda$ is isomorphic to
\[
	\freecog\operad (\Gamma \oplus \ideali )\isonat
	\freecog\operad (\Gamma) \times \freecog\operad (\ideali).
\]

The coderivation is entirely determined by its projection on
the cogenerators $\Gamma \oplus \ideali$ which
has three components given by
\begin{itemize}
	\item the coderivation of $\cobaradjfull\Gamma$;
	\item the differential of $\ideali$;
	\item a transfer component
	      \[
	      	\freecog\operad (\Gamma \oplus \ideali ) \twoheadrightarrow
	      	\freecog\operad \Gamma \to \ideali
	      \]
	      which is built out from $d_t$ and $a_t$.
\end{itemize}
The key idea is that, composing the transfer component with the degree
$1$ map $\ideali \xrightarrow{\degreeonemap \otimes \id{\ideali}}
\s\ideali$ gives us by coextension along
$\s\ideali \to \freecog\operad (\s\ideali)$, a morphism of $\operad
$-coalgebras
\[
	\tilde t: \cobaradjfull\Gamma \to \freecog\operad (\s\ideali).
\]
This will allow us to realise $\cobaradj\Lambda$ as the fibre product
\[
	\begin{tikzcd}[ampersand replacement=\&]
		\cobaradjfull\Lambda \arrow[r] \arrow[d]
		\arrow[rd, phantom, very near start, "\lrcorner"]
		\& \freecog\operad  (\dzero \otimes \s\ideali)\arrow[d]
		\\
		\cobaradjfull\Gamma \arrow[r,"\tilde t"]
		\& \freecog\operad  (\s\ideali).
	\end{tikzcd}
\]
Thus the map $\cobaradj \varpi$ is the pull-back of a
fibration of $\operad $-coalgebras. As such, it is also a fibration.
\end{proof}

Now that the proof has been sketched, let us go into the details. In the
following we shall denote by $\specialprojto {\ideali}$ the composite
map
\[
	\begin{tikzcd}
		\specialprojto {\ideali} :
		\freecog\operad (\Gamma\oplus\ideali)
		\arrow[rrr, twoheadrightarrow,"\freecog\operad (\Gamma
		\oplus\ideali
		\twoheadrightarrow \ideali)"]
		&&&
		\freecog\operad (\ideali)
		\arrow[r, "\canonicalproj"]
		& \ideali
	\end{tikzcd}
\]
and let us define $\specialprojto \Gamma$ in the same way.

\begin{definition}[(Transfer morphism)]%
\label{def:transfert_morphism}
Let $t$ be the transfer component of the differential of
$\cobaradj \Lambda$, that is the following degree $0$ map
\[
	\begin{tikzcd}[ampersand replacement=\&]
		\freecog\operad \Gamma
		\arrow[r, "\freecog\operad  s", hookrightarrow]
		\&
		\freecog\operad \Lambda
		\arrow[r, "d_{\cobaradj\Lambda}"]
		\& \freecog\operad \Lambda
		\arrow[r, twoheadrightarrow, "\specialprojto {\ideali}"]
		\& \ideali
		\arrow[r]
		\&
		\s\ideali
	\end{tikzcd}
\]
Let $\tilde t : \freecog\operad  \Gamma \to \freecog\operad
(\s\ideali)$ be the morphism
of graded $\operad $-coalgebras obtained by coextension of $t$.
\end{definition}

\begin{remark}%
\label{remark:ecriture-du-morphisme-de-transfer-t}
The map $t$ can be described as the sum
\[
	t =(\degreeonemap \otimes \id{\ideali}) \circ
	\left(a_t \circ \Gamma^\alpha + d_t \circ \tau_\Gamma \right).
\]
\end{remark}

The following lemma is a direct consequence of the fact that $\varpi$ is
an infinitesimal extension.

\begin{lemma}%
\label{thm:reecriture-derivation}
The following equality between maps from $\cobaradj \Lambda$ to
$\s \ideali$
holds
\[
	(\degreeonemap \otimes \id{\ideali}) \circ \specialprojto {\ideali}
	\circ  {d_{\cobaradj \Lambda} }
	=
	 t \circ \freecog\operad \varpi
	 + (\degreeonemap \otimes \id{\ideali})\circ d_\ideali \circ
	 \specialprojto {\ideali} .
\]
\end{lemma}

\begin{lemma}%
\label{thm:morphisme_t}
The following morphism of graded $\operad $-coalgebras
\[
	\begin{tikzcd}[ampersand replacement=\&]
		\freecog\operad (\Lambda)
		\arrow[rr, ""]
		\arrow[d, equal]
		\& \& \freecog\operad (\s \ideali \oplus \ideali)
		\arrow[d,equal]
		\\
		\freecog\operad (\Gamma) \times \freecog\operad (\ideali)
		\arrow[rr, "\tilde t \times
		\id {\freecog\operad (\ideali)}"]
		\& \& \freecog\operad (\s\ideali) \times
		\freecog\operad (\ideali)
	\end{tikzcd}
\]
is actually a morphism of dg-$\operad $-coalgebras from
$\cobaradj \Lambda$ to $\freecog\operad (\dzero \otimes \s\ideali)$.
\end{lemma}

\begin{proof}
We have to show that the following diagram commutes
\[
	\begin{tikzcd}[ampersand replacement=\&]
		\freecog\operad (\Gamma \oplus \ideali)
		\arrow[rr,"{(t \circ  \freecog\operad \varpi)
		\times
		\specialprojto {\ideali}}"]
		\arrow[d,"{d_{\cobaradj \Lambda} }"']
		\&\& \s\ideali \oplus \ideali 
		\arrow[d,"{d_{\gradedd\otimes \s\ideali }}"]\\
		\freecog\operad (\Gamma \oplus \ideali)
		\arrow[rr,"{(t \circ  \freecog\operad \varpi) \times 
		\specialprojto {\ideali}}"']
		\&\& \s\ideali \oplus \ideali.
	\end{tikzcd}
\]
It suffices to show that both the projections on $\ideali$ and on
$\s\ideali$ of the two composite from $\freecog\operad (\Gamma \oplus
\ideali)$ to $\s\ideali \oplus \ideali$ are the same. On the one
hand, the difference between the two projections on $\s\ideali$ is given
by the following formula
\[
	t \circ \freecog\operad
	\varpi \circ d_{\cobaradj\Lambda} - (\id{\s} \otimes d_{\ideali})
	\circ t \circ \freecog\operad  \varpi.
\]
One has
\begin{align*}
	& t \circ \freecog\operad  \varpi \circ d_{\cobaradj\Lambda}
	- (\id{\s\monoidalunit} \otimes d_{\ideali}) \circ
	t \circ \freecog\operad  \varpi
	\\
	\begin{split}
		&= (\degreeonemap \otimes \id{\ideali})
		\circ \specialprojto {\ideali}
		\circ  {d_{\cobaradj \Lambda}^2 }
		- (\degreeonemap \otimes \id{\ideali})
		\circ d_\ideali \circ \specialprojto {\ideali} \circ
		{d_{\cobaradj \Lambda} }
		\\
		& \qquad - (\id{\s\monoidalunit} \otimes d_{\ideali}) \circ
		t \circ \freecog\operad  \varpi%
		~[\ref{thm:reecriture-derivation}]
	\end{split}
	\\
	\begin{split}
		& = 0 
		+ (\id{\s\monoidalunit} \otimes d_{\ideali})
		\circ (\degreeonemap \otimes \id{\ideali})
		\circ \specialprojto {\ideali} \circ {d_{\cobaradj \Lambda} }
		\\
		& \qquad - (\id{\s\monoidalunit} \otimes d_{\ideali}) \circ
		t \circ \freecog\operad  \varpi
	\end{split}
	\\
	&= (\id{\s\monoidalunit} \otimes d_{\ideali}) \circ
	(\degreeonemap \otimes \id{\ideali})\circ d_\ideali \circ
	\specialprojto {\ideali}
	\\
	&= - (\degreeonemap \otimes \id{\ideali})\circ d^2_\ideali \circ
	\specialprojto {\ideali}
	\\
	&= 0.
\end{align*}
On the other
hand, the two projections on $\ideali$ coincide, since%
~[\ref{thm:reecriture-derivation}] we get:
\[
	\specialprojto {\ideali} \circ d_{\cobaradj \Lambda} =
	(\degreeminusonemap
	\otimes \id{\ideali}) \circ t \circ \freecog\operad
	\varpi + d_{\ideali} \circ \specialprojto{\ideali}.
\]
\end{proof}

\begin{lemma} The following square
\begin{center}
	\begin{tikzpicture}
	\node (a) at (0,0) {
		\begin{tikzcd}[ampersand replacement=\&]
			\cobaradjfull\Lambda
			\arrow[rr, "\tilde t \times \id {\freecog\operad  \ideali}"]
			\arrow[d, "\cobaradj \varpi", swap]
			\& 
			\& 
			\freecog\operad (\dzero \otimes \s\ideali)
			\arrow[d, "\freecog\operad(\counitofd \otimes \s I)"] \\
			\cobaradjfull\Gamma
			\arrow[rr, "\tilde t", swap]
			\& 
			\& 
			\freecog\operad  (\s\ideali)
		\end{tikzcd}
	};
	\node[draw = black, rectangle, rounded corners] (b) at (-0.2,-0.05) {
	$\thesquares$
	};
	\end{tikzpicture}
\end{center}
is a commutative square of differential graded $\operad $-coalgebras.
\end{lemma}

\begin{proof}
By construction of $t$ this square is a commutative square of graded
$\operad $-coalgebras. Moreover all maps but $\tilde t$ are differential
graded maps~[\ref{thm:morphisme_t}]. We deduce that the composite
map $\tilde t \circ \cobaradj \varpi$ commutes with coderivations.
Since $\cobaradj \varpi$ also commutes with the coderivation and is an
epimorphism%
~[\ref{rmk:L_preserve_les_epi}], then
$\tilde t$ also commutes with coderivation and the square $\thesquares$
is a commutative square of differential graded $\operad$\=/coalgebras.
\end{proof}

\begin{lemma}%
\label{thm:une_extension_infinitesimale_est_une_fibration}
The above square $\thesquares$ is a fibre product in the category
of $\operad$-coalgebras.
\end{lemma}

\begin{proof}
This is actually a fibre product of graded $\operad $-coalgebras. We
conclude by the fact that the forgetful functor from differential graded
$\operad $-coalgebras to graded $\operad $-coalgebras preserves and reflects
limits%
~[\ref{thm:forget_coderivation_preserves_limits}].
\end{proof}

\begin{proposition}%
\label{thm:les_extensions_infinitesimales_sont_des_fibrations}
Infinitesimal extensions of $\coperad$-algebras are fibrations.
\end{proposition}

\begin{proof}
Since $\freecog\operad $ is a right adjoint it preserves fibrations.
Hence since the map $\tensorbycounitofd$ is a degree-wise epimorphism,
the induced map $\freecog\operad \tensorbycounitofd$ is a fibration
for the model structure on \mbox{$\operad $-coalgebras}. Since every
infinitesimal extension $\varpi$ is such that $\cobaradj \varpi$ can be
obtained
as a pull-back of a fibration of $\operad $-coalgebras~%
[\ref{thm:une_extension_infinitesimale_est_une_fibration}], it is
itself a fibration. This shows that infinitesimal extensions are
fibrations.
\end{proof}

\subsection{Fibrations}
%-----------------------------------------------------------------------

We shall now use what we know about infinitesimal extensions to
show the following characterisation of fibrations of complete
$\coperad$-algebras: degree-wise epimorphisms of complete
$\coperad$-algebras are always fibrations and in the canonical model
structure the two notions coincide.

\begin{lemma}
Let $\varpi : \Lambda \to \Lambda/\idealk$ be a degree-wise epimorphism
between complete $\coperad$-algebras. Then $\varpi$ is the ordinal
composition
\[
	\Lambda \to \cdots \to \Lambda/(\idealk \cap
	\ideal {n+1} \Lambda)
	\to \Lambda/(\idealk \cap \ideal n \Lambda) \to \cdots \to
	\Lambda/(\idealk \cap \ideal 0 \Lambda) = \Lambda/\idealk .
\]
\end{lemma}

\begin{proof}
Notice first that all the $\coperad$-algebras occurring in the above
diagram are complete. Indeed for any $n$, the $\coperad$-algebras
$\Lambda/\ideal n  \Lambda$ and $\Lambda/(\ideal n \Lambda +
\idealk)$ are nilpotent. Moreover $\Lambda/(\ideal n \Lambda \cap
\idealk)$ is obtained as the fibre product
\[
	\begin{tikzcd}[ampersand replacement=\&]
		\Lambda/(\idealk \cap \ideal n \Lambda )
		\arrow[r, ""]
		\arrow[d, "",swap]
		\arrow[rd, very near start, phantom, "\lrcorner"]
		\& \Lambda/\idealk
		\arrow[d, ""] \\
		\Lambda/\ideal n \Lambda
		\arrow[r, "", swap]
		\& \Lambda/(\idealk + \ideal n \Lambda) .
	\end{tikzcd}
\]
Finally a fibre product of complete $\coperad$-algebras is again
complete, since the category of complete $\coperad$-algebras is a
reflective localisation of the category of $\coperad$-algebras%
~[\ref{thm:la_categorie_des_algebres_completes_est_reflexive}].
In parallel, for $n \geq p$, we have the equivalence
\[
	\cofiltration p(\Lambda/(\idealk \cap \ideal n \Lambda))
	\isonat \cofiltration p \Lambda .
\]
which leads us, by limit switch, to the following computation:
\[
	\Lambda \isonat \limover{p\in\ordinalomega\op}
	\cofiltration p \Lambda \isonat
	\limover{p\in\ordinalomega\op}\limn 
	\cofiltration p(\Lambda/(\idealk \cap \ideal n \Lambda))
	\isonat \limn \Lambda/(\idealk \cap \ideal n \Lambda) .
\]
\end{proof}

\begin{lemma}%
\label{thm:fnelemfib}
Let $(\Lambda, a)$ be a $\coperad$-algebra
and let $\idealk$ be an ideal of
$\Lambda$.  Then, for any natural $n$, the morphism
\[
	\Lambda/\left(\idealk \cap \ideal {n+1}  \Lambda \right)
	\to \Lambda/\left(\idealk \cap \ideal {n} \Lambda \right)
\]
is an infinitesimal extension.
\end{lemma}

\begin{proof}
Given a natural $n$, by definition of the coradical filtration of
$\coperad$%
~[\ref{def:coradical_filtration}], the map
\[
	\begin{tikzcd}
		\overline{\coperad}
		\arrow[r, "\overline w"]
		&
		\overline{\coperad} \compofsymseq \coperad
		\arrow[r, two heads]
		&
		\overline{\coperad} \compofsymseq
		\sha(\coperad, \coperad/\filtration n \coperad)
	\end{tikzcd}
\]
factorises through
$\overline{\coperad}/ \filtration {{n+1}}
\overline{\coperad}$.
Thanks to this, one can write the commutative diagram
\[
	\begin{tikzcd}[ampersand replacement=\&]
		{\sha(\Lambda^\coperad,
		\Lambda^{\coperad/\filtration n \coperad})}^{\overline{\coperad}}
		\arrow[rr,"a^{\overline{\coperad}}"]
		\arrow[d]
		\& \& {\sha(\Lambda,\ideal n  \Lambda)}^{\overline{\coperad}}
		\arrow[d,"a"]\\
		\Lambda^{\overline{\coperad} \compofsymseq
		\sha(\coperad, \coperad/\filtration n \coperad)}
		\arrow[r,"\Lambda^{\overline{w}}"]
		\& \Lambda^{\coperad/\filtration {{n+1}} \coperad} \arrow[r,"a"]
		\& \Lambda.
	\end{tikzcd}
\]
Since by definition $\ideal {n+1}  \Lambda$ is the image of the
ideal $\Lambda^{\coperad/\filtration {{n+1}} \coperad}$ by $a$, the
composition
\[
	{\sha(\Lambda,\ideal n  \Lambda)}^{\overline{\coperad}}
	\xrightarrow{a} \Lambda
\]
factorises through $\ideal {n+1}  \Lambda$. So, since $\idealk$ is an
ideal, the map
\[
	{\sha(\Lambda, \idealk \cap \ideal n  \Lambda)}^{\overline{\coperad}}
	\xrightarrow{a} \Lambda
\]
factorises through $\idealk \cap \ideal {n+1} \Lambda$. Consider now the
following commutative square
\[
	\begin{tikzcd}[ampersand replacement=\&]
		{\sha\left(\Lambda, \idealk \cap \ideal {n}
		\Lambda \right)}^{\overline{\coperad}}
		\arrow[d,"a"]
		\arrow[r, two heads]
		\& {\sha\left(\Lambda/ \left(\idealk \cap \ideal {n+1}
		\Lambda \right),
		\left(\idealk \cap \ideal {n}  \Lambda \right) /
		\left(\idealk \cap \ideal {n+1}  \Lambda 
		\right)\right)}^{\overline{\coperad}}
		\arrow[d,"a"]
		\\
		\Lambda
		\arrow[r,two heads]
		\& \Lambda / \left(\idealk \cap \ideal {n+1}  \Lambda \right).
	\end{tikzcd}
\]
Since the composite maps from
${\sha\left(\Lambda, \idealk \cap \ideal {n}
\Lambda \right)}^{\overline{\coperad}}$
to $\Lambda / \left(\idealk \cap \ideal {n+1}  \Lambda \right)$
are both zero, and since the top horizontal
map is an epimorphism, then the right vertical map is zero.
\end{proof}

\begin{theorem}%
\label{thm:les_quotients_sont_des_fibrations}%
\label{thm:toutes_les_algebres_completes_sont_fibrantes}
A degree-wise epimorphism between complete $\coperad$\=/algebras is a
fibration. In particular all complete $\coperad$\=/algebras are fibrant.
\end{theorem}

\begin{proof}
Let $\varpi : \Lambda \to \Lambda/K$ be a degree-wise epimorphism
of complete $\coperad$-algebras. Then $\varpi$ is the ordinal
composition
\[
	\Lambda \to \cdots \to \Lambda/(\idealk \cap
	\ideal {n+1} \Lambda)
	\to \Lambda/(\idealk \cap \ideal n \Lambda) \to \cdots \to
	\Lambda/(\idealk \cap \ideal 0 \Lambda) = \Lambda/\idealk .
\]
Each term in this diagram is an infinitesimal extension%
~[\ref{thm:fnelemfib}],
hence a fibration%
~[\ref{thm:les_extensions_infinitesimales_sont_des_fibrations}].
Thus $\varpi$ is a fibration as it is an ordinal
composition of fibrations.
\end{proof}

\begin{corollary}%
\label{thm:cobar_preserves_les_fibrations}
The $\cobarfunctorfull$ functor preserves all fibrations.
\end{corollary}

\begin{proof}
Let $f : V \to W$ be a fibration between two $\operad$-coalgebras. Then
it is in particular a degree-wise epimorphism%
~[\ref{thm:les_fibrations_entre_cogebres_sont_des_epis}], it is thus
preserved by $\cobarfunctorfull$. Finally by the previous
proposition%
~[\ref{thm:les_quotients_sont_des_fibrations}], any degree-wise
epimorphism between complete $\coperad$-algebras is a fibration.
\end{proof}

\begin{proposition}%
\label{thm:sane_implies_fibrations_are_epis}
In the canonical model structure, every fibration of $\coperad$-algebras
is a degree-wise epimorphism.
\end{proposition}

\begin{proof}
Let $f : \Lambda \to \Gamma$ be a fibration of $\coperad$-algebras. Then
by definition $\cobaradj f$ is a fibration of
$\baradj\coperad$-coalgebras. By coadmissibility%
~[\ref{thm:les_fibrations_entre_cogebres_sont_des_epis}],
it is a degree-wise epimorphism.
We now contemplate the following counit diagram
\[
	\begin{tikzcd}[ampersand replacement=\&]
		\freecog\operad (\Lambda)
		\arrow[r, "\canonicalproj_\Lambda", twoheadrightarrow]
		\arrow[d, "\freecog\operad  f", twoheadrightarrow, swap]
		\& \Lambda
		\arrow[d, "f"] \\
		\freecog\operad (\Gamma)
		\arrow[r, "\canonicalproj_\Gamma", swap, twoheadrightarrow]
		\& \Gamma
	\end{tikzcd}
\]
to deduce that $f$ is also a degree-wise epimorphism.
\end{proof}

\subsection{Dévissage equivalences}%
%-----------------------------------------------------------------------
\label{sec:devissage_eq}

Now that we understand perfectly the set of fibrations in
the transferred model structure on the category of complete
$\coperad$-algebras, we need to focus on weak equivalences. We will
not have a complete knowledge of what they are but we can single out
a subset of equivalences that we can comprehend: the dévissage
equivalences. After giving the definition, we shall see that any
dévissage equivalence is a weak equivalence for the transferred model
structure on the category of complete $\coperad$-algebras.

\begin{lemma}%
\label{thm:gr_is_dg}
Let $\Lambda$ be any $\coperad$-algebra, then for every natural $n$
the quotient
\[
	\gr n \Lambda = \ideal n \Lambda/\ideal {n+1} \Lambda
\]
is a chain complex.
\end{lemma}

\begin{proof}
It follows from the fact that the morphism
\[
	\cofiltration {n+1} \Lambda \to 
	\cofiltration {n} \Lambda
\]
is an infinitesimal extension%
~[\ref{thm:fnelemfib}],
so that its kernel is a chain complex.
\end{proof}

\begin{definition}%
\label{def:devissage_eq}
A morphism of $\coperad$-algebras $f : \Lambda \to \Gamma$ is said to
be a
dévissage equivalence if for any natural number $n$, the induced map 
\[
	\gr n f : \gr n \Lambda \longrightarrow \gr n\Gamma
\]
is a quasi-isomorphism.
\end{definition}

Let us start with an elementary verification that will be useful later.
\begin{proposition}
Let $\Lambda$ be a $\coperad$-algebra, then the canonical map
\[
	\completionmap_\Lambda : \Lambda \longrightarrow \widehat{\Lambda}
\]
is a dévissage equivalence.
\end{proposition}

\begin{proof}
We have
$\widehat{\Lambda} \isonat \Lambda/\ideal \infty \Lambda
\Longrightarrow \cofiltration n\widehat{\Lambda}
\isonat \cofiltration n\Lambda$,
from which the dévissage equivalence follows.
\end{proof}

One can extend the notion of dévissage equivalence between complete
algebras to more general cofiltrations.

\begin{definition}
 A coladder of complete $\coperad$-algebras is a functor
 $$
 A : \ordinalomega\op \to \catofcompletealg\coperad
 $$
  so that for any natural integer $n$, the map
 $A(n) \to A(n-1)$ is an infinitesimal
 extension (with the convention $A(-1)=0$).
 In particular, for any $n$, the $\s^{-1}$-module
$\mathrm{Ker}(A(n) \to A(n-1))$ is a chain complex.
\end{definition}

\begin{definition}
 A morphism of coladder of complete $\coperad$-algebras
 $f: A \to A'$ is a coladder equivalence if for any natural integer $n$,
 the induced map
 $$
 \mathrm{Ker}(A(n) \to A(n-1)) \to \mathrm{Ker}(A'(n) \to A'(n-1))
 $$
 is a quasi-isomorphism.
 More generally, we say that a morphism of complete $\coperad$-algebras
is a coladder equivalence if it appears as the limit of a coladder equivalence between
coladders.
\end{definition}

\begin{example}
 Given a complete $\coperad$-algebras $\Lambda$, one can build naturally a coladder,
 that is the coladder of its radical cofiltration. Indeed, the
 transition maps $\cofiltration {n+1} \Lambda \to
\cofiltration n \Lambda$ are infinitesimal extensions for
any $\Lambda$ and any natural $n$%
~[\ref{thm:fnelemfib}]. Moreover, a morphism of complete $\coperad$-algebras
$f: \Lambda \to \Lambda'$ is a dévissage equivalence if and only if the 
induced morphism between their radical cofiltraiton coladders
is a coladder equivalence.
\end{example}

\begin{lemma}[(Weak $5$ lemma)]%
\label{thm:weak_five_lemma}
Consider the following diagram
\[
	\begin{tikzcd}[ampersand replacement=\&]
		0 \arrow[r] \arrow[d, equal] \&
		\idealj \arrow[r] \arrow[d,"a"] \&
		\Lambda \arrow[r,"\varpi"] \arrow[d, "b"] \& \Lambda/\idealj
		\arrow[d, "c"] \arrow[r]
		\& 0 \arrow[d,equal]
		\\
		0 \arrow[r] \& \idealk \arrow[r] \&
		\Gamma \arrow[r, "{\varpi'}"] \& \Gamma/\idealk
		\arrow[r] \& 0 
	\end{tikzcd}
\]
where $\varpi$ and $\varpi'$ are infinitesimal extensions of complete
$\coperad$-algebras and $b$ is a morphism of $\coperad$-algebras.

Assume that $a$ is a quasi-isomorphism and that $c$ is a weak
equivalence of complete $\coperad$-algebras. Then $b$ is also a weak
equivalence.
\end{lemma}

\begin{proof}
We proceed by reduction to the simpler case where $c$ is an
isomorphism. Let us denote by $P$ the following fibre product in the
category of $\coperad$-algebras
\[
	\begin{tikzcd}[ampersand replacement=\&]
		\Lambda \arrow[rd,"f"] \arrow[rdd, bend right,"b"']
		\arrow[rrd,"\varpi", bend left]
		\\
		\& P
		\arrow[rd, phantom, "\lrcorner", very near start]
		\arrow[r,"q"] \arrow[d,"r"'] \& \Lambda/\idealj \arrow[d,"c"]
		\\
		\&\Gamma \arrow[r, "{\varpi'}"] \& \Gamma/\idealk.
	\end{tikzcd}
\]
The image under the functor $\cobaradj$ of the
square is also a fibre product. Moreover
since any degree-wise
epimorphism of complete algebras is a fibration
[\ref{thm:les_quotients_sont_des_fibrations}],
the following square
\[
	\begin{tikzcd}[ampersand replacement=\&]
		\cobaradjfull P
		\arrow[rr, "\cobaradj q"]
		\arrow[d, "\cobaradj r",swap]
		\arrow[rrd, very near start, phantom, "\homotopypullbackmark"]
		\&\& \cobaradjfull\Lambda/\idealj
		\arrow[d, "\cobaradj c"] \\
		\cobaradjfull\Gamma
		\arrow[rr, "{\cobaradj \varpi'}", swap]
		\&\& \cobaradjfull \Gamma/\idealk
	\end{tikzcd}
\]
is a homotopy fibre square. As $\cobaradj c$ is a weak equivalence,
$\cobaradj r$ is also a weak equivalence. Besides, the map $q : P \to
\Lambda/\idealj$ is an infinitesimal extension with kernel $\idealk$.
We may then assume that $q = \varpi'$ and $c$ is the identity of
$\Lambda/\idealj$, that is we reduce to the following commutative
diagram:
\[
	\begin{tikzcd}[ampersand replacement=\&]
		0 \arrow[r]
		\arrow[d, equal]
		\&
		\idealj
		\arrow[r, ""]
		\arrow[d, "a"]
		\& 
		\Lambda
		\arrow[r,"\varpi"]
		\arrow[d, "f"]
		\&
		\arrow[d, equal]
		\Lambda/\idealj
		\arrow[r]
		\&
		0
		\arrow[d, equal]
		\\
		0 \arrow[r]
		\&
		\idealk
		\arrow[r, "", swap]
		\& 
		P
		\arrow[r,"{\varpi'}"]
		\&
		\Lambda/\idealj
		\arrow[r]
		\&
		0
	\end{tikzcd}
\]
Let us choose a graded
splitting $s$ of the infinitesimal extension $\varpi$.
This gives us a graded splitting $s'=f \circ s$ of the infinitesimal
extension $\varpi^\prime$.

Performing the construction $t$ of \cref{thm:morphisme_t}, since we
have carefully taken compatible graded splitting, we can draw a
commutative diagram of $\operad $-coalgebras:
\[
	\begin{tikzcd}[ampersand replacement=\&]
		\cobaradj (\Lambda/\idealj)
		\arrow[r, "\tilde t_s"]
		\arrow[d,equal]
		\& \freecog\operad (\s\idealj)
		\arrow[d, "\freecog\operad (\s a)"] \\
		\cobaradj (\Lambda/\idealj)
		\arrow[r, "\tilde t_{s'}", swap]
		\& \freecog\operad (\s\idealk)
	\end{tikzcd}
\]
This is part of the following cube diagram
\begin{center}
\begin{tikzpicture}[baseline= (a).base]
\node[scale=.915] (a) at (0,0){
	\begin{tikzcd}
		&
		\cobaradj P
		\arrow[rr]
		\arrow[ddd]
		&&
		\freecog {\operad } (\dzero \otimes \s\idealk)
		\arrow[ddd]
		\\
		\cobaradj \Lambda
		\arrow[rr, crossing over]
		\arrow[ddd, ""]
		\arrow[ru,"\cobaradj f"]
		&&
		\freecog {\operad } (\dzero \otimes \s\idealj)
		\arrow[ru]
		\\
		\\
		&
		\cobaradj(\Lambda/\idealj)
		\arrow[rr, "\tilde t_{s'}", near start]
		&&
		\freecog {\operad } (\s\idealk)
		\\
		\cobaradj(\Lambda/\idealj)
		\arrow[rr, "\tilde t_s"]
		\arrow[ru, equal]
		&&
		\freecog {\operad } (\s\idealj).
		\arrow[ru,"\freecog\operad (\s a)"']
		\arrow[from=uuu, crossing over]
	\end{tikzcd}
};
\end{tikzpicture}
\end{center}
Notice that both the front face and the back face are both homotopy
fibre product squares. Moreover since $\freecog\operad$ preserves
quasi\=/isomorphisms%
~[\ref{rmk:L_preserve_les_epi}],
the morphisms $\freecog \operad (\s a)$ and $\freecog\operad
(\dzero\otimes \s a)$ are weak equivalences, hence $\cobaradj f$ is also
a weak equivalence.
\end{proof}

\begin{lemma}
\label{thm:cofiltered-alg-equivalence}
 Let us consider a coladder equivalence between
 coladders of complete $\coperad$-algebras $f: A \to A'$.
 Then, the canonical morphism of complete $\coperad$-algebras
 $$
 \varprojlim_n A(n) \to  \varprojlim_n A'(n)
 $$
is a weak equivalence.
\end{lemma}

\begin{proof}
Using the weak 5 lemma%
~[\ref{thm:weak_five_lemma}] one can show by induction that
for any natural number $n$, the morphism of $\operad $-coalgebras
\[
	\cobaradjfull f (n) : 
	\cobaradjfull (A(n))
	\to \cobaradjfull (A'(n)).
\]
is a quasi-isomorphism.
Since the functor $\cobaradjfull$ is right adjoint, it commutes
with limits so we get:
\[
	\cobaradjfull (\limn A(n))
	\isonat \limn
	  \cobaradjfull( A(n)).
\]
Moreover, since for any $n$ the morphism
\[
	\cobaradjfull\left( A(n+1)\right)
	\to \cobaradjfull( A(n) \Lambda)
\]
is a fibration between fibrant objects, we have
\[
	\cobaradjfull (\limn A(n)) 
	\isonat \hlimn \cobaradjfull( A(n))
\]
and the same holds for $A'$. As a consequence,
the morphism
\[
	\cobaradjfull(f) : \cobaradjfull(\limn A(n))
	\longrightarrow \cobaradjfull(\limn A(n))
\]
is the homotopy limit of a weak equivalence of diagrams; it is a weak
equivalence.
\end{proof}

\begin{theorem}%
\label{thm:equivalence_de_devissage=>eq}
Any dévissage equivalence $f: \Lambda \to \Gamma$ between
complete $\coperad$-algebras is a weak equivalence.
\end{theorem}

\begin{proof}
The result is a direct consequence of \cref{thm:cofiltered-alg-equivalence}.
\end{proof}

\subsection{Planar path object}%
%-----------------------------------------------------------------------
\label{sec:planar_path_object}

Assuming that $\operad$ is a planar dg-operad and let
$(\Lambda,a)$ be a complete
$\coperad$\=/algebra, then $[\interval , \Lambda]$ as an object of
$\catofmod {\sinv}(\catss)$,
has a canonical structure of a complete $\coperad$-algebra. This
structure which involves the usual $\uass$-coalgebra structure on
$\interval$ is given by the following composition
\[
	\begin{tikzcd}
		{{[\interval, \Lambda]}^\coperad}
		\arrow[r, equal]
		&{\displaystyle\prod_{n\in\naturals} \left[\coperad (n),
		{[\interval, \Lambda]}^{\otimes n}\right]}
		\arrow[d, "\natabrev"]
		\\
		&{\displaystyle\prod_{n\in\naturals} \left[\interval^{\otimes n},
		\left[\coperad (n), \Lambda^{\otimes n}\right]\right]}
		\arrow[d, "\prod_{n\in\naturals}
		\internalhom{\deltacoprod_\interval^{(n)}} {\idfunctor}"]
		\\
		&{\displaystyle\prod_{n\in\naturals} \left[\interval,
		\left[\coperad (n), \Lambda^{\otimes n}\right]\right]}
		\arrow[d, "\natabrev"]
		\\
		& {\left[\interval, \displaystyle\prod_{n\in\naturals}
		\left[\coperad (n), \Lambda^{\otimes n}\right]\right]}
		\arrow[r,equal]
		& {\left[\interval,\Lambda^{\coperad}\right]}
		\ar[rr, "\internalhom \interval {a_\Lambda}"]
		&& \internalhom \interval \Lambda.
	\end{tikzcd}
\]

Moreover, the functor $[\interval, -]$ commutes with
the radical cofiltration and so $\internalhom \interval \Lambda$ is
complete.

The construction above is functorial and one gets a sequence of
maps
\[
	\Lambda \isonat [\monoidalunit, \Lambda]
	\to [\interval, \Lambda]
	\to [\monoidalunit \oplus \monoidalunit, \Lambda]
	\isonat \Lambda \oplus \Lambda
\]
providing a path object for $\Lambda$ because the map
$[\interval, \Lambda] \to \Lambda \oplus \Lambda$ is a degree\=/wise
epimorphism, hence a fibration%
~[\ref{thm:toutes_les_algebres_completes_sont_fibrantes}]
and $\Lambda \to [\interval,
\Lambda]$ is a dévissage equivalence, hence a weak equivalence%
~[\ref{thm:equivalence_de_devissage=>eq}].

\subsection{Trivial cofibrations in characteristic zero}%
%-----------------------------------------------------------------------
\label{sec:trivial_cofibrations}

The following theorem is true in general but we shall only give a
proof of it in the characteristic zero case.

\begin{theorem}%
\label{thm:trivial_cofibrations}
Any trivial cofibration is a dévissage equivalence.
\end{theorem}

\begin{proof}[Outline of the proof]
The generating trivial cofibrations of complete $\coperad$-algebras
are the image by $\cobarfunctorfull$ of the trivial cofibrations of
$\operad$\=/coalgebras, \ie{} those maps which are
both degree-wise monomorphisms and quasi-isomorphisms.

Since fibrations of complete $\coperad$-algebras are obtained by
right transfer along $\cobaradjfull$, we have
\[
	\fibrationofcompletealg
	= {\left(\cobarfunctorfull\left(\cofibrationofcog \cap
	\weakeqofcog\right)\right)}^\boxslash.
\]
Since $\catofcog\operad$ is combinatorial%
~[\ref{thm:coadmissible}] and $\catofcompletealg
\coperad$ is presentable%
~[\ref{thm:categorie_des_algebres_completes_courbees_est_presentable}],
one can use the small object argument to show
that any arrow having the left lifting property against all
fibrations can be obtained as a retract of a cellular trivial
cofibration, that is one obtained through transfinite composition of
push-outs along generating trivial cofibrations.

Meanwhile, dévissage equivalences are stable under retracts and
transfinite compositions. Hence we need only show that a push-out along
a generating trivial cofibration is a dévissage equivalence, which is
the subject of the next proposition.
\end{proof}

\begin{proposition}
Let $\sigma : W \to V$ be a trivial cofibration of $\operad$-coalgebras
and consider a push-out diagram in the category of complete
$\coperad$-algebras
\[
	\begin{tikzcd}[ampersand replacement=\&]
		\cobarfunctor W
		\arrow[r, "\cobarfunctor \sigma"]
		\arrow[d, "",swap] 
		\arrow[rd, very near end, phantom, "\ulcorner"]
		\& \cobarfunctor V
		\arrow[d, ""] \\
		\Lambda
		\arrow[r, "j", swap]
		\& \Gamma
	\end{tikzcd}
\]
then $j$ is a dévissage equivalence.
\end{proposition}

\begin{proof}[Outline of the proof]
We first operate a reduction to the canonical case.
Let us switch notations
just for this line and denote by
$\cobarfunctor$ the Cobar functor associated to the canonical
twisting morphism and let $\cobarfunctor_\alpha$ be the one associated
to the given twisting morphism $\alpha : \baradj\coperad \to \operad$.
Then one has $\cobarfunctor \circ \alpha^\ast \isonat
\cobarfunctor_\alpha$. Since
$\alpha^\ast$ preserves trivial
cofibrations, one can assume that $\operad = \baradj \coperad$.

Next we can reduce to the case where the map $g : \cobarfunctor
W \to \Lambda$ is a degree-wise epimorphism.
One can factor $g : \cobarfunctor W \to \Lambda$ as
\[
	\begin{tikzcd}[ampersand replacement=\&]
		\cobarfunctor W
		\arrow[r, ""]
		\arrow[d,hook]
		\arrow[rd, very near end, phantom, "\ulcorner"]
		\& \cobarfunctor V 
		\arrow[d, ""] \\
		\cobarfunctor W \amalg \cobarfunctor \cobaradj \Lambda
		\arrow[r, "j^\prime"']
		\arrow[d, two heads, "g \amalg \counitcobaradj(\Lambda)"']
		\arrow[rd, very near end, phantom, "\ulcorner"]
		\& \Gamma'
		\arrow[d] \\
		\Lambda
		\arrow[r, "j", swap]
		\& \Gamma
	\end{tikzcd}
\]
Meanwhile the functor
$\cobarfunctorfull$ preserves push-outs and trivial cofibrations,
moreover in $\catofcog{\baradj\coperad}$
a push-out of a trivial cofibration is again a trivial cofibration, so
$j^\prime$ is also a trivial cofibration.
Since $\counitcobaradj(\Lambda)$ is a degree-wise epimorphism%
~[\ref{thm:cobaradj_est_fidele}],
$g \amalg \counitcobaradj(\Lambda)$ is also a degree-wise epimorphism.

Going back to normal notations, since $\sigma$ is a trivial cofibration
of $\operad$-coalgebras, it
induces a trivial cofibration on the underlying chain complexes,
so one can find a dg\=/retraction $\rho : V \to W$
\[
\begin{tikzcd}
	W \arrow[r,"\sigma"]
	\arrow[rr,bend left,"\id{}"]
	& V \arrow[r,"\rho"]
	& W.
\end{tikzcd}
\]
This induces a degree-wise retraction $r \coloneqq \cobarfunctor \rho$
of $s \coloneqq \cobarfunctor \sigma$ and a degree-wise retraction $q :
\Gamma \to \Lambda$ of $j$.
\[
	\begin{tikzcd}
		\cobarfunctor W \arrow[r,"s"] \arrow[d, two heads]
		\arrow[rr,bend left,"\id{}"]
		\arrow[rd, very near end, phantom, "\ulcorner"]
		& \cobarfunctor V , \arrow[r,"r"] \arrow[d, two heads]
		& \cobarfunctor W\arrow[d, two heads]\\
		\Lambda \arrow[r,"j"] \arrow[rr,bend right,"\id{}"']
		& \Gamma  \arrow[r,"q"]
		&\Lambda.
	\end{tikzcd}
\]
Moreover both $\gr \ast r$ and $\gr \ast q$ are dg\=/retractions.
Finally let us setup a few more notations: let $\ideali$ denote the
ideal of $\cobarfunctor W$ and
let $\idealj$ the ideal of $\cobarfunctor V$ such that
\[
	\Lambda \isonat \cobarfunctor W/\ideali \qand \Gamma \isonat
	\cobarfunctor V/\idealj.
\]
Since $s$ is a degree-wise monomorphism, we shall
identify $\ideali$ with its image $s(\ideali)$ and say that
$\idealj$ is generated by $\ideali$.

Assuming now that one can produce a degree $1$ map $H : \cobarfunctor
V\to \cobarfunctor V$
\begin{itemize}
	\item that stabilises both $\idealj$ and the ideals
	      $\ideal n \cobarfunctor V$;
	\item and such that the map
	      \[
	      	\isophi \coloneqq r \circ s    + \partial H
	      \]
	      is a degree-wise isomorphism of $\cobarfunctor V$ which
	      restricts to degree-wise isomorphisms of both $\idealj$ and
	      the ideal $\ideal n \cobarfunctor V$,
\end{itemize}
then $H$ induces a degree $1$ map on $\Gamma$ that we still denote
by $H$ and which is such that the map
\[
	q \circ j + \partial H
\]
is a degree-wise isomorphism of $\Gamma$ which restricts to isomorphisms
of $\ideal n \Gamma$ for every natural $n$. Thus
$\gr \ast H$ is a degree $1$ map such that
\[
	\gr \ast q \circ \gr \ast j + \partial (\gr \ast H)
\]
is an isomorphism. Besides,
\[
	\gr \ast j \circ \gr \ast q = \id{\gr \ast \Lambda}.
\]
So, $\gr \ast j$ is an homotopy equivalence with homotopical
inverse $\gr \ast q$.
This proves that $j$ is a dévissage equivalence.
\end{proof}

We are now left with the construction of such a homotopy $H$. We shall
decompose the proof in two parts: first we shall assume that the
curved cooperad is planar and we shall then deduce the characteristic
zero case.

Let $K$ be the kernel of $\rho$ in the category of
chain complexes. Since $\rho$ is a weak equivalence
of $\operad $-coalgebras,
the chain complex $K$ is acyclic.
Using the section $\sigma$, one can build a degree $1$ map
$h : V \to V$ such that $\partial h = \projto K$.
This allows us to create a degree $1$ derivation
$\bigdiff_h : V^\coperad \longrightarrow V^\coperad$ relative to the
zero coderivation on $\coperad$ defined as
\[
	\bigdiff_h
	\coloneqq {\sha\left(\id V, h\right)}^\coperad.
\]
\begin{lemma}%
\label{thm:phi_is_an_isomorphism_as_well_as_its_restriction}
The degree $0$ map
\[
	\isophi \coloneqq s \circ r + \partial \bigdiff_h
\]
is a degree-wise isomorphism of $V^\coperad$ which restricts to a
degree-wise
isomorphism of $\ideal n V^\coperad$ for any integer $n$.
\end{lemma}

\begin{proof}
Let $n$ be natural number, since $\partial \bigdiff_h$ is a derivation
relatively to the zero coderivation on $\coperad$%
~[\ref{thm:bracket_derivation_algebras}], it preserves
$\ideal n \cobarfunctor$. As a consequence, the map
$\isophi : \cobarfunctor V \to \cobarfunctor V$ also preserves
$\ideal n \cobarfunctor V$.

By construction of $\bigdiff_h$,
the associated map $\gr n \isophi$ acts on the factor
\[
	{\internalhom{(\grfilt n
	\coperad)(k)}{{(W\oplus K)}^{\otimes k}}}^{\symgroup k}
\]
as
\[
	{\internalhom{(\grfilt n
	\coperad)(k)}{{{(\projto W)}^{\otimes k} + \sha(\id V, \projto
	K)}^{\otimes k}}}^{\symgroup k},
\]
hence it acts locally as an integer multiple of the identity. By the
characteristic zero assumption, $\gr n \isophi$ is an isomorphism.
By immediate induction, $\cofiltration n \isophi$
is a degree-wise isomorphism for any natural $n$ and thus
is $\isophi$. Finally, since the sequence
\[
	\ideal n \cobarfunctor V \to \cobarfunctor V
	\to \cofiltration n \cobarfunctor V
\]
is exact, $\ideal n \phi$ is a degree-wise isomorphism.
\end{proof}

Since the maps $s \circ r$ and $\partial \bigdiff_h$ stabilise
the subobject $\ideali \subset V^\coperad$, the
sum of maps
\[
	{\left(s \circ r \right)}^\coperad
	+ {\sha (\id{}, \partial \bigdiff_h)}^\coperad
\]
from $\power{(V^\coperad)}\coperad$ to itself
may be lifted to a map from
$\power{\sha (V^\coperad , \ideali)}\coperad$ to itself
that we also denote $
\power{(s \circ r)}\coperad
+ \power{\sha (\id{}, \partial \bigdiff_h)}\coperad $, so that the
following diagram commutes
\[
\begin{tikzcd}
	{\sha \left(V^\coperad , \ideali\right)}^\coperad
	\arrow[rrrr,"{\fullpower{s \circ r}\coperad
	+ {\sha(\id{}, \partial \bigdiff_h)}^\coperad}"]
	\arrow[d]
	&&&& {\sha \left(V^\coperad , \ideali\right)}^\coperad
	\arrow[d]
	\\
	\fullpower{V^\coperad}\coperad
	\arrow[rrrr,"{\fullpower{s \circ r}\coperad
	+ {\sha \left(\id{}, \partial \bigdiff_h\right)}^\coperad}"]
	&&&& \fullpower{V^\coperad}\coperad.
\end{tikzcd}
\]

\begin{lemma}%
\label{thm:psi_is_an_isomorphism}
The map $\isopsi$ defined by
\[
\begin{tikzcd}
	{\sha \left(V^\coperad , \ideali\right)}^\coperad
	\arrow[rrr, "{\fullpower{s \circ r}\coperad
	+ {\sha (\id{}, \partial \bigdiff_h)}^\coperad}"]
	&&&
	{\sha \left(V^\coperad , \ideali\right)}^\coperad
	\arrow[d,"\laxmapsha"]
	\\
	&&&{\sha \left(V , \ideali\right)}^{\coperad\compofsymseq \coperad}
	\arrow[d,"\id{}^{w}"]
	\\
	{\sha\left(V, \ideali\right)}^\coperad \ar[rrr, "\isopsi"]
	\ar[uu,hook]
	&&&{\sha \left(V , \ideali\right)}^{\coperad}.
\end{tikzcd}
\]
is a degree-wise isomorphism.
\end{lemma}

\begin{proof}
The map $\isopsi$ stabilises the subobjects
$\power{\sha (V , \ideali)}{\coperad / \filtration n \coperad}$.
For compatibilty reasons, the map $\grfilt n \coperad \to
(\grfilt n \coperad) \compofsymseq \coperad$ induced by $w$ factors
through
$(\grfilt n \coperad) \compofsymseq \monoidalunit$. Meanwhile, on the
generators $V$ the derivation $\partial \bigdiff_h$ is well understood:
one has
\[
	\begin{tikzcd}[ampersand replacement=\&]
		\cobarfunctor V
		\arrow[r, "\partial \bigdiff_h"]
		\& \cobarfunctor V
		\arrow[d, "V^\iota"] \\
		V
		\arrow[u, "V^\tau"]
		\arrow[r, "\projto K", swap]
		\& V.
	\end{tikzcd}
\]
On the object
$\power{\sha (V , \ideali)}{\grfilt n \coperad}$, the
map $\gr n \isopsi$ is then equal to the restriction of the map
\[
	\prod_{k \geq 1} \internalhom{(\grfilt n \coperad)(k)}{
	{(\projto{W^\coperad})}^{\otimes k}
	+ {\sha(\id {V^\coperad}, \projto K)}^{\otimes k}}^{\symgroup k},
\]
from $\power{\sha (\power{V}{\coperad},
\power{V}{\coperad})}{\grfilt n \coperad}$
to itself. Thanks to the characteristic zero assumption, it is an
isomoprhism of $\power{\sha (V , \ideali)}{\grfilt n \coperad}$.
We conclude by direct induction, using the fact that
\[
	{\sha (V , \ideali)}^\coperad
	= \limn {\sha (V , \ideali)}^{\filtration n \coperad}.
\]
\end{proof}

\begin{lemma}%
\label{thm:phi_is_an_isomorphism_of_J}
The map $\isophi$ restricts to an isomorphism of $\idealj$.
\end{lemma}

\begin{proof}
The
ideal $\idealj$ is stable through the map $s \circ r$
and through the structural derivation of the $\coperad$-algebra
$\cobarfunctor V$. Since the restriction
of $\bigdiff_h$ to $\ideali$ is zero, then $\idealj$ is
also stable through $\bigdiff_h$
~[\ref{thm:derivation_sur_ideal_engendre}].
So it is stable through $\isophi$.
Since $\isophi$ is a degree-wise monomorphism, it is also
the case of its restriction to $\idealj$. Let us prove that
this restriction is a degree-wise epimorphism.
Since
$ s \circ r$ is a morphism of
$\coperad$-algebras, and since the map $\partial \bigdiff_h$
is a derivation of the graded $\coperad$-algebra $V^\coperad$
relatively to the zero coderivation on the cooperad $\coperad$%
~[\ref{thm:bracket_derivation_algebras}]
the following square diagram commutes
\[
	\begin{tikzcd}
		{\sha\left(V, \ideali \right)}^{\coperad}
		\arrow[r]
		\arrow[d,"\isopsi"']
		& \idealj \arrow[d,"\isophi"]
		\\
		{\sha \left(V, \ideali \right)}^{\coperad}
		\arrow[r]
		& \idealj,
	\end{tikzcd}
\]
where $\isopsi$ is the map introduced in the previous lemma% 
~[\ref{thm:psi_is_an_isomorphism}].
The bottom horizontal arrow is a degree-wise epimorphism%
~[\ref{thm:other_presentation_of_the_ideal_j}]
and $\isopsi$ is a degree-wise isomorphism
~[\ref{thm:psi_is_an_isomorphism}],
so $\isophi$ is a degree-wise epimorphism.
\end{proof}

\subsection{Functoriality}

Let us consider a morphism of curved conilpotent cooperads
$f$, a morphism of operads $g$ and twisting morphisms $\alpha, \beta$
that make the following square diagram commute
$$
\begin{tikzcd}
\coperad
\ar[r,"f"]\ar[d, "\alpha"']
& \coperad'
\ar[d, "\beta"]
\\
\operad
\ar[r, "g"']
& \operad'.
\end{tikzcd}
$$
Let us endow the category of complete $\coperad$-algebras 
with the model structure transferred from that of $\operad$-coalgebras
and the category of complete $\coperad'$-algebras 
with the model structure transferred from that of $\operad'$-coalgebras.

Let us have a look at the adjunction
$$
\begin{tikzcd}
\catofcompletealg{\coperad'}
\ar[rr, shift left, "\freealg f"]
&& \catofcompletealg{\coperad}
\ar[ll, shift left, "\forget^f"]
\end{tikzcd}
$$

\begin{theorem}
The adjunction $\freealg f \dashv \forget^f$ is a model adjunction.
\end{theorem}

\begin{proof}
 It is clear that the functor $\forget^f$ preserves fibrations and weak equivalences.
\end{proof}

\subsection{A useful property of weak equivalences of conilpotent curved cooperads}

We show that given a morphism of conilpotent curved cooperads
$f: \coperad \to \coperad'$, the adjunction between the categories of algebras 
is closed to be a model equivalence when these categories of algebras
are equipped with the model structure transferred respectively from
that of  $\baradj \coperad$-coalgebras and that of $\baradj \coperad'$-coalgebras.
Actually, we will see that this adjunction is a model adjunction and the result 
we show here is a step towards showing that fact.

\begin{theorem}
\label{thm:pre-model-equivalence-algebras}
Let $f: \coperad \to \coperad'$ be a morphism of conilpotent curved cooperads.
Then, for any $\baradj \coperad'$-coalgebra $V$
the map
$$
\cobarfunctor_{\coperad'} V \to \forget^f \freealg f \cobarfunctor_{\coperad'} V
= \forget^f \cobarfunctor_{\coperad} \forget^{\baradj(f)} V
$$
is a weak equivalence for the model structure on complete $\coperad'$-algebras
transferred from that of $\baradj \coperad'$-coalgebras. 
\end{theorem}

\begin{proof}
 By \cref{thm:algebraic-we-ladder}
 and \cref{thm:model-quasi-equivalence-composition}, we know that such a set of weak equivalences
 (that are called algebraic in \cref{def:algebraic-we} below)
contains all ladder equivalences and follows the 2-out-of-3 rule. Hence
it contains all the weak equivalences.
\end{proof}

\begin{definition}
\label{def:algebraic-we}
Let us call algebraic a weak equivalence of conilpotent curved cooperads
that satisfies the property
described in \cref{thm:pre-model-equivalence-algebras}; that is such a 
weak equivalence $f: \coperad \to \coperad'$ is algebraic
if for any $\baradj \coperad'$-coalgebra $V$
the map
$$
\cobarfunctor_{\coperad'} V \to \forget^f \freealg f \cobarfunctor_{\coperad'} V
= \forget^f \cobarfunctor_{\coperad} \forget_{\baradj(f)} V
$$
is a weak equivalence for the model structure on complete $\coperad'$-algebras
transferred from that of $\baradj \coperad'$-coalgebras. 
\end{definition}

\begin{lemma}
\label{thm:model-quasi-equivalence-composition}
 Algebraic weak equivalences of conilpotent curved cooperads
  follow the 2-out-of-3 rule in the sense that
 for any two composable morphisms $f,g$, if two of the tree maps
 $f$, $g$, $g \circ f$ are algebraic weak equivalences, then so is the third.
\end{lemma}

\begin{proof}
Let us consider two composable morphisms of curved conilpotent cooperads
$$
\coperad \xrightarrow{f} \coperad' \xrightarrow{g} \coperad''
$$
and let $V$ be a $\baradj\coperad''$-coalgebra.
Let us contemplate the following composition.
$$
\cobarfunctor_{\coperad''} V  \xrightarrow{a}
\forget^g \freealg g \cobarfunctor_{\coperad''} V
\xrightarrow{b} \forget^g \forget^{f} \freealg{f} \freealg g \cobarfunctor_{\coperad''} V 
=\forget^{gf} \freealg{gf}  \cobarfunctor_{\coperad'} V  .
$$
If $f$ and $g$ are algebraic weak-equivalence, then $g \circ f$ is a weak equivalence and 
$a$ and $b$ are weak equivalences; so is $b \circ a$. Hence, $g\circ f$ is an algebraic weak equivalence.

If $f$ and $g\circ f$ are algebraic weak-equivalence, then $g$ is a weak equivalence and 
$b$ and $a \circ b$ are weak equivalences; so is $a$. Hence, $g$ is an algebraic weak equivalence.

Let us assume that $g$ and $g\circ f$ are algebraic weak-equivalence.
Then $f$ is a weak equivalence and 
$a$ and $a \circ b$ are weak equivalences; so is $b$. Let $V'$ be a $\baradj \coperad'$-coalgebra.
Since the model adjunction
$\forget_{\baradj(g)} \freecog{\baradj(g)}$
relating $\baradj \coperad'$-coalgebras to $\baradj \coperad'$-coalgebras is a model equivalence
[\cref{bigthm:weq_of_operads_eq_of_models}], there exists a $\baradj \coperad''$-coalgebra $V$ and a chain
of weak equivalences of $\baradj \coperad''$-coalgebras
$$
V' \to Z \leftarrow \forget_{\baradj(g)} V.
$$
Since the functor $\forget^g$ reflects weak equivalences (again because the
the model adjunction
$\forget_{\baradj(g)} \freecog{\baradj(g)}$ is a model equivalence) and since
the map $b$ is a weak equivalence, then the morphism
$$
\freealg g \cobarfunctor_{\coperad''} V
\to \forget^{f} \freealg{f} \freealg g \cobarfunctor_{\coperad''} V 
$$
is a weak equivalence. Let us consider now the following diagram
$$
\begin{tikzcd}
\freealg g \cobarfunctor_{\coperad''} V
\ar[r, equal] \ar[d]
& \cobarfunctor_{\coperad'} \forget_{\baradj(g)} V 
\ar[d] \ar[r]
& \cobarfunctor_{\coperad'} Z
\ar[d]
& \cobarfunctor_{\coperad'} V'
\ar[l] \ar[d]
\\
\forget^{f} \freealg{f}\freealg g \cobarfunctor_{\coperad''} V
\ar[r, equal]
& \forget^{f} \freealg{f}\cobarfunctor_{\coperad'} \forget_{\baradj(g)} V 
\ar[r]
&\forget^{f} \freealg{f} \cobarfunctor_{\coperad'} Z
& \forget^{f} \freealg{f}\cobarfunctor_{\coperad'} V'
\ar[l]
\end{tikzcd}
$$
The horizontal arrows are weak equivalences as well as the left vertical arrow.
By the 2-out-of-3 rule, all these arrows are weak equivalences, in particular the right
vertical one. So $f$ is an algebraic weak equivalence.
\end{proof}

\begin{lemma}
\label{thm:algebraic-we-ladder}
Let $f : (\coperad_n)_{n \in \ordinalomega} \to (\coperad'_n)_{n \in \ordinalomega}$
be a ladder equivalence of a ladder of conilpotent curved cooperads.
Then, the morphism
$$
f : \varinjlim_{n \in \ordinalomega} \coperad_n \to \varinjlim_{n \in \ordinalomega} \coperad'_n
$$
is model quasi-equivalence.
\end{lemma}

\begin{proof}
Let us denote
\begin{align*}
 \coperad &= \varinjlim_{n \in \ordinalomega} \coperad_n;
 \\
  \coperad' &= \varinjlim_{n \in \ordinalomega} \coperad'_n.
\end{align*}
Moreover, for any natural integer $n$, let us denote $i_n$ and $j_n$ respectively the inclusions
\begin{align*}
 i_n :& \coperad_n \to \coperad;
 \\
  j_n :& \coperad_n \to \coperad.
\end{align*}
We can notice that
$$
f \circ i_n = j_n \circ f, \quad n \in \ordinalomega.
$$
Let  $\Lambda$ be a quasi-free complete $\coperad'$-algebra. Its underlying graded algebra has the form
$$
\Lambda \isonat X^{\coperad'}
$$
The map $\Lambda \to  \forget^f \freealg f \Lambda$ fits in
the diagram
$$
\begin{tikzcd}
 \Lambda
 \ar[r]  \ar[d]
 & \cdots 
 \ar[r]  \ar[d]
 & {\forget^{j_n} \freealg{j_n}\Lambda}
 \ar[r]  \ar[d]
 & \cdots
 \ar[r]  \ar[d]
 & {\forget^{j_1} \freealg{j_1}\Lambda}
 \ar[r]  \ar[d]
  & {\forget^{j_0} \freealg{j_0}\Lambda}
   \ar[d]
  \\
   \forget^f \freealg f \Lambda
 \ar[r]  
 & \cdots 
 \ar[r]  
 & {\forget^{fi_n} \freealg{fi_n}\Lambda}
 \ar[r]  
 & \cdots
 \ar[r]  
 & {\forget^{fi_1} \freealg{fi_1}\Lambda}
 \ar[r]  
  & {\forget^{fi_0} \freealg{fi_0}\Lambda}
\end{tikzcd}
$$
For any natural integer $n$, the map
$$
{\forget^{j_n} \freealg{j_n}\Lambda} \to {\forget^{j_{n-1}} \freealg{j_{n-1}}\Lambda}
$$
is an infinitesimal extension since at the level of graded algebras, it is the map
$$
X^{\coperad'_n} \to X^{\coperad'_{n-1}}.
$$
Moreover, the map
$$
\Lambda \to \limn {\forget^{j_n} \freealg{j_n}\Lambda}
$$
is an isomorphism. Thus the first line of the above diagram represent a coladder of
complete $\coperad'$-algebras. The same holds for the second line.

For any natural integer $n$, the kernel $K_n$ of the map ${\forget^{j_n} \freealg{j_n}\Lambda}
\to {\forget^{j_{n-1}} \freealg{j_{n-1}}\Lambda}$ is the chain complex
$$
X^{\coperad'_n / \coperad'_{n-1}}
$$
whose differential is induced from that of $X$ (that is the map
$X \hookrightarrow \Lambda \xrightarrow{d_\Lambda} \Lambda \to X$) combined with that of $\coperad'_n$.
Similarly, the $K'_n$ of the map ${\forget^{fi_n} \freealg{fi_n}\Lambda}
\to {\forget^{fi_{n-1}} \freealg{fi_{n-1}}\Lambda}$ is the chain complex
$$
X^{\coperad_n / \coperad_{n-1}}
$$
whose differential is induced from that of $X$ combined with that of $\coperad'_n$.
Since, $f$ is a ladder equivalence, then the map $K'_n \to K_n$ is a quasi-isomorphism.
So the vertical maps of the above diagram represent a coladder equivalence.
So by \cref{thm:cofiltered-alg-equivalence}, the map
$$
\Lambda \to \forget^f \freealg f \Lambda
$$
is a weak equivalence.
\end{proof}

\section{The Cobar equivalence}%
%=======================================================================
\label{sec:equivalence_cobar}

The goal of this section is to prove the equivalence theorem below.

%bigtheorem
\begin{theorem}%
\label{bigthm:equivalence}%
\label{thm:equivalence_for_cofibrant}
Let $\coperad$ be a curved conilpotent cooperad and let $\alpha : \coperad
\to \baradj \coperad$ the canonical twisting morphism.
For any $\baradj \coperad$-coalgebra $(V,a)$, the unit morphism
$V \to \cobaradj \cobarfunctor V$ is a quasi-isomorphism.
Hence, the model adjunction
\[
	\begin{tikzcd}[ampersand replacement=\&]
		\catofcog{\baradj \coperad}
		\arrow[rr, shift left=1.5,"\cobarfunctorfull"]
		\&
		\&
		\catofcompletealg\coperad,
		\arrow[ll, shift left=1.5, "\cobaradjfull"]
	\end{tikzcd}
\]
is a model equivalence.
\end{theorem}

\subsection{The strategy}%
%-----------------------------------------------------------------------

Our proof follows two steps:
\begin{enumerate}
	\item first prove the result in a context close to
	the planar context and which encompass it, that we call
	the quasi-planar context ; 
	\item in the $\rationals$-linear framework, one
	replaces $\coperad$ by an equivalent
	curved conilpotent cooperad which is quasi-planar.
\end{enumerate}
Let us describe in details these two steps.

\subsubsection{Retraction and decomposition of the unit map}
\label{subsubsection-retraction-decomposition}

Consider a morphism of curved conilpotent cooperads $f: \coperad \to \coperad'$ and the following square
\[
\begin{tikzcd}
	\coperad
	\ar[r, "f"] \ar[d, "\alpha"']
	& \coperad'
	\ar[d, "{\alpha'}"]
	\\
	\baradj \coperad
	\ar[r, "\phi"']
	& \baradj \coperad'
\end{tikzcd}
\]
where $\phi = \baradj (f)$. It induces several adjunctions: for instance
$\cobarfunctor_\alpha \dashv \cobaradj_\alpha$,
$\cobarfunctor_{\alpha'} \dashv \cobaradj_{\alpha'}$,
$\freealg f \dashv \forget^f$ and 
$\forget_\phi \dashv \freecog\phi$.
We know from \cref{proposition:factorisation_unit_morphism} that the map
\[
\forget_g  \xrightarrow{\eta U_g}\cobaradj_\alpha \circ \cobarfunctor_\alpha \circ \forget_g 
\]
induced by the unit of the adjunction $\cobarfunctor_\alpha \dashv \cobaradj_\alpha$ is
equal to the following composite map
\[
\begin{tikzcd}
 	\forget_g 
	\arrow[r,"\forget_g \circ \eta"]
	& \forget_g \circ \cobaradj_{\alpha'} \circ \cobarfunctor_{\alpha'}
	\arrow[rr,"\forget_g \circ \cobaradj_{\alpha'} \circ \eta \circ \cobarfunctor_{\alpha'}"]
	&& \forget_g \circ \cobaradj_{\alpha'} \circ
	\forget^f \circ \freealg f \circ \cobarfunctor_{\alpha'}
	\arrow[d,equal]
	\\
	&&& \forget_g \circ \freecog g \circ \cobaradj_\alpha
	\circ \cobarfunctor_\alpha \circ \forget_g
	\arrow[d,"\epsilon \circ \cobaradj_\alpha
	\circ \cobarfunctor_\alpha \circ \forget_g"']
	\\
	&&& \cobaradj_\alpha \circ \cobarfunctor_\alpha \circ \forget_g .
\end{tikzcd}
\]
We are interested in two cases.
\begin{itemize}
 \item First, the case where $f$ is the morphism
 \[
 	\monoidalunit \to \coperad.
 \]
 Since complete $\monoidalunit$-algebras
 as well as $\baradj \monoidalunit$-coalgebras
 are just chain complexes and since the associated
 cobar adjunction is just the identity adjunction, then,
 applying \cref{proposition:factorisation_unit_morphism}
in this case leads us to factorise the identity of the
underlying chain complex $\forget V$ of $V$ as follows
\[
\begin{tikzcd}
 	\forget V \arrow[r,"(\forget \circ \eta) (V)"]
	\arrow[rr, bend right,"="]
	& \forget \cobaradj \cobarfunctor V 
	\arrow[r,"\specialquotient"]
	& \forget V .
\end{tikzcd}
\]
for any $\baradj \coperad$-coalgebra $V$.
\item Then the case where $f$ is a cofibration of curved
conilpotent cooperads
\[
	\coperad \xrightarrow{f} \coperad'
\]
that is a morphism so that
so that $\phi \coloneqq \baradj f$ is an acyclic cofibration of
dg operads. As so, it has a retraction $\rho : \baradj \coperad' \to
\baradj \coperad$. The morphism $\rho$ induces an adjunction
$\forget_\rho \dashv \freecog\rho$ relating $\baradj \coperad'$-coalgebras
and $\baradj \coperad$-coalgebras. We have
\[
	\forget_\phi \circ \forget_\rho \isonat \idfunctor .
\]
Thus, applying \cref{proposition:factorisation_unit_morphism}
and pre-composing by $\forget_\rho$, one obtains a decomposition of
the unit map
$$
\eta : \idfunctor \to \cobaradj_{\alpha} \circ \cobarfunctor_{\alpha}
$$
as follows
\[
\begin{tikzcd}
 	\idfunctor
	\ar[r, equal]
	& \forget_\phi \circ \forget_\rho
	\ar[r]
	&\forget_\phi \circ \cobaradj_{\alpha'} \circ \cobarfunctor_{\alpha'} \circ \forget_\rho
	\ar[d]
	\\
	&&\forget_\phi \circ \cobaradj_{\alpha'} \circ \forget^f \circ \freealg f \circ  \cobarfunctor_{\alpha'} \circ \forget_\rho
	\ar[d, "\simeq"]
	\\
	&& \forget_\phi \circ \freecog \phi \circ \cobaradj_{\alpha} \circ \cobarfunctor_{\alpha} \circ \forget_\phi \circ   \circ \forget_\rho
	\ar[d]
	\\
	&& \cobaradj_{\alpha} \circ \cobarfunctor_{\alpha} .
\end{tikzcd}
\]
\end{itemize}

\subsubsection{The strategy in the quasi-planar context}%

The quasi-planar step does not
require $\rationals$-linearity. Moreover,
for conciseness purposes,
we shall pretend that $\freecog\specialoperad  V^\coperad$ is
$\freefunctor_1^\specialoperad  V^\coperad$,
the proof for the
other restrictions $\freefunctor_n^\specialoperad  V^\coperad$
for $n>1$
being only much longer.

This quasi-planar context asserts that the conilpotent curved cooperad $\coperad$
so that the underlying graded conilpotent cooperad arises from a planar one
$$
\coperad = \pltosym{\coperad\planar}
$$
 and that the coderivation is tamed in some sense.

In the previous paragraph \ref{subsubsection-retraction-decomposition}
we have factorises the identity of the
underlying chain complex $\forget V$ of $V$ as follows
\[
\begin{tikzcd}
 	\forget V \arrow[r,"(\forget \circ \eta) (V)"]
	\arrow[rr, bend right,"="]
	& \forget \cobaradj \cobarfunctor V 
	\arrow[r,"\specialquotient"]
	& \forget V .
\end{tikzcd}
\]
Let us denote $W= \forget \cobaradj \cobarfunctor V $
and $\overline W$ the kernel
of the map $\specialquotient$.

Proving that
the unit morphism $V \to \cobaradj \cobarfunctor V$
is a quasi-isomorphism amounts to prove that the chain complex
$\overline W$ is acyclic. 
For that purpose, we shall exhibit a degree $1$ map
$\thebighomotopy$
from $\overline W$ to itself
such that the morphism of chain complexes
\[
	\partial H \coloneqq
	\bigdiff \circ H + H \circ \bigdiff
\]
is an isomorphism. Such a contracting homotopy
may be described as follows. The underlying graded object
of $\overline W$ embeds
into a graded object
of the form $V^T$ where $T$ is the made planar trees whose vertices
is labelled by $\sinv \overline{\coperad\planar}$ and some top vertices
are labelled by $\overline{\coperad\planar}$. Then, one can define the degree
$1$ endomorphism $\thehomotopy$ of $T$ that takes a labelled tree $t$ and remove
the $\sinv$ that may appear on the leftest top vertex. Otherwise it sends $t$ to $0$.
Then $\thebighomotopy = - V^\thehomotopy$.

\subsubsection{The quasi-planar approximation step}

In the second step we shall find a quasi-planar
curved conilpotent cooperad $\coperad'$
together with an acyclic cofibration cofibration of cooperads
\[
\begin{tikzcd}
	\coperad \arrow[r, "s"] 
	& \coperad'
\end{tikzcd}
\]
that is a morphism so that $\baradj s$ is
an acyclic cofibration of dg operads.
Let us denote $\specialoperad' \coloneqq \baradj \coperad'$
and $\sigma \coloneqq \baradj s$ and let $\alpha'$ be
the canonical twisting morphism from $\coperad'$
to $\specialoperad'$.

As an acyclic cofibration, the morphism $\sigma$ has a retraction
$\rho$. Using paragraph \ref{subsubsection-retraction-decomposition},
we decompose the unit map
$V \to \cobaradj \cobarfunctor V$
into
$$
V \to \forget_\sigma \cobaradj_{\alpha'} \cobarfunctor_{\alpha'} \forget_\rho (V)
\to \cobaradj  \cobarfunctor (V) .
$$
From the first quasi-planar step, we know that the first map
is a quasi-isomorphism. We then use
of \cref{thm:pre-model-equivalence-algebras} to prove that the second map is also a
quasi-isomorphism.

\subsection{The quasi-planar case}
%-----------------------------------------------------------------------
\label{sec:quasi_planar_case}

We shall first build a suitable contracting homotopy in context closed to
the planar context that we call the quasi-planar
context. It encompasses in particular the planar case.

\subsubsection{The quasi-planar context}

\begin{definition}
The curved conilpotent coperad $\coperad$ is said to be quasi-planar
 if
 \begin{itemize}
 \item the underlying graded conilpotent cooperad of $\coperad$
is the image under the functor $\pltosymfunctor$ of a graded planar conilpotent
cooperad $\coperad\planar$
\[
	\forgetd \coperad
	 \isonat \pltosym{\coperad\planar} , 
\]
where $\forgetd$ denotes the forgetful functor from curved
conilpotent cooperads
to graded conilpotent cooperads ;
\item for any natural integer $n$,
there exists a filtration on $\grfilt n \coperad\planar$
\[
	\filtrationd 0 \grfilt n \coperad\planar \hookrightarrow
	\filtrationd 1 \grfilt n \coperad\planar \hookrightarrow
	\cdots \hookrightarrow
	\filtrationd m \grfilt n \coperad\planar \hookrightarrow
	\cdots
\]
so that the induced filtration on $\grfilt n \coperad
\isonat \pltosym{\left(\grfilt n \coperad\planar\right)}$,
that is
\[
	\filtrationd m \grfilt n \coperad =  \pltosym{\left(
	\filtrationd m \grfilt n \coperad\planar\right)}
\]
is stable under $d_\coperad$ and so that
the resulting degree $-1$ endomophism 
 $\grfiltd m \grfilt n d_\coperad$
 of 
 \[
 	\grfiltd m \grfilt n \coperad \isonat 
	\pltosym{\left(\grfiltd m \grfilt n \coperad\planar \right)}
 \]
 has the form
 \[
 	\grfiltd m \grfilt n d_\coperad
	= d_{\coperad,\plan} \otimes \id{\catofpermutations}.
 \]
\end{itemize}
\end{definition}

\begin{example}
The image under the functor $\pltosymfunctor$
of a curved conilpotent cooperad is obviously
a quasi-planar cooperad.
\end{example}

In this quasi-planar context, the underlying graded operad of $\specialoperad$
has the form
\[
	\forgetd \specialoperad \isonat
	\treemodule \left(\pltosym{\sinv\overline{\coperad}\planar}\right)
	\isonat
	\pltosym{\left( \treemodule\planar \sinv\overline{\coperad}\planar
	\right)} ,
\]
Then, denoting $\specialoperad\planar= \treemodule\planar 
\sinv\overline{\coperad}\planar$,
we have the following canonical isomorphisms of
graded $\catofpermutations$-modules
\begin{align*}
 	\specialoperad \compofsymseq \specialoperad
	\isonat&
	\pltosym{\left( \specialoperad\planar
	\compofplanarseq \specialoperad\planar
	\right)} ;
	\\
	\specialoperad \compofsymseq \coperad\planar
	\isonat&
	\pltosym{\left( \specialoperad\planar
	\compofplanarseq \coperad\planar
	\right)} .
\end{align*}

\begin{notation}\label{notation-tree}
Let $t$ be a non trivial planar tree. We shall order the top vertices of
$t$ from left to right and we shall let $v$ be the leftest top vertex
of $t$.
We shall also denote the number of leaves of $v$ by $b$,
the number of leaves of the tree $t$ before $v$ by $a$, and
the number of leaves of the tree $t$ after $v$ by $c$.
Subsequently $l \coloneqq a+b+c$ is the total number of leaves of $t$:
\begin{center}
	\begin{tikzpicture}
		\tikzstyle{vertex}=[draw, circle, thick, inner sep=0, minimum
		size=4]
		\tikzstyle{every path}=[thick]
		\node [vertex] (a) at (0,0) {};
		\node [vertex] (c) at (0,1) {};
		\node [vertex] (d) at (2,1) {};
		\node [vertex] (e) at ($(d)+(110:1cm)$) {};
		\node [vertex] (f) at ($(e)+(110:1cm)$) {};
		\node [right of=c, node distance=.3cm] {$v$};
		\node at ($(c)+(0,1.4)$) {$b$};
		\node at ($(a)+(140:2.35cm)$) {$a$};
		\node at ($(d)+(80:2.35cm)$) {$c$};
		\draw (a) -- (c);
		\draw (a) -- ($(a)+(130:2cm)$);
		\draw (a) -- ($(a)+(150:2cm)$);
		\draw (a) -- (0,-1);
		\draw (c) -- (0,2);
		\draw (c) -- ($(c)+(60:1cm)$);
		\draw (c) -- ($(c)+(120:1cm)$);
		\draw (d) -- ($(d)+(70:2cm)$);
		\draw (e) -- ($(d)+(0,2)$);
		\draw (e) -- (f);
		\draw (c) -- (0,2);
		\draw (a) -- (d);
		\draw (d) -- (e);
		\draw ($(c)+(55:1.1cm)$) arc (55:125:1.1cm);
		\draw ($(a)+(125:2.1cm)$) arc (125:155:2.1cm);
		\draw ($(d)+(65:2.1cm)$) arc (65:95:2.1cm);
	\end{tikzpicture}
\end{center}
Besides, $t_1$ shall denote the planar tree made out of $t$ by removing
$v$.
\end{notation}

\subsubsection{The operad
\texorpdfstring{$\operad = \baradj\coperad$}{Btop Q}}%
%-----------------------------------------------------------------------
\label{sec:the_operad_BQ_and_the_cobar_resolution}

Let us introduce some
notations about the operad $\specialoperad= \baradj\coperad$. 
\begin{itemize}
	\item The graded operad underlying $\specialoperad$ is the free
	      operad $\treemodule\left(\sinv\coperad\right)$;
	\item Let $\augmentation: \specialoperad \to \monoidalunit$ be
	      the canonical augmentation of the underlying graded operad of
	      $\specialoperad$,
	      and let $\unitspecialoperad : \monoidalunit \to \specialoperad$
	      be its unit;
	\item We denote $\overline{\specialoperad}$ the kernel of $\augmentation$;
	\item The derivation of $\specialoperad $ is the sum of three
	      derivations $\smalldiff_w$,
	      $\smalldiff_{\sinv\overline{\coperad}}$
	      and $\smalldiff_\theta$ defined
	      on the generators $\sinv\overline{\coperad}$ as follows:
	      \begin{align*}
	      	& \smalldiff_w : \sinv\overline{\coperad}
	      	\xrightarrow{\sinv\degreeminusonemap}
	      	\s^{-2}\overline{\coperad}
	      	\xrightarrow{\s^{-2}\overline w_2} \s^{-2}
	      	\treemodule_2\overline{\coperad}
	      	\iso \treemodule_2\left(\sinv\overline{\coperad}\right);
	      	\\
	      	& \smalldiff_{\sinv\overline\coperad} : \sinv\overline
	      	\coperad \xrightarrow{d_{\sinv\overline\coperad}} \sinv
	      	\overline\coperad;
	      	\\
	      	& \smalldiff_\theta: \sinv\overline{\coperad}
	      	\xrightarrow{\degreeminusonemap \otimes \overline \coperad}
	      	\overline{\coperad} \xrightarrow{-\theta}
	      	\monoidalunit .
	      \end{align*}
	\item Via vertexwise linearity, the coradical filtration of $\coperad$
	      induces an exhaustive filtration on the operad
	      $\specialoperad $ that we shall denote
	      $\under{\{\filtration n \specialoperad\}}{n \in \ordinalomega}$.
	      That is, for $n \in \naturals$, 
	      the part $\filtration n \specialoperad$ of the filtration
	      is made up of trees labelled by $\sinv\overline\coperad$
	      such that the cumulated coradical filtration of the labellings
	      does not exceed $n$;
	\item We shall denote by $\canonicaltwistinginverse$ the `inverse’ of
	      $\canonicaltwisting$:
	      \[
	      	\begin{tikzcd}
	      		\canonicaltwistinginverse : \specialoperad
	      		= \treemodule\left(\sinv\overline{\coperad}\right)
	      		\arrow[rr, two heads,
	      		"\picking\left(\sinv\overline \coperad\right)"]
	      		&&
	      		\sinv\overline{\coperad}
	      		\arrow[r, "\degreeminusonemap"]
	      		& \overline{\coperad};
	      	\end{tikzcd}
	      \]
	\item We shall describe $\specialoperad
	      \compofsymseq \coperad = \pltosym{(\specialoperad\planar \compofplanarseq \coperad\planar)}$
	      as obtained by decorating planar trees
	      with either $\sinv\overline\coperad\planar$ or $\overline\coperad\planar$,
	      with the condition that only top vertices can be
	      decorated with $\overline\coperad$. Here is an example of
	      an admissible decoration
	      	      \begin{center}
	      	\begin{tikzpicture}
	      		\tikzstyle{vertex}=[draw, circle, thick, inner sep=0,
	      		minimum size=4]
	      		\tikzstyle{every path}=[thick]
	      		\node [vertex] (a) at (0,0) {};
	      		\node [vertex] (c) at (0,1) {};
	      		\node [vertex] (d) at ($(a)+(45:1cm)$) {};
	      		\node [left of=c, node distance=0.7cm] {$\overline Q(3)$};
	      		\node [left of=a, node distance=1cm]
	      		{$\sinv\overline Q(2)$};
	      		\node [right of=d, node distance=1cm]
	      		{$\sinv\overline Q(1)$};
	      		\draw (a) -- (c);
	      		\draw (a) -- (0,-1);
	      		\draw (c) -- ($(c)+(60:1cm)$);
	      		\draw (c) -- ($(c)+(120:1cm)$);
	      		\draw (c) -- (0,2);
	      		\draw (a) -- (d);
	      		\draw (d) -- ($(d)+(45:1cm)$);
	      	\end{tikzpicture}
	      \end{center}
\end{itemize}

\begin{definition}
 We call operad-cooperad diagram the following arrow diagram in the category of
 $\naturals$-modules in $\catss^\integers$
 \begin{center}
\begin{tikzpicture}
\node (a) at (0,0) {
	\begin{tikzcd}[ampersand replacement=\&]
		\specialoperad\planar\compofplanarseq
		\specialoperad\planar\compofplanarseq\coperad\planar
		\arrow[rr, "m \compofplanarseq \id_{\coperad\planar}"]
		\&\&
		\specialoperad\planar\compofplanarseq\coperad\planar
	\end{tikzcd}
	};
\node[draw = black, rectangle, rounded corners] (b) at (0.27,-0.4) {
$\thediagramox$
};
\node (c) at (5,0) {shortened to};
\node (d) at (7.4,0) {
	\begin{tikzcd}[ampersand replacement=\&]
		\operadcompositiontwo
		\arrow[r]
		\&
		\operadcompositionone .
	\end{tikzcd}
	};
\node[draw = black, rectangle, rounded corners] (f) at (7.5,-0.35) {
$\thediagramox$
};
\end{tikzpicture}
\end{center}
\end{definition}

\subsubsection{The underlying graded coalgebra of the Cobar resolution
\texorpdfstring{$W = \cobaradj\cobarfunctor V$}{Cobar res}}%
%-----------------------------------------------------------------------
\label{sec:the_operad_BQ_and_the_cobar_resolution}

Let $(V,a)$ be a $\specialoperad $-coalgebra. We denote
\begin{itemize}
 \item $\Lambda \coloneqq \cobarfunctor
V$;
\item $W \coloneqq \cobaradj \Lambda = \cobaradj \cobarfunctor
V$.
\end{itemize}
The underlying graded $\specialoperad $\=/coalgebra of $W$ is
	      \[
	      	\freecog\specialoperad  \left(V^\coperad \right) \subobject
	      	{\left(V^\coperad\right)}^\specialoperad 
	      	\subobject V^{\specialoperad  \compofsymseq \coperad};
	      \]
It may be presented as the pullback in $\catss^\integers$ of a
diagram called the restriction-extension diagram, that we
describe below.

\begin{definition}
The restriction-extension diagram $\thediagramre$ is the following diagram
in $\catss^\integers$
 \begin{center}
\begin{tikzpicture}
\node (a) at (0,0) {
	\begin{tikzcd}[ampersand replacement=\&]
		\freecog\specialoperad V^\coperad \arrow[r] \arrow[d]
		\& \left(\power{\left(V^{\coperad}\right)}{\specialoperad} 
		\right)^\specialoperad
		\arrow[dd]
		\\
		\power{\left(V^{\coperad}\right)}{\specialoperad} 
		\arrow[d]
		\\
		\power{V}{\specialoperad\compofsymseq\coperad}.
		\arrow[r]
		\&
		\power{V}{\specialoperad\compofsymseq\specialoperad\compofsymseq\coperad}.
	\end{tikzcd}
	};
\node[draw = black, rectangle, rounded corners] (b) at (-0.2,-0.3) {
$\thediagramre$
};
\node (c) at (3.5,0) {shortened to};
\node (d) at (6.25,0) {
	\begin{tikzcd}[ampersand replacement=\&]
		W \ar[d] \ar[r]
		\&
		\restrictiontwo
		\arrow[dd]
		\\
		\restrictionone
		\arrow[d]
		\\
		\extensionone
		\arrow[r]
		\&
		\extensiontwo
	\end{tikzcd}
	};
\node[draw = black, rectangle, rounded corners] (f) at (6.3,-0.2) {
$\thediagramre$
};
\end{tikzpicture}
\end{center}
\end{definition}

\begin{definition}
The extension diagram $\thediagramex$ is the subdiagram of the
restriction-extension diagram that contains the arrow relating
$\extensionone$ to $\extensiontwo$:
\begin{center}
\begin{tikzpicture}
\node (a) at (0,0) {
	\begin{tikzcd}[ampersand replacement=\&]
		\power{V}{\specialoperad\compofsymseq\coperad}
		\arrow[r]
		\&
		\power{V}{\specialoperad\compofsymseq\specialoperad\compofsymseq\coperad}
	\end{tikzcd}
	};
\node[draw = black, rectangle, rounded corners] (b) at (-0.15,-0.4) {
$\thediagramex$
};
\node (c) at (3.5,0) {shortened to};
\node (d) at (6.25,0) {
	\begin{tikzcd}[ampersand replacement=\&]
		\extensionone
		\arrow[r]
		\&
		\extensiontwo .
	\end{tikzcd}
	};
\node[draw = black, rectangle, rounded corners] (f) at (6.3,-0.33) {
$\thediagramex$
};
\end{tikzpicture}
\end{center}
This is also the image of the 
operad-cooperad diagram through the operation $V^{(-)}$.
\end{definition}

\begin{lemma}
 The diagram $\thediagramre$ is a pullback.
\end{lemma}

\begin{proof}
 Let us consider the following diagram
 $$
\begin{tikzcd}
 \freecog \specialoperad V^\coperad
 \ar[r]  \ar[d]
 & \left(\left(V^\coperad\right)^\specialoperad\right)^\specialoperad
 \ar[d] 
 \\ \left(V^\coperad\right)^\specialoperad
  \ar[r]  \ar[d]
 & \left(V^\coperad\right)^{\specialoperad\compofsymseq\specialoperad}
  \ar[d]
 \\
 V^{\specialoperad\compofsymseq\coperad}
  \ar[r]
 & V^{\specialoperad\compofsymseq\specialoperad\compofsymseq\coperad}
\end{tikzcd}
 $$
 The two small squares are pullbacks. Hence, the outer square is also a pullback. This proves the result.
\end{proof}

As we are in the quasi-planar context, we have the following isomorphisms
in $\catss^\integers$
\begin{align*}
	\Lambda =& V^\coperad = V^{\coperad\planar} 
	= \prod_n \left[\coperad\planar(n) ,
	V^{\otimes n} \right];
	\\
 	\extensionone =& V^{\operad \compofsymseq \coperad} =
	\prod_n \left[\specialoperad\planar
	\compofplanarseq \coperad\planar(n) ,
	V^{\otimes n} \right];
	\\
	\extensiontwo =& V^{\operad \compofsymseq\operad \compofsymseq \coperad} =
	\prod_n \left[\specialoperad\planar
	\compofplanarseq \specialoperad\planar
	\compofplanarseq \coperad\planar(n) ,
	V^{\otimes n} \right];
	\\
	\restrictionone =& \left(V^{\coperad}\right)^{\operad} 
	= \left(\Lambda\right)^{\operad} = \left(\Lambda\right)^{\operad\planar} ;
	\\
	\restrictiontwo =& \left(\left(V^{\coperad}\right)^{\operad}\right)^{\operad}
	= \restrictionone^{\operad} = \restrictionone^{\operad\planar}.
\end{align*}
Notice that the symmetric groups $\catofpermutations_n$
play no role in these formulas.

Let us described some maps and objects related to the diagram $\thediagramre$
and $V$.
\begin{itemize}
\item The diagram $\thediagramre$ has a cocone $\specialquotient : \thediagramre \longrightarrow V$
targeting $V$ resulting from the unit
$\unitspecialoperad$
of the operad $\specialoperad$ and from the coaugmentation $\iota$ of
the cooperad $\coperad$:
$$
\extensiontwo = V^{\specialoperad \compofsymseq\specialoperad\compofsymseq \coperad}
\xrightarrow{V^{\unitspecialoperad\compofsymseq\unitspecialoperad\compofsymseq\iota}}
V .
$$
This cocone actually extends the map 
$\specialquotient : \forget \cobaradj \cobarfunctor V \to \forget V$ described in Paragraph
\ref{subsubsection-retraction-decomposition}.
\item The diagram $\thediagramre$ has also a cone $\pseudounitcobaradj V: V \to \thediagramre$
given by the graded map
	      \[
	      	\pseudounitcobaradj V: 
	      	V \xrightarrow{V^\tau} V^\coperad
	      	\xrightarrow{\freecog \augmentation V^\coperad}
	      	\freecog\specialoperad  V^\coperad.
	      \]
Composing the cocone $\specialquotient$
with the cone $\pseudounitcobaradj V$, one obtain an endomorphism
of the restriction-extension diagram
\[
	\projto V : \thediagramre \to \thediagramre, \quad \projto V = \pseudounitcobaradj \circ \specialquotient .
\]
\item Let $\overline \Lambda = V^{\overline \coperad}$ be the kernel of the map
$$
\Lambda = V^{\coperad} \xrightarrow{V^\iota} V
$$
 and let $\overline{\thediagramre}$ be the kernel of
	      \[
	      	\specialquotient :
	      	\thediagramre \to V.
	      \]
The components of the diagram $\overline{\thediagramre}$ are 
 \begin{align*}
 	\overline{\extensionone} =& V^{\overline{\operad \compofsymseq \coperad } } =
	V^{(\overline{\operad} \compofsymseq \coperad ) \oplus \overline{\coperad} }
	= V^{\overline{\operad} \compofsymseq \coperad } \oplus
	\overline \Lambda;
	\\
	\overline{\extensiontwo} =& V^{\overline{\operad \compofsymseq\operad \compofsymseq \coperad}} =
	V^{(\overline{\operad} \compofsymseq\operad \compofsymseq \coperad) \oplus
	(\overline{\operad} \compofsymseq \coperad)  \oplus \overline{\coperad} }
	= V^{\overline{\operad} \compofsymseq\operad \compofsymseq \coperad} \oplus \overline{\extensionone};
	\\
	\overline{\restrictionone} =& \left(V^{\coperad}\right)^{\overline\operad} \oplus 
	\overline \Lambda;
	\\
	\overline{\restrictiontwo} =& \left(\left(V^{\coperad}\right)^{\operad}\right)^{\overline\operad}
	\oplus \overline{\restrictionone}.
\end{align*}
\end{itemize}
	      
\subsubsection{The coderivation of the Cobar resolution
\texorpdfstring{$W = \cobaradj\cobarfunctor V$}{Cobar res}}%
%-----------------------------------------------------------------------

Let us describe the coderivation $\bigdiff$ of
	      $\cobaradj \cobarfunctor V$.
	      It decomposes into the following six degree
	      $-1$ maps from $V^{\specialoperad \compofsymseq \coperad}$ to
	      itself which actually induce coderivations on
	      $\freecog\specialoperad (V^\coperad )$.
	      
\begin{definition}
Let $\smalldiffex$ be the degree $-1$ endomorphism of
the operad-cooperad diagram $\thediagramox$ whose component
on $\operadcompositionone$ is
the composition
\[
	      	      	\begin{tikzcd}
	      	      		\specialoperad\planar \compofplanarseq \coperad\planar
	      	      		\arrow[d,"{\id_{\specialoperad\planar}
	      	      		\compofplanarseq \sha\left(\unitspecialoperad
	      	      		\compofplanarseq \id_{\coperad\planar}
	      	      		,(\canonicaltwisting \compofplanarseq \id_{\coperad\planar})
	      	      		\circ w\right)}"]
	      	      		\\
	      	      		{\specialoperad\planar  \compofplanarseq \specialoperad\planar
	      	      		\compofplanarseq \coperad\planar}
	      	      		\arrow[d,"{{m \compofplanarseq \id_{\coperad\planar}}}"]
				\\
	      	      		{\specialoperad\planar\compofplanarseq
	      	      		\coperad\planar}.
	      	      	\end{tikzcd}
\]
and so that
$$
\smalldiffex(\operadcompositiontwo)
= \id_{\specialoperad\planar} \compofplanarseq
\sha\left(\unitspecialoperad
	      	      		\compofplanarseq \id_{\coperad\planar}
	      	      		,\smalldiffex(\operadcompositionone)\right) .
$$
Let $\bigdiffex$ be the degree $-1$ endomorphism of the extension diagram
$\thediagramex$
 given by the formula
	      	      \[
	      	      	\bigdiffex(\extensionone) = V^{\smalldiffex}.
	      	      \]
\end{definition}

The map $\bigdiffex$ restricts to $W$ into the degree $-1$ coderivation
whose projection onto $V$ is the composition
$$
\freecog \specialoperad \Lambda \hookrightarrow \Lambda^\specialoperad
\xrightarrow{\Lambda^\alpha} \Lambda^\coperad 
\xrightarrow{a_\Lambda} \Lambda .
$$
However, it does not restrict in general to the whole restriction extension diagram.

\begin{definition}
Let $\bigdiffa$ be the degree $-1$ endomorphism of the restriction-extension diagram
$\thediagramre$ that results from the
$\specialoperad$-coalgebra structure $a$ on $V$ in the sense
that $\bigdiffa(\extensionone)$ is the following decomposition
	      	      \[
	      	      	\begin{tikzcd}
	      	      		{\left(V^\specialoperad \right)}^{\specialoperad
	      	      		\compofsymseq \coperad}
	      	      		\arrow[r, "\laxmap"]
	      	      		&
	      	      		V^{\specialoperad  \compofsymseq \coperad
	      	      		\compofsymseq \specialoperad }
	      	      		\arrow[rrr,"{V^{(\id_{\specialoperad \compofsymseq
	      	      		\coperad})
	      	      		\compofsymseq \sha(\unitspecialoperad\circ\tau,
	      	      		\canonicaltwisting)}}"]
	      	      		&&&
	      	      		V^{\specialoperad  \compofsymseq
	      	      		\coperad \compofsymseq \coperad}
	      	      		\arrow[d,"{V^{\id_{\specialoperad}\compofsymseq w}}"]
	      	      		\\
	      	      		V^{\specialoperad  \compofsymseq \coperad}
	      	      		\arrow[u,"{-a^{\specialoperad
	      	      		\compofsymseq \coperad}}"]
	      	      		\ar[rrrr, "\bigdiffa"]
	      	      		&&&& V^{\specialoperad  \compofsymseq \coperad};
	      	      	\end{tikzcd}
	      	      \]
 The map $\bigdiffa(\extensiontwo)$ is defined similarly. Moreover, the restrictions are
 by the following formulas
\begin{align*}
 	\bigdiffa (\restrictionone) &=
	      	-{\sha\left(\idfunctor,
	      	a_{V^\coperad} \circ 
	      	{\sha\left(V^\tau,V^\canonicaltwisting \circ a_V
	      	\right)}^\coperad\right)}^{\specialoperad};
		\\
	\bigdiffa (\restrictiontwo) &=
	      	-{\sha\left(\idfunctor,
	      	\bigdiffa (\restrictionone)\right)}^{\specialoperad}.
\end{align*}
\end{definition}

\begin{definition}
Let us $\bigdiffw$, $\bigdifftheta$,
	      	      $\bigdiffcop$ and $\bigdiffv$ 
		      the degree $-1$ endomorphisms
		      of the restriction-extension diagram
		      that
		      result respectively from the
	      	      derivations $\smalldiff_w$ and $\smalldiff_\theta$ of
	      	      the operad $\specialoperad$, the coderivation
	      	      $d_\coperad$ of the cooperad $\coperad$ and the
	      	      coderivation $d_V$ of the $\specialoperad$-coalgebra $V$
	      	      \begin{align*}
	      	      	\bigdiffw(\extensionone) &\coloneqq
	      	      	- V^{\smalldiff_w\compofsymseq\idfunctor};
	      	      	\\
			\bigdiffw (\extensiontwo) &\coloneqq
	      	         - \fullpower{V^\coperad}{\smalldiff_w \compofsymseq \idfunctor \compofsymseq \idfunctor
	      		+ (\idfunctor \compofsymseq' \smalldiff_w) \compofsymseq \idfunctor }  ;
			\\
	      	      	\bigdifftheta(\extensionone) &\coloneqq
	      	      	- V^{\smalldiff_\theta \compofsymseq
	      	      	\coperad};
			\\
			\bigdifftheta(\extensiontwo) &\coloneqq
	      	      	- V^{\smalldiff_\theta \compofsymseq
	      	      	\idfunctor_\specialoperad \compofsymseq
	      	      	\idfunctor_\coperad + (\idfunctor \compofsymseq' \smalldiff_\theta) \compofsymseq \idfunctor};
			\\
	      	      	\bigdiffcop(\extensionone) &\coloneqq
	      	      	-V^{\smalldiff_{\sinv\overline{\coperad}}
	      	      	\compofsymseq
	      	      	\id_\coperad
	      	      	+ \id_\specialoperad  \compofsymseq' d_\coperad};
	      	      	\\
			\bigdiffcop(\extensiontwo) &\coloneqq
	      	      	-V^{\smalldiff_{\sinv\overline{\coperad}}
	      	      	\compofsymseq
	      	      	\id_\specialoperad
			\compofsymseq
	      	      	\id_\coperad
			+ (\id_\specialoperad  \compofsymseq'\smalldiff_{\sinv\overline{\coperad}})
			\compofsymseq
	      	      	\id_\coperad
	      	      	+ \id_{\specialoperad \compofsymseq  \specialoperad}  \compofsymseq' d_\coperad};
			\\
	      	      	\bigdiffv(\extensionone) &\coloneqq
	      	      	{\sha(\id V, d_V)}^{\specialoperad
	      	      	\compofsymseq \coperad}
			\\
			\bigdiffv(\extensiontwo) &\coloneqq
	      	      	{\sha(\id V, d_V)}^{\specialoperad
	      	      	\compofsymseq
			\specialoperad
	      	      	\compofsymseq \coperad}.
	      	      \end{align*}
	      	      Note that the resultant of the derivation
	      	      $\smalldiff_{\sinv\coperad}$
	      	      actually appears in the coderivation $\bigdiffcop$.
\end{definition}

\begin{definition}
Let $\bigdiffin$ be the degree $-1$ endomorphism of the restriction-extension diagram
that is the sum of the previous maps except $\bigdiffex$
	      $$
	      \bigdiffin \coloneqq  \bigdiffw + \bigdiffcop +
	      	\bigdiffv +\bigdiffa + \bigdifftheta.
	      $$
\end{definition}

\begin{lemma}
 The coderivation of $\cobaradj \cobarfunctor V$ is the restriction
 to $\freecog\specialoperad  (V^\coperad )$ of the map
	      \[
	      	\bigdiff \coloneqq \bigdiffex + \bigdiffin (\extensionone).
	      \]
\end{lemma}

\begin{proof}
This follows from the definitions of the functors
$\cobarfunctorfull$ and $\cobaradjfull$.
\end{proof}

\subsubsection{The map \texorpdfstring{$\thebighomotopy$}{H}}%
\label{section:definition_de_H_planaire}

The sequence $\operadcompositionone = \specialoperad\planar
\compofsymseq\planar \coperad\planar$ is made up of planar trees
decorated
by $\sinv\overline \coperad\planar$ and $\overline\coperad\planar$ so that
non top vertices are decorated by $\sinv\overline \coperad\planar$.
For such a tree $t$, let $\thehomotopy_t$ be the degree $1$ operation defined as
follows: it is zero on the trivial tree and
\begin{itemize}
	\item if the leftest top vertex is labelled by
	      $\overline{\coperad}\planar$, then $\thehomotopy(\operadcompositionone)$ is zero;
	\item otherwise, $\thehomotopy(\operadcompositionone)$ applies $\canonicaltwistinginverse$
	      to the leftest top vertex:
\end{itemize}
\begin{center}
	\begin{tikzpicture}
		\tikzstyle{vertex}=[draw, circle, thick, inner sep=0,
		minimum size=4]
		\tikzstyle{every path}=[thick]
		\node [vertex] (a) at (0,0) {};
		\node [vertex] (c) at (0,1) {};
		\node [vertex] (d) at ($(a)+(45:1cm)$) {};
		\node [left of=c, node distance=1cm] {$\sinv\overline \coperad\planar(3)$};
		\node [left of=a, node distance=1cm] {$\sinv\overline \coperad\planar(2)$};
		\node [right of=d, node distance=0.7cm] {$\overline \coperad\planar(1)$};
		\node at (3,0.8) {$\thehomotopy(\operadcompositionone)$};
		\draw (a) -- (c);
		\draw (a) -- (0,-1);
		\draw (c) -- ($(c)+(60:1cm)$);
		\draw (c) -- ($(c)+(120:1cm)$);
		\draw (c) -- (0,2);
		\draw (a) -- (d);
		\draw (d) -- ($(d)+(45:1cm)$);
		\draw[|->] (2.5,0.5) -- (3.5,0.5);
		\node [vertex] (a2) at (6,0) {};
		\node [vertex] (c2) at (6,1) {};
		\node [vertex] (d2) at ($(a2)+(45:1cm)$) {};
		\node [left of=c2, node distance=0.7cm] {$\overline \coperad\planar(3)$};
		\node [left of=a2, node distance=1cm] {$\sinv\overline \coperad\planar(2)$};
		\node [right of=d2, node distance=0.7cm] {$\overline \coperad\planar(1)$};
		\draw (a2) -- (c2);
		\draw (a2) -- (6,-1);
		\draw (c2) -- ($(c2)+(60:1cm)$);
		\draw (c2) -- ($(c2)+(120:1cm)$);
		\draw (c2) -- (6,2);
		\draw (a2) -- (d2);
		\draw (d2) -- ($(d2)+(45:1cm)$);
	\end{tikzpicture}
\end{center}
Let us define the map $\thehomotopy_t$ in a way more compatible with
the cotensor.

\begin{definition}
 Let $t$ be a planar non trivial tree as in \cref{notation-tree}.
 We define $\thehomotopy_t$ as the degree $1$ morphism
 of $\naturals$-modules in $\catss^{\integers}$ defined in arity $n$ as
\[
	\begin{tikzcd}
		t\left(\sinv \overline\coperad\planar\right)
		\otimes \coperad\planar^{\convolution l} (n)
		\arrow[d, equal]
		\\
		{t_1\left(\sinv \overline\coperad\planar\right) \otimes 
		\left(\monoidalunit^{\otimes a} \otimes \sinv \overline\coperad\planar(b)
		\otimes \monoidalunit^{\otimes c}\right) \otimes 
		\coperad\planar^{\convolution a+b+c}}(n)
		\arrow[d,"{\id{} \otimes \left(\id{}^{\otimes a} \otimes
		\canonicaltwistinginverse
		\otimes \id{}^{\otimes c}\right) \otimes \left(\tau^{\convolution a}
		\convolution \tau^{\convolution b} \convolution
		\idfunctor^{\convolution c}\right)}"]
		\\
		t_1\left(\sinv \overline\coperad\planar\right) \otimes 
		\left(\monoidalunit^{\otimes a} \otimes \overline\coperad\planar(b)
		\otimes \monoidalunit^{\otimes c}\right) \otimes 
		\left(\monoidalunit^{\otimes a} \otimes
		\monoidalunit^{\otimes b} \otimes 
		\coperad\planar^{\convolution c}(n-a-b)\right)
		\arrow[d, equal]
		\\
		t_1\left(\sinv \overline\coperad\planar\right) \otimes 
		\left(\monoidalunit^{\otimes a} \otimes
		\overline\coperad\planar(b) \otimes 
		\coperad\planar^{\convolution c}(n-a-b)\right)
		\arrow[d, hook]
		\\
		t_1\left(\sinv \overline\coperad\planar\right) \otimes 
		\coperad\planar^{\convolution a+ 1 + c}(n).
	\end{tikzcd}
\]
The map $\thehomotopy_t$ is defined to be zero if $t$ is the trivial tree.
\end{definition}

\begin{definition}
 Let $\thehomotopy$ be the degree $1$ endomorphism of the operad-cooperad diagram 
defined as follows. 
\begin{itemize}
 \item On the one hand, $\operadcompositionone$ may be decomposed as the sum
 $$
\operadcompositionone =  \specialoperad\planar \compofplanarseq \specialoperad\planar
 \compofplanarseq \coperad\planar
 =
 \bigoplus_{t} t\left(\sinv \overline\coperad\planar\right) \compofplanarseq \coperad\planar
 $$
 where the sum is taken over the equivalence classes of planar trees $t$. Then,
 $\thehomotopy(\operadcompositionone)$ is the sum of the maps $h_t$
 $$
\thehomotopy(\operadcompositionone) = \bigoplus_t h_t.
 $$
 \item On the other hand, $\operadcompositiontwo$ may be decomposed as the sum
 $$
 \specialoperad\planar \compofplanarseq \specialoperad\planar
 \compofplanarseq \coperad\planar
 =
 \bigoplus_{t' \subset t} t\left(\sinv \overline\coperad\planar\right) \compofplanarseq \coperad\planar
 $$
 where the sum is taken pairs $(t', t)$ of a planar tree $t$ and a subtree $t'$ that contains the root edge.
 For such a pair, let us denote $t'_1$ the intersection of $t'$ and $t_1$ in $t$, that is the subtree of $t$
 whose edges are those that belong to both $t'$ and $t_1$.
 Then, $\thehomotopy(\operadcompositiontwo)$ is the map whose restriction to the summand $t' \subset t$
 is the composition
 $$
 t\left(\sinv \overline\coperad\planar\right) \compofplanarseq \coperad\planar \xrightarrow{h_t}
  t_1\left(\sinv \overline\coperad\planar\right) \compofplanarseq \coperad\planar
  \hookrightarrow 
   \specialoperad\planar \compofplanarseq \specialoperad\planar
 \compofplanarseq \coperad\planar
 $$
 whose second part is the inclusion into the $t_1' \subset t_1$ summand.
\end{itemize}
\end{definition}

\begin{definition}
Let $\thebighomotopy$ be the degree $1$ endomorphism
of the extension diagram defined by the formula
$$
\thebighomotopy \coloneqq - V^\thehomotopy.
$$
\end{definition}

\begin{proposition}
 The map $\thebighomotopy$ extends uniquely to the whole restriction-extension diagram
 $\thediagramre$.
\end{proposition}

\begin{proof}
 This is equivalent to the fact that $-\thebighomotopy$ extends
 to the restriction-extension diagram, which is equivalent to the fact that
 $-\thebighomotopy (\extensionone)$ restrict to the subobject
$\restrictionone$ and the fact that
$-\thebighomotopy (\extensiontwo)$ restrict to the subobject $\restrictiontwo$.

On the one hand, one can check that $-\thebighomotopy (\extensionone)$
restricts to $\restrictionone$ into the degree $1$ endomorphism of
$\restrictionone$ whose projection on
$[ t(\sinv\overline\coperad), \power{(V^\coperad)}{\otimes l} ]$
is given by the composition
\[
	\begin{tikzcd}
		{\left(V^\coperad\right)}^\specialoperad  \arrow[d, two heads]
		\\
		{\left[t_1\left(\sinv\overline{\coperad}\planar\right),
		{\left(V^\coperad\right)}^{\otimes a+1+c}\right]}
		\arrow[d,"{\left[\id{} , \fullpower{V^\iota}{\otimes a}
		\otimes V^\canonicaltwistinginverse
		\otimes \id{}^{\otimes c}\right]}"]
		\\
		{\left[t_1\left(\sinv\overline{\coperad}\planar\right),
		V^{\otimes a}\otimes
		V^{\sinv \overline{\coperad}}
		\otimes \fullpower{V^\coperad}{\otimes c} \right]}
		\arrow[d, two heads]
		\\
		{\left[t_1\left(\sinv\overline{\coperad}\planar\right),
		V^{\otimes a}\otimes
		\left[\sinv \overline{\coperad}\planar(b),V^{\otimes b}\right]
		\otimes \fullpower{V^\coperad}{\otimes c} \right]}
		\arrow[d, hook]
		\\
		{\left[t_1\left(\sinv\overline{\coperad}\planar\right)
		\otimes \left(\monoidalunit^{\otimes a} \otimes
		{\sinv \overline{\coperad}\planar}(b) \otimes
		\monoidalunit^{\otimes c}\right) ,
		V^{\otimes a+b} \otimes \fullpower{V^\coperad}{\otimes c}\right]}
		\arrow[d,equal]
		\\
		{\left[t\left(\sinv\overline{\coperad}\planar\right),
		V^{\otimes a+b} \otimes \fullpower{V^\coperad}{\otimes c}\right]}
		\arrow[d, hook]
		\\
		{\left[t\left(\sinv\overline{\coperad}\planar\right),
		\fullpower{V^\coperad}{\otimes l} \right]}.
	\end{tikzcd}
\]

On the other hand, one can check that $-\thebighomotopy (\extensiontwo)$
restricts to $\restrictiontwo$ into the degree $1$ endomorphism of
$\restrictiontwo$ whose projection on $[ t(\sinv\overline{\coperad}\planar),
\power{(\power{(V^\coperad)}\specialoperad)}{\otimes l} ]$
is the sum of the map
\[
	- \left[\id{}, \sum_{k<a+b} \projto{V}^{\otimes k} \otimes 
	\thebighomotopy (\restrictionone)
	\otimes \id{}^{\otimes l- k-1} \right],
\]
together with the same composite map, mutatis mutandis, as the one that
defines $\thebighomotopy (\restrictionone)$ just above in this proof.
More precisely, to build this second component one only needs to
change in the composition
\begin{itemize}
	\item $V^\coperad$ by $\restrictionone$;
	\item $V^\iota$ by $\fullpower{V^\iota}\unitspecialoperad$;
	\item $V^\canonicaltwistinginverse$ by
	      $V^\canonicaltwistinginverse \circ \fullpower{V^\coperad}
	      \unitspecialoperad$.
\end{itemize}
\end{proof}

\subsubsection{The garbage map
\texorpdfstring{$\bigmapg$}{the map G}}

Although the degree $-1$ endomorphism of the extension
diagram $\thediagramex$ 
may not be extended to the diagram
$\thediagramre$,
the boundary map
\[
	\bigdiffex \circ \thebighomotopy +
	\thebighomotopy \circ \bigdiffex
\]
does. We actually, show that such a map is the sum
$\id - \projto V + \bigmapg$ where $\bigmapg$ is a endomorphism of
the extension diagram
that extends to the whole restriction-extension diagram.

\begin{definition}
 Let $\mapg$ be the endomorphism of the operad-composition
 diagram defined by the formula
\[
	\mapg = \smalldiffex \circ \thehomotopy +
\thehomotopy\circ\smalldiffex - \idfunctor + \projto {\monoidalunit} .
\]
Moreover, let $\bigmapg$ be the endomorphism of the extension diagram
defined by the formula
$$
\bigmapg \coloneqq V^\mapg  .
$$
We have in particular
\begin{align*}
 \bigmapg 
 &= V^{\smalldiffex \circ \thehomotopy +
\thehomotopy\circ\smalldiffex - \idfunctor + \projto {\monoidalunit}}
 \\
 &= V^{\smalldiffex \circ \thehomotopy +
\thehomotopy\circ\smalldiffex} - \idfunctor + \projto {V}
\\
&= - V^{\thehomotopy} \circ V^{\smalldiffex} - V^{\smalldiffex} \circ V^{\thehomotopy}
- \idfunctor + \projto {V}
\\
&= \bigdiffex \circ \thebighomotopy +
	\thebighomotopy \circ \bigdiffex - \idfunctor + \projto {V} .
\end{align*}
\end{definition}

Let us unfold what is this map $\mapg$. This is the degree $0$ map
defined on a planar tree decorated by $\sinv\overline\coperad\planar$
and $\overline\coperad\planar$ as follows: it is zero on the trivial
tree and
\begin{itemize}
	\item if the leftest top vertex is decorated by
	      $\overline\coperad\planar$ then $\mapg$ is zero;
	\item otherwise, $\mapg$ applies
	      \[
	      	\sinv\overline w : \sinv\overline\coperad\planar
	      	\longrightarrow \sinv\left(\overline\coperad\planar
	      	\compofplanarseq
	      	\coperad\planar\right) \iso \left(\sinv\overline\coperad
	      	\planar\right)\compofplanarseq \coperad\planar
	      \]
	      to that leftest top vertex.
\end{itemize}
Let us describe $\mapg$ is a more precise way.

\begin{definition}
 Let $t$ be a non trivial planar tree as in \cref{notation-tree}.
 Let $\mapg_t$ be the degree $0$ map of
 $\naturals$-modules defined in arity $n$ by
 the composition
% Some manual spacing has been done below
\[
	\begin{tikzcd}
		t\left(\sinv \overline\coperad\planar\right)
		\otimes \coperad\planar^{\convolution l} (n)
		\arrow[d, equal]
		\\
		{t_1\left(\sinv \overline\coperad\planar\right) \otimes 
		\left(\monoidalunit^{\otimes a}\otimes\sinv \overline\coperad\planar(b)
		\otimes \monoidalunit^{\otimes c}\right) \otimes 
		\coperad^{\convolution a+b+c}}(n)
		\arrow[d,"{\id{} \otimes \left(\id{}^{\otimes a} \otimes
		\sinv\overline w
		\otimes \idfunctor^{\otimes c}\right) \otimes
		\left(\tau^{\convolution a+b} \convolution
		\idfunctor^{\convolution c}\right)}"]
		\\
		\!\!\!\!\!
		\displaystyle\bigoplus_{k\in\naturals}
		t_1\left(\sinv \overline\coperad\planar\right) \otimes 
		\left(\monoidalunit^{\otimes a} \otimes
		\sinv\overline\coperad\planar(k)\otimes
		\coperad\planar^{\convolution k}(b)
		\otimes \monoidalunit^{\otimes c}\right) \otimes 
		\left(\monoidalunit^{\otimes a+b} \otimes
		\coperad\planar^{\convolution c}(n-a-b)\right)
		\arrow[d, equal]
		\\
		\displaystyle\bigoplus_{k \in \naturals}
		t'_k\left(\sinv \overline\coperad\planar\right) \otimes 
		\left(\monoidalunit^{\otimes a} \otimes
		\coperad\planar^{\convolution k}(b) \otimes 
		\coperad\planar^{\convolution c}(n-a-b)\right)
		\arrow[d, hook]
		\\
		\displaystyle \bigoplus_{k\in\naturals}
		t'_k\left(\sinv \overline\coperad\planar\right) \otimes 
		\coperad\planar^{\convolution a+ k + c}(n),
	\end{tikzcd}
\]
where $t'_k$ is the planar tree obtained by replacing the leftest top
vertex of $t$ by a corolla vertex of valence $k$.
If $t$ is trivial, we define $\mapg_t$ to be the unique map from
$t\left(\sinv \overline\coperad\planar\right)\compofplanarseq \coperad\planar$
to $0$.
\end{definition}

\begin{lemma}
 The endomorphism $\mapg(\operadcompositionone)$ is the sum of the maps $\mapg_t$,
 that is the morphism
 whose restriction to $t\left(\sinv \overline\coperad\planar\right)
 \compofplanarseq \coperad\planar$ is
 $$
 t\left(\sinv \overline\coperad\planar\right)
 \compofplanarseq \coperad\planar \xrightarrow{\mapg_t}
 \bigoplus_{k\in\naturals}
		t'_k\left(\sinv \overline\coperad\planar\right)
		 \compofplanarseq \coperad\planar
\hookrightarrow \specialoperad\planar \compofplanarseq \coperad\planar.
 $$
Similarly,  $\mapg(\operadcompositiontwo)$ is the endomorphism whose restriction to the
summand indexed $t' \subset t$ is the composition of $\mapg_t$
 $$
 t\left(\sinv \overline\coperad\planar\right)
 \compofplanarseq \coperad\planar \xrightarrow{\mapg_t}
 \bigoplus_{k\in\naturals}
		t'_k\left(\sinv \overline\coperad\planar\right)
		 \compofplanarseq \coperad\planar		 
		 \hookrightarrow \specialoperad\planar \compofplanarseq
		 \specialoperad\planar \compofplanarseq \coperad\planar;
$$
with the inclusions into the summands indexed by $t''_k \subset t'_k$
where $t''_k$ is the subtree of $t$ whose edges are those which belong
to both $t_1$ and $t'$.
\end{lemma}

\begin{proof}
 This follows from the definition of $\smalldiffex$ and $\thehomotopy$ and
 from a straightforward computation.
\end{proof}

\begin{lemma}
 The map $\bigmapg$ extends uniquely to the restriction-extension diagram
 $\thediagramre$.
\end{lemma}

\begin{proof}
It suffices to show that $\bigmapg(\extensionone)$ restricts
to $\restrictionone$ and that $\bigmapg(\extensiontwo)$ restricts
to $\restrictiontwo$.

On the one hand, the restriction of $\bigmapg(\extensionone)$
to $\restrictionone$ is the degree $0$ endomorphism of
$\restrictionone$ whose projection on $[t(\sinv\overline\coperad),
\power{(V^\coperad)}{\otimes l}]$ is $0$ for a trivial tree $t$,
and is given by the following composition for a non-trivial tree $t$:
\[
	\begin{tikzcd}
		{\left(V^{\coperad\planar}\right)}^{\specialoperad\planar}  \arrow[d, two heads]
		\\
		\displaystyle \prod_{k\in\naturals}
		{\left[t'_k\left(\sinv\overline{\coperad\planar}\right),
		{\left(V^{\coperad\planar}\right)}^{\otimes a+k+c}\right]}
		\arrow[d,"{\left[\id{} , \fullpower{V^\iota}{\otimes a}
		\otimes V^{\laxmap^\convolution}
		\otimes \idfunctor^{\otimes c}\right]}"]
		\\
		\displaystyle \prod_{k\in\naturals}
		{\left[t'_k\left(\sinv\overline{\coperad\planar}\right),
		V^{\otimes a}\otimes
		V^{\power{\coperad\planar}{\convolution k}}
		\otimes \fullpower{V^\coperad\planar}{\otimes c} \right]}
		\arrow[d, two heads]
		\\
		\displaystyle \prod_{k\in\naturals}
		{\left[t'_k\left(\sinv\overline{\coperad\planar}\right),
		V^{\otimes a}\otimes
		\left[\coperad\planar^{\convolution k}(b),V^{\otimes b}\right]
		\otimes \fullpower{V^{\coperad\planar}}{\otimes c} \right]}
		\arrow[d, hook]
		\\
		{\left[
		\displaystyle \bigoplus_{k\in\naturals}
		t_1\left(\sinv\overline{\coperad\planar}\right)
		\otimes \left(\monoidalunit^{\otimes a} \otimes
		{\sinv \overline{\coperad\planar}}(k) \otimes
		\coperad\planar^{\convolution k}(b) \otimes
		\monoidalunit^{\otimes c}\right) ,
		V^{\otimes a+b} \otimes \fullpower{V^{\coperad\planar}}{\otimes c}\right]}
		\arrow[d,"\internalhom{\idfunctor \otimes \idfunctor^{\otimes a}
		\otimes\s^{1}\overline w \otimes\idfunctor^{\otimes c}} \idfunctor"]
		\\
		{\left[t_1\left(\sinv\overline{\coperad\planar}\right)\otimes
		\monoidalunit^{\otimes a}\otimes \sinv\overline\coperad\planar(b) \otimes
		\monoidalunit^{\otimes c},
		V^{\otimes a+b} \otimes \fullpower{V^{\coperad\planar}}{\otimes c}\right]}
		\arrow[d, hook]
		\\
		{\left[t\left(\sinv\overline{\coperad\planar}\right),
		\fullpower{V^{\coperad\planar}}{\otimes l} \right]}.
	\end{tikzcd}
\] 

On the other hand, the restriction of $-\thebighomotopy (\extensiontwo)$
to $\restrictiontwo$ is the degree $1$ endomorphism of
$\restrictiontwo$ whose projection on $[ t(\sinv\overline{\coperad}\planar),
\power{(\power{(V^{\coperad\planar})}{\specialoperad\planar})}{\otimes l} ]$
is the sum of the map
\[
	\left[\id{}, \sum_{k<a+b} \projto{V}^{\otimes k} \otimes 
	\bigmapg (\restrictionone)
	\otimes \id{}^{\otimes l- k-1} \right],
\]
together with the same composite map, mutatis mutandis, as the one that
defines $\bigmapg (\restrictionone)$ just above in this proof.
More precisely, to build this second component one only needs to
change in the composition
\begin{itemize}
	\item $V^\coperad$ by $\restrictionone$;
	\item $V^\iota$ by $\fullpower{V^\iota}\unitspecialoperad$;
	\item the map
	$$
	V^{\laxmap^\convolution}: \left(V^{\coperad\planar}\right)^{\otimes k} \to V^{\coperad\planar^{\convolution k}}
	$$
	is replaced by the map
	$$
	\left(\fullpower{V^{\coperad\planar}}{\specialoperad\planar}\right)^{\otimes k} 
	\xrightarrow{\left(\fullpower{V^{\coperad\planar}}{\unitspecialoperad}\right)^{\otimes k}} 
	\left(V^{\coperad\planar}\right)^{\otimes k} 
	\xrightarrow{V^{\laxmap^\convolution}}
	V^{\coperad\planar^{\convolution k}}.
	$$
\end{itemize}
\end{proof}

\subsubsection{The boundary map
\texorpdfstring{$\theboundary$}{the boundary of H}}

\begin{definition}
 Let us consider the following endomorphism of the restriction-extension
 diagram $\thediagramre$
\begin{align*}
 \partial_a\thebighomotopy &\coloneqq \bigdiffa \circ \thebighomotopy + \thebighomotopy \circ \bigdiffa;
 \\
 \partial_w\thebighomotopy &\coloneqq \bigdiffw \circ \thebighomotopy + \thebighomotopy \circ \bigdiffw;
 \\
  \partial_\coperad\thebighomotopy &\coloneqq \bigdiffcop \circ \thebighomotopy + \thebighomotopy \circ \bigdiffcop;
  \\
  \partial_\theta\thebighomotopy &\coloneqq \bigdifftheta \circ \thebighomotopy + \thebighomotopy \circ \bigdifftheta;
  \\
   \partial_V\thebighomotopy &\coloneqq \bigdiffv \circ \thebighomotopy + \thebighomotopy \circ \bigdiffv;
   \\
    \partial_{\mathrm{in}}\thebighomotopy &\coloneqq \bigdiffin \circ \thebighomotopy + \thebighomotopy \circ \bigdiffin
    =\partial_a\thebighomotopy + \partial_w\thebighomotopy + \partial_\coperad\thebighomotopy
    + \partial_\theta\thebighomotopy + \partial_V\thebighomotopy.
\end{align*}
\end{definition}

\begin{definition}
Let $\theboundary$ be the endomorphism of the restriction-extension diagram
$\thediagramre$ given
by
\[
	\theboundary \coloneqq \partial_{\mathrm{in}}\thebighomotopy +
	\bigmapg(-) + \idfunctor.
\]
\end{definition}

Using the sign rules
~[\ref{thm:sign}], one has $\partial_V \thebighomotopy = 0$,
which implies the following equality 
\[
	\theboundary =
	\partial_a\thebighomotopy +
	\partial_w\thebighomotopy +
	\partial_\coperad \thebighomotopy +
	\partial_\theta\thebighomotopy + \bigmapg +
	\idfunctor.
\]

Besides, the subobject $\overline \thediagramre \hookrightarrow \thediagramre$
is stable through the maps $\bigdiffa, \bigdiffw, \bigdiffcop, \bigdifftheta, \bigdiffv, \bigmapg,
\thebighomotopy$ and hence through the maps $\bigdiffin, \theboundary$.

Moreover, on the kernel part of the extension diagram $\overline \thediagramex \hookrightarrow \thediagramex$
and thus on $\overline W$, one has the following equality
$$
\bigdiff \circ \thebighomotopy + \thebighomotopy \circ \bigdiff
= \theboundary .
$$

\begin{lemma}\label{thm-isoimpliesiso}
 Let us suppose that the morphism $\theboundary(W)$ is an automorphism.
 Then, its resrtiction to $\overline W$ is also an automorphism.
 The same result holds if we replace $W$ by $\extensionone, \extensiontwo, \restrictionone$
 or $\restrictiontwo$.
\end{lemma}

\begin{proof}
 Since $\id_V \circ \specialquotient = \specialquotient \circ \theboundary(W)$
 is an isomorphism, it induces together with the identity of $V$
 an automorphism of the cospan diagram
 $$
 W \xrightarrow{\specialquotient} V \leftarrow 0.
 $$
 This induces an automorphism of the pullback of this diagram and this map
 is precisely which is the restriction
 of $\theboundary(W)$ to $\overline W$.
\end{proof}

\subsubsection{The map \texorpdfstring{$\thebighomotopy$}{H}
is a contracting homotopy}
%-----------------------------------------------------------------------

\begin{proposition}
\label{thm:B-est-un-automorphisme-general}
 The map $\theboundary  : {\thediagramre} \to {\thediagramre}$ is an automorphism.
\end{proposition}

\begin{proof}
 The fact that it is an isomorphism is a consequence of
 \cref{thm:B_est_un_automorphisme_E}, 
 \cref{B_est_un_automorphisme_R1} and \cref{thm:B_est_un_automorphism_de_R2}.
 The fact that its restriction to $\overline{\thediagramre}$ is also an automorphism
 is a consequence of \cref{thm-isoimpliesiso}.
\end{proof}

\begin{lemma}%
\label{thm:B_est_un_automorphisme_E}
The map $\theboundary(\extensionone)$ and $\theboundary(\extensiontwo)$
are automorphisms.
\end{lemma}

\begin{proof}
 Let us prove the case of $\extensionone$.
 We can cofilter $\extensionone$ using the radical filtration on
$\specialoperad$
\[
	{\extensionone} \isonat \limn \left(
	V^{{(\filtration n \specialoperad) \compofsymseq \coperad}}\right).
\]
All the maps
$\thebighomotopy$, $\bigdiffw$, $\bigdiffa$, $\bigdifftheta$, $\bigdiffcop$ and
$\bigmapg$ are compatible with this cofiltration.
Then one has for every natural $n$
\[
\gr n \thebighomotopy = \gr n \bigmapg = 0\Longrightarrow
	\gr n \theboundary = \idfunctor.
\]
The claimed result then follows by induction.

In the case of $\extensiontwo$, one can filter
$\specialoperad \compofsymseq \specialoperad$
in a similar fashion as $\specialoperad$, that is using the coradical
filtration of $\coperad$. Then the proof follows from the same arguments.
\end{proof}

\begin{lemma}%
\label{B_est_un_automorphisme_R1}
The map $\theboundary(\restrictionone)$ is an automorphism of $\restrictionone$.
\end{lemma}

\begin{proof}
We can cofilter $\restrictionone$ using the radical filtration on
$\specialoperad$
\[
	\restrictionone \isonat \limn
	\left(
	\fullpower{V^{\coperad}}{{\filtration n \specialoperad}}
	\right).
\]
All the maps
$\thebighomotopy$, $\bigdiffw$, $\bigdiffa$, $\bigdifftheta$, $\bigdiffcop$ and
$\bigmapg$ are compatible with this cofiltration.
Then one has for every natural $n$
\[
\gr n \thebighomotopy = \gr n \bigmapg = 0\Longrightarrow
	\gr n \theboundary = \idfunctor.
\]
The claimed result follows by induction.
\end{proof}

\begin{lemma}%
\label{thm:B_est_un_automorphism_de_R2}
The map $ \theboundary(\restrictiontwo)$ is an automorphism
of $\restrictiontwo$.
\end{lemma}

\begin{proof}
As before we shall start by cofiltering $\restrictiontwo$ using the radical filtration
on $\specialoperad$:
\[
	\restrictiontwo \isonat 
	\limn{\left( \restrictionone^{{\filtration n \specialoperad}}\right)}.
\]
Such a cofiltration is preserved by the components of $\bigdiffin$, by $\thebighomotopy$
and $\bigmapg$. One can notice that
$$
\gr n \restrictiontwo = \restrictionone^{{\grfilt n \specialoperad}}, \quad \forall n \in \naturals.
$$
In particular, $\gr 0 \restrictiontwo = \restrictionone$ and the map
$\gr 0 \theboundary$
is the isomorphism $\theboundary(\restrictionone)$ [\cref{B_est_un_automorphisme_R1}].

Let $n >0$ be a natural integer.
We shall further filter $\grfilt n \specialoperad$ two times.
\begin{itemize}
 \item  First, we filter it by the opposite of the
number of vertices of the trees. The induced cofiltration of $\gr n \restrictiontwo$
is also preserved by the map $\gr n\thebighomotopy$, the components of $\gr n \bigdiffin$
and by $\gr n\bigmapg$.
One then has
\[
	\gr {-v} \gr n \overline\restrictiontwo =
	\restrictionone^{\grfilt n \treemodule_v\left(\sinv \overline{\coperad}\right)},
	\quad 1 \leq v\leq n.
\]
We fix a natural integer $0 \leq v\leq n$.
This second filtration on $\specialoperad$ has the added benefit that
\[
	\grfilt {-v} \grfilt n \smalldiffw = 0.
\]

\item Then, we filter by the d-filtration, that is the filtration induced by the quasi-planar structure
on $\coperad$. Again,
the induced cofiltration of $\gr {-v} \gr n\restrictiontwo$
is also preserved by all the map $\gr n\thebighomotopy$, the components of $\gr n \bigdiffin$
and by $\gr n\bigmapg$. One has
\[
	\gr e \gr {-v} \gr n \restrictiontwo =
	\restrictionone^{\grfilt e \grfilt n \treemodule_v\left(\sinv \overline{\coperad}\right)},
	\quad e \geq 0.
\]
We fix a natural integer $e\geq 0$.
This third filtration on $\specialoperad$ has the added benefit that
the map
\[
	\grfilt e \grfilt {-v} \grfilt n d_{s^{-1}\coperad}
\]
is a planar map since it decomposes as
$$
\begin{tikzcd}
\grfilt e \grfilt n \treemodule_v\left(\sinv \overline{\coperad}\right)
\ar[rr, "{\grfilt e \grfilt {-v} \grfilt n d_{s^{-1}\coperad}}"] \ar[d, equal]
&& \grfilt e \grfilt n \treemodule_v\left(\sinv \overline{\coperad}\right)
\ar[d, equal]
\\
\grfilt e \grfilt n \pltosym{\treemodule_{v,\plan}\left(\sinv \overline{\coperad}\planar\right)}
\ar[rr,"{d_{\sinv\overline{\coperad}\plan} \otimes \id{\catofpermutations}}"]
&&\grfilt e \grfilt n \pltosym{\treemodule_{v,\planar}\left(\sinv \overline{\coperad}\planar\right)}
\end{tikzcd}
$$
The map $d_{\sinv\overline{\coperad}\planar}$
is the sum of degree $-1$ endomorphisms
$$
d_t : \grfilt e \grfilt n t \left(\sinv \overline{\coperad}\planar\right)
\to \grfilt e \grfilt n t \left(\sinv \overline{\coperad}\planar\right),
$$
indexed by planar trees with $v$ vertices.
\end{itemize}
The object $\gr e \gr {-v} \gr n \restrictiontwo$ is a product indexed by
planar trees with $v$ vertices
$$
\gr e\gr {-v} \gr n \restrictiontwo \isonat \prod_{t} {\restrictiontwo}_t,
$$
where
$$
{\restrictiontwo}_t = \internalhom{\grfilt e\grfilt n t\left(\sinv\overline\coperad\planar\right)(l)}
{\fullpower{\restrictionone}{\otimes l}}.
$$
Remember that $l = a +b+c$ is the number of leaves of the tree $t$.
In that perspective, the maps $\gr e \gr {-v} \gr n\thebighomotopy(\restrictiontwo)$,
$\gr e \gr {-v} \gr n \bigmapg(\restrictiontwo)$ and $\gr e \gr {-v} \gr n\bigdiffin(\restrictiontwo)$
can be described as the product over the planar trees $t$ with $v$ vertices of the
respective endomorphisms $\thebighomotopy_t$, $\bigmapg_t$ and ${\bigdiffin}_t$
given by
of ${\restrictiontwo}_t$
\begin{align*}
 	\thebighomotopy_t 
	&=  \left[\idfunctor, \thebighomotopy_{a+b,c}\right];
	\\
	\bigmapg_t 
	&= \left[\idfunctor, \bigmapg_{a+b,c}\right];
	\\
	{\bigdiffin}_t
	&= - \left[d_t, \idfunctor^{\otimes a+b+c}
	\right]
	+
	\left[\idfunctor, {\bigdiffin}_{a+b,c}\right].
\end{align*}
where the endomorphisms $\thebighomotopy_{a+b,c}$, $\bigmapg_{a+b,c}$
and ${\bigdiffin}_{a+b,c}$ of $\restrictionone^{\otimes (a+b+c)}$
are given by
\begin{align*}
 	\thebighomotopy_{a+b,c}
	&= \sum_{p+1+q=a+b} \projto{V}^{\otimes p}
	\otimes \thebighomotopy(\restrictionone)
	\otimes \idfunctor^{\otimes q+c};
	\\
	\bigmapg_{a+b,c}
	&= \sum_{p+1+q=a+b} \projto{V}^{\otimes p}
	\otimes \bigmapg(\restrictionone)
	\otimes \idfunctor^{\otimes q+c};
	\\
	{\bigdiffin}_{a+b,c}
	&=
	\sum_{p+1+q=l}  \idfunctor^{\otimes p}
	\otimes \bigdiffin(\restrictionone)
	\otimes \idfunctor^{\otimes q}.
\end{align*}
Then, the map
${\gr e \gr {-v} \gr n\theboundary(\restrictiontwo)}$ is equal to
the product over planar trees with $v$ nodes of the sum
\[
	\theboundary_t = \idfunctor + \bigmapg_t 
	+ {\bigdiffin}_t \circ \thebighomotopy_t  
	+ \thebighomotopy_t  \circ {\bigdiffin}_t.
\]
One can notice that for a planar tree $t$ with no leaf, we have
$\theboundary_t  = \idfunctor$.
Besides, for any planar tree $t$, since
$$
-\thebighomotopy_t  \circ \left[d_t, \idfunctor^{\otimes a+b+c}\right]
- \left[d_t, \idfunctor^{\otimes a+b+c}\right] \circ -\thebighomotopy_t = 0 ,
$$
then the map $\theboundary_t$
has the form
$$
\theboundary_t = \left[\idfunctor, \theboundary_{a+b,c}\right]
$$
where 
$$
\theboundary_{a+b,c} = \idfunctor + \bigmapg_{a+b,c}
	+ {\bigdiffin}_{a+b,c} \circ \thebighomotopy_{a+b,c}
	+ \thebighomotopy_{a+b,c}  \circ {\bigdiffin}_{a+b,c}.
$$

Let us show by indunction on the number of leaves $l$ of any planar tree $t$
that the map $\theboundary_{a+b,c}$ is an automorphism of
the tensor product $\restrictionone^{\otimes l}$. This result is clear if $l=0$. For
$l=1$, it follows from \cref{B_est_un_automorphisme_R1}. So, let us consider a planar tree
with $l > 1$ leaves. Then, necessarily, $b \geq 1$. The tensor product $\restrictionone^{\otimes l}$
may be 
filtered as
$$
\filtration k \left(\restrictionone^{\otimes l}\right) =  {\overline \restrictionone}^{\otimes a+b - k}
\otimes  \restrictionone^{c + k},
\quad 0 \leq k \leq a+b .
$$
This filtration is preserved by the maps $\bigmapg_{a+b,c}$, ${\bigdiffin}_{a+b,c}$
and $\thebighomotopy_{a+b,c}$ and so by $\theboundary_{a+b,c}$. We have
\begin{align*}
	\grfilt 0 \left(\restrictionone^{\otimes l}\right)& =
 	{\overline \restrictionone}^{\otimes a+b} \otimes  \restrictionone^{c} 
	\\
	\grfilt k \left(\restrictionone^{\otimes l}\right)
	&=
	{\overline \restrictionone}^{\otimes a+b - k} \otimes V \otimes
	\restrictionone^{\otimes c+k-1}
	\text{ for } 0<k \leq a+b ,
\end{align*}
Let $0 \leq k < a+b$. In this case, we can notice successively that
\begin{align*}
 \grfilt k \bigmapg_{a+b,c} &= \bigmapg (\restrictionone) \otimes \id{}^{\otimes l-1};
 \\
  \grfilt k \thebighomotopy_{a+b,c} &= \thebighomotopy(\restrictionone) \otimes \id{}^{\otimes l-1};
  \\
  \grfilt k (\partial_{\mathrm{in}}\thebighomotopy)_{a+b,c} &= 
	 \partial_{\mathrm{in}}\thebighomotopy(\restrictionone)\otimes \id{}^{\otimes l-1};
\\
\grfilt k \theboundary_{a+b,c} &= \theboundary(\restrictionone) \otimes \id{}^{\otimes l-1}.
\end{align*}
Hence, by \cref{B_est_un_automorphisme_R1}, $\grfilt k \theboundary_{a+b,c}$ is invertible.
Now, let $k=a+b$. In this case,
$$
\grfilt k \left(\restrictionone^{\otimes l}\right)
	=
	V \otimes
	\restrictionone^{\otimes l-1}
$$
We have two subcases.
\begin{itemize}
 \item If $k=a+b =1$, then $\grfilt k \bigmapg_{a+b,c} = \grfilt k \thebighomotopy_{a+b,c}=0$;
 so $\grfilt k \theboundary_{a+b,c}= \id{}$ is invertible.
 \item If $k > 1$, then we can notice successively that
\begin{align*}
 \grfilt k \bigmapg_{a+b,c} &= \id{} \otimes \bigmapg_{a+b-1,c};
 \\
  \grfilt k \thebighomotopy_{a+b,c} &= \id \otimes \thebighomotopy_{a+b-1,c};
  \\
  \grfilt k (\partial_{\mathrm{in}}\thebighomotopy)_{a+b,c} &= \id{} \otimes
	(\partial_{\mathrm{in}}\thebighomotopy)_{a+b-1,c};
\\
\grfilt k \theboundary_{a+b,c} &= \id{} \otimes \theboundary_{a+b-1,c}.
\end{align*}
By the induction hypothesis $\grfilt k \theboundary_{a+b,c}$ is invertible.
\end{itemize}

We have shown that for any planar tree, the map
map $\grfilt k \theboundary_t$ is an automorphism.
Thus one concludes
with
\begin{align*}
	& \quad \forall k, \enskip
	\grfilt k \theboundary_{a+b,c}
	\text{ is an isomorphism}
	\\
	\Longrightarrow & \quad \forall t,\enskip
	\theboundary_{a+b,c}
	\text{ is an isomorphism}
	\\
	\Longrightarrow & \quad \forall t,\enskip
	\theboundary_t
	\text{ is an isomorphism}
	\\
	\Longrightarrow & \quad \forall e, \enskip
	\gr e \gr {-v} \gr n \theboundary(\restrictiontwo)
	\text{ is an isomorphism}
	\\
	\Longrightarrow & \quad \forall v, \enskip
	\gr {-v} \gr n \theboundary(\restrictiontwo)
	\text{ is an isomorphism}
	\\
	\Longrightarrow & \quad \forall n, \enskip
	\gr n \theboundary(\restrictiontwo)
	\text{ is an isomorphism}
	\\
	\Longrightarrow & \quad 
	\theboundary(\restrictiontwo)
	\text{ is an isomorphism.}
\end{align*}
\end{proof}

\subsubsection{Conclusion}

In the quasi-planar context, we have introduced a degree $1$ endomorphism
$\thebighomotopy(W)$
of $W$ that restricts
to $\overline W$. Moreover, on this subobject, the boundary of $\thebighomotopy(W)$
$$
\bigdiff \circ \thebighomotopy + \thebighomotopy \circ \bigdiff
= \theboundary
$$
is an isomorphism [\cref{thm:B-est-un-automorphisme-general}
and \cref{thm-isoimpliesiso}]. Hence the zero map is homotopic to an automorphism
and so $\overline W$ is contractible.
Finally, the map $V \to W$ is a quasi-isomorphism.

\subsection{The characteristic zero case}%
%-----------------------------------------------------------------------
\label{sec:charac_zero_case}

In this section, we suppose that the category $\catss$ is $\rationals$-linear.
Our strategy to prove that the unit morphism $V \to \cobaradj\cobarfunctor V$
is a quasi-isomorphism is to replace $\coperad$ by an equivalent quasi-planar
curved conilpotent coperad.

\subsubsection{The quasi-planar resolution}

Remember that the underlying symmetric sequence of the Barratt-Eccles operad $\eoperad$
is so that for any two natural integers $n, d \geq 0$
$$
\eoperad(n)_d = \bigoplus_{(\sigma_0, \ldots,\sigma_{d}) } \monoidalunit
$$
where the sum is taken over sequence $(\sigma_0, \ldots,\sigma_{d})$
of $d+1$ permutations in $\catofpermutations_n$ so that $\sigma_i \neq \sigma_{i+1}$ for $0 \geq i < d$.

\begin{definition}
 Let $\eoperad\planar$ be the graded planar sequence of $\eoperad$ so that
 for any two natural integers $n, d \geq 0$
 $$
 \eoperad\planar(n)_d = \bigoplus_{(\sigma_0, \ldots,\sigma_{d}),\  \sigma_0= \id } \monoidalunit
 $$
is the subobject of $\eoperad(n)_d$ that corresponds the sequences $(\sigma_0, \ldots,\sigma_{d})$
so that $\sigma_0 = \id$. One has a canonical inclusion of graded planar sequences
$$
\eoperad\planar \hookrightarrow \eoperad
$$
that gives an isomorphism of graded symmetric sequences
$$
\pltosym{\eoperad\planar} \simeq \eoperad.
$$
\end{definition}

\begin{lemma}
For any dg-symmetric sequence $Y$, one has a canonical isomorphism
of graded symmetric sequence
$$
Y \diagtensor \eoperad \to \pltosym{(\forget_{\catofpermutations} (Y) \diagtensor\planar \eoperad\planar)}
$$
that commutes with the map $d_Y \otimes \id$.
\end{lemma}

\begin{proof}
Such an isomorphism is the composition
$$
Y \diagtensor \eoperad \to Y \diagtensor (\pltosym{\eoperad\planar})
\to \pltosym{(\forget_{\catofpermutations} (Y) \diagtensor\planar \eoperad\planar)}
$$
where the last map is given for any natural
integer $n$,by
\[
	Y(n) \otimes  \eoperad\planar(n) \otimes \catofpermutations_n
	= \bigoplus_{\sigma \in \catofpermutations_n}
	Y(n) \otimes  \eoperad\planar(n) \otimes \{\sigma\}
	\xrightarrow{\sigma^{-1} \otimes \id \otimes \id}
	\bigoplus_{\sigma \in \catofpermutations_n}
	Y(n) \otimes  \eoperad\planar(n) \otimes \{\sigma\}.
\]
One can also notice also that this last map commutes with $d_Y \otimes \id \otimes \id$.
\end{proof}

\begin{lemma}
For any dg-operad $\operad$, the curved cocomplete
cooperad $\barfunctor \left(\operad \diagtensor \eoperad\right)$
is quasi-planar.
\end{lemma}

\begin{proof}
On has a canonical isomorphism of graded symmetric sequences
$$
\operad \diagtensor \eoperad \simeq \pltosym{ (\forget_\catofpermutations (\operad) \diagtensor\planar \eoperad\planar)}
$$
that commutes with the map $d_\operad \diagtensor \id$.

A direct consequence of the such an isomorphsim
is that the underlying graded cooperad of
 $\coperad' = \barfunctor \left(\operad \diagtensor \eoperad\right)$
 arise from
a planar graded cooperad ~[\ref{thm:planar_trees_vs_trees}]
$$
\coperad'\planar= \planartreemodule
(\s^2 \monoidalunit \oplus \s \forget_\catofpermutations (\operad) \diagtensor\planar \eoperad\planar) .
$$
in the sense that we have a canonical isomorphism of graded conilpotent cooperads
$$
\coperad = \pltosym{\coperad'\planar}.
$$

Besides, one can filter $\eoperad\planar$ by the degree
 \[
 	\left(\filtrationd m \eoperad\planar(k)\right)_n \coloneqq 	
	\begin{cases}
 		\eoperad\planar (k)_n \text{ if }n\leq m ,\\
		0 \text{ otherwise} .
	\end{cases}
 \]
For any natural integer $n \geq 1$, this induces a filtration on
$$
\grfilt n \coperad'\planar = \treemodule_{n,\plan} 
(\s^2 \monoidalunit \oplus \s \forget_\catofpermutations (\operad) \diagtensor\planar \eoperad\planar).
$$
and then on 
$$
\grfilt n \coperad' = \pltosym{\grfilt n \coperad'\planar} 
= \treemodule_{n}  (\s^2 \monoidalunit \oplus \s \operad \diagtensor \eoperad).
$$
Moreover, for any natural integer $m \geq0$, the map
$\grfilt m \grfilt n d_{\coperad'}$ is planar as we have an equality
$$
\grfilt m \grfilt n d_{\coperad'} = \pltosym{\grfilt m \grfilt n d_{\coperad'\planar}}
$$
where $d_{\coperad'\planar}$ is the degree $-1$ coderivation of the
graded conilpotent planar cooperad $\coperad'\planar$ the projects the map
$$
\treemodule_{n,\plan} 
(\s^2 \monoidalunit \oplus \s \forget_\catofpermutations (\operad) \diagtensor\planar \eoperad\planar)
\twoheadrightarrow \s^2 \monoidalunit \oplus \s \forget_\catofpermutations (\operad) \diagtensor\planar
\eoperad\planar 
\xrightarrow{\s \eta_{\operad\planar} \diagtensor\planar \oplus d_\operad \diagtensor\planar \id_{\eoperad\planar}}
\s \forget_\catofpermutations (\operad) \diagtensor\planar \eoperad\planar.
$$
\end{proof}

\begin{lemma}
For any conilpotent curved cooperad $\coperad$
there exists a quasi-planar conilpotent curved cooperad $\coperad'$ and an acyclic cofibration
$$
\coperad \to \coperad'.
$$
One calls such a morphism a quasi-planar resolution of $\coperad$.
\end{lemma}

\begin{proof}
 Let us take
\[
\coperad'  \coloneqq \barfunctor
\left( \specialoperad \diagtensor \eoperad \right) .
\]
Since $\specialoperad$ is cofibrant, the canonical morphism
$\specialoperad \diagtensor \eoperad \to \specialoperad$
has a section $s$. We thus obtain a morphism from $\coperad$
to $\coperad'$ as follows
\[
	\coperad \to \barfunctor\baradj \coperad
	= \barfunctor \specialoperad
	\xrightarrow{\barfunctor s} 
	\barfunctor
	\left( \specialoperad \diagtensor \eoperad \right) .
\]
This is an acyclic cofibration as a composition of two acyclic cofibrations.
\end{proof}

\subsubsection{Conclusion}

Let us consider a quasi-planar resolution
\[
\begin{tikzcd}
	\coperad \arrow[r, "s"] 
	& \coperad'
\end{tikzcd}
\]
of the conilpotent curved cooperad $\coperad$.
Let us denote $\specialoperad' \coloneqq \baradj \coperad'$
and $\sigma \coloneqq \baradj s$ and let $\alpha'$ be
the canonical twisting morphism from $\coperad'$
to $\specialoperad'$.

As an acyclic cofibration, the morphism $\sigma$ has a retraction
$\rho$. We know from paragraph \ref{subsubsection-retraction-decomposition}
that the unit map
$$
V \to \cobaradj \cobarfunctor V
$$
decomposes as
$$
\begin{tikzcd}
	V 
	\ar[d, equal]
	\\
	\forget_\sigma \forget_\rho (V)
	\ar[d]
	\\
	\forget_\sigma \cobaradj_{\alpha'} \cobarfunctor_{\alpha'} \forget_\rho (V)
	\ar[d]
	\\
	\forget_\sigma \cobaradj_{\alpha'} \forget^s  \freealg s \cobarfunctor_{\alpha'} \forget_\rho (V)
	\ar[d, "\simeq"]
	\\
	\forget_\sigma \freecog \sigma \cobaradj  \cobarfunctor \forget_\sigma \forget_\rho (V)
	\ar[d,  equal]
	\\
	\forget_\sigma \freecog \sigma \cobaradj  \cobarfunctor (V)
	\ar[d]
	\\
	 \cobaradj  \cobarfunctor (V) .
\end{tikzcd}
$$
From the quasi-planar step and since the adjunction $\forget_\sigma \freecog \sigma$ is a Quillen equivalence,
we know that the maps
\begin{align*}
 &\forget_\sigma \forget_\rho (V) \to \forget_\sigma \cobaradj_{\alpha'} \cobarfunctor_{\alpha'} \forget_\rho (V);
 \\
 & \forget_\sigma \freecog \sigma \cobaradj  \cobarfunctor (V) \to  \cobaradj  \cobarfunctor (V);
\end{align*}
are quasi-isomorphisms. So proving that the unit map is a quasi isomorphism amounts to prove that the map
$$
\forget_\sigma \cobaradj_{\alpha'} \cobarfunctor_{\alpha'} \forget_\rho (V)
\to \forget_\sigma \cobaradj_{\alpha'} \forget^s  \freealg s \cobarfunctor_{\alpha'} \forget_\rho (V)
$$
is one. Since the composite functor $\forget_\sigma \cobaradj_{\alpha'}$ reflects weak equivalences,
this is equivalent to the fact that the map
$$
\cobarfunctor_{\alpha'} \forget_\rho (V)
\to \forget^s  \freealg s \cobarfunctor_{\alpha'} \forget_\rho (V)
$$
is a weak equivalence of complete $\coperad'$-algebras for the model structure transferred
from that of $\specialoperad'$-coalgebras. This is actually a consequence
of \cref{thm:pre-model-equivalence-algebras}.

%%%%%%%%%%%%%%%%%%%%%%%%%%%
%%%%%%%%%%%%%%%%%%%%%%%%%%%

\section{Homotopy theory of linear coalgebras}%
%=======================================================================
\label{sec:theorie_homotopique_des_cogebres_lineaires}

In this section we assume that $\specialoperad \coloneqq \baradj
\coperad$ and that $\canonicaltwisting$ is the canonical twisting
morphism. The results obtained in the previous section tell us that
the homotopy theory of $\specialoperad$-coalgebras is equivalent
to the homotopy theory of complete $\coperad$-algebras. We shall
now see that this homotopy theory can be explicitly described.

%bigtheorem
\begin{theorem}%
\label{bigthm:homotopy_theory_of_coalgebras}
When the category of complete $\coperad$-algebras is endowed with
its canonical model structure, the
three following full subcategories are equivalent:
\begin{itemize}
	\item the full subcategory of cofibrant objects;
	\item the full subcategory whose objects are images of
	      $\cobarfunctorfull$;
	\item the full subcategory whose objects are complete
	      $\coperad$-algebras
	      that are free as graded $\coperad$-algebras.
\end{itemize}
Moreover, given a morphism $f : \cobarfunctor V \to \cobarfunctor W$
between two cofibrant complete $\coperad$-algebras,
\begin{itemize}
	\item $f$ is a weak equivalence if and only if $\gr 0 f : V \to W$ is
	      a quasi-isomorphism;
	\item $f$ is a fibration if and only if $\gr 0 f : V\to W$ is a
	      degree-wise epimorphism;
	\item $f$ is a cofibration if and only if $\gr 0 f : V \to W$ is a
	      degree-wise monomorphism.
\end{itemize}
\end{theorem}

The remaining of this section will be dedicated to the proof of the
statements contained in this theorem.
The corresponding theorem, relating fibrant-cofibrant
$\coperad$-coalgebras with the image of the $\barfunctorfull$ functor
was proved by Loday \& Vallette in the augmented case%
~\cite[10.1.22]{doi:10.1007/978-3-642-30362-3} and proved in
\emph{Homotopy theory of unital algebras} in the general case%
~\cite[Proposition 33]{arXiv:1612.02254}

\subsection{The Rosetta stone}
%-----------------------------------------------------------------------

\begin{theorem}[(Rosetta stone)]%
\label{thm:Rosetta_stone}
Given a chain complex $(V,d_V)$, the $\cobarfunctorfull$ functor
induces a bijection
\[
	\cobarfunctorfull : \modulispace {\catofcog\specialoperad}
	\left(V, d_V\right) \longrightarrow \modulispace {\setofder \coperad}
	\left(V^\coperad, d_V\right)
\]
between the moduli set of $\specialoperad$-coalgebra structure on
$(V, d_V)$ and the moduli set of $\coperad$\=/derivations on
$V^\coperad$ which restrict to $d_V$ on $V \subset V^\coperad$.
\end{theorem}

\begin{proof}
On one hand, any $\coperad$-derivation $d$ of $V^\coperad$ is determined
by its restriction $d \circ \canonicalinj : V \to V^\coperad$%
~[\ref{thm:algebre_courbee_libre}]. If moreover $d$ restricts to
$d_V$ on $V$, then such a derivation is in fact given by a degree $-1$
map $f :V \to V^{\overline\coperad} \subset V^\coperad$.
The map $d$ is a
$\coperad$-derivation if and only if the map $f$ satisfies the equation
\[
	\induceddiff {(f+\canonicalinj \circ d_V)}
	\circ (f + \canonicalinj \circ d_V) + V^\theta = 0~%
	[\ref{thm:algebre_courbee_libre}].
\]
This equation is equivalent to
\begin{equation}
	f \circ d_V =
	V^{d_\coperad} \circ f - {\sha(\id V, d_V)}^\coperad \circ f
	- \canonicalaction_{V^\coperad} \circ
	{\sha(\canonicalinj, f)}^\coperad \circ f - V^\theta.
	\tag{\dertag}
\end{equation}

On the other hand, since $\specialoperad$ is free as a graded operad,
a structure of $\specialoperad$-coalgebra on
$(V, d_V)$ amounts to the data of a map $g : V \to V^\specialoperad$
that satisfies the equation
\[
	g \circ d_V = {\sha(\id V, d_V)}^\specialoperad \circ g
	- V^{\smalldiff_\specialoperad} \circ g~%
	[\ref{def:coderivation}]
\]
which equivalent to
\begin{equation}
	g \circ d_V = {\sha(\id V, d_V)}^\specialoperad \circ g
	- V^{\smalldiff_w} \circ g - V^{\smalldiff_{\sinv\overline\coperad}}
	\circ g - V^{\smalldiff_\theta} \circ g.
	\tag{\cogtag}
\end{equation}
Since $\specialoperad$ is free as a graded operad, a morphism of
graded operads
\[
	\specialoperad = \treemodule \left(\sinv\overline{\coperad}\right)
	\to \coend V.
\]
is in fact given by a graded map $\sinv\overline\coperad \to
\coend V$. As a consequence any $\specialoperad$-coalgebra structure $g$
on $(V, d_V)$ is generated by the corestriction
\[
	V^{\attach\left(\sinv\overline\coperad\right)}
	\circ g : V \longrightarrow V^{\sinv\overline \coperad}
\]
Since $g$ satisfies (\cogtag), the degree $-1$ map
\[
	f \coloneqq
	\left(\degreeminusonemap \otimes V^{\overline\coperad}\right)
	\circ
	V^{\attach(\sinv\overline\coperad)} \circ g : V \longrightarrow
	V^{\overline \coperad}
\]
satisfies (\dertag). Conversely, if $f : V \to V^{\overline \coperad}$
satisfies (\dertag), then the extension of the map
\[
	\left(\degreeonemap \otimes V^{\overline\coperad}\right) \circ f : V
	\to V^{\sinv\overline\coperad}
\]
to a structure $g$ of graded $\specialoperad$-structure on $V$ satisfies
(\cogtag).
\end{proof}

The $\coperad$-algebras that are free as graded algebras coincide with
the fibrant-cofibrant objects in the model category of complete
$\coperad$-algebras. This is what we are going to show now.

\begin{lemma}%
\label{thm:factorisation_for_free=fibrant}
Let $\Lambda$ be a complete $\coperad$-algebra. Then for any $n \in
\naturals$, the morphism $\Lambda^{\grfilt n \coperad} \to \gr n
\Lambda$ induced by the structure map $a_\Lambda$ factorises through
\[
	\fullpower{\cofiltration 0 \Lambda}{\grfilt n\coperad} .
\]
\end{lemma}

\begin{proof}
For any $n \in \naturals$ the composite map
\[
	\sha \fullpower{\Lambda, \Lambda^{\overline{\coperad}}}{\coperad /
	\filtration n \coperad}
	\xrightarrow{\sha\fullpower{\idfunctor, a_\Lambda}{\coperad /
	\filtration n \coperad}} \Lambda^{\coperad / \filtration n \coperad}
	\to \ideal n  \Lambda
\]
factorises through $\Lambda^{\coperad / \filtration {n+1} \coperad}$
and so through $\ideal {n+1} \Lambda$. Subsequently, the composite map
\[
	\sha \fullpower{\Lambda, \Lambda^{\overline{\coperad}}}{\coperad /
	\filtration n \coperad} \to
	\sha \fullpower{\Lambda, \Lambda^{\overline{\coperad}}}{\grfilt
	n\coperad} \to \Lambda^{\grfilt n\coperad} \to \gr n\Lambda,
\]
is zero. Since the first map is an epimorphism, then the composition of
the two last maps is zero. So the morphism $\Lambda^{\grfilt n \coperad}
\to \gr n \Lambda$ factorises through the kernel of the map
\[
	\sha \fullpower{\Lambda, \Lambda^{\overline{\coperad}}}{\coperad /
	\filtration n \coperad} \to
	\sha \fullpower{\Lambda, \Lambda^{\overline{\coperad}}}{\grfilt
	n\coperad},
\]
that is $\power{(\cofiltration 0 \Lambda)}{\grfilt n\coperad}$.
\end{proof}

\begin{lemma}%
\label{thm:retract_of_free}
A retract of a free graded $\coperad$-algebra is free.
\end{lemma}

\begin{proof}
Consider a retract diagram of graded $\coperad$-algebras as follows
\[
	\begin{tikzcd}[ampersand replacement=\&]
		(\Lambda,a) \arrow[r,"s"]  \&
		(X^{\coperad}, \canonicalaction) \arrow[r,"r"] \& (\Lambda,a).
	\end{tikzcd}
\]
The composite map
\[
	\cofiltration 0 \Lambda \xrightarrow{\cofiltration 0(s)}
	\cofiltration 0 X^\coperad \isonat X
	\xrightarrow{X^\tau}
	X^{\coperad} \xrightarrow{r} \Lambda.
\]
induces a morphism of graded $\coperad$-algebras
\[
	f:{\left(\cofiltration 0 \Lambda\right)}^\coperad \to \Lambda.
\]
Let us show that this is an isomorphism. Since $f$ is a morphism
between two complete algebras,
we need only show
that $\gr \ast f$ is an isomorphism: first the map
\[
	\gr 0 (f): \cofiltration 0\left({\left(\cofiltration 0
	\Lambda\right)}^\coperad \right) \to \cofiltration 0\Lambda
\]
is an isomorphism. Besides,
for every $n \in \naturals$ the morphism
$\Lambda^{\grfilt n \coperad} \to \gr n \Lambda$ induced
by the structure map $a_\Lambda$ actually factorises through
$\power{(\cofiltration 0 \Lambda)}{\grfilt n\coperad}$%
~[\ref{thm:factorisation_for_free=fibrant}].
Then for any natural $n$, consider the following
retract diagram:
\[
	\begin{tikzcd}[ampersand replacement=\&]
		\fullpower{\cofiltration 0 \Lambda}{\grfilt n\coperad} 
		\arrow[rr,"\fullpower{\cofiltration 0 s}{\grfilt n\coperad}"]
		\arrow[d,"{a_{\Lambda}}"]
		\&\&
		\fullpower{\cofiltration 0 X^{\coperad}}{\grfilt n\coperad} 
		\arrow[rr,"{\fullpower{\cofiltration 0 r}{\grfilt n\coperad}}"]
		\arrow[d,"a_{X^{\coperad}}"]
		\&\&
		\fullpower{\cofiltration 0 \Lambda}{\grfilt n \coperad}
		\arrow[d,"{a_\Lambda}"]
		\\
		\gr n \Lambda  
		\arrow[rr,"\gr n s"]
		\&\& \gr n X^\coperad 
		\arrow[rr,"\gr n r"]
		\&\& \gr n \Lambda.
	\end{tikzcd}
\]
Since the vertical middle arrow is an isomorphism, then the left
vertical arrow is also an isomorphism. Therefore, three of the four
faces of the following square are isomorphisms
\[
	\begin{tikzcd}[ampersand replacement=\&]
		\fullpower{\cofiltration 0 \Lambda}{\grfilt n\coperad} 
		\arrow[rr,"\fullpower{\cofiltration 0 f}{\grfilt n\coperad}"]
		\arrow[d,"{a_{\Lambda}}"]
		\&\& \fullpower{\cofiltration 0
		\fullpower{\cofiltration 0 X}{\coperad}}{\grfilt n\coperad} 
		\arrow[d,"a_{\fullpower{\cofiltration 0 X}{\coperad}}"]
		\\
		\gr n \Lambda  
		\arrow[rr,"\gr n f"]
		\&\& \gr n \fullpower{\cofiltration 0 X}\coperad 
	\end{tikzcd}
\]
By the 2-out-of-3 rule, the map
$\gr n (f)$ is an isomorphism.
\end{proof}

\begin{theorem}%
\label{thm:objets_fibrants_dans_les_algebres_completes}
The three following full subcategories of the category of
the model category of complete $\coperad$-algebras are
equivalent:
\begin{itemize}
	\item the full subcategory of cofibrant objects;
	\item the full subcategory whose object are images of Cobar;
	\item the full subcategory whose objects are complete
	      $\coperad$-algebras that
	      are free as graded \mbox{$\coperad$-algebra}.
\end{itemize}
\end{theorem}

\begin{proof}
Using the Rosetta stone%
~[\ref{thm:Rosetta_stone}], we already know that the full subcategory
whose object are images of Cobar is equivalent to the full subcategory
whose objects are complete $\coperad$-algebras that are free as graded
$\coperad$\=/algebra. Besides, let $\Lambda$ be a cofibrant object.
The following square has a lifting
\[
	\begin{tikzcd}[ampersand replacement=\&]
		0 \arrow[r] \arrow[d]
		\& \cobarfunctor\cobaradj \Lambda \arrow[d]
		\\
		\Lambda \arrow[r,equal]
		\& \Lambda .
	\end{tikzcd}
\]
Thus, $\Lambda$ is a retract of $\cobarfunctor\cobaradj
\Lambda$ which is free as a graded $\coperad$\=/algebra.
Hence it is free as a graded $\coperad$\=/algebra%
~[\ref{thm:retract_of_free}]. We conclude by noticing that any image of
the functor $\cobarfunctorfull{}$ is cofibrant.
\end{proof}

\subsection{\texorpdfstring{∞}{Infinity}-Morphisms}
%-----------------------------------------------------------------------

We have seen in the previous subsection that cofibrant complete
$\coperad$\=/algebras are essentially images of the functor
$\cobarfunctorfull$. We study here morphisms between images of this
functor.

\begin{definition}
An ∞\=/morphism 
$f : V \inftymorphism W$ between two $\specialoperad$\=/coalgebras
is a morphism of $\coperad$-algebras from $\cobarfunctor V$
to $\cobarfunctor W$.
\end{definition}

As the underlying graded $\coperad$-algebra $V^\coperad$ of 
$\cobarfunctor V$ is free an ∞\=/morphism $f: V\to
W$ is determined by its restriction $f_V$ to the generators $V$,
that is a map
\[
	(f_\dg,f_\h) : V \to
	W^\coperad \isonat W \oplus W^{\overline \coperad},
\]
which satisfies some conditions. In particular, the map $f_\dg : V \to
W$ is a morphism of chain complexes. This one does not commutes with
the structure of $\specialoperad$-coalgebras; but it commutes up to some
homotopy,
which is actually encoded in the map $f_\h : V \to W^{\overline
\coperad}$ together with higher coherences.

Notice that, as a morphisms of chain
complexes from $V$ to $W$
\[
	f_\dg = \gr 0 f .
\]

\begin{definition}
An ∞-morphism $f: (V,a) \inftymorphism (W,b)$ between two
$\specialoperad$-coalgebras is said to be
\begin{itemize}
	\item a strict ∞-morphism if $f_\h=0$;
	\item an ∞-isomorphism if $f_\dg$ is an
	      isomorphism of chain complexes;
	\item an ∞-quasi-isomorphism if $f_\dg$ is a
	      quasi-isomorphism of chain complexes;
	\item an ∞-monomorphism if $f_\dg$ is a
	      degreewise monomorphism;
	\item an ∞-epimorphism if $f_\dg$ is a
	      degreewise epimorphism;
	\item an ∞-isotopy if $f_\dg$ is the identity of
	      the chain complex $V$. Note that in general
	      $f_\dg^\specialoperad \circ a \neq b \circ f_\dg$.
\end{itemize}
\end{definition}

\begin{remark}%
\label{rmk:composition_of_infinity_arrows}
Assume one wants to compose two ∞-morphisms $f$ and $g$, then
understanding $\under{(f \circ g)}\dg$ and $\under{(f \circ g)}\h$
is easy when $g$ is a strict ∞-morphism. Indeed in this case one has
\[
	\fullunder{f \circ g}\dg = f_\dg \circ g_\dg \qand
	\fullunder{f \circ g}\h = f_\h \circ g_\dg.
\]
\end{remark}

\begin{proposition}%
\label{thm:strict_morphisms}
Given two $\specialoperad$-coalgebras $V$ and $W$, the functor
$\cobarfunctorfull$ gives an isomorphism between
\begin{itemize}
	\item the set of morphisms of $\specialoperad$-coalgebras
	      $\hombracket{\catofcog{\specialoperad}}{V}{W}$;
	\item the subset of $\hombracket{\catofcompletealg{\coperad}}
	      {\cobarfunctor V}{\cobarfunctor  W}$
	      made up of strict ∞-morphisms, that is those $f$
	      such that $f_\h=0$.
\end{itemize}
\end{proposition}

\begin{proof}
First, notice that the function
\[
	\hombracket{\catofcog{\specialoperad}}{V}{W} \to 
	\hombracket{\catofcompletealg{\coperad}}
	{\cobarfunctor V}{\cobarfunctor  W}
\]
given by the $\cobarfunctorfull$ functor is injective%
~[\ref{thm:cobar_est_fidele}] and that it image lies
in the subset of strict ∞-morphisms. Now, let $f: V \inftymorphism W$
be a strict ∞-morphisms.
Then, as a morphism of graded complete $\coperad$-algebras $f =
f_\dg^\coperad$. The fact that $f$ commutes with the derivations
of $\cobarfunctor V$ and $\cobarfunctor W$ is encompassed
in the fact that $f_\dg$ is a chain map and in the commutation 
of the following diagram
\[
	\begin{tikzcd}
		V \arrow[r,"f_\dg"] \arrow[d,"a_V"']
		& W \arrow[d,"a_W"]
		\\
		V^\specialoperad
		\arrow[d,"V^\canonicaltwisting"']
		& W^\specialoperad \arrow[d,"W^\canonicaltwisting"]
		\\
		V^{\sinv\overline{\coperad}}
		\arrow[r,"{f_\dg^{\sinv\overline{\coperad}}}"]
		& W^{\sinv\overline{\coperad}} .
	\end{tikzcd}
\]
The commutation of this square implies that $f_\dg$ commutes with the
$\specialoperad$-coalgebras structures on $V$ and $W$ since the operad
$\specialoperad$ is generated by $s^{-1}\overline{\coperad}$. Thus
$f_\dg$ is a morphism of differential graded $\specialoperad$-coalgebras.
Finally $f = \cobarfunctor (f_\dg)$.
\end{proof}

Besides, for any $\specialoperad$-coalgebra $V$, we can notice that there
exists canonical isomorphisms of chain complexes
\[
	V \isonat \gr 0 \cobarfunctor V,
\]
and for $n \geq 1$
\[
	\gr n \cobarfunctor V \isonat V^{\grfilt n \coperad}
	\isonat \fullpower{\gr 0 \cobarfunctor V}{\grfilt n \coperad}.
\]
Then, for any morphism of complete $\coperad$-algebras $f: \cobarfunctor
V \to \cobarfunctor W$ the graded map associated to $f$ is
\[
	\gr n f  = \fullpower{\gr 0 f}{\grfilt n \coperad}
	 = \fullpower{f_\dg}{\grfilt n \coperad}.
\]
As a consequence, the functor $\gr 0$ is conservative, which is 
rephrased in the following proposition.

\begin{proposition}%
\label{thm:dg_iso_implies_infinity_iso}
An ∞-morphism $f : V \inftymorphism W$ is an isomorphism of
$\coperad$-algebras if and only if it is an ∞-isomorphism.
\end{proposition}

\begin{theorem}%
\label{thm:weq_for_cofibrant_algebras}
An ∞-morphism $f : V \inftymorphism W$
is a weak equivalence of complete $\coperad$-algebras
if and only it is an ∞-quasi-isomorphism.
\end{theorem}

\begin{proof}
First, suppose that $f_\dg $ is a quasi-isomorphism. Then for any
natural $n$, the morphism of chain complexes $\gr n f = \power{(f_\dg
f)}{\grfilt n \coperad}$ is also a quasi-isomorphism. Hence
$f$ is a dévissage equivalence and so this is a weak equivalence%
~[\ref{thm:equivalence_de_devissage=>eq}].

Conversely, suppose that $f$ is a weak equivalence. This means that
the morphism of $\specialoperad$-coalgebras $\cobaradj \cobarfunctor V \to
\cobaradj \cobarfunctor W$ is a quasi-isomorphism. Besides, we know
since%
~\sectionref{sec:the_operad_BQ_and_the_cobar_resolution} that the
composite map
\[
	\cobaradj \cobarfunctor V
	\xrightarrow{\canonicalproj^\specialoperad(V^\coperad)}
	V^\coperad 
	\to \cofiltration 0 V^\coperad \isonat V,
\]
is actually a chain map and as
a left inverse of the counit morphism $V \to \cobaradj \cobarfunctor
V$, it is a quasi-isomorphism.
Then, consider the following commutative diagram in the category
$\catss^\integers$.
\[
	\begin{tikzcd}[ampersand replacement=\&] 
		\cobaradj \cobarfunctor V \arrow[d] \arrow[r,"\cobaradj f "]
		\& \cobaradj \cobarfunctor W \arrow[d] 
		\\
		\cobarfunctor V \arrow[r,"f"'] \arrow[d]
		\& \cobarfunctor W \arrow[d]
		\\
		V \arrow[r,"\gr 0 f"']
		\& W.
	\end{tikzcd}
\]
The vertical composite maps are chain maps and quasi-isomorphisms.
This is also the case of $\cobaradj f$. So $\gr 0 f = f_\dg$ is a
quasi-isomorphism.
\end{proof}

\begin{theorem}
An ∞-morphism $f : V \inftymorphism W$, that is
a morphism of $\coperad$\=/algebras from $\cobarfunctor V$
to $\cobarfunctor W$, is a
fibration
if and only it is an ∞\=/epimorphism.
\end{theorem}

\begin{proof}
In the standard model structure, the fibrations of complete
$\coperad$-algebras coincide with the degree-wise epimorphisms%
~[\ref{thm:toutes_les_algebres_completes_sont_fibrantes},%
~\ref{thm:sane_implies_fibrations_are_epis}].
Suppose that $f$ is a degree-wise epimorphism, then using the diagram
\[
	\begin{tikzcd}[ampersand replacement=\&] 
		\cobarfunctor V \arrow[r,"f"] \arrow[d]
		\& \cobarfunctor W \arrow[d]
		\\
		V \arrow[r,"\gr 0 f"']
		\& W.
	\end{tikzcd}
\]
one sees that $\gr 0 f = f_\dg$ becomes also a degree-wise epimorphism.
So $f$ is an ∞-epimorphism.

Conversely, suppose that $f$ is an ∞-epimorphism.
For any natural $n$, the equality
\[
	\gr {n} f = \fullpower{f_\dg}{\grfilt n \coperad} 
\] 
shows that $\gr {n} f$ is also a degree-wise epimorphism.
Thus $f$ itself is also a degree-wise epimorphism%
~[\ref{thm:epi-cofiltration}].
\end{proof}

\begin{lemma}[(Perturbation)]%
\label{thm:rectification_of_cofibrations_and_fibrations}
Let $f:  (V,a) \inftymorphism  (W,b)$ be an ∞\=/morphism. Then
$f$ can be rewritten as the composite of a strict ∞-monomorphism
followed by an ∞-isotopy:
\[
	\begin{tikzcd}
		\cobarfunctor V \ar[rr, "f"]
		\ar[rd, hook, "{g=\cobarfunctor (f_\dg)}"']& &
		\cobarfunctor (W,b)
		\\
		& \cobarfunctor (W,b') \ar[ur, "t"']
		\ar[ur, phantom, shift left=1.5, "\rotatebox{30}{$\iso$}"] &
	\end{tikzcd}
\]
\end{lemma}

\begin{proof}
Since the map $f_\dg : V \to W$ is a degree-wise monomorphism, it has
a left inverse $f_\dg^{-1}$ in the category $\catss^\integers$. Let
$t$ be the endomorphism of the graded complete $\coperad$-algebra
given by
\[
t_\dg \coloneqq \id W \qand
\begin{tikzcd}
	t_\h : 
	W \arrow[r,"f_\dg^{-1}"]
	& V \arrow[r,"f_\h"]
	& W^{\overline\coperad}
\end{tikzcd}
\]
Since $\gr 0 t = \id{W}$, $t$ is an isomorphism%
~[\ref{thm:dg_iso_implies_infinity_iso}]. We can thus
transfer the derivation of $\cobarfunctor W$ along $t$ to obtain a new
derivation $d'$ on $W^\coperad$
\[
	W^\coperad \xrightarrow{t} W^\coperad
	\xrightarrow{d_{\cobarfunctor W}} W^\coperad
	\xrightarrow{t^{-1}} W^\coperad.
\]
Using the
Rosetta stone%
~[\ref{thm:Rosetta_stone}], this defines another structure of
a $\specialoperad$\=/coalgebra on $W$. We denote by $b'$ this new
structure. Moreover, the maps $t$ and $t^{-1}$ are $\infty$-isotopies
between $(W,b)$ and $(W,b')$. Let $g: (V,a)
\inftymorphism (W,b')$ be the $\infty$-morphism
which is the composition of $f$ with $t^{-1}$
\[
	g = t^{-1} \circ f .
\]
The equality of morphisms of graded complete
$\coperad$-algebras from $V^\coperad$ to $W^\coperad$
\[
	f = t \circ \fullpower{\gr 0 f}{\coperad}~%
	[\ref{rmk:composition_of_infinity_arrows}] 
\]
implies that $g$ equals $f_\dg^\coperad$ and so
is a strict $\infty$-morphism.
\end{proof}

\begin{theorem}
An $\infty$-morphism $f : V \inftymorphism W$, that is
a morphism of $\coperad$-algebras from $\cobarfunctor V$
to $\cobarfunctor W$, is a
cofibration
if and only it is an ∞\=/monomorphism.
\end{theorem}

\begin{proof}
Suppose first that $f$ is a cofibration.
Consider the following
factorisation
\[
	\begin{tikzcd}
		V
		\arrow[r, hookrightarrow,"c"]
		& X
		\arrow[r]
		\arrow[r,phantom,"\sim", shift left]
		& \cobaradj 0
	\end{tikzcd}
\]
in the category of $\specialoperad$-coalgebras, given by a cofibration
followed by a weak equivalence. Then, since the adjunction
$\cobarfunctor \dashv \cobaradj$ is a model equivalence%
~[\ref{thm:equivalence_for_cofibrant}] and since
the functor $\cobarfunctor$ preserves weak equivalences%
~[\ref{thm:weq_for_cofibrant_algebras}], the composite
morphism
\[
	\cobarfunctor X \to  \cobarfunctor \cobaradj 0  \to 0,
\]
is a weak equivalence. As it is also degree-wise surjective, it is an
acyclic fibration. Moreover, it fits in the following square diagram
\[
	\begin{tikzcd}
		\cobarfunctor V
		\arrow[d, "f"']
		\arrow[r, hook,"{\cobarfunctor c}"]
		& \cobarfunctor X
		\arrow[d]
		\\
		\cobarfunctor W \arrow[r]
		& { 0} .
	\end{tikzcd}
\]
This square has a lifting $g : \cobarfunctor W \to \cobarfunctor X$.
Since $\gr 0 c = \gr 0 g \circ \gr 0 f$ is a degreewise monomorphism,
then $\gr 0 f$ is a degreewise monomorphism, that is,
$f$ is an $\infty$-monomorphism.
Conversely,
suppose that $f$ is an ∞\=/monomorphism. By
perturbation%
~[\ref{thm:rectification_of_cofibrations_and_fibrations}], there exists
a strict ∞\=/monomorphism 
$g: (V,a) \inftymorphism (W,b')$ and an $\infty$-isotopy
$t :  (W,b') \inftymorphism (W,b)$ such that
\[
	f = t \circ \cobarfunctor g.
\]
Since $g$ is the image through the Cobar functor of a cofibration it is
a cofibration. Since $t$ is an isomorphism, $f$ is also a cofibration.
\end{proof}

\section*{Acknowledgements}
%-----------------------------------------------------------------------
The authors would like to thank both Utrecht University and the IBS
Center for Geometry and Physics which provided financial support for
this long term research project. In particular they allowed the authors
to meet regularly either in Korea or in the Netherlands.
The first author was supported by
the {NWO}; the second author was supported by IBS-R$003$-D$1$.

They would also like to thank Gabriel Drummond-Cole for encouraging
comments, showing them nice \LaTeX{} tricks and
for sharing thoughts on model categories of coalgebras, Bruno Vallette
for his careful attention to details and
Geoffroy Horel for suggesting applications to rational homotopy theory.
The first author would like to thank Victor Roca Lucio for sharing his thoughts about
the first version of this paper. The second author would
like to thank Sori Lee for helpful conversations, Daniel
Robert-Nicoud for exchanging thoughts on complete algebras, and
Simon Henry
for introducing him to the theory of modules over non-unital rings.

%%%%%%%

%====[ Bibliographie ]==================================================
\bibliography{dl}

\begin{thebibliography}{10}

\bibitem{doi:10.2307/1993609}
James Stasheff.
\newblock Homotopy associativity of {H}-spaces. {II}.
\newblock {\em Transactions of the American Mathematical Society}, 108(2):293,
  August 1963.

\bibitem{doi:10.1007/bfb0097438}
Daniel Quillen.
\newblock {\em Homotopical Algebra}.
\newblock Springer Berlin Heidelberg, 1967.

\bibitem{doi:10.1090/amsip/046.1}
Kenji Fukaya, Yong-Geun Oh, Hiroshi Ohta, and Kaoru Ono.
\newblock {\em Lagrangian Intersection Floer Theory}.
\newblock American Mathematical Society, June 2010.

\bibitem{doi:10.1007/s00220-010-0987-x}
Sergei Merkulov.
\newblock Wheeled pro(p)file of {B}atalin-{V}ilkovisky formalism.
\newblock {\em Communications in Mathematical Physics}, 295(3):585--638, May
  2010.

\bibitem{arXiv:1409.4688}
Mathieu Anel.
\newblock Cofree coalgebras over operads and representative functions,
  September 2014.

\bibitem{doi:10.1090/conm/641/12859}
Marzieh Bayeh, Kathryn Hess, Varvara Karpova, Magdalena Kedziorek, Emily Riehl,
  and Brooke Shipley.
\newblock Left-induced model structures and diagram categories.
\newblock {\em Contemporary Mathematics}, 641:49--81, 2015.

\bibitem{doi:10.1080/00927879708826055}
Vladimir Hinich.
\newblock Homological algebra of homotopy algebras.
\newblock {\em Communications in Algebra}, 25(10):3291--3323, 1997.

\bibitem{arXiv:1411.5533}
Bruno Vallette.
\newblock Homotopy theory of homotopy algebras, November 2014.

\bibitem{arXiv:0310337}
Kenji Lef{\`e}vre-Hasegawa.
\newblock {\em Sur les A-infini catégories}.
\newblock PhD thesis, universit{\'e} Paris 7, October 2003.

\bibitem{arXiv:1612.02254}
Brice {Le Grignou}.
\newblock Homotopy theory of unital algebras.
\newblock {\em Algebraic and Geometric Topology}, 19:1541--1618, 2019.

\bibitem{mathoverflow:148626}
Mauro Porta.
\newblock Model structure for cooperads and for coalgebras.
\newblock MathOverflow.

\bibitem{doi:10.1007/978-3-642-30362-3}
Jean-Louis Loday and Bruno Vallette.
\newblock {\em Algebraic Operads}, volume 346 of {\em Grundlehren der
  mathematischen Wissenschften}.
\newblock Springer Berlin Heidelberg, Berlin, Heidelberg, 2012.

\bibitem{doi:10.1090/conm/504/09879}
Benoit Fresse.
\newblock Operadic cobar constructions, cylinder objects and homotopy morphisms
  of algebras over operads.
\newblock In Christian Ausoni, Kathryn Hess, and J\'er\^ome Scherer, editors,
  {\em Proceedings of the 3rd {C}onference on {A}lgebraic {T}opology held in
  {A}rolla, {A}ugust 18--24, 2008}, volume 504 of {\em Contemporary
  Mathematics}, pages 129--188. American Mathematical Society, Providence, RI,
  2009.

\bibitem{doi:10.1007/s002220100197}
Amnon Neeman.
\newblock A counterexample to a 1961 "theorem" in homological algebra.
\newblock {\em Inventiones Mathematicae}, 148(2):397--420, May 2002.

\bibitem{doi:10.1112/S0024610705022416}
Jan-Erik Roos.
\newblock Derived functors of inverse limits revisited.
\newblock {\em Journal of the London Mathematical Society}, 73(1):65--83,
  February 2006.

\bibitem{doi:10.1007/978-3-0348-0052-5}
Ieke Moerdĳk and Bertrand Toën.
\newblock {\em Simplicial Methods for Operads and Algebraic Geometry}.
\newblock Springer Basel, 2010.

\bibitem{doi:10.1016/s0166-8641(03)00037-3}
Justin Smith.
\newblock Cofree coalgebras over operads.
\newblock {\em Topology and its Applications}, 133(2):105--138, September 2003.

\bibitem{doi:10.1016/0022-4049(93)90038-u}
Thomas Fox.
\newblock The construction of cofree coalgebras.
\newblock {\em Journal of Pure and Applied Algebra}, 84(2):191--198, February
  1993.

\bibitem{doi:10.1017/CBO9780511600579}
{Ji{\v r}\'\i} {Ad\'amek} and {Ji{\v r}\'\i} {Rosick\'y}.
\newblock {\em Locally Presentable and Accessible Categories}, volume 189 of
  {\em London Mathematical Society Lecture Notes Series}.
\newblock Cambridge University Press, Cambridge, 1994.

\bibitem{doi:10.4310/HHA.2014.v16.n2.a9}
Michael Ching and Emily Riehl.
\newblock Coalgebraic models for combinatorial model categories.
\newblock {\em Homology, Homotopy and Applications}, 16(2):171--184, 2014.

\bibitem{doi:10.1515/CRELLE.2008.051}
Bruno Vallette.
\newblock Manin products, {K}oszul duality, {L}oday algebras and {D}eligne
  conjecture.
\newblock {\em Journal f{\"u}r die reine und angewandte Mathematik (Crelles
  Journal)}, 2008(620):105--164, 2008.

\bibitem{MR:125148}
Samuel Eilenberg.
\newblock Abstract description of some basic functors.
\newblock {\em Journal of the Indian Mathematical Society}, 24(1-2):231--234,
  1960.

\bibitem{doi:10.1090/s0002-9939-1960-0118757-0}
Charles Watts.
\newblock Intrinsic characterizations of some additive functors.
\newblock {\em Proceedings of the American Mathematical Society}, 11(1):5--5,
  January 1960.

\bibitem{doi:10.1016/j.jpaa.2010.06.024}
Mark Hovey.
\newblock Additive closed symmetric monoidal structures on {R}-modules.
\newblock {\em Journal of Pure and Applied Algebra}, 215(5):789--805, May 2011.

\bibitem{zbMATH:01203017}
Françoise Grandjean and Enrico Vitale.
\newblock {Morita equivalence for regular algebras}.
\newblock {\em {Cahiers de Topologie et G\'eom\'etrie Diff\'erentielle
  Cat\'egoriques}}, 39(2):137--153, 1998.

\bibitem{doi:10.1215/s0012-7094-94-07608-4}
Victor Ginzburg and Mikhail Kapranov.
\newblock Koszul duality for operads.
\newblock {\em Duke Mathematical Journal}, 76(1):203--272, oct 1994.

\bibitem{arXiv:hep-th/9403055}
Ezra Getzler and J.~D.~S. Jones.
\newblock Operads, homotopy algebra and iterated integrals for double loop
  spaces, 1994.

\bibitem{doi:10.1007/s00208-011-0766-9}
Joseph Hirsh and Joan Millès.
\newblock Curved {K}oszul duality theory.
\newblock {\em Mathematische Annalen}, 354(4):1465--1520, January 2012.

\bibitem{arXiv:1707.03465}
Brice {Le Grignou}.
\newblock Algebraic operads up to homotopy, December 2021.

\bibitem{doi:10.1007/s00014-003-0772-y}
Clemens Berger and Ieke Moerdĳk.
\newblock Axiomatic homotopy theory for operads.
\newblock {\em Commentarii Mathematici Helvetici}, 78(4):805--831, October
  2003.

\bibitem{goerss1999model}
Ezra Getzler and Paul Goerss.
\newblock A model category structure for differential graded coalgebras.
\newblock Preprint, 1999.

\bibitem{zbMATH:05934769}
Justin Smith.
\newblock Model-categories of coalgebras over operads.
\newblock {\em Theory and Applications of Categories}, 25:189--246, 2011.

\bibitem{doi:10.1017/s0305004114000437}
Sinan Yalin.
\newblock Simplicial localisation of homotopy algebras over a prop.
\newblock {\em Mathematical Proceedings of the Cambridge Philosophical
  Society}, 157(3):457--468, October 2014.

\bibitem{zbMATH:06572171}
Vladimir {Hinich}.
\newblock Rectification of algebras and modules.
\newblock {\em Documenta Mathematica}, 20:879--926, 2015.

\bibitem{doi:10.1016/j.jpaa.2014.01.005}
Michael Makkai and {Ji{\v r}\'\i} Rosick{\'{y}}.
\newblock Cellular categories.
\newblock {\em Journal of Pure and Applied Algebra}, 218(9):1652--1664,
  September 2014.

\bibitem{doi:10.1112/topo.12011}
Kathryn Hess, Magdalena Kedziorek, Emily Riehl, and Brooke Shipley.
\newblock A necessary and sufficient condition for induced model structures.
\newblock {\em Journal of Topology}, 10(2):324--369, April 2017.

\bibitem{doi:10.1017/S0305004103007138}
Clemens Berger and Benoit Fresse.
\newblock Combinatorial operad actions on cochains.
\newblock {\em Mathematical Proceedings of the Cambridge Philosophical
  Society}, 137(1):135–174, 2004.

\bibitem{doi:10.1016/0040-9383(74)90036-6}
Michael Barratt and Peter Eccles.
\newblock $\gamma^+$-structures{\textemdash}i: a free group functor for stable
  homotopy theory.
\newblock {\em Topology}, 13(1):25--45, March 1974.

\bibitem{doi:10.1007/s40062-018-0210-x}
Hermann Soré.
\newblock On a {Q}uillen adjunction between the categories of differential
  graded and simplicial coalgebras.
\newblock {\em Journal of Homotopy and Related Structures}, page~{}, June 2018.

\bibitem{doi:10.1016/S0022-4049(00)00121-3}
Vladimir Hinich.
\newblock {DG} coalgebras as formal stacks.
\newblock {\em Journal of Pure and Applied Algebra}, 162(2-3):209--250, aug
  2001.

\bibitem{doi:10.1090/proc/12823}
Gabriel Drummond-Cole and Joseph Hirsh.
\newblock Model structures for coalgebras.
\newblock {\em Proceedings of the American Mathematical Society},
  144(4):1467--1481, October 2015.

\end{thebibliography}

\tocconfig%

%====[ Adresses ]=======================================================

\addresses%

\end{document}